%% file: structure_2019_03_29.tex
\documentclass[12pt]{amsart}
\usepackage{amsmath,amsfonts,euscript,dsfont,amscd,amsthm,amssymb,upref,graphics,mathdots}
\usepackage[all]{xy}
\usepackage{tikz}
\usepackage[normalem]{ulem}
\usepackage{mathrsfs}
\newcommand{\Pscr}{\mathscr{P}}
\newcommand{\Qscr}{\mathscr{Q}}

\usepackage{color}

\theoremstyle{definition}

\theoremstyle{plain}

\newtheorem{corollaryspecial}{Corollary}

\swapnumbers

\newtheorem{theorem}{Theorem}[section]
\newtheorem{proposition}[theorem]{Proposition}
\newtheorem{lemma}[theorem]{Lemma}
\newtheorem{corollary}[theorem]{Corollary}

\newtheorem*{claim}{Claim}

\newtheorem{sub}{}[theorem] % This creates the counter "sub"

\newtheorem{sublemma}		[sub]{Lemma}

\theoremstyle{definition}

\newtheorem{definition}[theorem]{Definition}
\newtheorem{definitions}[theorem]{Definitions}
\newtheorem{parag}[theorem]{}
\newtheorem{example}[theorem]{Example}

\newtheorem{notation}[theorem]{Notation}
\newtheorem{notations}[theorem]{Notations}

\newtheorem{remark}[theorem]{Remark}
\newtheorem{remarks}[theorem]{Remarks}

\theoremstyle{remark}

\newtheorem*{smallremark}{Remark}
{\begin{enumerate}\setlength{\itemsep}{#1}}{\end{enumerate}}
\newenvironment{enumerata}%
{\begin{enumerate}

}{\end{enumerate}}
{\begin{enumerata}\setlength{\itemsep}{#1}}{\end{enumerata}}
\newcommand{\Nd}{	\operatorname{{\rm Nd}}}

\newcommand{\tD}{\tilde\Delta}
\newcommand{\bD}{\bar\Delta}

\newcommand{\In}{	\operatorname{{\rm In}}}

\newcommand{\qdic}{q_{\text{\rm dic}}}

\newcommand{\setspec}[2]{\big\{\,#1\, \mid \,#2\, \big\}}

\newcommand{\Integ}{\ensuremath{\mathbb{Z}}}
\newcommand{\Nat}{\ensuremath{\mathbb{N}}}
\newcommand{\Rat}{\ensuremath{\mathbb{Q}}}
\newcommand{\Comp}{\ensuremath{\mathbb{C}}}
\newcommand{\Reals}{\ensuremath{\mathbb{R}}}

\newcommand{\proj}{\ensuremath{\mathbb{P}}}

\newcommand{\Aeul}{\EuScript{A}}

\newcommand{\Ceul}{\EuScript{C}}
\newcommand{\Deul}{\EuScript{D}}
\newcommand{\Eeul}{\EuScript{E}}

\newcommand{\Neul}{\EuScript{N}}
\newcommand{\Oeul}{\EuScript{O}}

\newcommand{\Teul}{\EuScript{T}}

\newcommand{\Veul}{\EuScript{V}}

\renewcommand{\epsilon}{\varepsilon}
\renewcommand{\phi}{\varphi}
\renewcommand{\emptyset}{\varnothing}

\newcommand{\rien}[1]{}

\CompileMatrices

\begin{document}
\renewcommand{\baselinestretch}{1.07}

%%%%%%	TOPMATTER:   %%%%%%%%%%%%%%%%%%%%%%%%%

\title[Structure of Newton trees at infinity]{Structure of the Newton tree at infinity \\ of a polynomial in two variables}

\author{Pierrette Cassou-Nogu\`es}
\author{Daniel Daigle}
\address{Univ.\ Bordeaux, CNRS, Bordeaux INP, IMB, UMR 5251,  F-33400, Talence, France}
\email{Pierrette.Cassou-nogues@math.u-bordeaux.fr}

\address{Department of Mathematics and Statistics\\
University of Ottawa\\
Ottawa, Canada K1N 6N5}
\email{ddaigle@uottawa.ca}

\thanks{Research of the first author partially supported by  Spanish grants
MTM2016-76868-C2-2-P (directors Pedro Gonzalez Perez and Alejandro Melle),
MTM2016-76868-C2-1-P (directors Enrique Artal and Jose Ignacio Cogolludo),
and the group {\it Geometria, Topologia, Algebra y Cryptografia en singularidades y sus aplicaciones.}}
\thanks{Research of the second author supported by grant 04539/RGPIN/2015 from NSERC Canada.}

{\renewcommand{\thefootnote}{}
}

{\renewcommand{\thefootnote}{}
}

\begin{abstract}
Let $f : \Comp^2 \to \Comp$ be a primitive polynomial.
Extend $f$ to a morphism $\Phi : X \to \proj^1$ where $X$ is a nonsingular projective surface that contains $\Comp^2$ as
an open set and where $\mathcal{D} = X \setminus \Comp^2$ is an SNC-divisor of $X$. 
The dual graph of $\mathcal{D}$ is a tree. 
We analyse the structure and complexity of this tree in terms of the genus of the generic fiber of $f$.
\end{abstract}

\maketitle
  
\vfuzz=2pt

\section*{Introduction}

This article studies a type of combinatorial object called
an {\it abstract Newton tree at infinity}.
It develops a graph theory whose results are valid for
all abstract Newton trees at infinity that are {\it minimally complete}
(these notions are defined in Sec.\ \ref{Sec:PreliminariesAbstractNewtontreesatinfinity}).
The motivation for investigating these trees comes from the study of polynomial maps $\Comp^2 \to \Comp$.
The motivation is discussed in this introduction, and the rest of the article
is purely combinatorial:
polynomials are almost never mentioned outside of this introduction
and of paragraphs \ref{918235071yrsj2dhry}--\ref{v87bnd5hmzvfdjsm6k7oxdg}.

Let $f \in \Comp[x,y]$ be a primitive polynomial and let us also view $f$ as a map $f:\Comp^2 \to \Comp$.
One calls $f$ a \textit{rational polynomial} if its generic fiber is a rational curve.
The study of rational polynomials has a long history,
but the problem of classifying them is still very much open.
So far, only three families of rational polynomials have been described:
(i)~simple rational polynomials (i.e., rational polynomials all of whose dicriticals have degree 1)
by Miyanishi and Sugie \cite{MiySugie:GenRatPolys}, 
then by Neumann and Norbury \cite{NeumannNorbury:simple} (who corrected an error of \cite{MiySugie:GenRatPolys}),
and finally by the authors \cite{CND15:simples} (who corrected an error of \cite{NeumannNorbury:simple});
(ii)~rational polynomials with a $\Comp^*$-fiber, by Kaliman \cite{Kaliman:CstarFiber96};
(iii)~quasi-simple rational polynomials (rational polynomials with one dicritical of degree $2$ and all others of degree $1$),
by Sasao \cite{Sasao_QuasiSimple2006}. 

The results of  \cite{MiySugie:GenRatPolys},  \cite{Kaliman:CstarFiber96} and \cite{Sasao_QuasiSimple2006} give equations for $f$.
The fact that Sasao's equations are already quite complicated suggests that, for more general hypotheses,
giving equations should be very difficult and perhaps unuseful.
Neumann and Norbury work with the splice diagram of $f$; they first give these trees and then deduce equations using ``educated guesses''.

It seems more realistic to study classifications of polynomials using their trees.  It is what we do in this article.
To each primitive polynomial $f \in \Comp[x,y]$, we associate a minimally complete abstract Newton tree at infinity $\Teul(f;x,y)$
(these are a variant of the splice diagrams used in \cite{NeumannNorbury:simple});
Rem.\ \ref{pc09v4rdh8drcj7yu} recalls how this association is done.
Since the results of this paper apply to all minimally complete abstract Newton trees at infinity,
they apply in particular to all trees $\Teul(f;x,y)$ where $f \in \Comp[x,y]$ is a primitive polynomial.
The genus $g$ of the generic fiber of $f$ can be computed from the information contained in $\Teul(f;x,y)$.
We decompose $\Teul(f;x,y)$ in shapes that we call ``combs'' and we show that the number of combs
is $\le 2g+1$. In particular, for rational polynomials the tree has a \textit{one comb} shape.
Our aim in a subsequent article is to describe birational maps that construct combs, which should show that all polynomials
can be deduced from simple ones and standard birational maps. This is a way to get equations without writing them.

It is well known that each primitive polynomial  $f \in \Comp[x,y]$ determines a finite collection of ``dicritical curves,'' 
that there is a notion of ``degree of a dicritical curve,''
and that the gcd of the degrees of all dicritical curves of $f$ is equal to $1$.
If $\Teul$ is an abstract Newton tree at infinity then 
one defines a notion of ``dicritical vertex'' of $\Teul$, and of ``degree of a dicritical vertex'' (Def.\ \ref{823yi867jcn73}).
In the special case where $\Teul = \Teul(f;x,y)$ with $f$ primitive,
there is a degree-preserving bijection between the set of dicritical vertices of $\Teul$ and 
the set of dicritical curves of $f$.
So the tree $\Teul(f;x,y)$, in addition to being minimally complete, also satisfies the condition that 
the gcd of the degrees of all dicritical vertices is equal to $1$.
Note that our theory is valid for all  minimally complete abstract Newton trees at infinity, without assuming that they
satisfy this gcd condition;
Cor.\ \ref{pc0wbyrjo79e8rnn9}, Prop.\ \ref{c0viwjytsdDJjLlxdifFebg} and Section~\ref{Section:RationalCase} 
are the only places where the gcd condition is assumed.

{\setlength{\unitlength}{1mm}
The reader unfamiliar with  abstract Newton trees at infinity should look at Ex.\ \ref{nv63jfy64nvy3}
and take note that these trees consist of vertices, arrows, edges, and decorations.
Dicritical vertices are represented by 
``\begin{picture}(2,1)(-1,-.5) \put(0,0){\circle*{1}} \end{picture}'' 
and vertices which are not dicritical are represented by ``\begin{picture}(2,1)(-1,-.5) \put(0,0){\circle{1}} \end{picture}''.
Each arrow is represented by an arrowhead
``\begin{picture}(2,1)(-1,-.5) \put(0.5,0){\vector(1,0){0}} \end{picture}''
(so ``\begin{picture}(5.5,1)(-.5,-.5) \put(0,0){\circle*{1}} \put(0.5,0){\vector(1,0){4.5}} \end{picture}''
represents an edge joining a vertex 
``\begin{picture}(2,1)(-1,-.5) \put(0,0){\circle*{1}} \end{picture}''
to an arrow
``\begin{picture}(2,1)(-1,-.5) \put(0.5,0){\vector(1,0){0}} \end{picture}'').
The edges and arrows are decorated.
A more complicated example is shown in Fig.~\ref{723dfvcjp2q98ewdywe},
but that one uses the abbreviation \begin{picture}(9.6,1)(-.9,-.5) \put(0,0){\circle*{1}} 
\put(0.5,0){\line(1,0){4.5}} \put(4.8,0){\makebox(0,0)[l]{\footnotesize $<$}} 
\end{picture}
defined in \ref{pd0f123wdjwe}.
}

Given a minimally complete abstract Newton tree at infinity $\Teul$,
let $\Neul$ be the subtree of $\Teul$ obtained by deleting all dicritical vertices and arrows of $\Teul$, and all edges incident to a
dicritical vertex or to an arrow. 
For each vertex $v$ of $\Neul$, a number $\tD(v) \in \Integ$ is defined in  \ref{pc09vg349fbv8682};
we write $\tD( \Neul )$ as an abbreviation for $\sum_{v \in \Neul} \tD(v)$.
Thus every minimally complete abstract Newton tree at infinity $\Teul$ determines an integer $\tD( \Neul )$.
This integer is important because in the special case where $\Teul = \Teul(f;x,y)$ with $f \in \Comp[x,y]$ primitive,
we have (by Rem.\ \ref{v87bnd5hmzvfdjsm6k7oxdg})
\begin{equation}  \label {ckjvbo293djbf0q2}
\tD( \Neul ) = 2g
\end{equation}
where $g$ is the genus of the generic fiber of $f$;
in particular, $f$ is a rational polynomial if and only if $\tD( \Neul ) = 0$.
Thanks to Eq.\ \eqref{ckjvbo293djbf0q2}, we may now forget polynomials and turn ourselves into graph theorists.
What we accomplish in this article can be summarized as follows:
{\it Given a minimally complete abstract Newton tree at infinity $\Teul$,}
\begin{enumerate}

\item[(i)] {\it we show that the subtree $\Neul$ of $\Teul$ can be decomposed into combs;}
\item[(ii)] {\it we explain how the structure of the decomposition {\rm (i)} is related to the number $\tD( \Neul )$;}
\item[(iii)] {\it for small values of  $\tD( \Neul )$, we derive a more detailed description of $\Teul$.}

\end{enumerate}

For the purpose of this introduction, it suffices to say that 
a \textit{comb} is a subtree of $\Neul$ of the type depicted in Fig.\  \ref{902br8hrjkderydyukhyg},
where the vertical or oblique lines are called \textit{teeth}.\footnote{These definitions of ``comb'' and ``tooth'' are oversimplified.
The correct definitions are given in Sec.\ \ref{SectionCombs} and involve delicate arithmetical conditions that cannot be stated here.}
There can be several teeth attached to $z_n$, but at most one tooth is attached to each $z_i$ with $1\le i <n$.
\begin{figure}[htb]
$$
\scalebox{.8}{
\setlength{\unitlength}{1mm}
\begin{picture}(116,29)(-1,-26)
\put(0,0){\circle{1}}
\put(20,0){\circle{1}}
\put(30,0){\circle{1}}
\put(40,0){\circle{1}}
\put(60,0){\circle{1}}
\put(70,0){\circle{1}}
\put(80,0){\circle{1}}
\put(100,0){\circle{1}}
\put(110,0){\circle{1}}
\put(.5,0){\line(1,0){6}}
\put(40.5,0){\line(1,0){6}}
\put(80.5,0){\line(1,0){6}}
\put(19.5,0){\line(-1,0){6}}
\put(59.5,0){\line(-1,0){6}}
\put(99.5,0){\line(-1,0){6}}
\put(10,0){\makebox(0,0){\dots}}
\put(50,0){\makebox(0,0){\dots}}
\put(90,0){\makebox(0,0){\dots}}
\put(20.5,0){\line(1,0){9}}
\put(30.5,0){\line(1,0){9}}
\put(60.5,0){\line(1,0){9}}
\put(70.5,0){\line(1,0){9}}
\put(100.5,0){\line(1,0){9}}
\put(30,-.5){\line(0,-1){6}} \put(30,-9){\makebox(0,0){$\vdots$}} \put(30,-19.5){\line(0,1){6}} \put(30,-20){\circle{1}}
\put(70,-.5){\line(0,-1){6}} \put(70,-9){\makebox(0,0){$\vdots$}} \put(70,-19.5){\line(0,1){6}} \put(70,-20){\circle{1}}
\put(109.8787,-.4851){\line(-1,-4){1.5}} \put(105.12,-19.51){\line(1,4){1.5}}
\put(110.1213,-.4851){\line(1,-4){1.5}} \put(114.88,-19.51){\line(-1,4){1.5}}
\multiput(107.2575,-10.97)(.2425,.97){3}{\makebox(0,0){$\cdot$}}
\multiput(112.62,-10.485)(-.2425,.97){3}{\makebox(0,0){$\cdot$}}
\put(105,-20){\circle{1}}
\put(115,-20){\circle{1}}
\put(0,1.5){\makebox(0,0)[b]{\tiny $z_1$}}
\put(30,1.5){\makebox(0,0)[b]{\tiny $z_{i_1}$}}
\put(70,1.5){\makebox(0,0)[b]{\tiny $z_{i_k}$}}
\put(110,1.5){\makebox(0,0)[b]{\tiny $z_n$}}
\put(110,-21){\makebox(0,0)[t]{$\underbrace{\rule{11mm}{0mm}}_{t(z_n)}$}}
\end{picture} \qquad
\raisebox{20mm}{\begin{minipage}{4cm}
\footnotesize
$n\ge1$, $k\ge0$,  \\
$1 \le i_1 < \cdots < i_k < n$, \\
$t(z_n) \ge 0$.
\end{minipage}}}
$$
\caption{A comb.}
\label {902br8hrjkderydyukhyg}
\end{figure}

Again, let $\Teul$ be a minimally complete abstract Newton tree at infinity and let $\Neul$ be the subtree
of $\Teul$ defined as before.
The decomposition mentioned in (i) implies that $\Neul$ is a tree of combs,
where the $z_n$ of one comb may be adjacent to the $z_1$ of another comb (or to the $z_1$s of several other combs).
Cor.\ \ref{pTo9v2q3hZYa9r1ge0cX3rg}(a) states that
\begin{equation} \label {pv09b34587r}
\text{the total number of combs in the decomposition is $\le 1+\max(0,\tD(\Neul))$;}
\end{equation}
this is in fact the simplest of our results that would go in item (ii), in the above list.
Other results show that the value of $\tD(\Neul)$ imposes conditions on the internal properties of the combs and on how the different combs 
are organized together to form $\Neul$.
For instance, Fig.\ \ref{d9be3yie8dfdokj6rdosefd} (located near \ref{0vbb359gb}) shows all cases where the number of combs is between $2$ and $5$.
In each row of the table, the tree is a simplified representation of $\Neul$ where each vertex represents a comb,
and the number $H$ satisfies $H \le \tD(\Neul)$.

\rien{
Recall that $\Neul$ is obtained from $\Teul$ by deleting all dicritical vertices, all arrows and some edges.
So if $v$ is a vertex of $\Neul$ then there may exist dicritical vertices $u_1, \dots, u_s$ adjacent to $v$ (where the $u_i$ are in $\Teul$ but not in $\Neul$);
if $s\neq 0$, and if $d_j$ is the degree of the dicritical $u_j$, we say that $v$ is a node of type $[d_1,\dots,d_s]$;
if $s=0$, we say that $v$ is not a node.
We have $\Nd^*(\Teul) \subseteq \Nd(\Teul) \subseteq \Neul$,
where $\Nd(\Teul)$ is the set of all nodes and 
$\Nd^*(\Teul)$ is the set of all nodes whose type $[d_1,\dots,d_s]$ satisfies $\gcd(d_1,\dots,d_s)=1$.
Observe that if $u$ is a dicritical vertex of degree $1$, then $u$ is adjacent to some element of $\Nd^*(\Teul)$;
so $\Nd^*(\Teul)$ gives information on the  possible locations of dicritical vertices of degree $1$ in $\Teul$.
This is relevant in view of the notions of simple and quasi-simple rational polynomials,
and also in view of the dichotomy between good and bad field generators (in characteristic zero, a field generator is the same thing
as a rational polynomial; by the remark after 1.3 in \cite{Rus:fg}, a field generator is ``good'' if 
and only if at least one of its dicriticals has degree $1$).
}

In order to be able to state some results, let us now introduce some ideas and notations.
Recall that $\Neul$ is obtained from $\Teul$ by deleting all dicritical vertices, all arrows and some edges.
So if $v$ is a vertex of $\Neul$ then there may exist dicritical vertices $u_1, \dots, u_s$ adjacent to $v$ (where the $u_i$ are in $\Teul$ but not in $\Neul$);
if $s\neq 0$, and if the degrees of the dicriticals $u_1,\dots,u_s$ are denoted $d_1\le \cdots \le d_s$,
we say that $v$ is a node of type $[d_1,\dots,d_s]$;
if $s=0$, we say that $v$ is not a node.
By Lemma \ref{90q932r8dhd89cnr9}\eqref{o82y387jw9e23},
if $v$ is a node of type  $[d_1,\dots,d_s]$ then each $d_j$ is a divisor of the integer $N_v$ defined in \ref{c9v39rf0eX9e4np8glr9t8}.
For any vertex $v$ of $\Neul$ we define 
$$
\epsilon'(v) = a_v^* + b_v + \text{ valency of $v$ in $\Neul$} \, ,
$$
where $a_{v}^* = 0$ (resp.\  $a_{v}^* = 1$) if no arrow
(resp.\ some arrow\footnote{Such an arrow would be in $\Teul$ but not in $\Neul$.})
decorated by $(0)$ is adjacent to $v$, and $b_{v} =$ number of dicritical vertices adjacent to $v$ and of degree $< N_{v}$.

We have $\Nd^*(\Teul) \subseteq \Nd(\Teul) \subseteq \Neul$,
where $\Nd(\Teul)$ is the set of all nodes and 
$\Nd^*(\Teul)$ is the set of all nodes whose type $[d_1,\dots,d_s]$ satisfies $\gcd(d_1,\dots,d_s)=1$.
Observe that if $u$ is a dicritical vertex of degree $1$, then $u$ is adjacent to some element of $\Nd^*(\Teul)$;
so $\Nd^*(\Teul)$ gives information on the  possible locations of dicritical vertices of degree $1$ in $\Teul$.
This is relevant in view of the notions of simple and quasi-simple rational polynomials,
and also in view of the dichotomy between good and bad field generators (in characteristic zero, a field generator is the same thing
as a rational polynomial; by the remark after 1.3 in \cite{Rus:fg}, a field generator is ``good'' if 
and only if at least one of its dicriticals has degree $1$).

\medskip

\noindent{\bf Some results for rational polynomials.}
Let us give some examples of new results for rational polynomials, obtained as corollaries of our theory.
Corollaries \ref{0v3y46y7iwoujbfv8e} and \ref{p09verfgukeikdseur} are immediate consequences of, respectively,
Thm \ref{c03hbr9dfgsjlxnzwesmchp} and Prop.~\ref{9988zhgGUjfdlkq4q965vsaqWd39r0}.

By \eqref{ckjvbo293djbf0q2}, rational polynomials satisfy $\tD(\Neul)=2g=0$,
so \eqref{pv09b34587r} implies that the decomposition of $\Neul$ into combs has exactly one comb.
This explains assertion (a) of:

\begin{corollaryspecial}  \label {0v3y46y7iwoujbfv8e}
Let $\Teul = \Teul(f;x,y)$ where $f \in \Comp[x,y]$ is a rational polynomial. 
Then either $\Teul$ is one of the trees of Ex.\ \ref{nv63jfy64nvy3} or the following hold.
\begin{enumerata}

\item $\Neul$ looks like the tree of Fig.\ \ref{902br8hrjkderydyukhyg} with $n>1$.

\item The root $v_0$ of $\Teul$ is one of $z_1, \dots, z_{n-1}$, say $v_0 = z_{i_0}$ with $1 \le i_0 < n$, and if $k\neq 0$ then $i_0 < i_1$.
Moreover, the valency of $v_0$ in $\Teul$ is at most $2$.

\item No tooth is attached to $z_1$ and $\epsilon'(z_1) \le 2$. 

\item We have $\epsilon'(z_n) \le 4$, so in particular $t(z_n) \in \{0,1,2,3\}$. 

\item We have $\epsilon'(z_i) \le 3$ for all $i$ such that $1<i<n$.

\item If $v$ is a vertex in a tooth and $v$ is distinct from the $z_i$ to which the tooth is attached,
then $v$ is a node and $\epsilon'(v) \le 2$.

\item We have $| \Nd^*(\Teul) | \le 2$ and one of the following holds:
\begin{itemize}

\item $\Neul = \{ z_1, \dots, z_n\}$, $\epsilon'(z_i) \le 2$ for all $i$, and $\Nd^*(\Teul) \subseteq \{z_1, z_n\}$.

\item Let $V(z_n)$ denote the set of all vertices  $v$  belonging to some tooth attached to $z_n$ and such that $v$ has valency $1$ in $\Neul$;
then $\Nd^*(\Teul) \subseteq V(z_n) \cup \{z_n\}$,
and if $| \Nd^*(\Teul) | = 2$ then $V(z_n) \subseteq \Nd^*(\Teul)$.

\end{itemize}

\end{enumerata}
\end{corollaryspecial}

For primitive polynomials, the topological structure of $\Neul$ can be arbitrarily complicated;
so it is remarkable that, for all rational polynomials, $\Neul$ is as simple as is claimed in (a).
We also stress that all assertions of Cor.\ \ref{0v3y46y7iwoujbfv8e} are new results, except for the claim that $v_0$ has valency at most $2$.

The next result is concerned with the class of rational polynomials whose trees satisfy $\Nd(\Teul) = \Nd^*(\Teul)$
(this is considerably larger than the class of simple rational polynomials).
To understand the statement, first recall that a polynomial $f \in \Comp[x,y]$ is called a \textit{variable} if $\Comp[x,y]=\Comp[f,g]$ for some $g$.
By \cite[Thm 4.5]{Rus:fg}, if a rational polynomial $f\in \Comp[x,y]$ is not a variable then one can choose
the pair $(x,y)$ so that $f$ has two points at infinity.

\begin{corollaryspecial} \label {p09verfgukeikdseur}
Let $f \in \Comp[x,y]$ be a rational polynomial which is not a variable,
assume that $(x,y)$ has been chosen so that $f$ has two points at infinity, and let $\Teul = \Teul(f;x,y)$.
If $\Nd(\Teul) = \Nd^*(\Teul)$ then either $\Teul$ is one of the trees of Ex.\ \ref{nv63jfy64nvy3} or the following hold,
where the notation for $\Neul$ is that of Fig.\ \ref{902br8hrjkderydyukhyg}.
\begin{enumerata}

\item $\Neul = \{ z_1, \dots, z_n\}$,  $n \in \{2,3\}$ and if $n=3$ then $v_0=z_2$.

\item $\epsilon'(z_i) \le 2$ and $a_{z_i}^*=0$ for all $i = 1, \dots, n$.

\item $\Nd(\Teul) = \{ z_1, z_n \}$ and each node $v$ has type $[d_1, \dots, d_s] = [1, N_v, \dots, N_v]$, where $s=1 \Leftrightarrow v = v_0$.

\end{enumerata}
\end{corollaryspecial}

\medskip

\noindent{\bf Strategy and organization of the text.}
Let $\Teul$ be a minimally complete abstract Newton tree at infinity.
To study the structure of $\Teul$, we work at three different levels:
the first is that of $\Teul$ itself,
the second is the subtree $\Neul$ of $\Teul$,
and the third is the subtree $S(\Teul)$ of $\Neul$, called the skeleton of $\Teul$,
obtained from $\Neul$ by deleting all teeth of all combs.\footnote{The skeleton can be arbitrarily complicated:
it can be shown that if $T$ is an arbitrary finite tree and $x$ is any vertex of $T$, then there exists a minimally complete
abstract Newton tree at infinity $\Teul$ whose skeleton $S(\Teul)$ is $T$ and whose root $v_0$ is $x$.}

Before deleting dicritical vertices and focussing our attention on $\Neul$, we carefully analyse
how dicritical vertices cluster around vertices called ``nodes'', and how these clusters contribute to the value of $\tD(\Neul)$.
This analysis of nodes is carried out in Sec.\ \ref{SEC:Nodes}.

The second level of our analysis (i.e., the study of the tree $\Neul$) is carried out
in Sections \ref{Section:Characteristicnumbers} and \ref{Sec:LocalstructureofNeul}.
Sec.\ \ref{Section:Characteristicnumbers} defines the set $P(\Teul)$ whose elements are
all pairs $(u,e)$ such that $u$ is a vertex of $\Neul$ and $e$ is an edge of $\Neul$ incident to $u$.
We define a partial order $\preceq$ on $P(\Teul)$ by declaring that $(u,e) \succ (u',e')$ if the path from $u$ to $u'$ traverses $e$ but not $e'$.
Then several definitions and proofs proceed by induction on the poset $(P(\Teul),\preceq)$.
For each $(u,e) \in P(\Teul)$, we define a characteristic number $c(u,e) \in \Rat_{>0} = \setspec{ x \in \Rat }{ x>0 }$. 
Properties of characteristic numbers are established in Sec.\ \ref{Section:Characteristicnumbers}.
Thm \ref{xncoo9qwdx9} appears to be a deep result about Newton trees.
It allows us to define the integer $M(u,e)>0$ for each $(u,e) \in P(\Teul)$ (see Notation \ref{c9v8mmxzKRknrscxe4ftdes}).

Using characteristic numbers, Sec.\ \ref{Sec:LocalstructureofNeul} develops a ``calculus'' for
computing $\tD$ of certain subsets of $\Neul$ (see in particular Thm \ref{P90werd23ewods0ci}).
These tools allow us to discover the unexpected fact (Lemma \ref{p0c9vin12q09wsc})
that three important functions $P(\Teul) \to \Rat$ are monotonic with respect to $\preceq$.
The rest of Sec.\ \ref{Sec:LocalstructureofNeul} looks at the relation between the constancy of these functions in a given region of $\Neul$
and the presence of certain substructures of $\Neul$ (teeth, combs, etc.) in that region.

Studying the skeleton $S(\Teul)$ is the third level of our analysis, and is the subject of
Sections \ref{Section:Globalstructurefirststeps} and \ref{Section:GlobalstructureofNeul}.
Sec.\ \ref{Section:Globalstructurefirststeps} deals with the case where $S(\Teul)$ is a single vertex
and Sec.\ \ref{Section:GlobalstructureofNeul} does the general case.
It is in Sec.\ \ref{Section:GlobalstructureofNeul} that we prove the existence of the decomposition into combs, and that
we describe how the structure of that decomposition is related to the number $\tD( \Neul )$.

Sec.\ \ref{Sec:Nd_etoile} gives information about the set $\Nd^*(\Teul)$ (whose relevance is briefly discussed before the
statement of Cor.~\ref{0v3y46y7iwoujbfv8e}).

Sec.\ \ref{Section:RationalCase} describes the class of rational trees,
which contains in particular $\Teul(f;x,y)$ for all rational polynomials $f \in \Comp[x,y]$.
We give a complete explicit description in the case $| S(\Teul) | = 1$,
and a partial description in the general case.
Cor.~\ref{9vbq34ritgf7ZAOdb9rv} and  Prop.~\ref{9988zhgGUjfdlkq4q965vsaqWd39r0} give information about the set $\Nd^*(\Teul)$.

Sec.\ \ref{Section:Genus1case} gives a partial description of $\Teul$ in the case $\tD( \Neul ) = 2$, and an even more partial description 
when $\tD( \Neul ) = 4$.
The description is complete when  $\tD( \Neul ) = 2$ and  $| S(\Teul) | = 1$.

\medskip

Almost everything in this article is new.
The only exceptions are:
Section \ref{Sec:PreliminariesAbstractNewtontreesatinfinity} (which recalls some definitions and results from \cite{CND15:simples}),
Notations \ref{Ppc09wberyrji76awe} and \ref{kjfoqwsdjkwef} (also from \cite{CND15:simples}),
and the geometric interpretation discussed at the end of Sec.\ \ref{SEC:Nodes}.

The theory of minimally complete abstract Newton trees at infinity is surprisingly rich and interesting,
and it is not easy to single out the ``main results'' of this work.
Theorems \ref{xncoo9qwdx9} and \ref{P90werd23ewods0ci} are powerful tools.
The fact that a decomposition into combs always exists is a consequence of Thm \ref{cvnv7nd6ykawsujwryf9v},
and is a new insight in the structure of Newton trees at infinity.
The decomposition into combs leads to Thm \ref{0d2h39fh29834jfp0dfgq} and its consequences,
which reveal the relation between the structure of $\Teul$ and the value of  $\tD(\Neul)$.

\section{Preliminaries: Abstract Newton trees at infinity}
\label {Sec:PreliminariesAbstractNewtontreesatinfinity}

The aim of this section is to recall the definition of an abstract Newton tree at infinity that is minimally complete.
We begin by some terminologies and notations for general graphs and trees.
CAUTION: our use of the word ``vertex'' differs from the standard usage of graph theory.

\begin{parag}
In this paper, a {\it graph\/} is a pair $X = (X_0,X_1)$ where $X_0$ and $X_1$ are finite sets and each element
of $X_1$ is a subset of $X_0$ of cardinality exactly $2$. The elements of $X_1$ are called the {\it edges},
and those of $X_0$ are called {\it $0$-dimensional cells\/} (we do NOT use the word ``vertex'' here!).
If $x \in X_0$ and $e \in X_1$ are such that $x \in e$, we say that the edge $e$ is {\it incident\/} to $x$.
If $x \in X_0$, the number of edges incident to $x$ is called the {\it valency\/} of $x$ and is denoted $\delta_{x}$.
If $x,y \in X_0$ are such that $\{x,y\}$ is an edge, we say that $x,y$ are {\it adjacent},
or that $x,y$ are \textit{linked by an edge}. If $e=\{x,y\}$ is an edge then $x,y$ are the {\it endpoints\/} of $e$.

A {\it path\/} in the graph $X=(X_0,X_1)$ is an ordered tuple $(x_0,\dots,x_n)$ of elements of $X_0$ satisfying
$n\ge0$ and the two conditions:
\begin{itemize}

\item if $n\ge1$ then for each $i \in \{0,\dots,n-1\}$ we have $\{ x_i, x_{i+1} \} \in X_1$;
\item if $n\ge2$ then the edges $\{x_0,x_1\}$, \dots, $\{x_{n-1},x_n\}$ are distinct.

\end{itemize}
Note that what we call a ``path'' is called a ``simple path'' in standard terminology of graph theory.
Given $x,y \in X_0$, a {\it path from $x$ to $y$} is a path $(x_0,\dots,x_n)$ satisfying 
$x_0=x$ and $x_n=y$.

Suppose that $\gamma = (x_0,\dots,x_n)$ is a path. 
We say that a $0$-dimensional cell $x \in X_0$ ``is in $\gamma$'' if $x \in \{ x_0,\dots,x_n \}$;
we say that an edge $e \in X_1$ ``is in $\gamma$'' if $e$ is one of the edges
$\{x_0,x_1\}$, \dots, $\{x_{n-1},x_n\}$.
If an edge $e$ is in a path $\gamma$, we also say that $\gamma$ \textit{traverses} $e$.

A graph $X = (X_0,X_1)$ is a {\it tree\/} if for every choice of $x,y \in X_0$ there exists a unique path from $x$ to $y$.
A {\it rooted tree\/} is a triple $\Teul=(X_0,X_1,v_0)$ such that $(X_0,X_1)$ is a tree and $v_0 \in X_0$;
then $v_0$ is called the {\it root\/} of $\Teul$ (we always denote the root by $v_0$).
\end{parag}

\begin{parag} \label {p0cb2398gfv821wid9bgd6}
Given a rooted tree $\Teul = (X_0,X_1,v_0)$, we define
$\Aeul = \setspec{ x \in X_0 }{ \delta_x=1 }\setminus\{v_0\}$
and
$\Veul = \setspec{ x \in X_0 }{ \delta_x>1 }\cup\{v_0\}$;
the elements of $\Aeul$ are called {\it arrows\/} and those of $\Veul$ are called {\it vertices\/}
(so the root $v_0$ is a vertex).
Observe that $X_0 = \Veul \cup \Aeul$ and $\Veul \cap \Aeul = \emptyset$,
that all arrows have valency $1$, and that all elements of $\Veul\setminus\{v_0\}$ have valency $\ge2$.
We define a partial order on the set $X_0 = \Veul \cup \Aeul$ by stipulating that, given distinct $x,y \in X_0$, 
$$
x < y \iff \text{$x$ is on the path from the root $v_0$ to $y$.}
$$
It follows in particular that $v_0 < y$ for all $y \in X_0 \setminus \{ v_0 \}$.
\end{parag}

\begin{remark} \label {jdhbf2i3p9e}
Let $\Teul=(X_0,X_1,v_0)$ be a rooted tree.
We shall say that a subset $S$ of $X_0$ is {\it connected\/} if every path
$(x_0,\dots,x_n)$ in $\Teul$ that satisfies $x_0,x_n \in S$ also satisfies $x_i \in S$ for all $i=0,\dots,n$.
Any connected set $S$ can be viewed as a subtree of $\Teul$ (define the edge-set of $S$ to be the set of all edges $e$ of $\Teul$ such
that both endpoints of $e$ are in $S$).

Note in particular that $\Veul$ is connected, in any rooted tree.
\end{remark}

\begin{parag}
A {\it decorated rooted tree\/} is a triple $(\Teul,f,q)$ where
\begin{itemize}
\item $\Teul = (X_0,X_1,v_0)$ is a rooted tree,
\item $f : \Aeul \to \{ (0), (1) \}$ is any map (where $(0)$ and $(1)$ are just two distinct symbols),
\item $q : \setspec{ (e,x) \in X_1 \times X_0 }{ x \in e } \to \Integ$ is any map.
\end{itemize}
\end{parag}

\begin{parag}
Let $(\Teul,f,q)$ be a decorated rooted tree, where $\Teul = (X_0,X_1,v_0)$.

If $\alpha \in \Aeul$ and $f(\alpha) = (0)$ (resp.\ $f(\alpha)=(1)$),
we say that {\it $\alpha$ is decorated by $(0)$} (resp.\ {\it by\/} $(1)$).
We write $\Aeul_0 = \setspec{ \alpha \in \Aeul }{ \text{$\alpha$ is decorated by $(0)$} }$;
consequently, $\Aeul \setminus \Aeul_0$ is the set of arrows decorated by $(1)$.
A {\it dead end\/} is an edge incident to an element of $\Aeul_0$.
Note that if $\{x,\alpha\}$ is a dead end and $\alpha \in \Aeul_0$ then $x \in \Veul$.

If $e = \{x,y\}$ is an edge, the number
$q(e,x)$ is called the {\it decoration of $e$ near $x$}
(so each edge is decorated near each one of its endpoints).

Given $x \in X_0$, let $E_x$ temporarily denote the set of all edges incident to $x$.
For each $e \in E_x$, define 
$Q(e,x) = \prod_{e' \in E_x \setminus \{e\}} q(e',x)$.
{\bf Here and throughout this paper, empty products of numbers are equal to \boldmath $1$.}
This defines a map
$$
Q : \setspec{ (e,x) \in X_1 \times X_0 }{ x \in e } \to \Integ .
$$
Note that $q$ and $Q$ have the same domain and that $Q$ is determined by $q$.

If $e=\{x,y\}$ is an edge such that $x,y \in \Veul$, we define the {\it determinant of $e$} by 
$$
\det(e) = q(e,x)q(e,y) - Q(e,x)Q(e,y) .
$$
\end{parag}

\begin{definition} \label {p9823p98p2d}
An  {\it abstract Newton tree at infinity\/} is a decorated rooted tree $\Teul$ that satisfies 
the following requirements.
\begin{enumerate}

\item \label {fho8wiw9s-1} 
For each vertex $v \in \Veul$, there exists $\alpha \in \Aeul\setminus\Aeul_0$ such that $\alpha>v$.

\item \label {fho8wiw9s-2} 
There is at most one dead end incident to a given vertex.

\item \label {fho8wiw9s-3} 
If $e$ is an edge incident to the root $v_0$ then $q(e,v_0)=1$.

\item  \label {fho8wiw9s-4} 
If $e$ is an edge incident to an arrow $\alpha \in \Aeul$ then  $q(e,\alpha)=1$.

\item \label {fho8wiw9s-5} 
Let $v\in \Veul$. If $e,e'$ are distinct edges incident to $v$ then 
$q(e,v)$ and $q(e',v)$ are relatively prime.
Let $E_v^+$ temporarily denote the set of edges of the form $\{v,x\}$ where $x \in \Veul \cup \Aeul$ and $v<x$,
and observe that $E_v^+ \neq \emptyset$ by condition (1). It is required that $q(e,v) \ge1$ for all $e \in E_v^+$,
and that at most one element $e$ of $E_v^+$ satisfies $q(e,v)>1$.
Moreover, if there is a dead end $\epsilon$ incident to $v$ then
$q(\epsilon,v) = \max_{e \in E_v^+} q(e,v)$ (note that $\epsilon \in E_v^+$).

\item \label {fho8wiw9s-6} 
For each edge $e=\{x,y\}$ such that $x,y \in \Veul$, we have $\det(e)<0$.

\end{enumerate}
\end{definition}

\smallskip

In all that follows, we assume that $\Teul$ is an abstract Newton tree at infinity with notations ($\Veul$, $\Aeul$, etc.) as 
defined in the above paragraphs.

\smallskip

\begin{definition} \label {c9v39rf0eX9e4np8glr9t8}
(Recall that empty products of numbers are equal to $1$.)
\mbox{\ }
\begin{enumerate}

\item[(i)] 
We say that an edge $\epsilon$ is {\it incident\/} to a path $\gamma$ if $\epsilon$ is not in $\gamma$
and $\epsilon$ is incident to some vertex $u$ of $\gamma$.
If $\epsilon$ is incident to $\gamma$,
we define $q(\epsilon,\gamma) = q(\epsilon,u)$ where $u$ is the unique vertex of $\gamma$ to which
$\epsilon$ is incident.

\item[(ii)] Given $v \in \Veul \cup \Aeul_0$ and $\alpha \in \Aeul\setminus\Aeul_0$,
we set
$$
\textstyle
x_{v,\alpha} = \prod_{\epsilon \in E} q(\epsilon,\gamma)
\quad \text{and} \quad
\hat x_{v,\alpha} = \prod_{\epsilon \in \hat E} q(\epsilon,\gamma)
$$
where $\gamma$ is the path from $v$ to $\alpha$,
$E$ is the set of edges incident to $\gamma$ and
$$
\hat E =  \text{set of edges incident to $\gamma$ but not incident to $v$}.
$$
Observe that 
$x_{v,\alpha} = Q(e,v) \hat x_{v,\alpha}$,
where $e$ is the unique edge incident to $v$ which is in $\gamma$.

\item[(iii)] Given $v\in \Veul\cup \Aeul_0$, we define the {\it multiplicity $N_v$ of $v$} by
$ N_v = \sum_{ \alpha \in \Aeul \setminus \Aeul_0} x_{v,\alpha}$.

\item[(iv)]  We define the {\it multiplicity\/} $M(\Teul)$ of $\Teul$ 
by $ M(\Teul) = -\sum_{v \in \Veul\cup \Aeul_0} N_v(\delta_v-2) $.

\end{enumerate}
\end{definition}

\begin{definition} \label {823yi867jcn73}
A {\it dicritical vertex\/} is a vertex $v \in \Veul$ satisfying $N_v=0$.
If $v$ is a dicritical vertex satisfying
\begin{equation} \tag {$ * $}
\setspec{ x \in \Veul \cup \Aeul }{ x > v } \subseteq \Aeul ,
\end{equation}
we define the {\it degree of the dicritical\/} $v$
to be the number of edges $\{ v,\alpha \}$ where $\alpha \in \Aeul \setminus \Aeul_0$.
By part \eqref{fho8wiw9s-1} of Def.\ \ref{p9823p98p2d}, the degree of a dicritical is strictly positive.
\end{definition}

\begin{parag} \label {pd0f123wdjwe}
{\bf Pictures.} \setlength{\unitlength}{1.5mm}
When an abstract Newton tree at infinity is represented by a picture,
dicritical vertices are represented by 
``\begin{picture}(2,1)(-1,-.5) \put(0,0){\circle*{1}} \end{picture}'' 
and vertices which are not dicritical are represented by ``\begin{picture}(2,1)(-1,-.5) \put(0,0){\circle{1}} \end{picture}''.
Each arrow is represented by an arrowhead
``\begin{picture}(2,1)(-1,-.5) \put(0.5,0){\vector(1,0){0}} \end{picture}'', not by an arrow
``\begin{picture}(5,1)(0,-.5) \put(0,0){\vector(1,0){5}} \end{picture}'',
so that ``\begin{picture}(5.5,1)(-.5,-.5) \put(0,0){\circle*{1}} \put(0.5,0){\vector(1,0){4.5}} \end{picture}''
represents an edge joining a vertex 
``\begin{picture}(2,1)(-1,-.5) \put(0,0){\circle*{1}} \end{picture}''
to an arrow
``\begin{picture}(2,1)(-1,-.5) \put(0.5,0){\vector(1,0){0}} \end{picture}''.
When the decoration of an arrow does not appear in the picture, that decoration is assumed to be $(1)$.
If $e$ is an edge incident to $x \in \Veul \cup \Aeul$,
and if the decoration $q(e,x)$ does not appear in the picture, that decoration is assumed to be $1$.
For a simple example of an abstract Newton tree at infinity, the reader is referred to Ex.\ \ref{nv63jfy64nvy3}.
A more complicated example is given in Fig.\  \ref{723dfvcjp2q98ewdywe}, where the following convention is used:
$$
\begin{picture}(9.6,1)(-.9,-.5) \put(0,0){\circle*{1}} \put(-.4,-.7){\makebox(0,0)[t]{\footnotesize $v$}}
\put(0.5,0){\line(1,0){4.5}} \put(4.8,0){\makebox(0,0)[l]{\footnotesize $<$}} \put(7,0){\makebox(0,0)[l]{\footnotesize $d$}} \end{picture}
\qquad \text{is an abbreviation for} \qquad
\raisebox{-3.0\unitlength}{\begin{picture}(9.6,7.4)(-.9,-3.5) \put(0,0){\circle*{1}} \put(-.4,-.7){\makebox(0,0)[t]{\footnotesize $v$}}
\put(0,0){\vector(2,1){5}} \put(0,0){\vector(2,-1){5}} \put(4,.6){\makebox(0,0)[c]{\footnotesize $\vdots$}}
\put(5.5,2.3){\makebox(0,0)[l]{\footnotesize $\alpha_1$}}\put(5.5,-2.3){\makebox(0,0)[l]{\footnotesize $\alpha_d$}}
\end{picture}}
$$
where $v \in \Veul$ is a dicritical vertex of degree $d$,
$\alpha_1, \dots, \alpha_d$ are the distinct elements of  $\Aeul\setminus\Aeul_0$ which are adjacent to $v$,
and $q( \{ v, \alpha_i\} , v ) = 1$ for all $i = 1, \dots, d$.

The tree of Fig.\ \ref{723dfvcjp2q98ewdywe} will be revisited in  Ex.\ \ref{Pc09vbq3j7gabXrZiA}, \ref{cIuCbgepwej6DsJ37655enm},
\ref{cov7btiw67d78vujew9}, \ref{oOo8cvn239vn3p98fgIW}, \ref{FFlAkjcvwiuer2n3Z3cxs7df9} and \ref{09cq13cgrl38n9ZAyey}.
The reader may verify directly that
$N_{v_0}= 252$, $N_{v_1}= 126$, $N_{v_2}= 63$, $N_{v_3}= 9$, $N_{v_4}= 3$, $N_{v_5}= 2$ and $N_{v_6}=9$,
and that all other vertices $v$ satisfy $N_v=0$ (and hence are dicritical).
Keep in mind that $v_0$ always denotes the root.
\end{parag}

\begin{figure}[htb]
 \centering
	\scalebox{0.7}{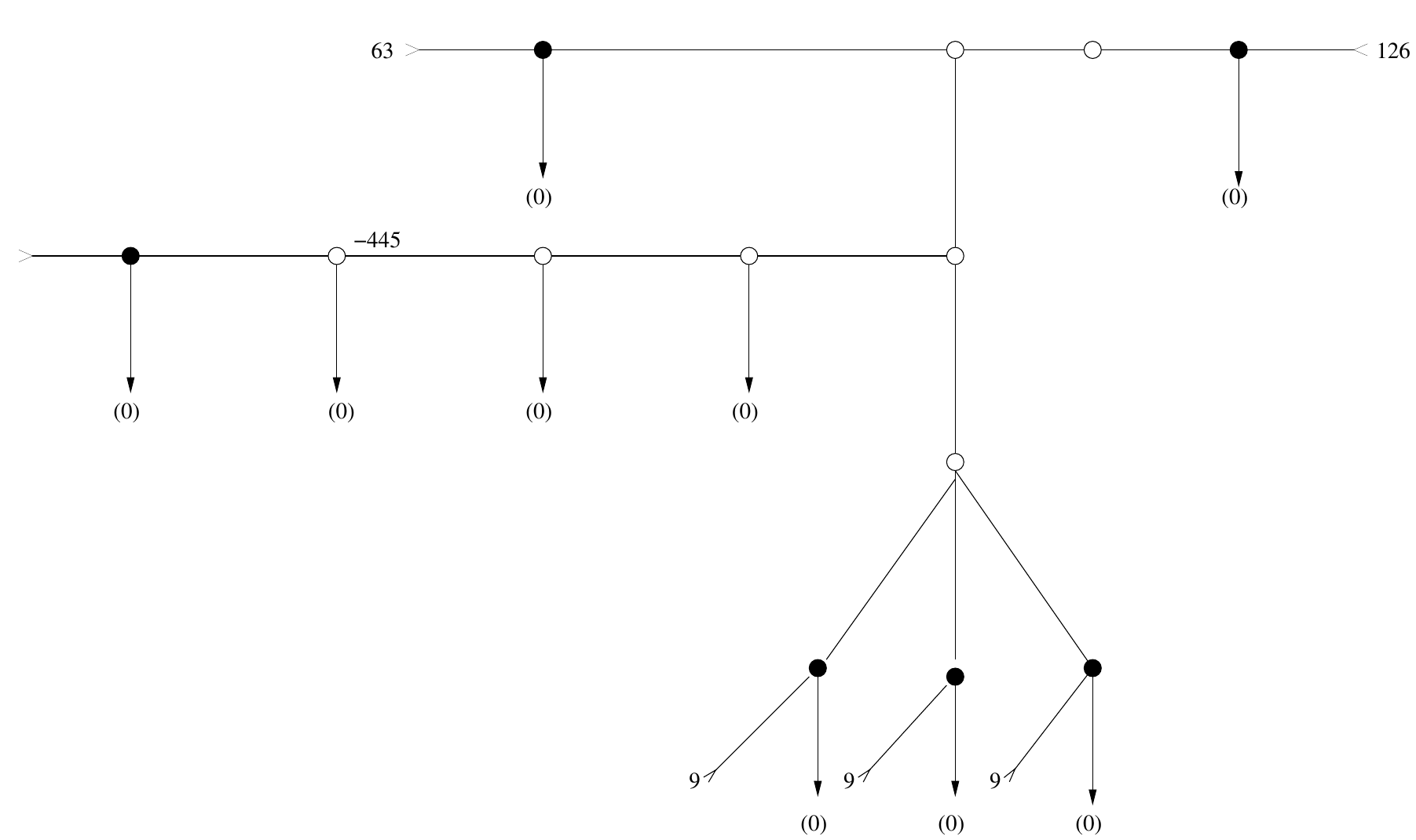}
 \caption{An abstract Newton tree at infinity with $|\Veul|=13$ and $|\Aeul|=227$.}
 \label {723dfvcjp2q98ewdywe}
\end{figure}

\begin{definition} \label {9c8yvqewjft7d}
The {\it number of points at infinity\/} of $\Teul$
is the valency of the root when there is no dead end incident to the root and the valency minus $1$ otherwise.
It follows from part \eqref{fho8wiw9s-1} of \ref{p9823p98p2d} that $\Teul$ has at least one point at infinity.
\end{definition}

\begin{remark} \label {efy872392eujf}
If $\Teul$ has $n$ points at infinity then
$N_{v_0} \ge | \Aeul \setminus \Aeul_0 | \ge n \ge 1$.
Indeed, part \eqref{fho8wiw9s-1} of \ref{p9823p98p2d} implies that 
$| \Aeul \setminus \Aeul_0 | \ge n \ge 1$ and
part \eqref{fho8wiw9s-5}  of \ref{p9823p98p2d} implies
that $x_{v_0,\alpha}\ge1$ for each $\alpha \in \Aeul\setminus\Aeul_0$.
\end{remark}

\begin{remark} \label {es5s6xjc17lcy4l}
Let $e=\{v,\alpha\}$ be a dead end, where $v \in \Veul$ and $\alpha \in \Aeul_0$.  Then 
$$
N_v = q(e,v) N_{\alpha}.
$$
To see this, simply observe that $x_{v,\beta}=q(e,v)x_{\alpha,\beta}$ for all $\beta \in \Aeul \setminus \Aeul_0$.
\end{remark}

\begin{definition} \label {d23ed9wei9d23}
For a vertex $v$, we define $a_v=1$ if there is no dead end incident to $v$, and $a_v=q(e,v)$ if $e$
is a dead end incident to $v$ (recall that there is at most one dead end incident to $v$).
Note that  for all $v \in \Veul$ we have $a_v \in \Nat\setminus\{0\}$ and (by Rem.\ \ref{es5s6xjc17lcy4l})
$a_v \mid N_v$.
\end{definition}

\begin{definition} 
Consider distinct $v, v' \in \Veul$ and the path $\gamma$ from $v$ to $v'$.
We say that $\gamma$ is a {\it linear path\/} if each vertex $u$ in $\gamma$ but different from $v,v'$ has valency $2$.
If $\gamma$ is a linear path from $v$ to $v'$ then we define $\det(\gamma) = q(e,v) q(e',v') - Q(e,v)Q(e',v')$,
where $e$ (resp.\ $e'$) is the unique edge in $\gamma$ which is incident to $v$ (resp.\ $v'$).
\end{definition}

The determinant of a linear path generalizes the determinant of an edge.
We now quote Prop.\ 2.4 of \cite{CND15:simples}:

\begin{proposition}  \label {kuwdhr12778}
Consider distinct $v,v' \in \Veul$ and suppose that the path $\gamma$ from $v$ to $v'$ is linear. 
Let $q=q(e,v)$, $q'=q(e',v')$, $Q=Q(e,v)$ and $Q'=Q(e',v')$ 
where $e$ (resp.\ $e'$) is the unique edge in $\gamma$ which is incident to $v$ (resp.\ $v'$),
and let
\begin{align*}
A &= \setspec{ \alpha \in \Aeul\setminus\Aeul_0 }{ \text{the path from $v$ to $\alpha$ does not contain $v'$} }, \\
A' &= \setspec{ \alpha \in \Aeul\setminus\Aeul_0 }{ \text{the path from $v$ to $\alpha$ contains $v'$} } 
\end{align*}
(see Figure~\ref{83473yruer938d}).  Then the following hold.
\begin{enumerata}

\item \label {8723d4917t3rh} 
 For each $\alpha \in A$,
$x_{v,\alpha} = q \hat x_{v',\alpha}$
and $x_{v',\alpha} = Q' \hat x_{v',\alpha}$.

\item \label {u66d3swxsd}
For each $\alpha \in A'$, $x_{v',\alpha} = q' \hat x_{v,\alpha}$
and $x_{v,\alpha} = Q \hat x_{v,\alpha}$.

\item \label {238dhwi912F}
$\left| \begin{matrix} q & Q' \\ N_v & N_{v'} \end{matrix} \right| = 
\det(\gamma) \sum_{\alpha \in A'} \hat x_{v,\alpha}$\ \ and\ \ 
$\left| \begin{matrix} q' & Q \\ N_{v'} & N_{v} \end{matrix} \right| = 
\det(\gamma) \sum_{\alpha \in A} \hat x_{v',\alpha}$.

\item \label {61324d1g3r}
 If $v<v'$ then $q>0$, $Q'>0$, $\det\gamma<0$ and
$\left| \begin{matrix} q & Q' \\ N_v & N_{v'} \end{matrix} \right| < 0$.

\end{enumerata}
\end{proposition}

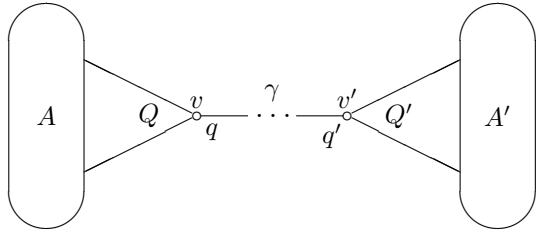
\begin{figure}[h]
\setlength{\unitlength}{1mm}
\raisebox{0\unitlength}{\begin{picture}(80,30)(-40,-15)
\put(-10,0){\circle{1}}
\put(10,0){\circle{1}}
\put(-9.5,0){\line(1,0){6.3}}
\put(9.5,0){\line(-1,0){6.3}}
\put(0.5,0){\makebox(0,0){\dots}}
\put(-10.632,0.316){\line(-2,1){14.367}}
\put(-10.632,-.316){\line(-2,-1){14.367}}
\put(10.632, 0.316){\line(2,1){14.367}}
\put(10.632,-0.316){\line(2,-1){14.367}}
\put(-30,0){\oval(10,30)}
\put(30,0){\oval(10,30)}
\put(-30,0){\makebox(0,0){\footnotesize $A$}}
\put(30,0){\makebox(0,0){\footnotesize $A'$}}
\put(-15,0){\makebox(0,0)[r]{\footnotesize $Q$}}
\put(-8,-1){\makebox(0,0)[t]{\footnotesize $q$}}
\put(-10,1){\makebox(0,0)[b]{\footnotesize $v$}}
\put(0,2){\makebox(0,0)[b]{\footnotesize $\gamma$}}
\put(10,1){\makebox(0,0)[b]{\footnotesize $v'$}}
\put(8,-1){\makebox(0,0)[t]{\footnotesize $q'$}}
\put(15,0){\makebox(0,0)[l]{\footnotesize $Q'$}}
\end{picture}}
\caption{Schematic representation of the situation of Prop.~\ref{kuwdhr12778}.}
\label {83473yruer938d}
\end{figure}

See Rem.\ \ref{jdhbf2i3p9e} for the concept of a connected set.

\begin{corollary}[\cite{CND15:simples}, 2.5] \label {i875863urj1dpi}
\begin{enumerata}

\item If $v,v' \in \Veul$, $v < v'$ and  $N_v\le0$  then $N_{v'}<0$.

\item The set of vertices whose multiplicity is nonnegative  is connected.

\item The set of vertices whose multiplicity is positive  is connected.

\end{enumerata}
\end{corollary}

\begin{proof}
By \ref{kuwdhr12778}\eqref{61324d1g3r}, (a) is true whenever $\{v,v'\}$ is an edge;
it easily follows that (a) is true in general.  Assertions (b) and (c) follow from (a).
\end{proof}

\begin{remark} \label {7623egd812e76f}
It follows from Cor.\ \ref{i875863urj1dpi}(a) that if $x,y$ are dicritical vertices then $x,y$ are not adjacent.
Also note that $N_{v_0}>0$ by Rem.\ \ref{efy872392eujf}, so the root $v_0$ is not a dicritical vertex.
\end{remark}

\begin{definition} \label {c0j23oed9vbaCo8f}
An abstract Newton tree at infinity $\Teul$ is {\it generic\/}  if every vertex of $\Teul$
has  nonnegative multiplicity.
\end{definition}

\begin{remark} \label {zfp2o39wdw0d}
Let $\Teul$ be a generic Newton tree at infinity.
Then Cor.\ \ref{i875863urj1dpi}(a) 
implies that each dicritical vertex $v$ of $\Teul$ satisfies condition $(*)$ of Def.\ \ref{823yi867jcn73}.
Consequently, each dicritical vertex of $\Teul$ has a well-defined degree.
\end{remark}

\begin{lemma} \label {923ip38kjfbve09}
In a generic Newton tree at infinity, 
let $v$ be a dicritical vertex of degree $d$ and let
$B_v = \setspec{ \alpha \in \Aeul\setminus\Aeul_0 }{ \text{$\{v,\alpha\}$ is an edge} }$.
\begin{enumerata}

\item \label {jtiq348439}  $d = | B_v | \ge 1$

\item \label {93f394ejf9i} $q( \{v,\alpha\}, v ) = 1$ for all $\alpha \in B_v$.

\item \label {9v8bdzkc93ry7f} There exists exactly one $w \in \Veul$ such that $\{w,v\}$ is an edge.
Moreover, this unique $w$ satisfies $N_w>0$ and $w<v$.

\end{enumerata}
\end{lemma}

\begin{proof}
Since our tree is generic,
Cor.\ \ref{i875863urj1dpi}(a) implies that $U_v \subseteq \Aeul$, where we define
$U_v = \setspec{ x \in \Veul \cup \Aeul }{ x > v }$;
so \cite[Lemma 2.9]{CND15:simples} implies that \eqref{93f394ejf9i}  is true.
Also note that 
each element of $U_v$ is linked to $v$ by an edge.
By part (1) of Def.\ \ref{p9823p98p2d}, there exists $\alpha \in \Aeul \setminus \Aeul_0$ such that $\alpha>v$;
then $\alpha \in U_v$, so $\{v,\alpha\}$ is an edge, so $\alpha \in B_v$, showing that $B_v \neq \emptyset$.
By definition of the degree of a dicritical, it is clear that $d=|B_v|$.

To prove \eqref{9v8bdzkc93ry7f}, observe that $v_0 \neq v$ (by Rem.\ \ref{7623egd812e76f}),
and that consequently there exists exactly one $w \in \Veul$ such that $\{w,v\}$ is an edge and $w<v$.
If there exists $w' \in \Veul \setminus \{w\}$ such that $\{w',v\}$ is an edge
then $w' \in U_v \subseteq \Aeul$, a contradiction.
This shows that there exists exactly one $w \in \Veul$ such that $\{w,v\}$ is an edge,
and that this $w$ satisfies $w<v$. We have $N_w > 0$ by Cor.\ \ref{i875863urj1dpi}(a) or Rem.\ \ref{7623egd812e76f},
so \eqref{9v8bdzkc93ry7f} is proved.
\end{proof}

\begin{definition} \label {cn1o29cnds9}
An abstract Newton tree at infinity $\Teul$ is {\it complete\/}  if it is generic and if
each arrow of $\Teul$ decorated with $(1)$ is adjacent to a dicritical vertex.
\end{definition}

\begin{remark} 
Let $\Teul$ be an abstract Newton tree at infinity.
By Rem.\  \ref{7623egd812e76f}, the root is not a dicritical vertex.
By Cor.\ \ref{i875863urj1dpi}, 
there is at most one dicritical on any path from the root to an arrow.
If $\Teul$ is complete, there is exactly one dicritical on any path from the root to an arrow decorated by $(1)$.
\end{remark}

Consider a dead end $e=\{v,\alpha\}$, where $v \in \Veul$ and $\alpha \in \Aeul_0$.
If the decoration of $e$ near $v$ is $1$, we call $e$ a {\it dead end decorated by $1$}.

\begin{definition} \label {e5e6e7w8e9wee9}
An abstract Newton tree at infinity $\Teul$ is {\it minimally complete\/} if it satisfies:
\begin{itemize}
\item[(i)] $\Teul$ is complete;

\item[(ii)] if $v$ is a dicritical then there is a dead end incident to $v$;

\item[(iii)] if a dead end decorated by $1$ is incident to a vertex $v$, then $v$ is a dicritical;

\item[(iv)] every element of $\Veul \setminus \{v_0\}$ has valency different from $2$.

\end{itemize}
\end{definition}

\begin{lemma}[{\cite[Lemma 2.24]{CND15:simples}}] \label {8123oiuwhqd8d2}
Consider an edge $\{v,v'\}$ in a minimally complete abstract Newton tree at infinity,
where $v$ is a vertex with $N_v=1$, $v'$ is a vertex, and $v<v'$.
Then $v'$ is a dicritical of degree $1$ and the edge determinant of $\{v,v'\}$ is $-1$.
\end{lemma}

\medskip
\noindent{\bf Equivalence of abstract  Newton trees at infinity.}
Definition 2.13 of  \cite{CND15:simples} defines an equivalence relation on the set of 
abstract  Newton trees at infinity.
The main properties of that equivalence relation are given in Lemmas 2.14 and 2.18 of \cite{CND15:simples},
which we reproduce here as Lemmas \ref{0923923ujds0} and \ref{cvp9wpe9p9d}:

\begin{lemma} \label {0923923ujds0}
Equivalent abstract  Newton trees at infinity have the same number of points at infinity and the same multiplicity.
\end{lemma}

\begin{lemma}  \label {cvp9wpe9p9d}
Let $\Ceul$ be the equivalence class of a generic Newton tree at infinity.
Then $\Ceul$ contains exactly one minimally complete  Newton tree at infinity.
\end{lemma}

\section{Nodes}
\label {SEC:Nodes}

{\it In this section, $\Teul$ denotes an abstract Newton tree at infinity that is minimally complete.} 

\medskip

This section develops a theory of ``nodes''  (see Def.\ \ref{oifnue09e2039}) that can be applied
to any tree $\Teul$ satisfying the above assumption.
We also define the subtree $\Neul$ of $\Teul$ (\ref{p0c9vh394envpwI2})
and the map $\tD : \Neul \to \Integ$ (\ref{pc09vg349fbv8682}),
both of which play a crucial role in subsequent sections.
Lemmas \ref{90q932r8dhd89cnr9}, \ref{C0qowbdowX932epug7fnvd} and \ref{8ey8cdody27}
give interesting information on vertices $v$ satisfying $\tD (v) \le 0$.

\begin{notation}  \label {p0c9vh394envpwI2}
Let $\Neul = \setspec{ v \in \Veul }{ N_v>0 }$ and $\Deul = \setspec{ v \in \Veul }{ N_v=0 }$.
Note that $\Deul$ is the set of dicritical vertices of $\Teul$ and that $\{ \Neul, \Deul \}$ is a partition of $\Veul$.
\end{notation}

We noted in \ref{i875863urj1dpi} that $\Neul$ is connected; so we may view $\Neul$ as a subtree of $\Teul$
(but strictly speaking, $\Neul$ is only a set of vertices).

\begin{remark} \label {p092pof90e9rf}
Note that $v_0 \in \Neul$  (by Rem.\ \ref{7623egd812e76f}) and that
\begin{enumerate}

\item[(i)] no arrow is adjacent to $v_0$,
\item[(ii)]  $\delta_{v_0}$ is equal to the number of points at infinity,
\item[(iii)]  $a_{v_0} = 1$,

\end{enumerate}
where (ii) and (iii) are immediate consequences of (i) (in fact (iii) also follows from Def.\ \ref{p9823p98p2d}\eqref{fho8wiw9s-3}).
To prove (i) by contradiction, assume that $\alpha \in \Aeul$ is such that $e = \{v_0,\alpha\}$ is an edge.
If $\alpha$ is decorated by $(1)$ then (since $\Teul$ is complete) $v_0$ is a dicritical, which is not the case.
So $\alpha$ is decorated by $(0)$; since $q(e,v_0)=1$, $e$ is a dead end decorated by $1$, so condition \ref{e5e6e7w8e9wee9}(iii)
implies that $v_0$ is  a dicritical, which is not the case.
\end{remark}

\begin{definition} \label {oifnue09e2039}
\mbox{\ }
\begin{enumerate}

\item By a {\it node\/} of $\Teul$, we mean a vertex which is adjacent to at least one dicritical vertex.
The set of all nodes of $\Teul$ is denoted $\Nd(\Teul)$.
By Rem.\ \ref{7623egd812e76f}, dicritical vertices cannot be adjacent; so $\Nd(\Teul) \subseteq \Neul$.
It is clear that $\Nd(\Teul) \neq \emptyset$.

\item Given $v \in \Neul$, let $\Deul_v$ be the set of dicritical vertices adjacent to $v$.
Note that $\Deul_v \neq \emptyset$ if and only if $v \in \Nd(\Teul)$.

\item Given a dicritical vertex $u$, let $d_u \in \Nat\setminus\{0\}$ be the degree of $u$,
i.e., the number of elements of $\Aeul \setminus \Aeul_0$ that are adjacent to $u$ (see \ref{823yi867jcn73} and \ref{zfp2o39wdw0d}).

\item Let $v \in \Nd(\Teul)$.
Let $u_1, \dots, u_s$ ($s \ge 1$) be the distinct elements of $\Deul_v$,
labeled so as to have  $d_{u_1} \le \cdots \le d_{u_s}$,
where $d_{u_i} \in \Nat\setminus\{0\}$ is the degree of the dicritical $u_i$.
Then we define the {\it type\/} of the node $v$ to be the finite sequence $[d_{u_1}, \dots, d_{u_s}]$, enclosed in square brackets.

\end{enumerate}
\end{definition}

\begin{example} \label {Pc09vbq3j7gabXrZiA}
Let $\Teul$ be the abstract Newton tree at infinity depicted in Fig.\ \ref{723dfvcjp2q98ewdywe}.
Then $\Teul$ is minimally complete.
The subtree $\Neul$ of $\Teul$ is shown in the picture below.
The numbers written near each vertex $v \in \Neul$ (in the picture)
are $N_v, a_v$ if $v$ is not a node,
and $N_v,a_v, \tau_v$ if $v$ is a node, where $\tau_v$ is the type of the node $v$.
$$
\scalebox{1}{
\setlength{\unitlength}{1mm}%
\begin{picture}(95,21)(-5,-16)
%%%%%%%%%%%%%%%%%%%%%%%%%%%%%%%%%%%%%%%%%%%%%%
\multiput(0,0)(20,0){5}{\circle{1}}
\put(0,-15){\circle{1}}
\multiput(0.5,0)(20,0){4}{\line(1,0){19}}
\put(0,-0.5){\line(0,-1){14}}
\put(40,-15){\circle{1}}
\put(40,-0.5){\line(0,-1){14}}
%%%%%%%%%%%%%%%%%%%%%%%%%%%%%%%%%%%%%%%%%%%%%%
\put(0,2){\makebox(0,0)[b]{\tiny ${ 3,3 }$}}
\put(20,2){\makebox(0,0)[b]{\tiny ${ 9,3 }$}}
\put(40,2){\makebox(0,0)[b]{\tiny ${ 63,1 }$}}
\put(60,2){\makebox(0,0)[b]{\tiny ${ 126,1,[63] }$}}
\put(80,2){\makebox(0,0)[b]{\tiny ${ 252,1,[126] }$}}
\put(2,-15){\makebox(0,0)[l]{\tiny ${ 2,2,[2] }$}}
\put(42,-15){\makebox(0,0)[l]{\tiny ${ 9,1,[9,9,9] }$}}
\put(38,-15){\makebox(0,0)[r]{\tiny ${ v_6 }$}}
\put(-2,-15){\makebox(0,0)[r]{\tiny ${ v_5 }$}}
\put(-2,-1.5){\makebox(0,0)[t]{\tiny ${ v_4 }$}}
\put(20,-1.5){\makebox(0,0)[t]{\tiny ${ v_3 }$}}
\put(38,-1.5){\makebox(0,0)[t]{\tiny ${ v_2 }$}}
\put(60,-1.5){\makebox(0,0)[t]{\tiny ${ v_1 }$}}
\put(80,-1.5){\makebox(0,0)[t]{\tiny ${ v_0 }$}}
%%%%%%%%%%%%%%%%%%%%%%%%%%%%%%%%%%%%%%%%%%%%%%
\end{picture} \qquad\quad
\raisebox{12\unitlength}{\begin{minipage}[t]{7cm} $\Neul = \{v_0, v_1, v_2, v_3, v_4, v_5, v_6 \}$ 
\end{minipage}}
}
$$
\end{example}

\begin{remark}
In fact there exists a rational polynomial $f \in \Comp[x,y]$ such that $\Teul(f;x,y)$ is the tree $\Teul$ of Fig.\ \ref{723dfvcjp2q98ewdywe}.
So, according to Cor.~\ref{0v3y46y7iwoujbfv8e} in the Introduction, the tree $\Neul$ depicted in Ex.~\ref{Pc09vbq3j7gabXrZiA} should
look like the tree of Fig.\ \ref{902br8hrjkderydyukhyg}.
This is the case, with $(z_1, \dots, z_n) = (v_0,v_1,v_2,v_3,v_4)$.
\end{remark}

\begin{definition}  \label {SJyFDkjyftGkp0c9vb349vby3f6w}
An \textit{f-partition} of a set $X$ is a family $( X_i )_{i \in I}$, indexed by a set $I$, satisfying:
\begin{itemize}

\item for each $i \in I$, $X_i$ is a (possibly empty) subset of $X$;
\item we have $X_i \cap X_j = \emptyset$ whenever $i,j$ are distinct elements of $I$;
\item $X = \bigcup_{i \in I} X_i$.

\end{itemize}
Clearly, if $( X_i )_{i \in I}$ is an f-partition of a finite set $X$ and $g : X \to \Reals$ is any
function then $\sum_{ x \in X} g(x) = \sum_{i \in I} \sum_{x \in X_i} g(x)$.
\end{definition}

\begin{remark} \label {92k3cbq3dbeoqw}
$( \Deul_v )_{v \in \Nd(\Teul)}$ is an f-partition of $\Deul$
and $( A_v )_{v \in \Nd(\Teul)}$ is an f-partition of $\Aeul \setminus \Aeul_0$,
where we define $ A_v = \setspec{ \alpha \in \Aeul \setminus \Aeul_0 }{ \text{$\alpha$ is adjacent to some element of $\Deul_v$}} $
for each $v \in \Nd(\Teul)$.
\end{remark}

\begin{notation} \label {Ppc09wberyrji76awe}
Following Paragraph 2.19 of \cite{CND15:simples} we define
$$
r_v = -1 + | \setspec{v' \in \Veul}{\text{$\{v,v'\}$ is an edge}} |,  \quad \text{for all $v \in \Neul$.}
$$
We note that $r_v \in \Nat$ for all $v \in \Neul$, and that $r_v=0$ if and only if $v = v_0$ and $\delta_{v_0}=1$.
Indeed, no arrow is adjacent to $v_0$ by  Rem.\ \ref{p092pof90e9rf}, so  $r_{v_0} = -1 + \delta_{v_0}$, so the claim is true for $v=v_0$.
If $v \neq v_0$ then we have $\delta_v \ge3$, at most one dead end is incident to $v$, and (since $v$ is not a dicritical) 
no element of $\Aeul\setminus\Aeul_0$ is adjacent to $v$, so $r_v\ge1$.
\end{notation}

\begin{notation} \label {kjfoqwsdjkwef}
We define the set map $\Delta : \Neul \to \Integ$ by
\begin{equation} \label {SODjc9wedc}
\textstyle   \Delta(v) = (r_v-1)(N_v-1) + N_v\big( 1-\frac1{a_v} \big) \qquad \text{for all $v \in \Neul$.}
\end{equation}
Since $a_v \mid N_v$, the right-hand-side of \eqref{SODjc9wedc} is indeed an integer.
We also define $ \textstyle  \Delta(S) = \sum_{v \in S} \Delta(v)$, for any subset $S$ of $\Neul$.
\end{notation}

The definition of $\Delta$ given in  \cite{CND15:simples} is different from this one,
but Equality \eqref{SODjc9wedc} is proved in Lemma 2.20 of \cite{CND15:simples},
so $\Delta$ is the same in the two articles.\footnote{Lemma 2.20 of \cite{CND15:simples} assumes that $\Teul$ has at least two points at infinity,
but that assumption is not used in the proof of \eqref{SODjc9wedc} so $\Delta$ is indeed the same in the two articles.}
Note that \cite{CND15:simples} does not consider the map $\tD : \Neul \to \Integ$, that we now define.

\begin{notation} \label {pc09vg349fbv8682}
We define $\tD(v) = \Delta(v) - \sum_{u \in \Deul_v} (d_u-1)$  for all $v \in \Neul$,
and $\tD(S) = \sum_{v \in S} \tD(v)$ for any subset $S$ of $\Neul$.
Note that $\tD(v) \in \Integ$ and $\tD(S) \in \Integ$.
Also observe that if $v \in \Neul \setminus \Nd(\Teul)$ then $\tD(v) = \Delta(v)$, because $\Deul_v = \emptyset$.
\end{notation}

\begin{definition} \label {aowi3efkWEGRGa9e93}
We shall now define numbers $\sigma(v), \epsilon(v) \in \Nat$ for each $v \in \Neul$,
and $k_u \in \Nat\setminus\{0\}$ for each $u \in \Deul$.
\begin{enumerate}

\item Suppose that $v \in \Nd(\Teul)$. For each $u \in \Deul_v$,  
let $k_u = -\det(\{u,v\}) \in \Nat\setminus \{0\}$ and let $d_u \in \Nat\setminus\{0\}$ be the degree of the dicritical $u$.
Then we define $\sigma(v) = \sum_{u \in \Deul_v} (k_u-1)d_u$. 

\smallskip

\noindent If $v \in \Neul$ is not a node we define $\sigma(v)=0$. It follows that the formula
$$
\textstyle
\sigma(v) = \sum_{u \in \Deul_v} (k_u-1)d_u
$$
is valid for all $v \in \Neul$, because if $v$ is not a node then $\Deul_v=\emptyset$.  Note that $\sigma(v) \in \Nat$ for all $v \in \Neul$.

\item For any $v \in \Neul$, define $\epsilon(v) =$ the cardinality of $\setspec{ x \in \Neul }{ \text{$x$ is adjacent to $v$} }$.

\end{enumerate}
\end{definition}

\begin{remark} \label {c0v9jn23rf9vXkfC2e}
We have $\epsilon(v)= r_v +1- | \Deul_v |$  for all $v \in \Neul$.
\end{remark}

\begin{lemma} \label {90q932r8dhd89cnr9}
Let $v \in \Neul$. 
\begin{enumerata}

\item \label {o82y387jw9e23} If $v \in \Nd(\Teul)$ then for all $u \in \Deul_v$ we have
$N_v = k_ud_u$ and $k_u,d_u \in \Nat\setminus\{0\}$.

\item \label {pc9vp23r09}     $\tD(v) = \sigma(v)  + (\epsilon(v)-2)(N_v-1) + N_v(1 - \frac1{a_v})$ 

\item  \label {0239u9e78923}
If $N_v=1$ then $\tD(v)=0 = \sigma(v)$,
$a_v=1$, $\epsilon(v) \le 1$ and $v$ is a node of type $[1,\dots,1]$.

\item \label {p9f239h209f9} If $\epsilon(v) >2$ then $\tD(v) > 0$.

\item \label {WFh2h8heq8whoq9u} If $\tD(v) < 0$ then $\epsilon(v) \in \{0,1\}$.

\end{enumerata}
\end{lemma}

\begin{proof}
\eqref{o82y387jw9e23} Let $u\in \Deul_v$.  Let $e = \{v,u\}$, then \ref{kuwdhr12778} gives
$$
\left| \begin{matrix} q(e,v) & Q(e,u) \\ N_v & N_{u} \end{matrix} \right| = 
\displaystyle \det(e) \sum_{\alpha \in A'} {\hat x_{v,\alpha}}
$$
where $A' = \setspec{ \alpha \in \Aeul \setminus \Aeul_0 }{ \text{$\{u,\alpha\}$ is an edge} }$ and  $N_{u}=0$.
Lemma \ref{923ip38kjfbve09} implies that $Q(e,{u}) = a_{u}$
and that ${\hat x_{v,\alpha}} = a_{u}$ for each $\alpha \in A'$.
So $- a_{u} N_v = \det(e) d_u a_{u} = -k_u d_u a_{u}$ and the equality $N_v = k_u d_u$ follows.
It is clear that $k_u,d_u \in \Nat\setminus\{0\}$.
This proves \eqref{o82y387jw9e23}.

\eqref{pc9vp23r09}     The following argument is valid whether $v$ is a node or not.
Let $s = | \Deul_v |$ and note that $\epsilon(v) =  r_v +1- s$ by Rem.\ \ref{c0v9jn23rf9vXkfC2e}.
By \ref{kjfoqwsdjkwef},
\begin{align*}
\tD(v)
&= \textstyle \Delta(v) -  \sum_{u \in \Deul_v} (d_u-1) \\
&= \textstyle  (r_v-1)(N_v-1) + N_v\big( 1-\frac1{a_v} \big) - \sum_{u \in \Deul_v} (d_u-1) \\
&= \textstyle  (s-2+\epsilon(v))(N_v-1) + N_v\big( 1-\frac1{a_v} \big) - \sum_{u \in \Deul_v} (d_u-1) \\
&= \textstyle  (\epsilon(v)-2)(N_v-1) + N_v\big( 1-\frac1{a_v} \big) + s(N_v-1) - \sum_{u \in \Deul_v} (d_u-1) \\
&= \textstyle  (\epsilon(v)-2)(N_v-1) + N_v\big( 1-\frac1{a_v} \big) + \sum_{u \in \Deul_v} (N_v-d_u) \\
&= \textstyle   (\epsilon(v)-2)(N_v-1) + N_v(1 - \frac1{a_v}) + \sum_{u \in \Deul_v} (k_u-1) d_u  \\
&= \textstyle  (\epsilon(v)-2)(N_v-1) + N_v(1 - \frac1{a_v}) + \sigma(v) .
\end{align*}

\eqref{0239u9e78923} Assume that $N_v=1$.
Let $B$ be the set of vertices $u$ adjacent to $v$ and such that $v<u$.
As $v$ is not a dicritical, $B \neq \emptyset$.
Since $N_v=1$, Lemma \ref{8123oiuwhqd8d2} implies that $B \subseteq \Deul_v$, so $v$ is a node.
For each $u \in \Deul_v$ we have $k_u d_u = N_v=1$ by  \eqref{o82y387jw9e23}, so $k_u=1=d_u$;
it follows that the type of $v$ is $[1,\dots,1]$ and that $\sigma(v)=0$. 
Since $N_v=1$ we also have $a_v=1$ (because $a_v \mid N_v$). 
Then \eqref{pc9vp23r09}     gives $\tD(v)=0$.
The fact that $\epsilon(v) \le 1$ follows from $B \subseteq \Deul_v$
and the observation that at most one vertex is adjacent to $v$ and less than $v$.

(d) If $\epsilon(v) > 2$ then (c) implies that $N_v>1$, so (b) gives $\tD(v) > 0$.

Assertion (e) follows from (b).
\end{proof}

\begin{example} \label {cIuCbgepwej6DsJ37655enm}
Let $\Teul$ be the tree considered in Fig.\ \ref{723dfvcjp2q98ewdywe} and Ex.\ \ref{Pc09vbq3j7gabXrZiA}.
Using only the information contained in Ex.\ \ref{Pc09vbq3j7gabXrZiA}, the reader may calculate (with Lemma \ref{90q932r8dhd89cnr9}) that 
$\tD(v_0) = -125$, $\tD(v_1) = 63$, $\tD(v_2) = 62$, $\tD(v_3) = 6$, $\tD(v_4) = 2$, $\tD(v_5) = 0$ and $\tD(v_6) = -8$. 
Consequently, $\tD( \Neul ) = 0$.
\end{example}

\begin{lemma} \label {C0qowbdowX932epug7fnvd}
Let $v \in \Neul$.
\begin{enumerata}

\item If $\epsilon(v)=1$  and $v$ is not a node then $v=v_0$ and $\delta_{v_0}=1$.

\item If $\epsilon(v)=2$, $\tD(v)=0$ and $v \neq v_0$ then  $v$ is a node of type $[N_v,\dots,N_v]$ and $a_v=1$. 

\end{enumerata}
\end{lemma}

\begin{proof}
(a) Suppose that $v \in \Neul$ is such that $\epsilon(v)=1$  and $v \neq v_0$.
We have $\delta_v\ge3$, at most one element of $\Aeul_0$ is adjacent to $v$
and no element of $\Aeul \setminus \Aeul_0$ is adjacent to $v$; so
at least two vertices are adjacent to $v$, and exactly one of them belongs to $\Neul$ because $\epsilon(v)=1$;
so at least one element of $\Veul \setminus \Neul = \Deul$ is adjacent to $v$, so $v$ is a node.

By the above paragraph, if $v \in \Neul$, $\epsilon(v)=1$  and $v$ is not a node, then $v=v_0$.
Since $v_0$ is not a node and no arrow is adjacent to $v_0$, we have $\delta_{v_0} = \epsilon(v_0)=1$.
So (a) is true.

For (b), note that 
$$
\textstyle
0 = \tD(v) = \sigma(v) + (\epsilon(v)-2) (N_v-1) + N_v( 1 - \frac1{a_v} ) = \sigma(v) + N_v( 1 - \frac1{a_v} ),
$$
so $\sigma(v)=0$ and $a_v=1$.
Since $v$ is not a dicritical, there cannot be a dead end decorated by $1$ incident to $v$;
so the fact that  $a_v=1$ implies that no dead end is incident to $v$. 
Since $v$ is not a dicritical, no element of $\Aeul \setminus \Aeul_0$ is adjacent to $v$, so no arrow is adjacent to $v$.
Since $v \neq v_0$, we have $\delta_v > 2 = \epsilon(v)$, so there exists a $u \in \Veul \setminus \Neul$ adjacent to $v$.
Then $u$ is a dicritical and hence $v$ is a node.
The fact that $\sigma(v)=0$ implies that $k_x=1$ for all $x \in \Deul_v$, so $v$ is of type $[N_v, \dots, N_v]$.
\end{proof}

\begin{lemma} \label {lkjxfX0293wsod}
For all $v \in \Neul$ we have \  $\frac{\sigma(v)}{N_v} = \sum_{u \in \Deul_v}\left(1 - \frac1{k_u}\right)$.
\end{lemma}

\begin{proof}
If $v$ is not a node then both sides of the equality are equal to $0$.
If $v \in \Nd(\Teul)$ then $N_v = k_ud_u$ for all $u \in \Deul_v$, so 
$\frac{\sigma(v)}{N_v}
=\frac{\sum_{u \in \Deul_v} (k_u - 1)d_u}{N_v}
= \sum_{u \in \Deul_v} \frac{(k_u - 1)d_u}{k_ud_u}
= \sum_{u \in \Deul_v} \left(1 - \frac1{k_u}\right)$.
\end{proof}

The following triviality is very useful:

\begin{lemma} \label {trivialobs}
Let $(m_1,\dots,m_s)$ be a family of elements of $\Nat \setminus \{0\}$.
If $K \in \Reals$ is such that
$ \textstyle  \sum_{i=1}^s (1 - \frac1{m_i}) < K $,
then $| \setspec{ i }{ m_i>1 } | < 2K$.
In particular, if  $\sum_{i=1}^s (1 - \frac1{m_i}) < 1$ then at most one $i$ is such that $m_i>1$.
\end{lemma}

\begin{proof}
If $1 - \frac1{m_i} \neq0$ then $1 - \frac1{m_i} \ge \frac12$.
\end{proof}

\begin{notations} \label {92bty9df8349thg98}
Given $v \in \Neul$, we define $d(v)=N_v$ if $v$ is not a node,
and $d(v)=$ the gcd of the degrees of the dicriticals in $\Deul_v$ if $v$ is a node.
\end{notations}

\begin{remark} \label {F0934nofe8rg3p406egfh}
For all $v \in \Neul$ we have $d(v) \mid N_v$, and $d(v) = N_v$  if and only if $\sigma(v)=0$.
(See Lemma \ref{90q932r8dhd89cnr9}\eqref{o82y387jw9e23} and Def.\ \ref{aowi3efkWEGRGa9e93}.)
\end{remark}

\begin{lemma} \label {8ey8cdody27}
Let $v \in \Neul$ be such that $\epsilon(v)=1$ and $\tD(v) \le 0$. Then
$$
\tD(v) = 1 - \frac{d(v)}{a_v}
$$
and if $v \in \Nd( \Teul )$ then the following hold.
\begin{enumerata}

\item At most one $u \in \Deul_v$ satisfies $k_u>1$, and if such a $u$ exists then $a_v=1$.

\item The type of $v$ is $[d(v), N_v, \dots, N_v]$, and if $d(v) \neq N_v$ then $a_v=1$.

\end{enumerata}
\end{lemma}

\begin{proof}
Consider the case where $v \notin \Nd(\Teul)$.
By Lemma \ref{C0qowbdowX932epug7fnvd}, $v=v_0$ and $\delta_{v_0}=1$.
Since $\sigma(v_0)=0$ and $a_{v_0}=1$, Lemma \ref{90q932r8dhd89cnr9}\eqref{pc9vp23r09} gives  $\tD(v_0) = 1 - N_{v_0}$.
Since $d(v_0) = N_{v_0}$ it follows that the equality $\tD(v) = 1 - d(v)/a_v$ is true.

Until the end of the proof we assume that $v \in \Nd(\Teul)$.  We have
$$
\textstyle
0 \ge \tD(v) = \sigma(v) + (\epsilon(v)-2)(N_v-1) + N_v(1 - \frac1{a_v}) = \sigma(v) + (-1)(N_v-1) + N_v(1 - \frac1{a_v}) ,
$$
which implies that $\sigma(v) + N_v(1 - \frac1{a_v}) \le (N_v-1) < N_v$.
Dividing by $N_v$ gives
$$
\textstyle (1 - \frac1{a_v}) + \sum_{u \in \Deul_v} (1 - \frac1{k_u}) < 1 
$$
where we used Lemma \ref{lkjxfX0293wsod}.
If $u_1,\dots,u_s$ denote the distinct elements of $\Deul_v$ then Lemma \ref{trivialobs} says that at most one
member of the family $(a_v, k_{u_1}, \dots, k_{u_s})$ is greater than $1$.
This proves (a).  We may arrange the labelling of $u_1,\dots, u_s$ so that  $k_{u_1} \ge 1 = k_{u_2} = \cdots = k_{u_s}$.
Then the type of $v$ is $[d_{u_1}, \dots, d_{u_s}] = [\frac{N_{v}}{k_{u_1}}, N_{v}, \dots, N_v]$, so
$d(v) = \gcd(\frac{N_{v}}{k_{u_1}}, N_{v}, \dots, N_v) = \frac{N_{v}}{k_{u_1}}$,
which shows that the type of $v$ is $[d(v), N_v, \dots, N_v]$.
If $d(v) \neq N_v$ then $k_{u_1}>1$, so $a_v=1$.  This proves (b).
Note that $\sigma(v) = \sum_{u \in \Deul_v} (k_u-1)d_u = (k_{u_1}-1)d_{u_1} = N_v - d(v)$, so 
\begin{equation} \label {Dkjnxf923qwsd}
\textstyle  \tD(v) = (N_v-d(v)) + (-1)(N_v-1) + N_v(1 - \frac1{a_v}) = 1 + N_v -d(v) - \frac{N_v}{a_v}.
\end{equation}
If $d(v)=N_v$ then \eqref{Dkjnxf923qwsd} simplifies to $\tD(v) = 1 - \frac{d(v)}{a_v}$;
if $d(v) \neq N_v$ then $a_v=1$, and again \eqref{Dkjnxf923qwsd} simplifies to $\tD(v) = 1 - \frac{d(v)}{a_v}$. 
So $\tD(v) = 1 - \frac{d(v)}{a_v}$ is true and the Lemma is proved.
\end{proof}

\section*{Geometric interpretation of $\tD(\Neul)$}

Thm \ref{918235071yrsj2dhry} and Remarks \ref{pc09v4rdh8drcj7yu} and \ref{v87bnd5hmzvfdjsm6k7oxdg} are not used anywhere else in the article.
Their only purpose is to justify equality \eqref{ckjvbo293djbf0q2}, claimed in the introduction.
Refer to \ref{p0c9vh394envpwI2}, \ref{oifnue09e2039} and \ref{c9v39rf0eX9e4np8glr9t8} for the definitions of $\Deul$, $d_u$ and $M(\Teul)$.

\begin{theorem} \label {918235071yrsj2dhry}  
Let $\Teul$ be a minimally complete Newton tree at infinity.
Let $D(\Teul)=\sum_{u \in \Deul} d_u$ and  $D'(\Teul)=\sum_{u \in \Deul} (d_u-1)$. Then
$$
\tD(\Neul) = \Delta(\Neul) - D'(\Teul)  = 2 - M(\Teul) -  D(\Teul) .
$$
\end{theorem}

\begin{proof}
We have $\tD(\Neul) = \Delta(\Neul) - D'(\Teul)$ because $( \Deul_v )_{ v \in \Nd(\Teul) }$ is an  f-partition of $\Deul$.
Thm 2.22(1) of \cite{CND15:simples} states that $M(\Teul) + \Delta(\Neul) = 2- | \Deul |$;\footnote{Thm 2.22 of \cite{CND15:simples} assumes
that $\Teul$ has at least two points at infinity, but in fact that assumption is not used in the proof of part (1) of that result.}
since $D'(\Teul) = D(\Teul) - |\Deul|$, we are done.
\end{proof}

\begin{remark} \label {pc09v4rdh8drcj7yu}	
Let us recall how a primitive polynomial $f \in \Comp[x,y]$ determines a tree $\Teul(f;x,y)$.
First note that $f : \Comp^2 \to \Comp$ extends to a rational map $\Phi_0 : \proj^2 \dashrightarrow \proj^1$.
Choose a composition $\pi :X\to \proj^2$ of blowing-up maps centered at points of 
the line at infinity $\mathbb{L}_\infty = \proj^2 \setminus \Comp^2$
such that $\Phi=\Phi_0\circ \pi : X \to \proj^1$ is a morphism. Then the curve
$$
\mathcal{D}=\pi^{-1}(\mathbb{L}_{\infty}) 
$$
is a tree of projective lines with simple normal crossings and is the complement of $\Comp^2$ in the nonsingular projective surface $X$.
If $E$ is an irreducible component of $\mathcal{D}$ satisfying $\Phi(E)=\proj^1$, one says that $E$ is a  dicritical component of $\mathcal{D}$;\footnote{In the
introduction, we used the term ``dicritical curve of $f$'' to refer to these dicritical components.}
then, for generic $t \in \Comp$, the closure in $X$ of the affine curve $f(x,y)=t$
intersects $E$ in a finite (and nonzero) number of points that is called the  degree of the dicritical $E$.

We consider the dual graph of $\mathcal{D}$ in $X$.
This means that we represent each irreducible component of $\mathcal{D}$ by a vertex
and that we put an edge between two vertices when the corresponding components intersect.
Note that this is a tree.
We represent the branches of the curve by arrows;
more precisely,
for each vertex $u$ representing a dicritical component $E$ of degree $d_E$, we attach $d_E$ arrows to $u$
and we decorate these arrows with the symbol $(1)$;
note that $u$ now has valency $\ge d_E+1 \ge 2$.
The vertex representing $\mathbb{L}_\infty$ is denoted by $v_0$ and is called the root of the tree.
Each edge is decorated by two integers, one near each of its ends;
the integers decorating the edges are obtained from the self-intersection numbers of the divisors as explained in \cite[Chap.\ 5]{Neumann:LinksInfinity}. 
In the terminology of \cite{Neumann:LinksInfinity}, the object that we have at this point is called a rooted  RPI splice diagram.
For each vertex $v$ of valency $1$ and distinct from $v_0$, we replace $v$ by an arrow decorated with $(0)$.
The decorated rooted tree $\Teul_0$ obtained in this manner satisfies the definition of an abstract Newton tree at infinity (Def.\ \ref{p9823p98p2d}),
and is generic (Def.\ \ref{c0j23oed9vbaCo8f}). 
We stress that the only difference between this $\Teul_0$ and Neumann's splice diagram is that we replaced 
each vertex $v \neq v_0$ of valency $1$ by  an arrow decorated with $(0)$.

Alternatively, one can obtain $\Teul_0$
from the Puiseux pairs of the branches at infinity as explained in \cite{Eisenbud_Neumann_1985} and \cite{Neumann:LinksInfinity},
or from the equations of the successive Newton polygons in the Newton algorithm as explained in \cite{CN11-Russellfest}.
In all cases, $\Teul_0$ is an abstract Newton tree at infinity and is generic.

We define $\Teul(f;x,y)$ to be the unique minimally complete Newton tree at infinity which is equivalent to $\Teul_0$ (Lemma \ref{cvp9wpe9p9d}).

One can see that there is a natural injective map $v \mapsto C_v$ from the set of vertices of $\Teul(f;x,y)$
to the set of irreducible components of $\mathcal{D}$,
and that a vertex $v$ of $\Teul(f;x,y)$ belongs to $\Neul$ if and only if $C_v$ is a component of the fiber of $\Phi$ over the point $\infty \in \proj^1$
(more precisely, the multiplicity of $C_v$ in the divisor $\Phi^{-1}(\infty)$ is equal to $N_v$).
\end{remark}

\begin{remark} \label {v87bnd5hmzvfdjsm6k7oxdg}
Let $f \in \Comp[x,y]$ be a primitive polynomial and let $\Teul = \Teul(f;x,y)$.
As explained in Rem.\ \ref{pc09v4rdh8drcj7yu}, $\Teul$ is a minimally complete abstract Newton tree at infinity.
We claim that $\tD(\Neul) = 2g$, where $g$ is the genus of the generic fiber of $f : \Comp^2 \to \Comp$.
To see this, first observe that  $\tD(\Neul) = 2 - M(\Teul) -  D(\Teul) $ by Thm \ref{918235071yrsj2dhry}; it therefore suffices to show that
\begin{equation} \label {pco8v76jretryusie45n54tf}
2 - M(\Teul) -  D(\Teul) = 2g.
\end{equation}
Let $C \subset \Comp^2$ denote a generic fiber of $f$;
then Thm 4.3 of \cite{Neumann:LinksInfinity} shows that the Euler characteristic of $C$ is equal to $M(\Teul)$.
Since $D(\Teul)$ is equal to the number $n_\infty$ of places at infinity of $C$, \eqref{pco8v76jretryusie45n54tf}
follows from the well-known equality $\chi(C) = 2 - 2g - n_\infty$.
So $\tD(\Neul) = 2g$.
\end{remark}

\section{Characteristic numbers}
\label {Section:Characteristicnumbers}

{\it Throughout this section, $\Teul$ is a minimally complete abstract Newton tree at infinity.}

\medskip

In this section we consider the set $P=P(\Teul)$ whose elements are all pairs $(u,e)$ such that $u\in\Neul$
and $e$ is an edge in $\Neul$ that is incident to $u$. We also define a partial order $\preceq$ on the set $P$.
By induction on the poset $(P,\preceq)$, we define a family $\big( c(u,e) \big)_{ (u,e) \in P }$ of positive rational numbers
and prove several of its properties.
We call $c(u,e)$ the characteristic number of $(u,e)$.
The most important property of characteristic numbers is given in Thm \ref{xncoo9qwdx9}, which is a fundamental result in the 
theory of Newton trees.

\begin{notation}
If $x,y \in \Veul \cup \Aeul$ then the unique path from $x$ to $y$ is denoted $\gamma_{x,y}$.
\end{notation}

\begin{notation} \label {dp9cvp0293ewd}
Let $w \in \Veul$ and let $A$ be a subset of $\Aeul \setminus \Aeul_0$.
Let $H(w,A)$ be the set of all pairs $(e,u)$ such that $u$ is a vertex, $e$ is an edge of $\Teul$ incident to $u$, and:
\begin{gather*}
\text{for all $\alpha \in A$, $u$ is in $\gamma_{w,\alpha}$ but $e$ is not in $\gamma_{w,\alpha}$.} 
\end{gather*}
Let $\hat H(w,A) = \setspec{ (e,u) \in H(w,A) }{ u \neq w }$.
Also define the integers
$$
\textstyle
h(w,A) = \prod_{(e,u) \in H(w,A)} q(e,u) \quad \text{and} \quad \hat h(w,A) = \prod_{(e,u) \in \hat H(w,A)} q(e,u)
$$
and observe that $h(w,A) \mid x_{w,\alpha}$ and $\hat h(w,A) \mid \hat x_{w,\alpha}$ for all $\alpha \in A$.
\end{notation}

\begin{lemma} \label {civ023edkdkfpw}  
Let $S$ be a nonempty collection of dicritical vertices,
let $d$ be the gcd of the degrees of the dicritical vertices belonging to $S$, and let
$$
A = \setspec{ \alpha \in \Aeul\setminus\Aeul_0 }{ \text{$\alpha$ is adjacent to some element of $S$} }.
$$
Then for every $w \in \Veul$,
$$
\textstyle \text{$h(w,A) d$ divides  $\sum_{\alpha \in A} x_{w,\alpha}$ \quad and \quad $\hat h(w,A) d$ divides  $\sum_{\alpha \in A} \hat x_{w,\alpha}$.}
$$
In particular,  $\sum_{\alpha \in A} \frac{x_{w,\alpha}}{a_w}$ and $\sum_{\alpha \in A} \hat x_{w,\alpha}$ belong to $d \Integ$.
\end{lemma}

\begin{proof}
Let $w \in \Veul$.
As noted in \ref{dp9cvp0293ewd}, we have
\begin{equation} \label {c9f29wudsf0we}
\text{$h(w,A) \mid x_{w,\alpha}$ \quad and \quad $\hat h(w,A) \mid \hat x_{w,\alpha}$ \qquad for all $\alpha \in A$.}
\end{equation}

For each $v \in S$, let $B_v = \setspec{ \alpha \in \Aeul\setminus\Aeul_0 }{ \text{$\alpha$ is adjacent to $v$} }$.

Let $v \in S$. Note that $|B_v|$ is the degree of the dicritical $v$, so $|B_v| = c_v d$ for some $c_v  \in \Integ$.
It follows from Lemma \ref{923ip38kjfbve09}\eqref{93f394ejf9i} that $x_{w,\alpha}$ has the same value for all $\alpha \in B_v$,
so \eqref{c9f29wudsf0we} implies that there exists $g_v \in \Integ$ such that $x_{w,\alpha} = h(w,A) g_v$ for all $\alpha \in B_v$.
Similarly, $\hat x_{w,\alpha}$ has the same value for all $\alpha \in B_v$,
so \eqref{c9f29wudsf0we} implies that there exists $\hat g_v \in \Integ$ such that $\hat x_{w,\alpha} = \hat h(w,A) \hat g_v$ for all $\alpha \in B_v$.
Thus 
\begin{align*}
\textstyle \sum_{\alpha \in B_v} x_{w,\alpha} & = \textstyle \sum_{\alpha \in B_v} h(w,A) g_v = |B_v| h(w,A) g_v = c_v d h(w,A) g_v, \\
\textstyle \sum_{\alpha \in B_v} \hat x_{w,\alpha} & = \textstyle \sum_{\alpha \in B_v} \hat h(w,A) \hat g_v
= |B_v| \hat h(w,A) \hat g_v = c_v d \hat h(w,A) \hat g_v, 
\end{align*}
showing that for each $v \in S$ we have
$$
\textstyle \text{$\sum_{\alpha \in B_v} x_{w,\alpha}  \in h(w,A) d \Integ$ \quad and \quad $\sum_{\alpha \in B_v} \hat x_{w,\alpha}  \in \hat h(w,A) d \Integ$.}
$$
Since $(B_v)_{ v \in S }$ is an f-partition of $A$, 
$\sum_{\alpha \in A} x_{w,\alpha} = \sum_{v \in S} \sum_{\alpha \in B_v} x_{w,\alpha} \in h(w,A) d  \Integ$
and similarly $\sum_{\alpha \in A} \hat x_{w,\alpha} \in \hat h(w,A) d  \Integ$.
The fact that  $\sum_{\alpha \in A} \frac{x_{w,\alpha}}{a_w} \in d \Integ$ is a consequence of
$\sum_{\alpha \in A} x_{w,\alpha} \in h(w,A) d  \Integ$ and of $a_w \mid h(w,A)$.
The fact that $\sum_{\alpha \in A} \hat x_{w,\alpha} \in d \Integ$ is clear.
\end{proof}

\begin{notation} \label {0db3ukifxoms7dcdOoOoxde}
Define $A_v = \setspec{ \alpha \in \Aeul \setminus \Aeul_0 }{ \text{$\alpha$ is adjacent to some element of $\Deul_v$}}$
for each $v \in \Neul$, and observe that if $v \notin \Nd(\Teul)$ then $A_v = \emptyset$ (because $\Deul_v = \emptyset$).
\end{notation}

\begin{corollary} \label {p9cv355yutekvCs7df}
If $v \in \Neul$ and $w \in \Veul$ then  $\sum_{ \alpha \in A_v } \frac{x_{w,\alpha}}{a_w}$ and
$\sum_{ \alpha \in A_v } \hat x_{w,\alpha}$ belong to $d(v)\Integ$.
\end{corollary}

\begin{proof}
If $v \notin \Nd(\Teul)$ then $A_v = \emptyset$, so the two sums are zero. 
If $v \in \Nd(\Teul)$ then the result follows from Lemma \ref{civ023edkdkfpw}.
\end{proof}

\begin{notations} \label {jkcnv9wisdxco}
If $x_1, \dots, x_n \in \Rat$, let $\langle x_1, \dots, x_n \rangle$ denote the $\Integ$-submodule of $\Rat$ generated by $x_1, \dots, x_n$.
Then $\langle x_1, \dots, x_n \rangle$ is a free $\Integ$-module of rank $1$ or $0$,
so there exists a unique $\xi \in \Rat$ such that $\xi\ge0$ and $\langle x_1, \dots, x_n \rangle = \xi\Integ$. 
We refer to this unique $\xi$ as the \textit{gcd} of $x_1, \dots, x_n$, and we write $\gcd(x_1, \dots, x_n) = \xi$.
In the special case where $x_1, \dots, x_n \in \Integ$, this definition agrees with the usual gcd of $x_1, \dots, x_n$.
Gcd's of rational numbers have (in particular) the following properties:
\begin{itemize}

\item If $a, x_1, \dots, x_n \in \Rat$ then  $\gcd(a x_1, \dots, a x_n) = |a| \gcd( x_1, \dots,  x_n)$.

\item If $m_1, \dots, m_n \in \Integ$ and $x \in \Rat$ then  $\gcd(x, m_1 x, \dots, m_n x) = |x|$.

\item If $x_1, \dots, x_n \in \Rat$, and if $m$ is any positive integer such that $mx_i \in \Integ$ for all $i$,
then $\gcd(x_1, \dots, x_n) = \frac1m \gcd(m x_1, \dots, mx_n)$ where  $\gcd(m x_1, \dots, mx_n)$ is the usual gcd of 
the integers $m x_1, \dots, mx_n$.

\end{itemize}
\end{notations}

\begin{notations} \label {kCfnp293wdw0}
Let $P=P(\Teul)$ be the set of pairs $(u,e)$ where $u\in\Neul$ and $e$ is
an edge in $\Neul$ that is incident to $u$.
Given $(u,e) \in P$, define:
\begin{itemize}
\item $\Nd(u,e) = 
\setspec{ z \in \Nd(\Teul) }{ \text{$e$ is in $\gamma_{u,z}$} }$
\item $A(u,e) = 
\setspec{ \alpha \in \Aeul \setminus \Aeul_0 }{ \text{$e$ is in $\gamma_{u,\alpha}$} }$
\item $A'(u,e) = (\Aeul \setminus \Aeul_0) \setminus A(u,e)
=\setspec{ \alpha \in \Aeul \setminus \Aeul_0 }{ \text{$e$ is not in $\gamma_{u,\alpha}$} }$
\item $p(u,e) = \sum_{ \alpha \in A(u,e) } \hat x_{ u,\alpha }$
\item $p'(u,e) = \sum_{ \alpha \in A'(u,e) } \hat x_{ v,\alpha }$ where  $v$ is defined by $e = \{u,v\}$.
\end{itemize}
\end{notations}

\begin{remark}
If $| \Neul | \ge 2$ then $P(\Teul) \neq \emptyset$, because $\Neul$ is connected.
If $| \Neul | < 2$ then $\Neul = \{ v_0 \}$ and $P(\Teul) = \emptyset$. 
This means that all results from here to the end of the section are trivially true when $| \Neul | < 2$.
\end{remark}

\begin{definition} \label {d9f239cw093eir}
We define a partial order $\prec$ on  $P=P(\Teul)$ by declaring that $(u',e') \prec (u,e)$
if and only if $\gamma_{u',u}$ contains $e$ but not $e'$.
This order relation allows us to use induction in definitions and proofs.
In this regard, the following remarks are useful.

Let $(u,e) \in P$. Then $e = \{ u, u_0 \}$ for some $u_0 \in \Neul$ uniquely determined by $(u,e)$.
Observe that $\epsilon(u_0)\ge1$, 
and that $(u,e)$ is minimal in $(P, \preceq)$ if and only if $\epsilon(u_0)=1$.

If $(u,e)$ is not minimal,
let $u,u_1,\dots,u_n$ be the distinct elements of $\Neul$ that are adjacent to $u_0$
and let $e_i = \{u_0,u_i\}$ for $i=1,\dots,n$.
Note that $\epsilon(u_0)=n+1$ and that $n \ge 1$.
Then $(u_0,e_i) \prec (u,e)$ for all $i=1,\dots,n$.
Moreover, $(u_0,e_1), \dots, (u_0,e_n)$
are the maximal elements of the set $\setspec{ (u',e') \in P }{ (u',e') \prec (u,e) }$.
We call $(u_0,e_1), \dots, (u_0,e_n)$ the \textit{immediate predecessors} of $(u,e)$.
\end{definition}

\begin{definition} \label {kcjfnp0293wd}
We inductively define a rational number $c(u,e)$ for each $(u,e) \in P$.
Consider $(u,e) \in P$ with $e = \{u,u_0\}$.
See \ref{92bty9df8349thg98} for the definition of $d(u_0) \in \Nat \setminus \{0\}$.
\begin{enumerate}

\item[(i)] If $(u,e)$ is minimal (i.e., $\epsilon(u_0)=1$), we define $c(u,e) = d(u_0) / a_{u_0}$.

\item[(ii)] If $(u,e)$ is non-minimal (i.e., $\epsilon(u_0)>1$), we define
$$
c(u,e) = \frac{\gcd(d(u_0), c(u_0,e_1), \dots, c(u_0,e_n))}{a_{u_0}},
$$
where $(u_0,e_1), \dots, (u_0,e_n)$ are the immediate predecessors of $(u,e)$ (cf.\ Def.\ \ref{d9f239cw093eir}) and where
we use gcds of rational numbers (cf.\ \ref{jkcnv9wisdxco}).

\end{enumerate}
We call $c(u,e)$ the \textit{characteristic number} of $(u,e)$. 
\end{definition}

\begin{example} \label {cov7btiw67d78vujew9}
Continuation of Fig.\ \ref{723dfvcjp2q98ewdywe}, Ex.\ \ref{Pc09vbq3j7gabXrZiA} and \ref{cIuCbgepwej6DsJ37655enm}.
Using the data contained in Ex.\ \ref{Pc09vbq3j7gabXrZiA}, we may calculate $c(u,e)$ for all $(u,e) \in P$.
The results are given here:
%%%%%%%%%%%%%%%%%%%%%%%%%%%%%%%%%%%%%%%%%%%%%%%%%%%%%%%%
$$
\scalebox{1}{
\setlength{\unitlength}{1mm}%
\begin{picture}(95,21)(-5,-16)
%%%%%%%%%%%%%%%%%%%%%%%%%%%%%%%%%%%%%%%%%%%%%%
\multiput(0,0)(20,0){5}{\circle{1}}
\put(0,-15){\circle{1}}
\multiput(0.5,0)(20,0){4}{\line(1,0){19}}
\put(0,-0.5){\line(0,-1){14}}
\put(40,-15){\circle{1}}
\put(40,-0.5){\line(0,-1){14}}
%%%%%%%%%%%%%%%%%%%%%%%%%%%%%%%%%%%%%%%%%%%%%%
\put(3,1){\makebox(0,0)[bl]{\tiny ${ 3 }$}}
\put(17,1){\makebox(0,0)[br]{\tiny ${ \frac13 }$}}
\put(23,1){\makebox(0,0)[bl]{\tiny ${ 9 }$}}
\put(37,1){\makebox(0,0)[br]{\tiny ${ \frac19 }$}}
\put(43,1){\makebox(0,0)[bl]{\tiny ${ 63 }$}}
\put(61,1){\makebox(0,0)[bl]{\tiny ${ 126 }$}}
\put(57,1){\makebox(0,0)[br]{\tiny ${ \frac19 }$}}
\put(77,1){\makebox(0,0)[br]{\tiny ${ \frac19 }$}}
\put(1,-3){\makebox(0,0)[tl]{\tiny ${ 1 }$}}
\put(41,-12){\makebox(0,0)[bl]{\tiny ${ \frac19 }$}}
\put(41,-3){\makebox(0,0)[tl]{\tiny ${ 9 }$}}
\put(1,-12){\makebox(0,0)[bl]{\tiny ${ 1 }$}}
\put(38,-15){\makebox(0,0)[r]{\tiny ${ v_6 }$}}
\put(-2,-15){\makebox(0,0)[r]{\tiny ${ v_5 }$}}
\put(-2,0){\makebox(0,0)[r]{\tiny ${ v_4 }$}}
\put(20,-1.5){\makebox(0,0)[t]{\tiny ${ v_3 }$}}
\put(38,-1.5){\makebox(0,0)[t]{\tiny ${ v_2 }$}}
\put(60,-1.5){\makebox(0,0)[t]{\tiny ${ v_1 }$}}
\put(80,-1.5){\makebox(0,0)[t]{\tiny ${ v_0 }$}}
%%%%%%%%%%%%%%%%%%%%%%%%%%%%%%%%%%%%%%%%%%%%%%
\end{picture} \quad
\raisebox{12\unitlength}{\begin{minipage}[t]{8cm} 
Meaning: $c(v_1, \{v_1,v_2\}) = \frac19$,\\ $c(v_2, \{v_1,v_2\}) = 63$, etc.
\end{minipage}}
}
$$
\end{example}

\begin{lemma} \label {Ocivbghn5dmkuyesK}
For each $(u,e) \in P$, the rational number $c(u,e)$ is strictly positive.
\end{lemma}

\begin{proof}
We write $\Rat_{>0} = \setspec{ x \in \Rat }{ x>0 }$.
Let $(u,e) \in P$, with $e = \{u,u_0\}$.
If $(u,e)$ is minimal then 
$c(u,e) = d(u_0)/a_{u_0}$ where $d(u_0), a_{u_0} \in \Nat\setminus\{0\}$, so $c(u,e) \in \Rat_{>0}$.
If $(u,e)$ is not minimal, 
and if its immediate predecessors $(u_0,e_1), \dots, (u_0,e_n)$ satisfy
$c(u_0,e_i) \in \Rat_{>0}$ for all $i=1,\dots,n$,
then $c(u,e) = \gcd(d(u_0), c(u_0,e_1), \dots, c(u_0,e_n)) / a_{u_0} \in \Rat_{>0}$.
It follows by induction that $c(u,e) \in \Rat_{>0}$ for all $(u,e) \in P$.
\end{proof}

\begin{notation}
If $x,y \in \Veul$ then $\alpha_{[x,y]} = \prod_{i=1}^n a_{v_i}$ where $v_1,\dots,v_n$ are the vertices defined by
$\gamma_{x,y} = (v_1, \dots, v_n)$.
We also define $\alpha_{[x,y)} = \prod_{i=1}^{n-1} a_{v_i}$ and $\alpha_{(x,y]} = \prod_{i=2}^{n} a_{v_i}$,
where $\alpha_{[x,y)} = 1 = \alpha_{(x,y]}$ when $x=y$.
\end{notation}

\begin{lemma}  \label {kjwoeid9cse9c}
Let $(u,e) \in P$.
\begin{enumerata}

\item If $(u',e') \in P$ satisfies $(u',e') \prec (u,e)$, then\ \ $c(u',e') \in \alpha_{[u',u)} c(u,e) \Integ$.

\item For each $z \in \Nd(u,e)$,\ \ $d(z) \in \alpha_{ [z,u) } {c(u,e)} \Integ $.

\item If $\Nd(u,e) \neq \emptyset$ then  $\gcd{ \setspec{ d(z) }{ z \in \Nd(u,e) } } \in  {c(u,e)} \Integ $. 

\end{enumerata}
\end{lemma}

\begin{proof}
Given $(u,e) \in P$, let $\Pscr_1(u,e)$ and $\Pscr_2(u,e)$ be the following statements:

\medskip

\noindent $\Pscr_1(u,e)$:\ \textit{Every $(u',e') \in P$ satisfying $(u',e') \prec (u,e)$  also satisfies\ \ $c(u',e') \in \alpha_{[u',u)} c(u,e) \Integ$.}

\smallskip

\noindent $\Pscr_2(u,e)$:\ \textit{For each $z \in \Nd(u,e)$,\ \ $d(z) \in \alpha_{ [z,u) } {c(u,e)} \Integ $.}

\medskip
Let us prove, by induction, that the statement ``$\Pscr_1(u,e)$ and $\Pscr_2(u,e)$'' is true for all $(u,e) \in P$.
This will prove that (a) and (b) are true.

Let $(u,e)$ be a minimal element of $P$.
Obviously, $\Pscr_1(u,e)$ is true.
To prove $\Pscr_2(u,e)$,  write $e=\{u,u_0\}$ and let $z \in \Nd(u,e) \subseteq \{u_0\}$.  
Then $z=u_0$, so $d(z) = d(u_0) = a_{u_0} c(u,e) = \alpha_{ [z,u) } c(u,e) \in \alpha_{ [z,u) } {c(u,e)} \Integ $,
so $\Pscr_2(u,e)$ is true.

Consider a non-minimal $(u,e) \in P$ and write $e=\{u,u_0\}$;
let $(u_0,e_1), \dots, (u_0,e_n)$ be the immediate predecessors of $(u,e)$
and assume that $\Pscr_1(u_0,e_i)$ and $\Pscr_2(u_0,e_i)$ are true for all $i = 1, \dots, n$;
to finish the proof, we have to deduce that $\Pscr_1(u,e)$ and $\Pscr_2(u,e)$ are true.

To prove that $\Pscr_1(u,e)$ is true,
we consider $(u',e') \in P$ satisfying $(u',e') \prec (u,e)$, and we have to show that
\begin{equation}  \label {Eocif2039wsdcx0}
c(u',e') \in \alpha_{[u',u)} c(u,e) \Integ.
\end{equation}
Since $(u',e') \prec (u,e)$, there exists $i$ such that $(u',e') \preceq (u_0,e_i)$.
Note that
\begin{equation} \label {kdF01923w9dqWjs}
c(u_0,e_i) \in \gcd( d(u_0),  c(u_0,e_1), \dots, c(u_0,e_n)) \Integ = a_{u_0} c(u,e) \Integ,
\end{equation}
so in particular \eqref{Eocif2039wsdcx0} is true when $(u',e') = (u_0,e_i)$.
If $(u',e') \prec (u_0,e_i)$ then, since $\Pscr_1(u_0,e_i)$ is assumed to be true,
we have $c(u',e') \in \alpha_{[u',u_0)} c(u_0,e_i) \Integ$.
We also have $c(u_0,e_i) \in a_{u_0} c(u,e) \Integ$ by \eqref{kdF01923w9dqWjs},
and it is clear that $\alpha_{[u',u_0)} a_{u_0} = \alpha_{[u',u)}$,
so  \eqref{Eocif2039wsdcx0} is true in this case too.
Since  \eqref{Eocif2039wsdcx0} is true in all cases, $\Pscr_1(u,e)$ is true.

To prove that $\Pscr_2(u,e)$ is true, we consider $z \in \Nd(u,e)$ and we have to show that
$d(z) \in \alpha_{ [z,u) } {c(u,e)} \Integ $.  Note that one of the following is true:
\begin{enumerate}

\item[(i)] $z = u_0$

\item[(ii)] $z \in \Nd(u_0,e_i)$ for some $i = 1, \dots, n$.

\end{enumerate}

In case (i), 
$ d(z) = d(u_0) \in  \gcd( d(u_0),  c(u_0,e_1), \dots, c(u_0,e_n)) \Integ = a_{u_0} c(u,e) \Integ
= \alpha_{ [z,u) } {c(u,e)} \Integ $.

In case (ii), the assumption that $\Pscr_2(u_0,e_i)$ is true implies that 
$d(z) \in \alpha_{[z,u_0)} c(u_0,e_i) \Integ$, and we have $c(u_0,e_i) \in a_{u_0} c(u,e) \Integ$ by \eqref{kdF01923w9dqWjs}, so
$d(z) \in \alpha_{ [z,u) } {c(u,e)} \Integ $.

This shows that $d(z) \in \alpha_{ [z,u) } {c(u,e)} \Integ $ in all cases, so $\Pscr_2(u,e)$ is true.
This completes the proof that $\Pscr_1(u,e)$ and $\Pscr_2(u,e)$ are true for all $(u,e) \in P$, so (a) and (b) are proved.
Assertion (c) follows from (b).
\end{proof}

\begin{notation} \label {Cjkf0293wdipwos}
Let $z \in \Nd(\Teul)$. 
For each $v \in \Deul_z$, let $e_v$ temporarily denote the edge $\{v,z\}$, and define
$$
y_v = \prod_{v' \in \Deul_z \setminus\{v\}} q(e_{v'},z).
$$
As the  $q(e_{v'},z)$ (with $v' \in \Deul_z$) are pairwise relatively prime,
it follows that 
$$
\gcd\setspec{ y_v }{ v \in \Deul_z } = 1
$$
(this is true even if $\Deul_z=\{v\}$, in which case we have $y_v=1$).  We also define
$$
\qdic(z) =  \prod_{v \in \Deul_z} q(e_{v},z) .
$$
\end{notation}

See Notations \ref{kCfnp293wdw0} for $p(u,e)$ and $p'(u,e)$.

\begin{theorem}  \label {xncoo9qwdx9}
For all $(u,e) \in P(\Teul)$, we have $p(u,e), p'(u,e), N_u \in c(u,e) \Integ$.
\end{theorem}

\begin{proof}
Let $(u,e) \in P = P(\Teul)$. Then $N_u = Q(e,u) p(u,e) + q(e,u) p'(u,e)$, so it suffices to show that $p(u,e), p'(u,e) \in c(u,e)\Integ$.
We first prove that  $p(u,e) \in c(u,e)\Integ$.
If $\Nd(u,e) = \emptyset$ then $A(u,e) = \emptyset$ and hence $p(u,e) = 0 \in c(u,e)\Integ$.
If $\Nd(u,e) \neq \emptyset$, define $d = \gcd{ \setspec{ d(z) }{ z \in \Nd(u,e) } }$;
we have $d \in c(u,e) \Integ $ by Lemma \ref{kjwoeid9cse9c} and  $p(u,e) \in d\Integ$ by Lemma \ref{civ023edkdkfpw},
so $p(u,e) \in c(u,e) \Integ$.

There remains to show that $p'(u,e) \in c(u,e) \Integ$.  We do this by induction.
For each $(u,e) \in P$, let $\Pscr(u,e)$ be the statement $p'(u,e) \in c(u,e) \Integ$.  
Let us prove that $\Pscr(u,e)$ is true for all $(u,e) \in P$.
In the argument below we use the following notation:
given $x,y \in \Rat$, we write $x \mid y$ if and only if there exists $n \in \Integ$ satisfying $nx=y$.
In other words, $x \mid y$ $\Leftrightarrow$ $y \in x\Integ$.

Consider the case where $(u,e)$ is a minimal element of $P$. Then $\epsilon(u_0)=1$, where $e = \{u,u_0\}$.
Now $u_0$ either is or is not a node; we distinguish the two cases.

If $u_0$ is not a node then $u_0 = v_0$ and $\delta_{ v_0 } = 1$ by Lemma \ref{C0qowbdowX932epug7fnvd}.
Then $c(u,e) = d(u_0)/ a_{u_0} = N_{v_0} = p'(u,e)$, so $\Pscr(u,e)$ is true.

Assume that $u_0$ is a node and observe that $\{ A_{u_0}, A'(u,e) \}$ is a partition of $\Aeul\setminus\Aeul_0$. For each $w \in \Deul_{u_0}$ we have
$$
0 = \frac{N_w}{a_w} = \sum_{\alpha \in A_{u_0}} \frac{x_{w,\alpha}}{a_w} + \sum_{\alpha \in A'(u,e)} \frac{x_{w,\alpha}}{a_w} .
$$
By Cor.\ \ref{p9cv355yutekvCs7df} there exists $K \in \Integ$ such that 
$\sum_{\alpha \in A_{u_0}} \frac{x_{w,\alpha}}{a_w} = K d(u_0) = K a_{u_0} c(u,e)$.
For each $\alpha \in A'(u,e)$ we have $\frac{x_{w,\alpha}}{a_w} = y_w a_{u_0} \hat x_{u_0,\alpha}$, so
$\sum_{\alpha \in A'(u,e)} \frac{x_{w,\alpha}}{a_w} =  y_w a_{u_0} p'(u,e)$. Thus
$0 =  K a_{u_0} c(u,e) +  y_w a_{u_0} p'(u,e)$ and hence $c(u,e) \mid y_w p'(u,e)$.
Since this holds for all $w \in \Deul_{u_0}$ and $\gcd\setspec{ y_w }{ w \in \Deul_{u_0} } = 1$,
it follows that $\Pscr(u,e)$ is true. 
So $\Pscr(u,e)$ is true whenever $(u,e)$ is minimal.

\medskip
Suppose that $(u,e)$ is non-minimal and that $\Pscr(u',e')$ is true for all $(u',e') \in P$ such
that $(u',e') \prec (u,e)$.
Write $e = \{u,u_0\}$ and let $(u_0,e_1), \dots, (u_0,e_n)$ be the immediate predecessors of $(u,e)$.
The inductive hypothesis implies that $\Pscr(u_0,e_i)$ is true for each $i = 1, \dots, n$, i.e.,
$$
c(u_0,e_i) \mid p'(u_0,e_i) \quad \text{for each $i \in \{1, \dots, n\}$.} 
$$
We also have $c(u_0,e_i) \mid p(u_0,e_i)$ (for all $i$) by the first part of the proof.
Since $a_{u_0} c(u,e) = \gcd( d(u_0), c(u_0,e_1), \dots, c(u_0,e_n) )$, it follows that 
\begin{equation} \label {JEILdc09vn320efeo}
a_{u_0} c(u,e) \mid  \gcd( d(u_0), p'(u_0,e_1), \dots, p'(u_0,e_n), p(u_0,e_1), \dots, p(u_0,e_n) ) .
\end{equation}
We use the following notation:
$$
\scalebox{.9}{
\setlength{\unitlength}{1mm}%
\begin{picture}(46,27)(-2,-16)
%%%%%%%%%%%%%%%%%%%%%%%%%%%%%%%%%%%%%%%%%%%%%%
\multiput(0,0)(20,0){2}{\circle{1}}
\put(40.8944,10.4472){\circle{1}}
\put(40.8944,-10.4472){\circle{1}}
\put(0.5,0){\line(1,0){19}}
\put(20,-0.5){\vector(0,-1){15}}
\put(19.64644,-0.353553){\line(-1,-2){4.5}}
\put(19.64644,-0.353553){\line(-1,-1){7}}
\put(20.44721,.2236){\line(2,1){20}}
\put(20.44721,-.2236){\line(2,-1){20}}
%%%%%%%%%%%%%%%%%%%%%%%%%%%%%%%%%%%%%%%%%%%%%%
\put(0,1){\makebox(0,0)[b]{\tiny ${ u }$}}
\put(40.8944,9){\makebox(0,0)[t]{\tiny ${ u_1 }$}}
\put(40.8944,-9){\makebox(0,0)[b]{\tiny ${ u_n }$}}
\put(8,1){\makebox(0,0)[b]{\tiny ${ e }$}}
\put(20,1){\makebox(0,0)[b]{\tiny ${ u_0 }$}}
\put(16,1){\makebox(0,0)[br]{\tiny ${ t }$}}
\put(26,4){\makebox(0,0)[br]{\tiny ${ z_1 }$}}
\put(23,-1.5){\makebox(0,0)[bl]{\tiny ${ z_n }$}}
\put(34,8){\makebox(0,0)[br]{\tiny ${ e_1 }$}}
\put(31,-5.5){\makebox(0,0)[bl]{\tiny ${ e_n }$}}
\put(21,-15){\makebox(0,0)[l]{\tiny ${ (0) }$}}
\put(20.5,-6){\makebox(0,0)[l]{\tiny ${ a_{u_0} }$}}
\put(12,-10.5){\circle{7}}
\put(12.5,-10.5){\makebox(0,0){\tiny ${ \Deul_{u_0} }$}}
%%%%%%%%%%%%%%%%%%%%%%%%%%%%%%%%%%%%%%%%%%%%%%
\end{picture}
\qquad
\raisebox{15mm}{\begin{minipage}{5cm}
$t = q(e,u_0)$\\
$z_i = q(e_i,u_0)$\\
$e_i = \{ u_0, u_i \}$\\
$\Deul_{u_0}$ may be empty.
\end{minipage}}}
$$
We also write
$\hat z_{ij} =$ product of all $z_k$ with $k \in \{1, \dots, n\} \setminus\{i,j\}$,
and $\hat z_{i} =$ product of all $z_k$ with $k \in \{1, \dots, n\} \setminus\{i\}$.
Note that
$$
\text{$\Aeul\setminus\Aeul_0 = A_{u_0} \cup A(u_0,e_1) \cup \cdots A(u_0,e_n) \cup A'(u,e)$ \ \  is a disjoint union.}
$$
Let $i \in \{ 1, \dots, n\}$ and $J_i = \{ 1, \dots, n\} \setminus \{i\}$. Then
$$
p'(u_0,e_i) =
\sum_{ \alpha \in A_{u_0} } \hat x_{u_i, \alpha}
+ \sum_{j \in J_i } \sum_{ \ \alpha \in A(u_0,e_j) } \hat x_{u_i, \alpha}
+ \sum_{ \alpha \in A'(u,e) } \hat x_{u_i, \alpha} .
$$
We have $\sum_{ \alpha \in A_{u_0} } \hat x_{u_i, \alpha} \in d(u_0)\Integ$ by Cor.\ \ref{p9cv355yutekvCs7df}.
For each $j \in J_i$, there exists $K_j \in \Integ$ such that $\hat x_{u_i, \alpha} = K_j \hat x_{u_0,\alpha}$ for all $\alpha \in A(u_0,e_j)$
(namely, $K_j = t \qdic(u_0) a_{u_0} \hat z_{ij}$);
thus $\sum_{ \alpha \in A(u_0,e_j) } \hat x_{u_i, \alpha} = K_j p(u_0,e_j) \in p(u_0,e_j)\Integ$.
It then follows from \eqref{JEILdc09vn320efeo} that $a_{u_0} c(u,e) \mid \sum_{ \alpha \in A'(u,e) } \hat x_{u_i, \alpha}$.
For each $\alpha \in A'(u,e)$ we have $\hat x_{u_i, \alpha} = \hat z_i \qdic(u_0) a_{u_0} \hat x_{u_0,\alpha}$,
so  $\sum_{ \alpha \in A'(u,e) } \hat x_{u_i, \alpha} =  \hat z_i \qdic(u_0) a_{u_0} p'(u,e)$ and hence
$ c(u,e) \mid  \hat z_i \qdic(u_0) p'(u,e) $.
Since this is true for all $i = 1, \dots, n$, and since $\gcd(\hat z_1, \dots, \hat z_n)=1$, we obtain
\begin{equation} \label {c09b2tij4waovddsvaweyu}
c(u,e) \mid \qdic(u_0) p'(u,e) .
\end{equation}

If $u_0$ is not a node then $\qdic(u_0)=1$ and hence $\Pscr(u,e)$ is true.
Suppose that $u_0$ is a node and let us show that $\Pscr(u,e)$ is true in this case as well.
Let $w \in \Deul_{u_0}$; then
$$
0 = \frac{N_w}{a_w} =
\sum_{\alpha \in A_{u_0}} \frac{x_{w,\alpha}}{a_w}
+ \sum_{j =1}^n \sum_{ \ \alpha \in A(u_0,e_j) } \frac{x_{w,\alpha}}{a_w} 
+ \sum_{\alpha \in A'(u,e)} \frac{x_{w,\alpha}}{a_w}  .
$$
We have $\sum_{\alpha \in A_{u_0}} \frac{x_{w,\alpha}}{a_w} \in d(u_0) \Integ$ by Cor.\ \ref{p9cv355yutekvCs7df}.
For each $j$ and each $\alpha \in A(u_0,e_j)$, we have $\frac{x_{w,\alpha}}{a_w} = y_w a_{u_0} t \hat z_j \hat x_{u_0,\alpha}$;
thus $\sum_{\alpha \in A(u_0,e_j) } \frac{x_{w,\alpha}}{a_w}= y_w a_{u_0} t \hat z_j p(u_0,e_j) \in p(u_0,e_j) \Integ$ for all $j$.
It follows from \eqref{JEILdc09vn320efeo} that $a_{u_0} c(u,e) \mid \sum_{\alpha \in A'(u,e)} \frac{x_{w,\alpha}}{a_w}$.
For each $\alpha \in A'(u,e)$ we have $\frac{x_{w,\alpha}}{a_w} = y_w a_{u_0} z_1 \cdots z_n \hat x_{u_0, \alpha}$, so
$\sum_{\alpha \in A'(u,e)} \frac{x_{w,\alpha}}{a_w} =  y_w a_{u_0} z_1 \cdots z_n p'(u,e)$ and hence
$a_{u_0} c(u,e) \mid  y_w a_{u_0} z_1 \cdots z_n p'(u,e)$.
Since this holds for every $w \in \Deul_{u_0}$, and since $\gcd\setspec{ y_w }{ w \in \Deul_{u_0} } = 1$,
we have $c(u,e) \mid  z_1 \cdots z_n p'(u,e)$. In view of  \eqref{c09b2tij4waovddsvaweyu} and of the fact
that $\gcd( z_1 \cdots z_n, \qdic(u_0) ) = 1$, we obtain that  $\Pscr(u,e)$ is true.
So the Theorem is proved.
\end{proof}

\begin{notation}  \label {c9v8mmxzKRknrscxe4ftdes}
For each $(u,e) \in P$, we define 
$$
M(u,e) = N_u / c(u,e) .
$$
Note that $M(u,e) \in \Nat \setminus \{0\}$, by Lemma \ref{Ocivbghn5dmkuyesK} and Thm \ref{xncoo9qwdx9}.
\end{notation}

\begin{proposition}    \label {0hp9fh023wbpchvgrsjs256}
If $(u,e) \in P$ satisfies $M(u,e) = 1$ then $u>u_0$, where $u_0 \in \Neul$ is defined by $e = \{u,u_0\}$.
\end{proposition}

\begin{proof}
Let $(u,e) \in P$ and write $e = \{u,u_0\}$.
We have to prove that $N_u = c(u,e)$ implies $u>u_0$.
Let $q = q(e,u)$ and $Q' = Q(e,u_0)$; then
$\left| \begin{smallmatrix} q & Q' \\ N_{u} & N_{u_{0}} \end{smallmatrix} \right|
= \det( e ) \sum_{\alpha \in A'} \hat x_{u,\alpha} $
by Prop.\ \ref{kuwdhr12778}\eqref{238dhwi912F},
where $A' = \setspec{ \alpha \in \Aeul\setminus\Aeul_0 }{ \text{the path $\gamma_{u,\alpha}$ contains $u_{0}$} }$.
So, using $\det(e) \le -1$, we get
$\sum_{\alpha \in A'} \hat x_{u,\alpha}
= \frac1{|\det(e)|} \left| \begin{smallmatrix} Q' & q \\ N_{u_{0}} & N_{u} \end{smallmatrix} \right|
\le \left| \begin{smallmatrix} Q' & q \\ N_{u_{0}} & N_{u} \end{smallmatrix} \right|$; so the implication
\begin{equation} \label {0cn3iehfc0F7CGBcCLIDD3Jhf828i3ubfi}
\text{if $N_u = c(u,e)$ then} \quad 
\textstyle
\sum_{\alpha \in A'} \hat x_{u,\alpha} \le Q' c(u,e) - q N_{u_{0}}  
\end{equation}
is valid for all $(u,e) \in P$.

We proceed by induction on $(u,e)$.
For each $(u,e) \in P$, let $\Pscr(u,e)$ be the statement ``if $N_u = c(u,e)$ then $u>u_0$, where $e = \{u,u_0\}$''.

Consider the case where $(u,e)$ is minimal in $(P,\preceq)$.
Then $c(u,e) = d(u_0) / a_{u_0}$.
We prove $\Pscr(u,e)$ by contradiction; assume that $N_u=c(u,e)$ and that it is not the case that $u>u_0$; then $u < u_0$ and
\begin{equation}  \label {pc09fn1o0q9wdnow9d}
\text{$Q'>0$, $q>0$, $A' \neq \emptyset$ and  $\hat x_{u,\alpha} > 0$ for all $\alpha \in A'$.}
\end{equation}
By \eqref{0cn3iehfc0F7CGBcCLIDD3Jhf828i3ubfi}, \eqref{pc09fn1o0q9wdnow9d} and $N_u = c(u,e) = d(u_0) / a_{u_0}$, we have
\begin{equation}  \label {0cvn23Oedhfp0CXew98we9}
\textstyle
Q' \frac{d(u_0)}{a_{u_0}} - q N_{u_0} \ge \sum_{\alpha \in A'} \hat x_{u,\alpha} > 0,
\end{equation}
so $Q' > a_{u_0} q \frac{N_{u_0}}{d(u_0)} \ge a_{u_0}$
(recall that $d(u_0)\mid N_{u_0}$ by Rem.\ \ref{F0934nofe8rg3p406egfh}) and hence $a_{u_0}< Q'$.
In view of part \eqref{fho8wiw9s-5} of Definition \ref{p9823p98p2d},
the conditions $u<u_0$ and $a_{u_0}<Q'$ imply that
no dead end is incident to $u_0$ (and consequently $a_{u_0}=1$).
Since $u_0 \in \Neul$, it follows that no arrow is adjacent to $u_0$.
Note that $u<u_0$ implies $u_0 \neq v_0$, so $\delta_{u_0}>2$;
since no arrow is adjacent to $u_0$ and $\epsilon(u_0)=1$, we have $|\Deul_{u_0}| \ge 2$.
Write $\Deul_{u_0} = \{ x_1, \dots, x_s \}$ (with $s\ge2$) and arrange the labelling so that
$q( \{u_0,x_1\}, u_0 ) = \max_{1 \le i \le s} q( \{u_0,x_i\}, u_0 )$
and hence $q( \{u_0,x_1\}, u_0 ) = Q'$.
Let $d_{x_i}$ be the degree of the dicritical $x_i$.
Using these facts and $d(u_0) = \gcd( d_{x_1}, \dots, d_{x_s})$, we get
$$
\textstyle
\sum_{\alpha \in A'} \hat x_{u,\alpha} 
= a_{x_1} d_{x_1} + Q' \sum_{i=2}^s a_{x_i} d_{x_i}
> Q' \sum_{i=2}^s a_{x_i} d_{x_i}
\ge Q' d(u_0) .
$$
Since \eqref{0cvn23Oedhfp0CXew98we9} gives 
$\sum_{\alpha \in A'} \hat x_{u,\alpha} \le Q' \frac{d(u_0)}{a_{u_0}} - q N_{u_0} <  Q' \frac{d(u_0)}{a_{u_0}} = Q' d(u_0)$,
this is absurd.
So $\Pscr(u,e)$ is true whenever $(u,e)$ is minimal in $(P,\preceq)$.

Now suppose that $(u,e)$ is not minimal and that $\Pscr(u',e')$ is true for all $(u',e') \in P$ such that $(u',e') \prec (u,e)$. 
We prove $\Pscr(u,e)$ by contradiction; assume that $N_u = c(u,e)$ and that it is not the case that $u>u_0$; then $u < u_0$
and \eqref{pc09fn1o0q9wdnow9d} is true.
Let $(u_0,e_1), \dots, (u_0,e_n)$ be the immediate predecessors of $(u,e)$ and write $e_i = \{u_0,u_i\}$ for $i = 1, \dots, n$.

For each $i \in \{1, \dots, n\}$,
the fact that $u<u_0$ implies that $u_0<u_i$; as $\Pscr(u_0,e_i)$ is true, we must have $N_{u_0} \neq c(u_0,e_i)$,
so   $N_{u_0} = M(u_0,e_i) c(u_0,e_i)$ where $M(u_0,e_i) \ge2$.
It follows that the number $a_{u_0} c(u,e) = \gcd( d(u_0), c(u_0,e_1), \dots, c(u_0,e_n))$ is a proper divisor of $N_{u_0}$, so
\begin{equation} \label {Cc0cipq0wdcno9w}
a_{u_0} c(u,e) < N_{u_0} .
\end{equation}
We have $\sum_{\alpha \in A'} \hat x_{u,\alpha}>0$ by \eqref{pc09fn1o0q9wdnow9d}, so
\eqref{0cn3iehfc0F7CGBcCLIDD3Jhf828i3ubfi} implies
$0 < \sum_{\alpha \in A'} \hat x_{u,\alpha} \le Q' c(u,e) - q N_{u_{0}}$, so $Q' c(u,e) > N_{u_{0}}$.
By this and \eqref{Cc0cipq0wdcno9w}, $a_{u_0}<Q'$.
As in the first part of the proof, this implies that no dead end is incident to $u_0$, and consequently $a_{u_0}=1$.
Another consequence of \eqref{0cn3iehfc0F7CGBcCLIDD3Jhf828i3ubfi} is:
\begin{equation} \label {d0o2Jcnqw9bafw829}
\textstyle
\sum_{\alpha \in A'} \hat x_{u,\alpha} < Q' c(u,e) .
\end{equation}
Define $A_{0}' = A_{u_0}$ and $A_{i}' = A(u_0,e_i)$ for each $i \in \{1, \dots, n\}$
(cf.\ \ref{0db3ukifxoms7dcdOoOoxde} and \ref{kCfnp293wdw0}).
Observe that $A' = \bigcup_{i=0}^n A_i'$ is a disjoint union and that $n\ge1$. Moreover, the fact that $u<u_0$ implies that
$A_i' \neq \emptyset$ for all $i > 0$.  One of the following must be true:
\begin{enumerate}

\item[(i)] $Q' = q( e_j , u_{0} )$ for some $j \in \{1, \dots, n\}$,

\item[(ii)] $Q' = q( \{ u_{0}, x \} , u_{0} )$ for some $x \in \Deul_{u_{0}}$.

\end{enumerate}
Define $A = \bigcup_{ i \in \{0,\dots,n\} \setminus \{j\} } A_i'$ in case (i), and $A = \bigcup_{ i=1 }^n A_i'$ in case (ii).
Observe that in both cases we have $A \subset A'$ (strict inclusion);
since $\hat x_{u,\alpha} > 0$ for all $\alpha \in A'$, it follows that
\begin{equation}  \label {0cbryj5ecxvn39fdtc}
\textstyle   \sum_{ \alpha \in A } \hat x_{u,\alpha} < \sum_{ \alpha \in A' } \hat x_{u,\alpha}.
\end{equation}
Let us prove:
\begin{equation}  \label {CKJv3l75q7mkmofhnbp4tZ68}
\textstyle
\sum_{ \alpha \in A } \hat x_{u,\alpha} \in Q' c(u,e) \Nat .
\end{equation}  
Let $S$ be the set of dicritical vertices that are adjacent to some arrow in $A$,
and let $d(S)$ be the gcd of the degrees of the dicriticals that belong to $S$.
Then $\sum_{ \alpha \in A } \hat x_{u,\alpha} \in \hat h(u,A) d(S) \Nat$
by Lemma \ref{civ023edkdkfpw}, so it suffices to show that  $\hat h(u,A) d(S) \in Q' c(u,e) \Nat$.
Because of how we defined $A$, we have $Q' \mid \hat h(u,A)$ in both cases; so it suffices to show that $d(S) \in c(u,e)\Nat$.
Since $S \subseteq \bigcup_{z \in \Nd(u,e)} \Deul_z$, $d(S)$ is a multiple of 
$g = \gcd{ \setspec{ d(z) }{ z \in \Nd(u,e) } }$; since $g \in  c(u,e) \Integ $ by Lemma \ref{kjwoeid9cse9c}(c),
we obtain $d(S) \in c(u,e)\Nat$, so  \eqref{CKJv3l75q7mkmofhnbp4tZ68} is proved.
So we have:
\begin{equation*} 
\textstyle
Q' c(u,e) > \sum_{\alpha \in A'} \hat x_{u,\alpha} > \sum_{\alpha \in A} \hat x_{u,\alpha} \in Q' c(u,e) \Nat,
\end{equation*}
where the first inequality is \eqref{d0o2Jcnqw9bafw829}, the second is \eqref{0cbryj5ecxvn39fdtc},
and the last claim is \eqref{CKJv3l75q7mkmofhnbp4tZ68}.
This implies that $\sum_{\alpha \in A} \hat x_{u,\alpha} = 0$, so $A = \emptyset$.
Since $n\ge1$ and $A_i' \neq \emptyset$ for all $i \ge 1$, case (ii) is impossible.
In case (i), $A=\emptyset$ means that $n=1$ and $A_0' = \emptyset$, so 
$\Deul_{u_0} = \emptyset$.
Since no dead end is incident to $u_0$, we obtain $\delta_{u_0}=2$, which is impossible because $u<u_0$ implies that $u_0$ is not the root.
This contradiction completes the proof.
\end{proof}

\section{Local structure of $\Neul$}
\label {Sec:LocalstructureofNeul}

{\it We assume that $\Teul$ is a minimally complete abstract Newton tree at infinity.}

\medskip

This section is divided into five parts. 
The first one introduces the rational numbers $\eta(u,e)$ and $R(u,A)$ and proves Thm \ref{P90werd23ewods0ci} and its corollaries.
These facts are the basis of many calculations carried out later.
The second subsection considers the three functions from $(P(\Teul), \preceq)$ to $(\Rat,\le)$ given by
$(u,e) \mapsto c(u,e)$,  $(u,e) \mapsto \eta(u,e)$ and $(u,e) \mapsto \tD( \Neul(u,e) )$,
and shows that the three are monotonic.
The third, fourth and fifth subsections introduce certain substructures of $\Neul$ 
($\tD$-trivial paths, teeth, combs, the set $\Omega(\Teul)$, etc.) and 
study the behavior of the three monotonic functions in relation with these substructures (in particular combs).

\begin{remark}
If $| \Neul | = 1$ then all claims made in this Section are trivially true.
In particular, the sets $P(\Teul)$, $Z(\Teul)$, $W(\Teul)$ and $\Omega(\Teul)$ are all empty, no path is $\tD$-trivial,
and there are no teeth and no combs.
Also note that if $| \Neul | = 1$ and $u \in \Neul$ then $\Eeul_u = \emptyset$ and $u=v_0$ ($\Eeul_u$ is defined in \ref{Bxdkj2A3wkwofdpfwend9}).
\end{remark}

\begin{definition} \label {Bxdkj2A3wkwofdpfwend9}
For each $(u,e) \in P$, we define the set $\Neul(u,e) = \setspec{ x \in \Neul }{ \text{$e$ is in $\gamma_{u,x}$} }$ and the rational number
$$
\eta(u,e) = \tD \big( \Neul(u,e) \big) - \big( 1 - c(u,e) \big).
$$
Given $u \in \Neul$, let $\Eeul_u$ be the set of all  edges of $\Neul$ incident to $u$.
Given a subset $A$ of $\Eeul_u$, define the numbers $\bD(u,A) \in \Integ$ and $R(u,A) \in \Rat_{\ge0}$ by:
\begin{align*}
\bD(u,A) &=  \textstyle  \tD\big( \{u\} \cup \bigcup_{e \in A} \Neul(u,e) \big) \\
R(u,A) &=  \textstyle  \sum_{x \in \Deul_u}(1-\frac1{k_x}) + (1-\frac1{a_u}) + \sum_{e \in A}(1-\frac1{M(u,e)}) .
\end{align*}
\end{definition}

Observe that $R(u,A)$ has the form $\sum_i (1 - \frac1{n_i})$ with $n_i \in \Nat\setminus \{0\}$ for all $i$.
Consequently, if $K$ is a number satisfying $R(u,A) < K$ then we have $| \setspec{ i }{ n_i \neq 1 } | < 2K$.
In particular, if $R(u,A)<1$ then at most one $i$ satisfies $n_i \neq 1$.

\begin{example} \label {oOo8cvn239vn3p98fgIW}
Continuation of Fig.\ \ref{723dfvcjp2q98ewdywe}, Ex.\ \ref{Pc09vbq3j7gabXrZiA}, \ref{cIuCbgepwej6DsJ37655enm} and \ref{cov7btiw67d78vujew9}.
Using the values of $\tD(v_i)$ given in Ex.\ \ref{cIuCbgepwej6DsJ37655enm} and those of $c(u,e)$ given in  Ex.\ \ref{cov7btiw67d78vujew9},
we may calculate $\eta(u,e)$ for all $(u,e) \in P$.
The results are given here:
%%%%%%%%%%%%%%%%%%%%%%%%%%%%%%%%%%%%%%%%%%%%%%%%%%%%%%%%
$$
\scalebox{1}{
\setlength{\unitlength}{1mm}%
\begin{picture}(85,21)(-5,-16)
%%%%%%%%%%%%%%%%%%%%%%%%%%%%%%%%%%%%%%%%%%%%%%
\multiput(0,0)(20,0){5}{\circle{1}}
\put(0,-15){\circle{1}}
\multiput(0.5,0)(20,0){4}{\line(1,0){19}}
\put(0,-0.5){\line(0,-1){14}}
\put(40,-15){\circle{1}}
\put(40,-0.5){\line(0,-1){14}}
%%%%%%%%%%%%%%%%%%%%%%%%%%%%%%%%%%%%%%%%%%%%%%
\put(3,1){\makebox(0,0)[bl]{\tiny ${ 0 }$}}
\put(17,1){\makebox(0,0)[br]{\tiny ${ \frac43 }$}}
\put(23,1){\makebox(0,0)[bl]{\tiny ${ 0 }$}}
\put(37,1){\makebox(0,0)[br]{\tiny ${ \frac{64}9 }$}}
\put(43,1){\makebox(0,0)[bl]{\tiny ${ 0 }$}}
\put(57,1){\makebox(0,0)[br]{\tiny ${ \frac{550}9 }$}}
\put(63,1){\makebox(0,0)[bl]{\tiny ${ 0 }$}}
\put(77,1){\makebox(0,0)[br]{\tiny ${ \frac{1117}9 }$}}
\put(1,-3){\makebox(0,0)[tl]{\tiny ${ 0 }$}}
\put(41,-12){\makebox(0,0)[bl]{\tiny ${ \frac{64}9 }$}}
\put(41,-3){\makebox(0,0)[tl]{\tiny ${ 0 }$}}
\put(1,-12){\makebox(0,0)[bl]{\tiny ${ 0 }$}}
\put(38,-15){\makebox(0,0)[r]{\tiny ${ v_6 }$}}
\put(-2,-15){\makebox(0,0)[r]{\tiny ${ v_5 }$}}
\put(-2,0){\makebox(0,0)[r]{\tiny ${ v_4 }$}}
\put(20,-1.5){\makebox(0,0)[t]{\tiny ${ v_3 }$}}
\put(38,-1.5){\makebox(0,0)[t]{\tiny ${ v_2 }$}}
\put(60,-1.5){\makebox(0,0)[t]{\tiny ${ v_1 }$}}
\put(80,-1.5){\makebox(0,0)[t]{\tiny ${ v_0 }$}}
%%%%%%%%%%%%%%%%%%%%%%%%%%%%%%%%%%%%%%%%%%%%%%
\end{picture}
\quad \raisebox{12\unitlength}{\begin{minipage}[t]{8cm} 
Meaning: $\eta(v_1, \{v_1,v_2\}) = \frac{550}9$,\\ $\eta(v_2, \{v_1,v_2\}) = 0$, etc.
\end{minipage}}
}
$$
\end{example}

\begin{theorem} \label {P90werd23ewods0ci}
Let $u \in \Neul$ and let $A$ be any subset of $\Eeul_u$. Then
$$
 R(u,A) + (\epsilon(u) - |A| - 1) { \textstyle (1 - \frac1{N_u}) }  =   1 + \frac{ \bD(u,A) - 1 - \sum_{e \in A}\eta(u,e)}{N_u} .
$$
\end{theorem}

\begin{proof}
By Lemma \ref{90q932r8dhd89cnr9} and the definitions of $\eta$ and $\bD$, we have
\begin{align*}
\bD(u,A) & \textstyle - \sum_{e \in A}\eta(u,e)
 = \tD(u) + \sum_{e \in A} \tD( \Neul(u,e) ) - \sum_{e \in A}\eta(u,e) \\
& \textstyle = \tD(u) + \sum_{e \in A} (1 - c(u,e)) \\
& \textstyle = \sigma(u) + (\epsilon(u)-2)(N_u-1) + N_u(1 - \frac1{a_u})  + \sum_{e \in A} (1 - c(u,e))\\
& \textstyle = \sigma(u) + (\epsilon(u)-|A|-2)(N_u-1) + N_u(1 - \frac1{a_u})   + |A|(N_u-1) + \sum_{e \in A} (1 - c(u,e))\\
& \textstyle = \sigma(u) + (\epsilon(u)-|A|-2)(N_u-1) + N_u(1 - \frac1{a_u})  + \sum_{e \in A} (N_u - c(u,e)) ,
\end{align*}
\begin{align*}
\textstyle \frac{\bD(u,A) - \sum_{e \in A}\eta(u,e)}{N_u} 
& \textstyle= \frac{\sigma(u)}{N_u} +  \sum_{e \in A} (1 - \frac1{M(u,e)}) + (1 - \frac1{a_u})+ (\epsilon(u)-|A|-2)(1 - \frac1{N_u}) \\
& \textstyle= R(u,A) + (\epsilon(u)-|A|-2)(1 - \frac1{N_u}) ,
\end{align*}
where Lemma \ref{lkjxfX0293wsod} is used in the last equality. So
\begin{align*}
\textstyle R(u,A) + (\epsilon(u)-|A|-1)(1 - \frac1{N_u})
& \textstyle = \frac{\bD(u,A) - \sum_{e \in A}\eta(u,e)}{N_u} +1 - \frac1{N_u} \\
& \textstyle =   1 + \frac{ \bD(u,A) - 1 - \sum_{e \in A}\eta(u,e)}{N_u},
\end{align*}
which proves the claim.
\end{proof}

\begin{corollary} \label {P90werd23ewods0ci-cor}
Let $u \in \Neul$ and $A \subseteq \Eeul_u$, and assume that $\bD(u,A) < 1 + \sum_{e \in A}\eta(u,e)$.
\begin{enumerata}

\item $\epsilon(u)-|A| \in \{0,1,2\}$

\item If $\epsilon(u)-|A|=2$ then $R(u,A)=0$ and $\bD(u,A) =\sum_{e \in A}\eta(u,e)$.

\item If $\epsilon(u)-|A|=1$ then $R(u,A)<1$ and $\bD(u,A) = \sum_{e \in A}\eta(u,e) + (1 - c(u^*,e^*))$, where $e^*$ is the unique element
of $\Eeul_u \setminus A$ and $u^*$ is the vertex defined by $e^* = \{u,u^*\}$.

\end{enumerata}
\end{corollary}

\begin{proof}
The hypothesis and Thm \ref{P90werd23ewods0ci} imply that $R(u,A) + (\epsilon(u) - |A| - 1) (1 - \frac1{N_u} )  < 1$, so
$(\epsilon(u) - |A| - 1) (1 - \frac1{N_u} )  < 1$.
If $\epsilon(u)-|A|>2$ then $\epsilon(u)>2$, so (by Lemma \ref{90q932r8dhd89cnr9}\eqref{0239u9e78923}) $N_u>1$ and 
$(\epsilon(u) - |A| - 1) (1 - \frac1{N_u} )  \ge1$, a contradiction.
So (a) is true.

\smallskip

\noindent (b) Assume that $\epsilon(u)-|A|=2$.
Then $R(u,A) + (1 - \frac1{N_u} )  < 1$.
Since  $\epsilon(u)\ge2$, we have $N_u>1$, so $R(u,A) < 1/2$ and hence $R(u,A)=0$.
Then we get $(1 - \frac1{N_u} )  =   1 + \frac{ \bD(u,A) - 1 - \sum_{e \in A}\eta(u,e)}{N_u}$, and it follows that 
 $\bD(u,A) =\sum_{e \in A}\eta(u,e)$.
So assertion (b) is proved.

\smallskip

\noindent (c) If $\epsilon(u)-|A|=1$ then
$R(u,A)  =   1 + \frac{ \bD(u,A) - 1 - \sum_{e \in A}\eta(u,e)}{N_u} < 1$.

Consider the case $A = \emptyset$.
Then $\epsilon(u)=1$, the assumption $\bD(u,A) < 1 + \sum_{e \in A}\eta(u,e)$ is equivalent to $\tD(u)\le0$,
and the desired conclusion  $\bD(u,A) = \sum_{e \in A}\eta(u,e) + (1 - c(u^*,e^*))$ is equivalent to $\tD(u) = 1-\frac{d(u)}{a_u}$,
which is true by Lemma \ref{8ey8cdody27}.
So assertion (c) is true when $A=\emptyset$.
From now-on, assume that $A \neq \emptyset$.
Since $R(u,A) < 1$, there are three possibilities:
\begin{enumerate}

\item[($\alpha$)] {\it $k_x=1$ for all $x \in \Deul_{u}$ and $M(u,e)=1$ for all $e \in A$.}

Then $R(u,A) = 1 -1/\theta$ with $\theta = a_{u}$.

\item[($\beta$)] {\it There exists $y \in \Deul_{u}$ such that $k_y>1$.}

Applying Lemma \ref{trivialobs} to $R(u,A)<1$ gives
$k_x=1$ for all $x \in \Deul_{u} \setminus \{y\}$, $a_{u}=1$, and $M(u,e)=1$ for all $e \in A$,
so $R(u,A) = 1 -1/\theta$ with $\theta = k_y$.

\item[($\gamma$)] {\it There exists $e_1 \in A$ such that $M(u,e_1)>1$.}

Lemma \ref{trivialobs} gives
$M(u,e)=1$ for all $e \in A \setminus \{e_1\}$, $k_x=1$ for all $x \in \Deul_{u}$, and $a_{u}=1$,
so $R(u,A) = 1 -1/\theta$ with $\theta = M(u,e_1)$.

\end{enumerate}
So $1 - \frac1{\theta} =   1 + \frac{ \bD(u,A) - 1 - \sum_{e \in A}\eta(u,e)}{N_u}$, where
$\theta= a_{u}$ in case ($\alpha$), $\theta= k_y$ in case ($\beta$), and $\theta= M(u,e_1)$ in case ($\gamma$).
So in all cases we have
$\bD(u,A) = \sum_{e \in A}\eta(u,e) + (1 - \frac{N_{u}}{\theta})$,
and to prove the claim it suffices to show that
\begin{equation} \label {09he123q9wd10W2wjdi}
\textstyle    c(u^*,e^*) =  \frac{N_{u}}{\theta}.
\end{equation}

In cases ($\alpha$) and ($\gamma$) we have $k_x=1$ for all $x \in \Deul_{u}$, so $\sigma(u)=0$ and hence $d(u) = N_u$ by Rem.\ \ref{F0934nofe8rg3p406egfh}.

In case ($\alpha$) we have $N_{u} = M(u,e) c(u,e) = c(u,e)$ for all $e \in A$; since $A \neq \emptyset$,
$$
\textstyle
c(u^*,e^*)
= \frac{\gcd(\{d(u)\} \cup \{ c(u,e)\, \mid \,  e \in A\} )}{a_{u}}
= \frac{\gcd( N_{u}, N_{u})}{ \theta }
= \frac{N_{u}}{\theta},
$$
proving \eqref{09he123q9wd10W2wjdi}.
In case ($\gamma$) we have $M(u,e_1)>1$ and $M(u,e)=1$ (so $N_{u}=c(u,e)$) for $e \in A \setminus \{e_1\}$, so
$$
\textstyle
c(u^*,e^*)
= \frac{\gcd(\{d(u)\} \cup \{ c(u,e)\, \mid \,  e \in A\} )}{a_{u}}
= \frac{\gcd( N_{u}, \frac{N_{u}}{M(u,e_1)}, N_{u}, \dots, N_{u})}{ 1 }
= \frac{N_{u}}{M(u,e_1)}
= \frac{N_{u}}{\theta},
$$
so  \eqref{09he123q9wd10W2wjdi} is proved.
In case ($\beta$), $u$ is a node of type $[ \frac{N_{u}}{k_y},  N_{u}, \dots,  N_{u} ]$
so $d(u) =\frac{N_{u}}{k_y}$. Also, $N_{u} = c(u,e)$  for all $e \in A$, so
$$
\textstyle
c(u^*,e^*)
= \frac{\gcd(\{d(u)\} \cup \{ c(u,e)\, \mid \,  e \in A\} )}{a_{u}}
= \frac{\gcd(\frac{N_{u}}{k_y}, N_{u}, \dots, N_{u})}{1}
= \frac{N_{u}}{k_y}
= \frac{N_{u}}{\theta},
$$
so again \eqref{09he123q9wd10W2wjdi} is proved.
So assertion (c) is proved and we are done.
\end{proof}

\begin{corollary} \label {dox9f80293ewf0di}
For every $u \in \Neul$ we have
$ \tD(\Neul) = \big( R(u, \Eeul_u) - 2 \big) N_u  + 2 +  \sum_{e \in \Eeul_u} \eta(u,e) $.
\end{corollary}

\begin{proof}
Thm \ref{P90werd23ewods0ci} gives
$ R(u,\Eeul_u) -{ \textstyle (1 - \frac1{N_u}) }  =   1 + \frac{ \tD(\Neul) - 1 - \sum_{e \in \Eeul_u}\eta(u,e)}{N_u} $.
\end{proof}

\section*{Monotonicity}

\begin{corollary} \label {okcf039wepdocS}
Let $(u,e)$ be a non-minimal element of $P = P(\Teul)$, write $e = \{u,u_0\}$, and let $(u_0,e_1), \dots, (u_0,e_n)$ be the immediate predecessors
of $(u,e)$. Then  the following hold.
\begin{enumerata}

\item \label {8938rf9h83wed2} If $\tD( \Neul(u,e) ) < 1 + \sum_{i=1}^n \eta(u_0,e_i)$ then $\eta(u,e) = \sum_{i=1}^n \eta(u_0,e_i)$.

\item \label {9vfh2867fyjmvqty} $\eta(u,e) \ge \sum_{i=1}^n \eta(u_0,e_i)$

\end{enumerata}
\end{corollary}

\begin{proof}
\eqref{8938rf9h83wed2} Suppose that $\tD( \Neul(u,e) ) < 1 + \sum_{i=1}^n \eta(u_0,e_i)$.
Let $A=\{e_1, \dots, e_n\}$ and note that $\tD( \Neul(u,e) ) = \bD(u_0,A)$.
Then $\bD(u_0,A) < 1 + \sum_{e' \in A} \eta(u_0,e')$ and $\epsilon(u_0)-|A|=1$,
so Cor.\ \ref{P90werd23ewods0ci-cor} gives
$\bD(u_0,A) = \sum_{e' \in A} \eta(u_0,e') + (1-c(u,e))$, so
$$
\textstyle \eta(u,e) + (1-c(u,e)) = \tD(\Neul(u,e)) = \bD(u_0,A) =  \sum_{i=1}^n \eta(u_0,e_i) + (1-c(u,e)),
$$
so $\eta(u,e) = \sum_{i=1}^n \eta(u_0,e_i)$.
This proves \eqref{8938rf9h83wed2}.

\noindent \eqref{9vfh2867fyjmvqty} Note that
\begin{align*}
\textstyle \tD(\Neul(u,e)) - 1 - \sum_{i=1}^n \eta(u_0,e_i) 
&= \textstyle \eta(u,e) + (1-c(u,e)) - 1 - \sum_{i=1}^n \eta(u_0,e_i) \\
&= \textstyle \eta(u,e) - \sum_{i=1}^n \eta(u_0,e_i) -c(u,e),
\end{align*}
so $\eta(u,e) \le \sum_{i=1}^n \eta(u_0,e_i)$ implies $\tD(\Neul(u,e)) < 1 + \sum_{i=1}^n \eta(u_0,e_i)$,
which (by part \eqref{8938rf9h83wed2}) implies $\eta(u,e) = \sum_{i=1}^n \eta(u_0,e_i)$.
This shows that $\eta(u,e) \le \sum_{i=1}^n \eta(u_0,e_i)$ implies $\eta(u,e) = \sum_{i=1}^n \eta(u_0,e_i)$,
so \eqref{9vfh2867fyjmvqty} is proved.
\end{proof}

\begin{corollary}  \label {90hJkHJHF7238ewHkjHiu93}
We have $\eta(u,e) \ge 0$ for all $(u,e) \in P$.
\end{corollary}

\begin{proof}
We prove this by induction on $(u,e) \in P$.

Consider an $(u,e) \in P$ which is minimal with respect to $\prec$.
Write $e=\{u,u_0\}$; then 
$\epsilon(u_0)=1$, $\Neul(u,e) = \{u_0\}$ and  $c(u,e)= \frac{d(u_0)}{a_{u_0}}$.
To show that $\eta(u,e)\ge0$, it suffices to show that if $\eta(u,e)\le0$ then $\eta(u,e)=0$.
Assume that $\eta(u,e) \le 0$. 
Then $\tD( u_0 ) = \tD\big( \Neul(u,e) \big) \le 1-c(u,e) < 1$, so $\tD(u_0) \le 0$,
so $\tD( u_0 ) = 1 - \frac{d(u_0)}{a_{u_0}}$ by Lemma \ref{8ey8cdody27}, so
$\tD\big( \Neul(u,e) \big) = 1 - c(u,e)$ and hence $\eta(u,e)=0$.
So the claim is true for all minimal $(u,e) \in P$.

Let $(u,e) \in P$ be non-minimal, and assume that $\eta(u',e') \ge 0$ is true for every $(u',e') \in P$ satisfying $(u',e') \prec (u,e)$.
Write $e = \{u,u_0\}$ and let $(u_0,e_1), \dots, (u_0,e_n)$ be the immediate predecessors of $(u,e)$.
We have $\eta(u,e) \ge \sum_{i=1}^n \eta(u_0,e_i)$ by Cor.\ \ref{okcf039wepdocS}\eqref{9vfh2867fyjmvqty}, 
and the inductive hypothesis gives $\eta(u_0,e_i) \ge 0$ for all $i \in \{1, \dots, n\}$; so $\eta(u,e) \ge0$.
\end{proof}

In what follows, we consider the monotonicity properties of the maps 
$$
\begin{array}{rcl}
c : P(\Teul) & \to & \Rat_{>0} \\
(u,e) & \mapsto & c(u,e)
\end{array}
\qquad
\begin{array}{rcl}
\eta : P(\Teul) & \to & \Rat_{\ge0} \\
(u,e) & \mapsto & \eta(u,e)
\end{array}
\qquad
\begin{array}{rcl}
P(\Teul) & \to & \Integ \\
(u,e) & \mapsto & \tD\big( \Neul(u,e) \big)
\end{array}
$$
where $\Rat_{>0} = \setspec{x \in \Rat}{x>0}$ and $\Rat_{\ge0} = \setspec{x \in \Rat}{x\ge0}$.
We show that $c$ is order-reversing, and that the other two maps are order-preserving.

\begin{lemma} \label {p0c9vin12q09wsc}
If $(u,e), (u',e') \in P$ satisfy $(u,e) \succeq (u',e')$, then the following hold.
\begin{enumerata}

\item $\Neul(u,e) \supseteq \Neul(u',e')$

\item $c(u,e) \le c(u',e')$ and $\eta(u,e) \ge \eta(u',e')$

\item $\tD\big( \Neul(u,e) \big) - \tD\big( \Neul(u',e') \big) = [c(u',e') - c(u,e)] + [\eta(u,e)-\eta(u',e')]$  

\item $\tD\big( \Neul(u,e) \big) \ge \tD\big( \Neul(u',e') \big)$, where equality holds if and only if 
$c(u,e) = c(u',e')$ and $\eta(u,e) = \eta(u',e')$.

\end{enumerata}
\end{lemma}

\begin{proof}
First consider the case where $(u',e')$ is an immediate predecessor of $(u,e)$.
Then (a) is obvious and  $\eta(u,e) \ge \eta(u',e')$ follows from Cor.\ \ref{okcf039wepdocS}\eqref{9vfh2867fyjmvqty} 
and the fact (Cor.\ \ref{90hJkHJHF7238ewHkjHiu93}) that $\eta(u'',e'') \ge 0$ for all $(u'',e'') \in P$.
Moreover, Def.\ \ref{kcjfnp0293wd} implies that there exists an integer $m\ge1$ such that $m c(u,e) = c(u',e')/a_{u'}$,
so in particular $c(u,e) \le c(u',e')$.
So (a) and (b) are true when $(u',e')$ is an immediate predecessor of $(u,e)$.
It follows that (a) and (b) are true in general.
Assertion (c) is a direct consequence of the definitions of $\eta(u,e)$ and $\eta(u',e')$ (cf.\ Def.\ \ref{Bxdkj2A3wkwofdpfwend9}),
and (d) follows from (b) and (c).
\end{proof}

\begin{definition}
An element $(u,e)$ of $P$ is said to be  \textit{nonpositive} if $\tD( \Neul(u,e) ) \le 0$.
Note that if $(u,e)$ is nonpositive then (by Lemma \ref{p0c9vin12q09wsc}) so is every $(u',e') \in P$ satisfying $(u',e') \preceq (u,e)$.
\end{definition}

\begin{proposition}  \label {DKxcnpw93sdo}
Let $(u,e) \in P$. Then the following are equivalent:
\begin{enumerata}

\item $(u,e)$ is nonpositive

\item $\eta(u,e)=0$

\item $\tD\big(\Neul(u,e)\big) = 1- c(u,e)$.

\end{enumerata}
\end{proposition}

\begin{proof}
We have $\eta(u,e)\ge0$ by Cor.\ \ref{90hJkHJHF7238ewHkjHiu93}.
Since $\tD\big(\Neul(u,e)\big) - (1- c(u,e)) = \eta(u,e) \ge 0$,
we have $\tD\big(\Neul(u,e)\big) \ge 1- c(u,e)$ and we also see that (b)$\Leftrightarrow$(c).
There remains to show that (a)$\Leftrightarrow$(b).

If $\eta(u,e)=0$ then $\tD\big( \Neul(u,e) \big) = \eta(u,e) + 1 - c(u,e) =  1 - c(u,e) < 1$, 
and since   $\tD\big( \Neul(u,e) \big) \in \Integ$ it follows that $\tD\big( \Neul(u,e) \big) \le 0$.
Thus (b)$\Rightarrow$(a).

We prove the converse by induction.
For each $(u,e) \in P$, let $\Pscr(u,e)$ be the assertion ``if $(u,e)$ is nonpositive then $\eta(u,e)=0$.''

Consider an $(u,e) \in P$ which is minimal with respect to $\prec$.
Write $e=\{u,u_0\}$; then  $\epsilon(u_0)=1$, $\Neul(u,e) = \{u_0\}$ and $c(u,e)= \frac{d(u_0)}{a_{u_0}}$.
If $(u,e)$ is nonpositive then
$\tD( u_0 ) = \tD\big( \Neul(u,e) \big) \le 0$,
so $\tD( u_0 ) = 1 - \frac{d(u_0)}{a_{u_0}}$ by Lemma \ref{8ey8cdody27}, so
$\tD\big( \Neul(u,e) \big) = 1 - c(u,e)$ and hence $\eta(u,e)=0$.
So $\Pscr(u,e)$ is true for all minimal $(u,e) \in P$.

Let $(u,e) \in P$ be non-minimal, and assume that $\Pscr(u',e')$ is true for every $(u',e') \prec (u,e)$.
Write $e = \{u,u_0\}$ and let $(u_0,e_1), \dots, (u_0,e_n)$ be the immediate predecessors of $(u,e)$.
To prove that $\Pscr(u,e)$ is true, assume that $(u,e)$ is nonpositive.
Then $(u_0,e_1), \dots, (u_0,e_n)$ are nonpositive and consequently $\eta(u_0,e_i)=0$ for all $i$, by the inductive hypothesis.
Since $\tD\big( \Neul(u,e) \big) \le 0$, we certainly have $\tD\big( \Neul(u,e) \big) < 1 + \sum_{i=1}^n \eta(u_0,e_i)$,
so Cor.\ \ref{okcf039wepdocS}\eqref{8938rf9h83wed2}  gives $\eta(u,e) = \sum_{i=1}^n \eta(u_0,e_i)$, so $\eta(u,e)=0$ and $\Pscr(u,e)$ is true.
\end{proof}

\begin{remark} \label {uyhmdytjwhwhrkdftraef4gf23h23}
By Prop.\ \ref{DKxcnpw93sdo}, if $(u,e) \in P$ is nonpositive then $c(u,e) \in \Nat\setminus \{0\}$.
\end{remark}

\begin{notation} \label {p09cv347dYF6Us98}
Given $u \in \Neul$, define
$$
a_u^* = \begin{cases} 
1 & \text{if $a_u>1$,} \\
0 & \text{if $a_u=1$}
\end{cases}
$$
and let $\#(u)$ denote the number of nonzero terms in the right-hand-side of
\begin{equation} \label {uibyu5orbfvo7gnrnsdruoq0heu}
R(u, \Eeul_u) =  \textstyle  \sum_{x \in \Deul_u}(1-\frac1{k_x}) + (1-\frac1{a_u}) + \sum_{e \in  \Eeul_u}(1-\frac1{M(u,e)}) .
\end{equation}
\end{notation}

\begin{corollary} \label {Xxkvp0wedifcwZepd}
For each $u \in \Neul$,
$$
| \setspec{ x \in \Deul_u }{ k_x>1 } | + a_u^* + \epsilon(u)-1 \le \#(u) \le \max(3, \tD(\Neul) + 2) .
$$
\end{corollary}

\begin{proof}
The definition of $\#(u)$ is
\begin{equation} \label {cv9udh6dierbnxmmcis8eeAaBgbo}
\#(u) =  | \setspec{ x \in \Deul_u }{ k_x>1 } | + a_u^* + | \setspec{ e \in \Eeul_{u} }{ M(u,e) > 1 } | .
\end{equation}
Prop.\ \ref{0hp9fh023wbpchvgrsjs256} implies that
$| \setspec{ e \in \Eeul_{u} }{ M(u,e) > 1 } | \ge \epsilon(u)-1$, 
so the first inequality is true. 
Let $n = \#(u)$.
To prove the second inequality, it's enough to show that if ${n} \ge 4$ then ${n} \le \tD(\Neul)+2$.
Assume that ${n} \ge 4$. We claim:
\begin{equation} \label {D9023vWsbv54Cnbrsidagiut}
N_u \ge 2 .
\end{equation}
Indeed, assume the contrary. Then $N_u =1$ and consequently $k_x=1$ for all $x \in \Deul_u$.
Then \eqref{cv9udh6dierbnxmmcis8eeAaBgbo} gives  $| \setspec{ e \in \Eeul_{u} }{ M(u,e) > 1 } | = n- a_u^* \ge3$,
so $\epsilon(u)\ge3$, so Lemma \ref{90q932r8dhd89cnr9}\eqref{0239u9e78923} gives $N_u>1$, a contradiction.
So \eqref{D9023vWsbv54Cnbrsidagiut} is true.
Since there are $n$ nonzero terms in the sum \eqref{uibyu5orbfvo7gnrnsdruoq0heu}
and each nonzero term is $\ge\frac12$, we have $R(u,\Eeul_u) \ge {n}/2$.
By Corollaries \ref{dox9f80293ewf0di} and \ref{90hJkHJHF7238ewHkjHiu93} we have
$ \tD(\Neul) = \big( R(u, \Eeul_u) - 2 \big) N_u  + 2 +  \sum_{e \in \Eeul_u} \eta(u,e) \ge \big( \frac {n}2 - 2 \big) N_u  + 2$.
As $\frac {n}2-2\ge0$ and $N_u\ge2$, we obtain $\tD(\Neul) \ge {n} - 2$.
\end{proof}

\begin{proposition} \label {0c9b45okjrdidjfo}
We have $\tD(\Neul) \ge (\delta_{v_0}-1) (\delta_{v_0}-2)$,
and if $\tD(\Neul) \ge 0$ then  
$$
\delta_{v_0} \, \le  \, \frac{3+ \sqrt{1+4\tD(\Neul)}}{2} \, .
$$
\end{proposition}

\begin{proof}
Define $A = \setspec{ e \in \Eeul_{v_0} }{ c(v_0,e) \ge 1 }$ and $B = \setspec{ e \in \Eeul_{v_0} }{ c(v_0,e) < 1 }$.
Note that
if $e \in B$ then $c(v_0,e) \notin \Integ$, so Rem.\ \ref{uyhmdytjwhwhrkdftraef4gf23h23} gives $\tD( \Neul(v_0,e) ) \ge 1$;
since $\tD( \Neul(v_0,e) ) = 1 + \eta(v_0,e) - c(v_0,e)$, this shows that $\eta(v_0,e) - c(v_0,e) \ge 0$ for all $e \in B$.
On the other hand, $\eta(v_0,e) - c(v_0,e) \ge -c(v_0,e)$ for all $e \in A$ by  Cor.\ \ref{90hJkHJHF7238ewHkjHiu93}, so 
\begin{equation}  \label {cF0iowo4vjjk4ixawe7gjndje}
\textstyle   \sum_{ e \in \Eeul_{v_0} } (\eta(v_0,e) - c(v_0,e)) \ge -\sum_{e \in A} c(v_0,e) .
\end{equation}
Also note that $N_{v_0} = \sum_{u \in \Deul_{v_0} } a_u d_u + \sum_{e \in \Eeul_{v_0} } p(v_0,e)$,
$p(v_0,e) \ge 1$ for all $e \in B$, and (by Thm \ref{xncoo9qwdx9}) $\frac{p(v_0,e)}{c(v_0,e)} \ge 1$ for all $e \in A$.
Thus Cor.\ \ref{dox9f80293ewf0di} gives
\begin{align*}
\tD(\Neul) &- 2 = \textstyle (R(v_0,\Eeul_{v_0}) - 2) N_{v_0} + \sum_{e \in \Eeul_{v_0} } \eta(v_0,e) \\
&= \textstyle -2 N_{v_0} + \sum_{u \in \Deul_{v_0} } (N_{v_0} - d_u) + \sum_{e \in \Eeul_{v_0} } (N_{v_0} - c(v_0,e)) + \sum_{e \in \Eeul_{v_0} } \eta(v_0,e) \\
&= \textstyle (\delta_{v_0} - 2) N_{v_0} - \sum_{u \in \Deul_{v_0} } d_u + \sum_{ e \in \Eeul_{v_0} } (\eta(v_0,e) - c(v_0,e)) \\
&\ge \textstyle (\delta_{v_0} - 2) N_{v_0}
- \sum_{u \in \Deul_{v_0} } d_u -\sum_{e \in A} c(v_0,e) \quad \text{(by \eqref{cF0iowo4vjjk4ixawe7gjndje})} \\
&\ge \textstyle (\delta_{v_0} - 2) \sum_{u \in \Deul_{v_0} } a_u d_u - \sum_{u \in \Deul_{v_0} } d_u
+ (\delta_{v_0} - 2)  \sum_{e \in \Eeul_{v_0} } p(v_0,e) -\sum_{e \in A} c(v_0,e) \\
&\ge \textstyle  \sum_{u \in \Deul_{v_0} } [(\delta_{v_0} - 2)a_u-1] d_u 
+  \sum_{e \in A } \big[ (\delta_{v_0} - 2)\frac{p(v_0,e)}{c(v_0,e)} - 1 \big] c(v_0,e)  + (\delta_{v_0} - 2)  \sum_{e \in B } p(v_0,e) \\
&\ge \textstyle  (\delta_{v_0} - 3) | \Deul_{v_0} | +  (\delta_{v_0} - 3) | A | +  (\delta_{v_0} - 2) | B | 
\ge  (\delta_{v_0} - 3)\delta_{v_0} + |B|
\ge  (\delta_{v_0} - 3)\delta_{v_0} ,
\end{align*}
so $\tD(\Neul) \ge (\delta_{v_0}-1) (\delta_{v_0}-2)$.
The last claim easily follows.
\end{proof}

\section*{$\tD$-trivial paths, teeth and brushes}

\begin{definition}
A path $(z_1, \dots, z_n)$ is \textit{$\tD$-trivial} if $n>1$, $z_1, \dots, z_n \in \Neul$, $\epsilon(z_1)=1$, and
$$
\text{$\epsilon(z_i)=2$ and $\tD(z_i)=0$ for all  $i$ such that $1<i<n$.}
$$
\end{definition}

\begin{proposition} \label {xfo230weidwods}
Let $(u_1, \dots, u_n)$ be a $\tD$-trivial path. 
\begin{enumerata}

\item $c(u_i, \{u_i,u_{i-1}\}) =  \frac{d(u_1)}{a_{u_1}}$ for all  $i \in \{ 2, \dots, n\}$.

\item If $N_{u_n} = \frac{d(u_1)}{a_{u_1}}$ then $u_n > u_{n-1}$.

\end{enumerata}
\end{proposition} 

\begin{proof}
Note that $n\ge2$ by definition of $\tD$-trivial path. 

(a) Let $e_i = \{ u_i, u_{i-1} \}$ for $i \in \{2, \dots, n\}$, then
$(u_2,e_2) \prec \cdots \prec (u_n,e_n)$ are elements of $P(\Teul)$.
By induction on $i$, we show that $c(u_i,e_i) =  \frac{d(u_1)}{a_{u_1}}$ is true for all  $i \in \{ 2 \dots, n\}$.

Since $(u_2,e_2)$ is minimal, we have $c(u_2,e_2) = d(u_1)/a_{u_1}$, so the case $i=2$ is true.

Let $i \in \{2, \dots,n-1\}$, assume that $c(u_i,e_i) =  \frac{d(u_1)}{a_{u_1}}$, and let us deduce that 
$c(u_{i+1},e_{i+1}) =  \frac{d(u_1)}{a_{u_1}}$.
Since $1 < i < n$, we have $\tD( u_{i} ) = 0$ and $\epsilon(u_{i})=2$, so
$$
\textstyle
0 = 
\tD(u_{i})
= \sigma( u_{i} ) + (\epsilon(u_{i})-2)(N_{u_{i}}-1) + N_{u_{i}}(1 - \frac1{a_{u_{i}}})
= \sigma( u_{i} ) + N_{u_{i}}(1 - \frac1{a_{u_{i}}})
$$
and consequently $\sigma( u_{i} ) = 0$ and  $a_{u_{i}}=1$.
Since $\sigma( u_{i} ) = 0$, we have $d(u_i) = N_{u_i}$ by Rem.\ \ref{F0934nofe8rg3p406egfh}.
Since $N_{u_i} \in c(u_i,e_i)\Integ$ by Thm \ref{xncoo9qwdx9}, we get $d(u_i) \in c(u_i,e_i)\Integ$.
This implies that $\gcd( d(u_i),  c(u_i,e_i) ) =  c(u_i,e_i)$.
Since $(u_i,e_i)$ is the only immediate predecessor of $(u_{i+1},e_{i+1})$, and since $a_{u_i}=1$,
$$
\textstyle
c(u_{i+1},e_{i+1})
= \frac{ \gcd( d(u_i),  c(u_i,e_i) )  }{ a_{u_{i}} }
= c(u_i,e_i) =  \frac{d(u_1)}{a_{u_1}},
$$
which completes the proof of (a).

(b) If $N_{u_n} = \frac{d(u_1)}{a_{u_1}}$ then $N_{u_n} = c(u_n,e_n)$ by (a), so $M(u_n,e_n)=1$, so $u_n > u_{n-1}$ by Prop.\ \ref{0hp9fh023wbpchvgrsjs256}.
\end{proof}

\begin{definition} \label {xkclqlwpd90cod}
\begin{enumerate}

\item $ Z = Z(\Teul) = \setspec{ z \in \Neul }{ \text{$\epsilon(z)=1$ and $\tD(z)\le0$} } $

\item Let $\Gamma = \Gamma(\Teul)$ be the set of $\tD$-trivial paths $(z_1,\dots,z_n)$ such that\footnote{See Lemma \ref{0cvg248f7qvq8hene44dsdi374}
for an equivalent definition of $\Gamma(\Teul)$.}
$$
\text{$\tD(z_1) \le 0 < \tD(\{ z_1, \dots, z_n \})$ \quad and \quad $z_{n-1}>z_{n}$.}
$$
Note that if $(z_1, \dots, z_n) \in \Gamma$ then $n\ge2$,  $z_1 \in Z$, $\{z_1,\dots,z_n\} \subseteq  \Neul$ and $\tD(z_n)>0$.

\item Let $W=W(\Teul)$ be the set of $w \in \Neul$ satisfying:
$$
\text{there exists $(z_1, \dots, z_n) \in \Gamma$ such that $z_n=w$.}
$$
Observe that $\tD(w)>0$ for all $w \in W$.

\item An element $(u,e)$ of $P(\Teul)$ is called a \textit{tooth} if 
$$
\text{there exists $(z_1, \dots, z_n) \in \Gamma$ such that $z_n=u$ and $\{z_n, z_{n-1}\} = e$.}
$$

\end{enumerate}
\end{definition}

Observe that if $(u,e)$ is a tooth then $u \in W(\Teul)$.
Also note that  $(z_1, \dots, z_n) \mapsto (z_n, \{z_n, z_{n-1} \})$ is a bijection from $\Gamma(\Teul)$ to the set of teeth of $\Teul$.

\begin{lemma}  \label {GRygergGREg8948r}
If $(u,e)$ is a tooth then
\begin{enumerata}

\item  $(u,e)$ is nonpositive, $\eta(u,e)=0$ and $\tD(\Neul(u,e)) = 1 - c(u,e)$

\item $\tD( \{u\} \cup \Neul(u,e) ) > 0$

\item $M(u,e)>1$.  

\end{enumerata}
\end{lemma}

\begin{proof}
There exists a $\tD$-trivial path $(z_1, \dots, z_n)$
such that $\tD(z_1)\le0<\tD(\{z_1,\dots,z_n\})$, $z_n=u$, $\{z_n,z_{n-1}\} = e$ and $z_n<z_{n-1}$.
We have $\Neul(u,e) = \{ z_1, \dots, z_{n-1} \}$,
so $\tD(\Neul(u,e)) = \tD( \{ z_1, \dots, z_{n-1} \} ) = \tD(z_1) \le 0$, so $(u,e)$ is nonpositive.
It follows that  $\tD(\Neul(u,e)) = 1 - c(u,e)$ and that $\eta(u,e)=0$ by Prop.\ \ref{DKxcnpw93sdo}.
We also have $\tD( \{u\} \cup \Neul(u,e) ) = \tD(\{z_1,\dots,z_n\}) > 0$.
If $M(u,e)=1$ then Prop.\ \ref{0hp9fh023wbpchvgrsjs256} implies that $z_n > z_{n-1}$, a contradiction.
So $M(u,e)>1$ and the Lemma is proved.
\end{proof}

\begin{lemma} \label {awe9ifc0we9} 
Let $w \in \Neul$ and let $A$ and $T$ be subsets of $\Eeul_w$. 
Assume:
\begin{itemize}

\item[(i)]  $A \cap T = \emptyset$ and $|A \cup T| < \epsilon(w)$;

\item[(ii)] for each $e \in T$, $(w,e)$ is a tooth;

\item[(iii)] for each $e \in A$, $(w,e)$ is nonpositive.

\end{itemize}
Then:
\begin{enumerata}

\item $\bD(w, A \cup T) > 0 \iff \bD(w, A) > 0$

\item If $|A| \le \epsilon(w)-3$ then  $\bD(w, A) > 0$.  

\end{enumerata}
\end{lemma}

\begin{proof}
Note that  $\bD(w, A \cup T) = \bD(w, A) + \sum_{e \in T} \tD( \Neul(w,e) )$ and that $\tD( \Neul(w,e) ) \le 0$ for each $e \in T$,
so  $\bD(w, A \cup T) \le \bD(w, A)$ and consequently ``$\Rightarrow$'' is clear, in assertion (a).

Let $\ell = \epsilon(w)-|A \cup T|-1 \ge 0$.
By assumptions (ii) and (iii) and Lemma \ref{GRygergGREg8948r} we have $\eta(w,e)=0$ for all $e \in A \cup T$, so Thm \ref{P90werd23ewods0ci} gives
$ R(w,A\cup T) + \ell (1 - \frac1{N_w} )  =   1 + \frac{ \bD(w,A\cup T) - 1 }{N_w} $
and $R(w,A) + (\ell+|T|) (1 - \frac1{N_w})  =   1 + \frac{ \bD(w,A) - 1 }{N_w}$, 
or equivalently
\begin{gather}
\label {d90cf23owedJX0} \textstyle
\sum_{x \in \Deul_w}(1 - \frac1{k_x}) + \sum_{e \in A \cup T} (1 - \frac1{M(w,e)}) + (1 - \frac1{a_w}) + \sum_{i=1}^\ell (1 - \frac1{N_w} ) 
=   1 + \frac{ \bD(w,A\cup T) - 1 }{N_w} \\
 \label {9d2b398ev2we9fu} \textstyle
\sum_{x \in \Deul_w}(1 - \frac1{k_x}) + \sum_{e \in A} (1 - \frac1{M(w,e)}) + (1 - \frac1{a_w}) + \sum_{i=1}^{\ell+|T|} (1 - \frac1{N_w}) 
 =   1 + \frac{ \bD(w,A) - 1 }{N_w} .
\end{gather}
By  Lemma \ref{GRygergGREg8948r} we have $M(w,e)>1$ for all $e \in T$.
Suppose that $\bD(w,A) \ge 1$. Then $1 + \frac{ \bD(w,A) - 1 }{N_w} \ge 1$,
so the left-hand-side of \eqref{9d2b398ev2we9fu} has at least two nonzero terms.
Since the left-hand-side of \eqref{d90cf23owedJX0} is obtained from that of \eqref{9d2b398ev2we9fu} by removing $|T|$ terms $(1 - \frac1{N_w})$
and adding $|T|$ \textit{nonzero} terms $(1-\frac1{M(w,e)})$, 
the left-hand-side of \eqref{d90cf23owedJX0} has at least two nonzero terms;
thus $1 + \frac{ \bD(w,A\cup T) - 1 }{N_w} \ge 1$ and hence $\bD(w,A\cup T) \ge 1$. This proves (a).

(b) If $|A| \le \epsilon(w)-3$ then $\epsilon(w)\ge3$, so $N_w>1$ and hence $1-\frac1{N_w}>0$;
since $\ell + |T| = \epsilon(w) - |A| - 1 \ge 2$,
the left-hand-side of \eqref{9d2b398ev2we9fu}  has at least two nonzero terms $(1-\frac1{N_w})$;
then the sum is $\ge1$ and consequently
$1 + \frac{ \bD(w,A) - 1 }{N_w} \ge 1$, so $\bD(w,A)\ge1$.
\end{proof}

\begin{definition} \label {pd09Yv3ned09Xse}
Let $Z = Z(\Teul)$, $\Gamma = \Gamma(\Teul)$ and $W = W(\Teul)$ (see Def.\ \ref{xkclqlwpd90cod}).
\begin{enumerate}

\item For each $w \in W$, define
$$
V(w) = \setspec{ v \in \Veul }{  \gamma_{v,w} \in \Gamma } \quad \text{and} \quad
\bar V(w) = \setspec{x \in \Veul}{ \text{$x$ is in $\gamma_{v,w}$ for some $v \in V(w)$} } .
$$
Observe that $\bigcup_{w \in W} V(w) \subseteq Z$
and $\bigcup_{w \in W} \bar V(w) \subseteq  \Neul$.
We also define 
$$
V(w) = \emptyset \text{ and } \bar V(w) = \{w\} \quad \text{for all $w \in \Neul \setminus W$.}
$$
So, for $w \in \Neul$, we have $w \in W$ if and only if $| \bar V(w) | > 1$.

\item We say that $\Teul$ is a \textit{brush} if there exists $w \in W(\Teul)$ such that $\bar V(w) = \Neul$.

\end{enumerate}
\end{definition}

\begin{remark} \label {pc09n3409vnZiEWOdhFp30ef}
\begin{enumerata}

\item {\it No element $w$ of $W(\Teul)$ satisfies $v_0 \in \bar V(w) \setminus \{w\}$.}

\item {\it If $\Teul$ is a brush then $| \Neul | > 1$ and $W(\Teul) = \{ v_0 \}$.}

\end{enumerata}
Assertion (a) is proved by contradiction: suppose that  $w \in W(\Teul)$ satisfies $v_0 \in \bar V(w) \setminus \{w\}$.
Then there exists $(z_1, \dots, z_n) \in \Gamma(\Teul)$ satisfying $v_0 \in \{z_1, \dots, z_{n-1}\}$,
but this is absurd because the definition of $\Gamma(\Teul)$ implies that $z_{n-1}>z_n$.
So (a) is true.  Consider (b). If $\Teul$ is a brush then obviously $W(\Teul) \neq \emptyset$, so $| \Neul | > 1$.
The fact that $W(\Teul) = \{ v_0 \}$ follows from (a).
\end{remark}

\begin{lemma} \label {p0293efp0cw23ep0hvj}
If $\Teul$ is not a brush then for each $w \in W(\Teul)$ we have
$$
\epsilon(w) > |V(w)| \quad \text{and} \quad  \tD\big( \bar V(w) \big) \ge \max(1, \epsilon(w)-2) .
$$
\end{lemma}

\begin{proof}
Let $w \in W$.
It is clear that  $\epsilon(w) \ge |V(w)|$.
If $\epsilon(w) = |V(w)|$ then
$W=\{w\}$ and $\bar V(w) = \Neul$, so $\Teul$ is a brush, a contradiction.
So $\epsilon(w) > |V(w)|$.
From $\epsilon(w)>1$, we deduce that $N_w\ge2$.

Let $z_1, \dots, z_k$ be the distinct elements of $V(w)$.
For each $i =1,\dots,k$, let $e_i$ be the edge of $\gamma_{z_i,w}$ which is incident to $w$.
Then $(w,e_1), \dots, (w,e_k)$ are distinct teeth  (because $\gamma_{z_i,w} \in \Gamma$).
Let us apply Lemma \ref{awe9ifc0we9} with $A=\emptyset$ and $T=\{e_1, \dots,e_k\}$;
we have $|A \cup T| = k = | V(w) | < \epsilon(w)$ and (since $w \in W$)  $\bD( w, A) = \tD(w) > 0$,
so the Lemma gives $\tD\big( \bar V(w) \big) = \tD\big( \{w\} \cup \bigcup_{i=1}^k\Neul(w,e_i) \big) = \bD( w, A \cup T) \ge 1$. 

To complete the proof, there remains to show that  $\tD\big( \bar V(w) \big) \ge \epsilon(w)-2$.
If $\epsilon(w)\le2$ then  $\tD\big( \bar V(w) \big) \ge 1 > \epsilon(w)-2$, so we may assume from now-on that $\epsilon(w)\ge3$.

With $A = \{ e_1, \dots, e_k \}$, Thm \ref{P90werd23ewods0ci} gives
$$
R(w,A) + (\epsilon(w) - |A| - 1) { \textstyle (1 - \frac1{N_w}) }  =   1 + \frac{ \bD(w,A) - 1 - \sum_{e \in A}\eta(w,e)}{N_w}
 =   1 + \frac{ \bD(w,A) - 1}{N_w}
$$
where the last equality follows from the fact that, for each $e \in A$, $(w,e)$ is a tooth and hence $\eta(w,e)=0$ by Lemma \ref{GRygergGREg8948r}.
We have $R(w,A) \ge \sum_{i=1}^k (1 - \frac1{M(w,e_i)}) \ge k/2$ by Lemma \ref{GRygergGREg8948r}.
Noting that $\bD(w,A) = \tD( \bar V(w) )$ and $|A| = k$, we obtain
\begin{align} \label {zdklxpw2oeijfdpd}
\tD( \bar V(w) )
&  = (R(w,A)-1) N_w +  (\epsilon(w) - k - 1) (N_w - 1) + 1 \\
\notag
& \textstyle \ge (\frac k2 -1) N_w +  (\epsilon(w) - k - 1) (N_w - 1) + 1 \\
\notag
& \textstyle \ge (\frac k2 -1) N_w +  \epsilon(w) - k   + (\epsilon(w) - k - 1) (N_w - 2)  .
\end{align}
Since we assumed that $\epsilon(w)\ge3$, we have $k\ge 2$ or $\epsilon(w) \ge k+2$.

If $k\ge2$ then $N_w\ge2$ implies that  $(\frac k2 -1) N_w \ge  (\frac k2 -1) 2 = k-2$, so \eqref{zdklxpw2oeijfdpd} gives
\begin{align*}
\tD( \bar V(w) )
& \textstyle \ge \epsilon(w)-2 +  (\epsilon(w) - k - 1) (N_w - 2) \ge \epsilon(w)-2 .
\end{align*}

If $\epsilon(w) \ge k+2$ then  \eqref{zdklxpw2oeijfdpd} gives
\begin{align*}
\tD( \bar V(w) )
& \textstyle \ge (\frac k2 -1) N_w +  \epsilon(w) - k    +  (N_w - 2) + (\epsilon(w) - k - 2) (N_w - 2)  \\
& \textstyle \ge (\frac k2 -1) N_w +  \epsilon(w) - k    +  (N_w - 2) \\
& \textstyle = \frac k2  N_w +  \epsilon(w) - k      - 2 
\ge \epsilon(w)-2 .
\end{align*}
\end{proof}

\section*{The set $\Omega(\Teul)$}

\begin{definition} \label {c9v8b34yk4icsovedbd}
$ \Omega(\Teul) = Z(\Teul) \setminus \bigcup_{w \in W(\Teul)} V(w)$
\end{definition}

\begin{lemma}
\begin{enumerata}

\item $v_0 \in \Omega(\Teul)$ if and only if $v_0 \in Z(\Teul)$.

\item If $\epsilon(v_0)=1$ and $| \Deul_{v_0} | \le 1$ then  $v_0 \in \Omega(\Teul)$.

\item If $\delta_{v_0}=1$ and $| \Neul | > 1$ then $v_0 \in \Omega(\Teul)$.

\end{enumerata}
\end{lemma}

\begin{proof}
(a) Clearly, $v_0 \in \Omega(\Teul)$ implies $v_0 \in Z(\Teul)$.
If $v_0$ belongs to $Z(\Teul)$ but not to $\Omega(\Teul)$ then $v_0 \in V(w)$ for some $w \in W(\Teul)$;
then $v_0 \in \bar V(w) \setminus \{w\}$, which contradicts Rem.\ \ref{pc09n3409vnZiEWOdhFp30ef}.

(b) The assumption  $| \Deul_{v_0} | \le 1$ implies that $\sigma(v_0) < N_{v_0}$.
Since $\epsilon(v_0)=1$ and $a_{v_0}=1$, Lemma \ref{90q932r8dhd89cnr9}\eqref{pc9vp23r09} gives $\tD(v_0) = \sigma(v_0) + (-1)(N_{v_0} - 1) \le 0$,
so $v_0 \in Z(\Teul)$, so $v_0 \in \Omega(\Teul)$ by (a).

(c) We have $\epsilon(v_0) + | \Deul_{v_0} | = \delta_{v_0}$, so the assumptions
$\delta_{v_0}=1$ and $| \Neul | > 1$ imply that $\epsilon(v_0)=1$ and $\Deul_{v_0} = \emptyset$, so the claim follows from (b).
\end{proof}

\begin{example} \label {FFlAkjcvwiuer2n3Z3cxs7df9}
Continuation of Fig.\ \ref{723dfvcjp2q98ewdywe} and Ex.\ \ref{Pc09vbq3j7gabXrZiA}, \ref{cIuCbgepwej6DsJ37655enm}, \ref{cov7btiw67d78vujew9}
and \ref{oOo8cvn239vn3p98fgIW}.
We have
$Z(\Teul) = \{ v_0, v_5, v_6 \}$, $\Gamma(\Teul) = \{ (v_5, v_4), (v_6, v_2) \}$, $W(\Teul) = \{v_2,v_4\}$, $\Omega(\Teul) = \{ v_0 \}$.
\end{example}

For the next result, it is good to note that if  $(z_1, \dots, z_n)$ is a $\tD$-trivial path
then $n\ge2$ and the condition $\epsilon(z_n)=1$ is equivalent to $\Neul = \{ z_1, \dots, z_n \}$.

\begin{theorem}  \label {Xcoikn23ifcdKJDluFYT937}
We have $| \Omega(\Teul) | \le 2$ and the following hold.
\begin{enumerata}

\item Assume that $| \Omega(\Teul) | = 2$.
Then there exists a $\tD$-trivial path $(z_1, \dots, z_n)$ satisfying $\epsilon(z_n)=1$,
and for any such path, the following hold:
\begin{itemize}

\item $\Neul = \{z_1, \dots, z_n \}$ and $\Omega(\Teul) =  \{ z_1, z_n \}$;

\item $\tD(z_1)\le0$, $\tD(z_n)\le0$ and $\tD(\Neul) = \tD(z_1) + \tD(z_n) \le 0$;

\item  $(z_n, \dots, z_1)$ is a $\tD$-trivial path satisfying $\epsilon(z_1)=1$.

\end{itemize}

\item Assume that no $\tD$-trivial path $(z_1, \dots, z_n)$ satisfies $\epsilon(z_n)=1$.
Then  $| \Omega(\Teul) | \le 1$, and if $\Omega(\Teul) \neq \emptyset$
then there exists a $\tD$-trivial path $(z_1,\dots,z_n)$
satisfying $\Omega(\Teul) = \{z_1\}$ and  $z_{n-1}<z_n$.

\end{enumerata}
\end{theorem}

\begin{proof}
We first prove (b). Let us assume that
\begin{equation} \label {03irfc0w3eX38f}
\textit{no $\tD$-trivial path $(z_1, \dots, z_n)$ satisfies  $\epsilon(z_n)=1$.}
\end{equation}
Then we claim that
\begin{equation} \label {slxp203qwsxic}
\begin{minipage}[t]{.8\textwidth}
\it
for each $z \in \Omega(\Teul)$ there exists a $\tD$-trivial path $\gamma_z = (z_1,\dots,z_n)$
satisfying $z_1=z$ and  $z_{n-1}<z_n$.
\end{minipage}
\end{equation}
Indeed, let $z \in \Omega(\Teul)$.
Define $z_1=z$.
Since $\epsilon(z_1)=1$, there is exactly one $z_2 \in \Neul$ that is adjacent to $z_1$.
Then $(z_1,z_2)$ is a $\tD$-trivial path.
Among all $\tD$-trivial paths $(z_1,z_2,\dots)$ such that $z_1=z$,
choose one,  $\gamma_z = (z_1, \dots, z_n)$, which maximizes $n$ (so $n\ge2$).
Let us prove that $\gamma_z$ satisfies  $z_{n-1}<z_n$. 
By contradiction, assume that 
$$
z_{n-1}>z_n.
$$
Let $e = \{ z_n, z_{n-1} \}$
and note that Prop.\ \ref{0hp9fh023wbpchvgrsjs256} implies that
\begin{equation} \label {KKKdfno239wdbsoo}
M(z_n,e)>1.
\end{equation}
As $z_1 \notin  \bigcup_{w \in W} V(w)$, we have  $(z_1,\dots,z_n) \notin \Gamma$, so $\tD( \{z_1,\dots,z_n\}) \le 0$.
Let $A = \{e\}$, then $\bD( z_n, A ) = \tD( \{z_1,\dots,z_n\}) \le 0$,
and since $\eta(z_n,e)\ge0$ by Cor.\ \ref{90hJkHJHF7238ewHkjHiu93}, 
we have $\bD( z_n, A ) < 1 + \sum_{e' \in A} \eta(z_n,e')$.
Then Cor.\ \ref{P90werd23ewods0ci-cor} gives $\epsilon(z_n) - |A| \in \{0,1,2\}$.
Since  $\epsilon(z_n) \ge 2$ by  assumption \eqref{03irfc0w3eX38f}, we have $\epsilon(z_n) - |A| \in \{1,2\}$.
If $\epsilon(z_n) - |A| = 2$ then Cor.\ \ref{P90werd23ewods0ci-cor}  implies that 
$R(z_n,A)=0$, which contradicts \eqref{KKKdfno239wdbsoo}; so $\epsilon(z_n) - |A| = 1$ and hence $\epsilon(z_n)=2$.
By Cor.\ \ref{P90werd23ewods0ci-cor}  we get $R(z_n,A)<1$; this and \eqref{KKKdfno239wdbsoo} imply that $\sigma(z_n)=0$
and $a_{z_n}=1$, so $\tD(z_n)=0$ by Lemma \ref{90q932r8dhd89cnr9}\eqref{pc9vp23r09}.
Let $y$ be the unique element of $\Neul\setminus\{z_{n-1}\}$ which is adjacent to $z_n$;
then $(z_1, \dots, z_n,y)$ is $\tD$-trivial, which contradicts the maximality of $\gamma_z$ and hence proves that $z_{n-1}<z_n$.
So \eqref{slxp203qwsxic} is proved.

We claim that $| \Omega(\Teul) | \le 1$.
Indeed, suppose that $z_1,z_1' \in \Omega(\Teul)$ and let us argue that $z_1=z_1'$.
Consider $\gamma_{z_1}= (z_1,\dots,z_n)$ and $\gamma_{z_1'} = (z_1',\dots,z_m')$ satisfying \eqref{slxp203qwsxic}.
Since $z_{n-1}<z_n$,
we have $v_0 \in \Neul(z_n, \{z_n, z_{n-1}\}) = \{ z_1, \dots, z_{n-1}\}$;
similarly, $v_0 \in \{ z_1', \dots, z_{m-1}' \}$.
So there exists $(i,j)$ satisfying
\begin{equation} \label {0djb23o0wesd9fd}
1\le i<n,  \ \   1\le j<m  \ \  \textit{and} \ \  z_i = z_j'.
\end{equation}
Among all pairs $(i,j)$ satisfying \eqref{0djb23o0wesd9fd}, choose one that minimizes $i$.
Assume that $i \neq 1$; then $1<i<n$, so $\epsilon(z_i)=2$, so $\epsilon(z_j')=2$ and hence $j \neq1$, so $1<j<m$.
Then $(z_{i-1}, z_i, z_{i+1})$ is equal to either $(z_{j-1}', z_j', z_{j+1}')$ or $(z_{j+1}', z_j', z_{j-1}')$.
If $(z_{i-1}, z_i, z_{i+1}) = (z_{j+1}', z_j', z_{j-1}')$ then $(z_1, \dots, z_{i-1}, z_i, z_{j-1}', \dots, z_1')$ is a $\tD$-trivial path
with $\epsilon(z_1')=1$, contradicting \eqref{03irfc0w3eX38f};
so we have $(z_{i-1}, z_i, z_{i+1}) = (z_{j-1}', z_j', z_{j+1}')$, but then $z_{i-1} =  z_{j-1}'$ contradicts the minimality of $i$.
These contradictions show that $i=1$.
Thus $\epsilon(z_i)=1$, so $\epsilon(z_j')=1$ where $1\le j < m$, so $j=1$ and hence $z_1 = z_1'$. 
So $| \Omega(\Teul) | \le 1$. Together with \eqref{slxp203qwsxic}, this proves assertion (b).

There remains to show that $| \Omega(\Teul) | \le 2$ in all cases, and that assertion (a) is true.
Obviously, it suffices to prove this under the additional assumption that  $| \Omega(\Teul) | \ge 2$.
This assumption and (b) imply that there exists a $\tD$-trivial path $(z_1, \dots, z_n)$ satisfying $\epsilon(z_n)=1$.
Then  $\Neul = \{z_1,\dots,z_n\}$ and $\Omega(\Teul) \subseteq \{z_1, z_n\}$, so $\Omega(\Teul) = \{z_1, z_n\}$ (so $| \Omega(\Teul) | \le 2$ is proved).
We have $\tD(z_1)\le0$ and $\tD(z_n)\le0$ because $z_1,z_n \in \Omega(\Teul) \subseteq Z(\Teul)$; thus $\tD(\Neul) = \tD(z_1) + \tD(z_n) \le 0$ is clear.
It is also clear that  $(z_n, \dots, z_1)$ is a $\tD$-trivial path satisfying $\epsilon(z_1)=1$.
\end{proof}

\begin{lemma} \label {0cvg248f7qvq8hene44dsdi374}
$\Gamma(\Teul)$ is the set of $\tD$-trivial paths $\gamma = (z_1,\dots,z_n)$ such that
\begin{equation} \label {pc0S2g94v3e9dhud7r4v}
\text{$\tD(z_1) \le 0 < \tD( z_n )$ \quad and \quad $z_{n-1}>z_{n}$.}
\end{equation}
\end{lemma}

\begin{smallremark}
We defined the set $\Gamma(\Teul)$ in \ref{xkclqlwpd90cod}. This Lemma gives an equivalent definition.
\end{smallremark}

\begin{proof}
Let $\Gamma'(\Teul)$ be the set of $\tD$-trivial paths satisfying \eqref{pc0S2g94v3e9dhud7r4v}.
It is clear that $\Gamma(\Teul) \subseteq \Gamma'(\Teul)$.
We prove the reverse inclusion by contradiction: assume that $\Gamma'(\Teul) \nsubseteq \Gamma(\Teul)$.
Choose $\gamma = (z_1,\dots,z_n) \in \Gamma'(\Teul) \setminus \Gamma(\Teul)$.
Since $\gamma \notin\Gamma(\Teul)$, no element $(u_1,\dots,u_m)$ of $\Gamma(\Teul)$ satisfies $u_1=z_1$.
Thus $z_1 \notin \bigcup_{w \in W(\Teul)} V(w)$;
as $z_1 \in Z(\Teul)$, it follows that $z_1 \in \Omega(\Teul)$.

Suppose that no $\tD$-trivial path $(y_1, \dots, y_m)$ satisfies $\epsilon(y_m)=1$.
Then (since $z_1 \in \Omega(\Teul)$) Thm \ref{Xcoikn23ifcdKJDluFYT937}(b) implies that 
there exists a $\tD$-trivial path $(y_1, \dots, y_m)$ such that $y_1=z_1$ and $y_{m-1}<y_m$.
Since $\tD(z_n)>0$, we must have $(y_1, \dots, y_m) = (z_1, \dots, z_i)$ for some $i\le n$.
Thus $z_{n-1}<z_n$, which contradicts the assumption \eqref{pc0S2g94v3e9dhud7r4v}.

So there must exist a $\tD$-trivial path $(y_1, \dots, y_m)$ such that $\epsilon(y_m)=1$.
Then $\Neul = \{ y_1, \dots, y_m\}$ and consequently  $\{z_1,\dots,z_n\} \subseteq \{y_1,\dots,y_m\}$.
Since $\tD(z_n)>0$, we must have $z_n \in \{y_1,y_m\}$; we also have $z_1 \in \{y_1,y_m\}$ because $\epsilon(z_1)=1$;
so $\Neul = \{z_1,\dots,z_n\}$.
As $\gamma \notin \Gamma(\Teul)$, we have $\tD(\{z_1,\dots,z_n\})\le 0$, so $\tD(\Neul) \le 0$. 
We have $\tD(\Neul) \ge (\delta_{v_0}-1) (\delta_{v_0}-2)$ by Prop.\ \ref{0c9b45okjrdidjfo}, so $\delta_{v_0} \le 2$
and consequently $| \Deul_{v_0} | \le 1$. This implies that $\sigma(v_0) = N_{v_0} - d(v_0)$.
Also note that $z_n=v_0$, because $z_{n-1}>z_n$ and $v_0 \in \{z_1,\dots,z_n\}$; so
$\tD(z_n) = \sigma(v_0) + (\epsilon(v_0)-2)(N_{v_0}-1) + N_{v_0}( 1 - \frac1{a_{v_0}}) =  (N_{v_0} - d(v_0)) + (-1) (N_{v_0} - 1) = 1-d(v_0) \le 0$,
a contradiction.  So $\Gamma(\Teul) = \Gamma'(\Teul)$.
\end{proof}

\section*{Combs}
\label {SectionCombs}

\begin{definition} \label {c09Nc23owsfcnp2q0wuXwoh}
Suppose that $(u,e), (u',e') \in P$.
We say that $(u,e)$ is a \textit{comb over $(u',e')$} if $(u,e) \succeq (u',e')$ and,
for every $(v,f) \in P$ satisfying  $(u,e) \succ (v,f) \succeq (u',e')$,
the following conditions hold:
\begin{enumerata}

\item $\epsilon(v) \in \{2,3\}$;

\item if $\epsilon(v)=2$ then $R(v,\{f\})<1$;

\item if $\epsilon(v)=3$ then  $R(v,\{f\})=0$ and  $(v,g)$ is a tooth,
where $g$ denotes the unique edge of $\Neul$ which is incident to $v$, not in $\gamma_{u,v}$ and distinct from $f$.

\end{enumerata}
\end{definition}

\begin{remark} \label {pv09w4v7AqOmksndfwoa}
Suppose that $(u,e), (u',e') \in P$ and that $(u,e)$ is a comb over $(u',e')$.
For each $(v,f) \in P$ satisfying  $(u,e) \succ (v,f) \succeq (u',e')$, we have $\epsilon(v) \in \{2,3\}$ and:
\begin{itemize}

\item if $\epsilon(v)=2$ and $v$ is a node then its type is $[d(v), N_v, \dots, N_v]$ and if $d(v) \neq N_v$ then $a_v=1 = M(v,f)$;

\item if $\epsilon(v)=3$ then $a_v=1=M(v,f)$ and if $v$ is a node then its type is $[N_v, \dots, N_v]$.

\end{itemize}
\end{remark}

\begin{remarks}  \label {c0Bj12Wsdh0982EChi}
\begin{enumerate}

\item Each element of $P$ is a comb over itself.

\item If $(u,e), (u',e'), (u'',e'') \in P$ satisfy $(u,e) \succeq (u',e') \succeq (u'',e'')$, then the following are equivalent:
\begin{itemize}

\item $(u,e)$ is a comb over $(u',e')$ and $(u',e')$ is a comb over $(u'',e'')$,

\item $(u,e)$ is a comb over $(u'',e'')$.

\end{itemize}
\end{enumerate}
\end{remarks}

Recall (from Lemma \ref{p0c9vin12q09wsc}) that $(u,e) \succeq (u',e')$ implies $\Neul(u,e) \supseteq \Neul(u',e')$,
$\tD\big( \Neul(u,e) \big) \ge \tD\big( \Neul(u',e') \big)$ and $\eta(u,e) \ge \eta(u',e')$.
We now describe what happens when $\eta(u,e) = \eta(u',e')$.

\begin{proposition} \label {0ci19KJTghL872je309} 
Suppose that $(u,e), (u',e') \in P$ satisfy $(u,e) \succeq (u',e')$.

\begin{enumerata}
\item The following are equivalent:
\begin{enumerata}

\item $\eta(u,e) = \eta(u',e')$ and $\Omega(\Teul)$ is disjoint from $\Neul(u,e) \setminus \Neul(u',e')$;

\item $(u,e)$ is a comb over $(u',e')$.

\end{enumerata}

\item The following are equivalent:
\begin{enumerata}

\item \label {98fUHhhf69gw2h8G72}  $\tD\big( \Neul(u,e) \big) = \tD\big( \Neul(u',e') \big)$
and $\Omega(\Teul)$ is disjoint from $\Neul(u,e) \setminus \Neul(u',e')$;

\item \label {98mdmosdberdtkj76lca} $(u,e)$ is a comb over $(u',e')$ and $c(u,e) = c(u',e')$;

\item  \label {pd90cfn2pw0s}   for every vertex $v$ in $\gamma_{u,u'}$ such that $v \neq u$, we have $\epsilon(v)=2$ and $\tD(v)=0$.

\end{enumerata}
\end{enumerata}
\end{proposition}

\begin{proof}
To prove assertion (a),
consider the unique sequence $(u_0,e_0) \succ \cdots \succ (u_r,e_r)$ such that $(u_0,e_0) = (u,e)$,  $(u_r,e_r) = (u',e')$
and, for each $j \in \{1, \dots, r\}$, $(u_j,e_j)$ is an immediate predecessor of $(u_{j-1},e_{j-1})$.
Note that $\eta(u,e) = \eta(u_0,e_0) \ge \eta(u_1,e_1) \ge \cdots \ge \eta(u_r,e_r) = \eta(u',e')$ and
$\Neul(u,e) \setminus \Neul(u',e') = \bigcup_{j=1}^r [ \Neul(u_{j-1},e_{j-1}) \setminus \Neul(u_{j},e_{j}) ]$;
so condition (ai) is equivalent to
``for all $j \in \{1, \dots, r\}$, (ai$_j$) is true'', where
\begin{enumerata}
\item[(ai$_j$)] \qquad $\eta(u_{j-1},e_{j-1})=\eta(u_{j},e_{j})$ and $\Omega(\Teul) \cap [ \Neul(u_{j-1},e_{j-1}) \setminus \Neul(u_{j},e_{j}) ] = \emptyset$.
\end{enumerata}
By part (2) of Rem.\ \ref{c0Bj12Wsdh0982EChi}, condition (aii) is equivalent to 
$(u_{j-1},e_{j-1})$ being a comb over $(u_{j},e_{j})$ for all $j \in \{1, \dots, r\}$.
So it suffices to prove (ai)$\Leftrightarrow$(aii) under the additional assumption that $(u',e')$ is an immediate predecessor of $(u,e)$.
In other words, the proof of (a) reduces to proving the following statement:

\begin{claim}
Let $(u,e)$ be a non minimal element of $P$,
write $e=\{u,u_0\}$, and let $(u_0,e_1), \dots, (u_0,e_n)$ ($n\ge1$) be the immediate predecessors of $(u,e)$.
Then the following are equivalent:
\begin{enumerata}

\item[$(\alpha)$] $\eta(u,e) = \eta(u_0,e_1)$ and $\Omega(\Teul)$ is disjoint from $\Neul(u,e) \setminus \Neul(u_0,e_1)$

\item[$(\beta)$]  $n \in \{ 1, 2 \}$ and:
\begin{enumerata}

\item if $n=1$ then $R(u_0, \{e_1\})< 1$;
\item if $n=2$ then  $R(u_0, \{e_1\})=0$ and $(u_0,e_2)$ is a tooth.

\end{enumerata}
\end{enumerata}
\end{claim}
Let us prove this Claim.
Assume that $(\alpha)$ is true.
We have $\eta(u,e) \ge \sum_{i=1}^n \eta(u_0,e_i)$ by Lemma \ref{okcf039wepdocS} and $\eta(u_0,e_i)\ge0$ for all $i$ by Cor.\ \ref{90hJkHJHF7238ewHkjHiu93}.
So $\eta(u_0,e_i)=0$ for all $i>1$ and $\eta(u,e) = \eta(u_0,e_1) = \sum_{i=1}^n \eta(u_0,e_i) = \sum_{e' \in A} \eta(u_0,e')$,
where $A = \{ e_1, \dots, e_n \}$. 
Then
$$
\textstyle
\bD(u_0,A) = \tD( \Neul(u,e) )
= 1 + \eta(u,e) - c(u,e) 
< 1 + \eta(u,e) 
= 1 +  \sum_{e' \in A} \eta(u_0,e').
$$
As $\epsilon({u_0})-|A| = 1$, Cor \ref{P90werd23ewods0ci-cor}  gives
$$
\textstyle    R({u_0},A)<1 .
$$

Note in particular that if $n=1$ then  $R(u_0, \{e_1\}) = R({u_0},A)<1$, so $(\beta)$ is true when $n=1$.

From now-on we assume that $n\ge2$  (then $\epsilon(u_0) \ge 3$).
Since $R({u_0},A)<1$, at most one $i \in \{1, \dots, n\}$ satisfies $M(u_0,e_i)>1$; it follows that $n=2$, by Prop.\ \ref{0hp9fh023wbpchvgrsjs256}.
So, to prove $(\beta)$, there remains to show that $R(u_0, \{e_1\})=0$ and that $(u_0,e_2)$ is a tooth.

Observe that $\Teul$ is not a brush.
Indeed, if $\Teul$ is a brush then $\epsilon(u_0) \ge 3$ implies that $W=\{u_0\}$ and $\Neul = \bar V(u_0)$,
so both $(u_0,e_1)$ and $(u_0,e_2)$ are teeth,
which contradicts the fact that  at most one $i \in \{1, 2\}$ satisfies $M(u_0,e_i)>1$. 
Also observe that $\eta(u_0,e_2)=0$, so Prop.\ \ref{DKxcnpw93sdo} gives
\begin{equation} \label {923r0fgfia3uhg6e83}
\tD( \Neul(u_0,e_2) ) \le 0 .
\end{equation}

Let $Z_2 = \setspec{ z \in \Neul(u_0,e_2) }{ \epsilon(z)=1 }$.
We have $\tD( \Neul(u_0,e_2) ) \ge \tD(Z_2)$, because 
each $x \in \Neul(u_0,e_2) \setminus Z_2$ satisfies $\epsilon(x)>1$ and hence $\tD(x)\ge0$ (Lemma \ref{90q932r8dhd89cnr9}).
Thus $\tD(Z_2) \le 0$, by \eqref{923r0fgfia3uhg6e83}.
Since $Z_2 \neq \emptyset$, it follows that some $z \in Z_2$ satisfies $\tD(z) \le 0$.
Thus $Z_2 \cap Z \neq \emptyset$.

For each $z \in Z_2 \cap Z$,
the assumption that $\Omega(\Teul)$ is disjoint from $\Neul(u,e) \setminus \Neul(u_0,e_1)$
implies that $z \notin \Omega(\Teul) = Z \setminus \bigcup_{w \in W} V(w)$,
so $z \in V(w_z)$ for some $w_z \in W$.
We claim that
\begin{equation} \label {0dfj2339ne8hqog8}
\text{$w_z = u_0$ for some $z \in Z_2 \cap Z$.}
\end{equation}
Indeed, assume that \eqref{0dfj2339ne8hqog8} is false.
Then for all  $z \in Z_2 \cap Z$ we have $\bar V(w_z) \subseteq \Neul(u_0,e_2)$.
Since Lemma \ref{p0293efp0cw23ep0hvj} gives $\tD( \bar V(w_z) ) > 0$ for each  $z \in Z_2 \cap Z \neq \emptyset$ (recall that $\Teul$ is not a brush),
it follows that the subset $Y =  \bigcup_{z \in Z_2\cap Z} \bar V(w_z)$ of $\Neul(u_0,e_2)$ satisfies $\tD(Y)>0$.
Then \eqref{923r0fgfia3uhg6e83} implies that there exists $x \in \Neul(u_0,e_2) \setminus Y$ satisfying  $\tD(x)<0$.
Then $\epsilon(x)=1$, so $x \in Z_2 \cap Z$ and hence $x \in \bar V(w_x) \subseteq Y$, a contradiction. 
So \eqref{0dfj2339ne8hqog8} must be true, and in particular $u_0 \in W$. 
Consider $z \in Z_2 \cap Z$ such that $w_z = u_0$.
Then $z \in V(u_0)$, so $\gamma_{z,u_0} \in \Gamma$; since $e_2$ is in $\gamma_{z,u_0}$, we obtain that $(u_0,e_2)$ is a tooth.
Finally, note that $R(u_0,A)<1$ together with the fact that $(u_0,e_2)$ is a tooth (so $M(u_0,e_2)>1$ by Lemma \ref{GRygergGREg8948r}) implies 
that $R(u_0,\{e_1\}) = 0$.  So $(\beta)$ is true.

Conversely, assume that $(\beta)$ is true.
Let $A = \{ e_1, \dots, e_n \}$ (where $n \in \{1,2\}$).
If $n=1$ then $R(u_0,A) = R(u_0,\{e_1\}) <1$ by ($\beta$-i), and if $n=2$ then
$R(u_0,A) = R(u_0,\{e_1\}) + (1 - \frac1{M(u_0,e_2)}) = (1 - \frac1{M(u_0,e_2)}) < 1$ by ($\beta$-ii),
so $R(u_0,A) < 1$ in all cases.
Since $R(u_0,A) < 1$  and $\epsilon(u_0) - |A| = 1$, Thm \ref{P90werd23ewods0ci} gives
$$
 1> R(u_0,A) = R(u_0,A) + (\epsilon(u_0) - |A| - 1) { \textstyle (1 - \frac1{N_{u_0}}) }
=   1 + \frac{ \bD(u_0,A) - 1 - \sum_{e' \in A}\eta(u_0,e')}{N_{u_0}},
$$
so   $\bD(u_0,A) < 1 + \sum_{e' \in A}\eta(u,e')$.
Since $\bD(u_0,A) = \tD( \Neul(u,e) )$, we have $\tD( \Neul(u,e) ) < 1 + \sum_{e' \in A}\eta(u,e')$,
so $\eta(u,e) =  \sum_{e' \in A}\eta(u_0,e') =  \sum_{i=1}^n \eta(u_0,e_i)$ by Cor.\ \ref{okcf039wepdocS}.
We consider the cases $n=1$, $n=2$ separately.

If $n=1$ then $\eta(u,e) =  \sum_{i=1}^n \eta(u_0,e_i) = \eta(u_0,e_1)$.
Moreover,  $\Neul(u,e) \setminus \Neul(u_0,e_1) = \{ u_0 \}$ and $u_0 \notin \Omega(\Teul)$ because $\epsilon(u_0) \neq 1$,
so  $\Omega(\Teul)$ is disjoint from $\Neul(u,e) \setminus \Neul(u_0,e_1)$ and $(\alpha)$ is true.

If $n=2$ then $(u_0,e_2)$ is a tooth by ($\beta$-ii), so $\eta(u_0,e_2)=0$ by Lemma \ref{GRygergGREg8948r} and consequently 
$\eta(u,e) =  \sum_{i=1}^n \eta(u_0,e_i) = \eta(u_0,e_1)$.
By contradiction, suppose that there exists an element $z$ of $\Omega(\Teul) \cap [ \Neul(u,e) \setminus \Neul(u_0,e_1)]$.
Then $\epsilon(z)=1$, so $z \neq u_0$. We have 
$\Neul(u,e) \setminus \Neul(u_0,e_1) = \{ u_0 \} \cup \Neul(u_0,e_2)$, so $z \in \Neul(u_0,e_2)$.
Since $(u_0,e_2)$ is a tooth, we must have $\gamma_{z,u_0} \in \Gamma$, so $u_0 \in W$ and $z \in V(u_0)$, contradicting the fact that $z \in \Omega(\Teul)$.
So $\Omega(\Teul) \cap [ \Neul(u,e) \setminus \Neul(u_0,e_1)] = \emptyset$, and $(\alpha)$ is true.

So $(\beta)$ implies $(\alpha)$. This proves the Claim, and completes the proof of part (a) of the Proposition.

\bigskip

\noindent (b) Since $\tD\big( \Neul(u,e) \big) = \tD\big( \Neul(u',e') \big)$ is equivalent to 
$\eta(u,e) = \eta(u',e')$ and  $c(u,e) = c(u',e')$ by Lemma \ref{p0c9vin12q09wsc},
and since (a) is true, it follows that \eqref{98fUHhhf69gw2h8G72} is equivalent to  \eqref{98mdmosdberdtkj76lca}.
Let us prove that \eqref{98fUHhhf69gw2h8G72} is equivalent to \eqref{pd90cfn2pw0s}.
We may assume that $u \neq u'$, otherwise the result is trivial.
Let $(u,u']$ denote the set of  vertices $v$ in $\gamma_{u,u'}$ such that $v \neq u$.

If \eqref{pd90cfn2pw0s} is true then $\Neul(u,e) \setminus \Neul(u',e') = (u,u']$,
so each $v \in \Neul(u,e) \setminus \Neul(u',e')$  satisfies $\epsilon(v)=2$ and $\tD(v)=0$.
Since $\epsilon(v)=2$ for all $v \in \Neul(u,e) \setminus \Neul(u',e')$,
$\Omega(\Teul)$ is disjoint from $\Neul(u,e) \setminus \Neul(u',e')$;
since $\tD(v)=0$ for all $v \in \Neul(u,e) \setminus \Neul(u',e')$,
we have $\tD\big( \Neul(u,e) \big) = \tD\big( \Neul(u',e') \big)$.
So \eqref{pd90cfn2pw0s} implies \eqref{98fUHhhf69gw2h8G72}.  

Conversely, assume that  \eqref{98fUHhhf69gw2h8G72} is true.  
Then $\eta(u,e)=\eta(u',e')$ and $c(u',e') = c(u,e)$, by Lemma \ref{p0c9vin12q09wsc}.
Consider $v \in (u,u']$; we have to show that $\epsilon(v)=2$ and $\tD(v)=0$.
Note that there is a unique edge $f$ such that $(v,f) \in P$ and  $(u,e) \succ (v,f) \succeq (u',e')$.
Since $\eta(u,e)=\eta(u',e')$ and $\Omega(\Teul)$ is disjoint from $\Neul(u,e) \setminus \Neul(u',e')$,
part (a) implies that $(u,e)$ is a comb over $(u',e')$, so
$\epsilon(v) \in \{2,3\}$ and if $\epsilon(v) = 3$ then $R(v,\{f\})=0$ and $(v,g)$ is a tooth (notation $g$ as in Def.\ \ref{c09Nc23owsfcnp2q0wuXwoh}).
Suppose that $\epsilon(v)=3$. The fact that $R(v,\{f\})=0$ implies that $c(v,f) = N_v$ and the fact that $(v,g)$ is a tooth implies that $c(v,g)<N_v$,
so $c(v,g)<c(v,f)$; since $(u,e) \succ (v,g)$, we have $c(u,e) \le c(v,g)$, so  $c(u,e) \le c(v,g) < c(v,f) \le c(u',e')$, a contradiction.
This shows that $\epsilon(v)=2$ (for every $v \in (u,u']$). Then   $\Neul(u,e) \setminus \Neul(u',e') = (u,u']$, so 
$$
\textstyle  \sum_{ v \in (u,u'] } \tD(v) =  \tD\big( \Neul(u,e) \big) - \tD\big( \Neul(u',e') \big) = 0 .
$$
For each  $v \in (u,u']$ we have $\epsilon(v)>1$, so $\tD(v)\ge0$; it follows that 
$\tD(v)=0$ for all $v \in (u,u']$, so \eqref{pd90cfn2pw0s} holds and we are done.
\end{proof}

The next result is a corollary of  Lemma \ref{0hp9fh023wbpchvgrsjs256}, and is closely related to Prop.\ \ref{0ci19KJTghL872je309}.

\begin{corollary}
Let $(u,e), (u',e') \in P$.
If $(u,e)$ is a comb over $(u',e')$ and $u<u'$ then we have $\epsilon(v)=2$ and $\tD(v)=0$ for every vertex $v$ in $\gamma_{u,u'}$ such that $v \neq u$.
\end{corollary}

\begin{proof}
Consider a vertex $v$ in $\gamma_{u,u'}$ such that $v \neq u$, and let $f$ be the edge satisfying $(u,e) \succ (v,f) \succeq (u',e')$.
If we define $w$ by $f = \{v,w\}$ then $v<w$, so $M(v,f)>1$ by Lemma \ref{0hp9fh023wbpchvgrsjs256}.
Then $R(v,\{f\})>0$, so the definition of comb implies that $\epsilon(v)=2$.
We have $R(v,\{f\})<1$ (again by definition of comb), so $M(v,f)>1$ implies that $\sigma(v)=0$ and $a_v=1$, so $\tD(v)=0$.
\end{proof}

\section{Global structure: first steps}
\label{Section:Globalstructurefirststeps}

{\it We continue to assume that $\Teul$ is a minimally complete abstract Newton tree at infinity.}

\begin{definition} \label {FTJOIhgrjHIO76r67yiN}
The set $S(\Teul) = \Neul \setminus \bigcup_{w \in W(\Teul)}\big( \bar V(w) \setminus \{w\} \big)$ is called the \textit{skeleton} of $\Teul$.
See Def.\ \ref{xkclqlwpd90cod} and \ref{pd09Yv3ned09Xse} for notations.
\end{definition}

\begin{lemma}  \label {p309ef2GFT485ryf}
\begin{enumerata}

\item $S(\Teul)$ is a nonempty subtree of $\Neul$.

\item $v_0 \in S(\Teul)$, $W(\Teul) \subseteq S(\Teul)$, $\Omega(\Teul) \subseteq S(\Teul)$ and if  $\Omega(\Teul) = S(\Teul)$ then $|\Omega(\Teul)| = 2$.
In particular, if $\Omega(\Teul) \neq \emptyset$ then $|S(\Teul)|>1$.

\item $\Neul = \bigcup_{v \in S(\Teul)} \bar V(v)$,
and $\bar V(v) \cap \bar V(v') = \emptyset$ for every choice of distinct  $v, v' \in S(\Teul)$.

\item $| S(\Teul) | = 1$ if and only if $\Teul$ is a brush or $| \Neul | = 1$.

\end{enumerata}
\end{lemma}

\begin{proof}
(b) We abbreviate $W(\Teul)$ to $W$. 
We know that $v_0 \in \Neul$.  If $v_0 \notin S(\Teul)$ then there exists $w \in W$ such that $v_0 \in \bar V(w) \setminus \{w\}$,
which contradicts Rem.\ \ref{pc09n3409vnZiEWOdhFp30ef}.
So $v_0 \in S(\Teul)$.

We have $\tD(x) \le 0$ for all $x \in \bigcup_{w \in W}\big( \bar V(w) \setminus \{w\} \big)$ and 
$\tD(x) > 0$ for all $x \in W$,
so $W$ is disjoint from  $\bigcup_{w \in W}\big( \bar V(w) \setminus \{w\} \big)$ and hence $W \subseteq S(\Teul)$.

Let $z \in \Omega(\Teul)$. Since $\epsilon(z)=1$, a unique $y \in \Neul$ is adjacent to $z$.
We prove $\{z,y\} \subseteq S(\Teul)$ by contradiction.
Suppose that $\{z,y\} \nsubseteq S(\Teul)$; then there exists $w \in W$ such that 
$\{z,y\} \cap ( \bar V(w) \setminus \{w\} ) \neq \emptyset$;
so there exists $(x_1,\dots,x_n) \in \Gamma(\Teul)$ such that $x_n=w$ and $\{z,y\} \cap \{x_1,\dots,x_{n-1}\} \neq \emptyset$;
note that $z \neq w$ (because $\tD(z)\le0$) and that if only one element of $\{z,y\}$ belongs to $\{x_1,\dots,x_{n-1}\}$ then the other element
must be $w$; it follows that $(z,y)=(x_1,x_2)$, so $z \in V(w)$, which contradicts $z \in \Omega(\Teul)$.
So $\{z,y\} \subseteq S(\Teul)$.
If $\Omega(\Teul) = S(\Teul)$ then $\{z,y\} \subseteq \Omega(\Teul)$, so  $|\Omega(\Teul)| = 2$.
This argument shows that $\Omega(\Teul) \subseteq S(\Teul)$ and that if $\Omega(\Teul) = S(\Teul)$ then $|\Omega(\Teul)| = 2$.
So (b) is proved.

(a) Since $\emptyset \neq S(\Teul) \subseteq \Neul$ and $\Neul$ is a tree, it suffices to show that $S(\Teul)$ is connected.
If $\Gamma(\Teul) = \emptyset$ then $W = \emptyset$ and hence $S(\Teul) = \Neul$ is a nonempty tree.
Assume that $\Gamma(\Teul) \neq \emptyset$ and let $\gamma_1, \dots, \gamma_r$ be the distinct elements of $\Gamma(\Teul)$,
where we write $\gamma_i = (x_{i,1}, \dots, x_{i,n_i})$ for each $i=1,\dots,r$.
Define a sequence $(S_0, \dots, S_r)$ of sets by $S_0 = \Neul$ and
$S_{i} = S_{i-1} \setminus \{ x_{i,1}, \dots, x_{i,n_i-1} \}$ for each $i=1,\dots,r$.
Then $S(\Teul) = S_r$.
We know that $S_0$ is connected, and it is quite clear that if $S_{i-1}$ is connected then so is $S_i$.
So $S_r = S(\Teul)$ is connected.

(c) It is clear that $S(\Teul) \subseteq \bigcup_{v \in S(\Teul)} \bar V(v) \subseteq \Neul$.
If $x \in \Neul \setminus S(\Teul)$ then $x \in  \bar V(w) \setminus \{w\}$ for some $w \in W$;
since $W \subseteq S(\Teul)$, we have $x \in  \bigcup_{v \in S(\Teul)} \bar V(v)$, so $\Neul = \bigcup_{v \in S(\Teul)} \bar V(v)$.
It is clear that $\bar V(v) \cap \bar V(v') = \emptyset$ whenever  $v,v'$ are distinct elements of $S(\Teul)$.

(d) Assume that $| S(\Teul) | = 1$ and $| \Neul | \neq 1$. Then $S(\Teul) \neq \Neul$, so $W \neq \emptyset$. 
Since $W \subseteq S(\Teul)$, we obtain that $W = S(\Teul)$ is a singleton $\{w\}$.
Then (c) gives $\Neul = \bar V(w)$, so $\Teul$ is a brush.
This shows that if $| S(\Teul) | = 1$ then $\Teul$ is a brush or $| \Neul | = 1$.
The converse is clear.
\end{proof}

\section*{The case $| S(\Teul) | = 1$}
\label {ThecaseSTeul1}

The remainder of this section is devoted to the special case  $| S(\Teul) | = 1$.
The general case ($| S(\Teul) | \ge 1$) is studied in the subsequent sections.

Recall from Lemma \ref{p309ef2GFT485ryf} that $v_0 \in S(\Teul)$ and that
$$
| S(\Teul) | = 1 \iff \text{ $\Teul$ is a brush or $| \Neul | = 1$,}
$$
where the two clauses  ``$\Teul$ is a brush'' and  ``$| \Neul | = 1$'' are mutually exclusive.
Also note that $| S(\Teul) | = 1$ implies that, for each $e \in \Eeul_{v_0}$, $(v_0,e)$ is a tooth and hence is nonpositive.

\begin{notation} \label {FpPpc09bne0fviIwZemf90}
The following special notation is used in the present section (and only in the present section).
Let $E$  denote the set of all edges of $\Teul$ incident to $v_0$.
For each $e = \{ v_0, u\} \in E$, define
\begin{gather*}
a(e) = \begin{cases} a_u, & \text{if $u \in \Deul_{v_0}$,} \\ p(v_0,e)/c(v_0,e), & \text{if $u \notin \Deul_{v_0}$,} \end{cases}
\qquad
d(e) = \begin{cases} d_u, & \text{if $u \in \Deul_{v_0}$,} \\ c(v_0,e), & \text{if $u \notin \Deul_{v_0}$,} \end{cases} \\[1mm]
k(e) = N_{v_0}/d(e),   \qquad x(e) = q(e,u) .
\end{gather*}
Note that in the case where $u \in \Deul_{v_0}$ we have $k(e) = k_u = -\det(e) = a(e) - x(e)$.
We also use the abbreviations $N = N_{v_0}$ and $\delta = \delta_{v_0} = |E|$.
\end{notation}

\begin{lemma}  \label {a0v2b48fvwogydhdbfsr034j}
If $| S(\Teul) | = 1$ then the following hold.
\begin{enumerata}

\item $\delta, N \in \Nat\setminus \{0\}$ 

\item For all $e \in E$ we have   $a(e), d(e), k(e) \in \Nat\setminus \{0\}$ and $N = k(e) d(e)$.

\item $N = \sum_{e \in E} a(e) d(e)$ and $R(v_0, \Eeul_{v_0}) = \sum_{e \in E} (1 - \frac{1}{k(e)})$

\item If $\gcd\setspec{ d_u }{ u \in \Deul } = 1$ then  $\gcd\setspec{ d(e) }{ e \in E } = 1$.

\end{enumerata}

\end{lemma}

\begin{proof}
(a) is clear.

(b) It is clear that  $N = k(e) d(e)$ for all $e \in E$.
For each $e \in \Eeul_{v_0}$, $k(e) = N/c(v_0,e) = M(v_0,e)$ and $a(e) = p(v_0,e)/c(v_0,e)$ are positive integers by  Thm \ref{xncoo9qwdx9}.
The assumption $| S(\Teul) | = 1$ implies that (for each  $e \in \Eeul_{v_0}$) $(v_0,e)$ is nonpositive,
so $d(e) = c(v_0,e) \in \Nat\setminus\{0\}$ by Rem.\ \ref{uyhmdytjwhwhrkdftraef4gf23h23}.
If $e \in E \setminus \Eeul_{v_0}$ then $e = \{ v_0, u \}$ with 
$u \in \Deul_{v_0}$, so $a(e) = a_u$, $d(e) = d_u$ and $k(e) = N/d_u = k_u$ all belong to $\Nat \setminus \{0\}$.

(c) We have $N = \sum_{u \in \Deul_{v_0}} a_u d_u  +  \sum_{e \in \Eeul_{v_0}} p(v_0,e) =  \sum_{e \in E} a(e) d(e)$ and
$R(v_0, \Eeul_{v_0})
= \sum_{u \in \Deul_{v_0}} (1 - \frac{1}{k_u}) + \sum_{e \in \Eeul_{v_0}} (1 - \frac{1}{M(v_0,e)})
= \sum_{e \in E} (1 - \frac{1}{k(e)})$.

(d) We may assume that $\Eeul_{v_0} \neq \emptyset$, otherwise $\Deul = \Deul_{v_0}$ and the claim is clear.
For each $e \in \Eeul_{v_0}$, let $\Deul(e) = \bigcup_{z \in \Nd(v_0,e)} \Deul_z$ and note that $\Nd(v_0,e) \neq \emptyset$;
by Lemma \ref{kjwoeid9cse9c}(c), the positive integer $d(e) = c(v_0,e)$ divides $\gcd\setspec{d(z)}{ z \in \Nd(v_0,e) } = \gcd\setspec{ d_u }{ u \in \Deul(e)}$.
So the integer $\gcd\setspec{ d(e) }{ e \in \Eeul_{v_0} }$ divides  $\gcd\setspec{ d_u }{ u \in \Deul' }$, where we define 
$\Deul' = \bigcup_{e \in \Eeul_{v_0}} \Deul(e)$.
It follows that  $\gcd\setspec{ d(e) }{ e \in E }$ divides $\gcd\setspec{ d_u }{ u \in \Deul_{v_0} \cup \Deul' }$.
As $\Deul_{v_0} \cup \Deul' = \Deul$, we are done.
\end{proof}

\begin{lemma} \label {PpPppc0v9vv32487fCd}
If $| S(\Teul) | = 1$ then
$\tD(\Neul) = 2 + \sum_{e \in E} [(\delta-2)a(e) - 1] d(e)$.
\end{lemma}

\begin{proof}
The assumption $| S(\Teul) | = 1$ implies that $\eta(v_0,e) = 0$ for all $e \in \Eeul_{v_0}$,
so Cor.\ \ref{dox9f80293ewf0di} gives the first equality in:
\begin{multline*}
\tD(\Neul) - 2 = ( R(v_0, \Eeul_{v_0}) - 2 ) N
=  \textstyle   N \sum_{e \in E} (1 - \frac1{k(e)}) - 2N
=  \textstyle  \sum_{e \in E} (N - {d(e)}) - 2N \\
=  \textstyle  (\delta-2)N - \sum_{e \in E} d(e)
=  \textstyle  (\delta-2) \sum_{e \in E} a(e) d(e) - \sum_{e \in E} d(e) = \sum_{e \in E} [(\delta-2)a(e) - 1] d(e) .
\end{multline*}
\end{proof}

\begin{corollary} \label {pc09vnw3oe9c}
If $| S(\Teul) | = 1$ then the following hold.
\begin{enumerata}

\item $\tD(\Neul) > 0 \iff\tD(\Neul) \ge 2 \iff \delta > 2$

\item If $\delta > 3$ then $\delta \le \delta(\delta-3) \le (\delta-3) \sum_{e \in E} d(e) \le \tD(\Neul) - 2$.

\end{enumerata}
\end{corollary}

\begin{proof}
We have $\delta>2$ $\Leftrightarrow$ $\sum_{e \in E} [(\delta-2)a(e) - 1] d(e) \ge 0$
$\Leftrightarrow$ $\tD(\Neul) \ge 2$ by Lemma \ref{PpPppc0v9vv32487fCd}.
To finish the proof of (a), it remains to show that  $\tD(\Neul) \neq 1$.
Assume that  $\tD(\Neul) = 1$; then  $\sum_{e \in E} [(\delta-2)a(e) - 1] d(e) = -1$, so $\delta \in \{1,2\}$.
If $\delta=1$ then $E = \{e\}$ and $(a(e)+1)d(e)=1$, which is impossible.
If $\delta=2$ then $E = \{ e_1, e_2 \}$ and $d(e_1)+d(e_2)=1$, which is impossible.
So (a) is proved.

Assume that $\delta > 3$. The first inequality in (b) is clear, and the second one is obtained by mutiplying $\delta \le \sum_{e \in E} d(e)$ by $(\delta-3)$.
We get the third inequality by noting that $a(e) \ge 1$ implies $(\delta-2)a(e) - 1 \ge \delta-3$, so 
$
\tD(\Neul) - 2 = \sum_{e \in E} [ (\delta-2)a(e) - 1 ] d(e) \ge \sum_{e \in E} (\delta-3) d(e) .
$
\end{proof}

\begin{corollary}  \label {9vbrtyukey6idckFfFugus}
If $\Teul$ is a brush then $\tD(\Neul) \ge 2$.
\end{corollary}

\begin{proof}
By contradiction, assume that $\Teul$ is a brush and $\tD(\Neul) < 2$.
Since $\Teul$ is a brush, we have $| S(\Teul) | = 1$ and $| \Neul | > 1$.
Since $\tD(\Neul) < 2$, we have $\delta \le 2$ by Cor.\ \ref{pc09vnw3oe9c}, so $| \Deul_{v_0} | + \epsilon(v_0) \le 2$. 
Since $| \Neul | > 1$, we have $\epsilon(v_0) > 0$ and hence $| \Deul_{v_0} | \le 1$, so $\sigma(v_0) < N$.
Either $\epsilon(v_0)=2$ and $| \Deul_{v_0} | = 0$ (in which case $\tD(v_0) = 0$)
or $\epsilon(v_0)=1$ and $| \Deul_{v_0}  | \le 1$ (in which case  $\tD(v_0) = \sigma(v_0) + (-1)(N-1) \le 0$ because $\sigma(v_0) < N$).
So in all cases we have  $\tD(v_0) \le 0$, so $v_0 \notin W(\Teul)$, contradicting that $\Teul$ is a brush.
\end{proof}

\begin{proposition} \label {c0wkudmFDdkKkserfdkzfmvXxlie}
If $| \Neul | = 1$ then $\tD(\Neul)$ is even.
\end{proposition}

\begin{proof}
We proceed by contradiction: assume that $| \Neul | = 1$ and $\tD(\Neul)$ is odd.
Use the notations $E = \{ e_1, \dots, e_\delta\}$,  $a_i = a(e_i)$, $d_i = d(e_i)$, $k_i = k(e_i)$  and $x_i = x(e_i)$.
Since $\gcd(a_i,x_i)=1$ and $k_i = -\det(e_i) = a_i - x_i$, we have $\gcd( a_i, k_i ) = 1$ for all $i =1, \dots, \delta$.

Consider the case where $\delta$ is even.
Then Lemma \ref{PpPppc0v9vv32487fCd} implies that $\sum_{i=1}^\delta d_i$ is odd; since this sum is odd and the number $\delta$ of terms is even,
some $d_j$ must be even.
Then $N$ is even, since $d_j \mid N$.
Then $\sum_{i=1}^\delta (a_i-1)d_i = N - \sum_{i=1}^\delta d_i$ is odd and consequently there must exist an $r$ such that $(a_r-1) d_r$ is odd.
Then $d_r$ is odd and $a_r$ is even.
Since $\gcd(a_r,k_r)=1$, it follows that $k_r$ is odd, so $N = k_r d_r$ is odd, a contradiction.

Consider the case where $\delta$ is odd.
Then Lemma \ref{PpPppc0v9vv32487fCd} implies that $\sum_{i=1}^\delta (a_i-1) d_i \equiv 1 \pmod{2}$ and consequently
there must exist an $r$ such that $(a_r-1) d_r$ is odd.
Then $d_r$ is odd and $a_r$ is even.
Since $\gcd(a_r,k_r)=1$, it follows that $k_r$ is odd, so $N = k_r d_r$ is odd.
So $d_i$ is odd for all $i$ (because $d_i \mid N$).
It then follows from Lemma \ref{PpPppc0v9vv32487fCd} that $1 \equiv \sum_{i=1}^\delta ((\delta-2)a_i-1) d_i \equiv  \sum_{i=1}^\delta (a_i-1)
\equiv  \sum_{i=1}^\delta a_i - \delta \equiv  \sum_{i=1}^\delta a_i - 1 \pmod{2}$, 
so $\sum_{i=1}^\delta a_i$ is even.
Then $1 \equiv N \equiv \sum_{i=1}^\delta a_i d_i \equiv \sum_{i=1}^\delta a_i \equiv 0 \pmod2$, a contradiction.
\end{proof}

\begin{example}  \label {98vvhb349rfcJfe9F}
Let  $(d_1,d_2,d_3) \in \{ (1,1,1), (1,1,2), (1,2,3) \}$ and set
\begin{equation} \label {O0oZ8nvdfjiejyfanscix6f9f}
x_1 = - \frac{d_2 + d_3}{d_1} , \quad x_2 = - \frac{d_1 + d_3}{d_2} , \quad x_3 = - \frac{d_1 + d_2}{d_3} .
\end{equation}
Then the following is a minimally complete Newton tree at infinity satisfying $| \Neul | = 1$,\\
$\gcd\setspec{ d_u }{ u \in \Deul } = 1$ and $\tD(\Neul)=2$.
$$
\setlength{\unitlength}{1mm}
\begin{picture}(60,20)(-30,-14)
\put(0,0){\circle{1}}
\put(-15,0){\circle*{1}}
\put(.4472,.2236){\line(2,1){12}}
\put(.4472,-.2236){\line(2,-1){12}}
\put(12,6){\circle*{1}}
\put(12,6){\vector(0,-1){8}}
\put(12,6){\line(1,0){8}}
\put(12,-6){\circle*{1}}
\put(12,-6){\vector(0,-1){8}}
\put(12,-6){\line(1,0){8}}
\put(-.5,0){\line(-1,0){14.5}}
\put(-15,0){\line(-1,0){8}}
\put(-15,0){\vector(0,-1){8}}
\put(-14,-7){\makebox(0,0)[l]{\footnotesize $(0)$}}
\put(13,-2){\makebox(0,0)[bl]{\footnotesize $(0)$}}
\put(13,-14){\makebox(0,0)[bl]{\footnotesize $(0)$}}
\put(-0.8,1){\makebox(0,0)[b]{\footnotesize $v_0$}}
\put(-13,1){\makebox(0,0)[bl]{\tiny $x_1$}}
\put(10,5.5){\makebox(0,0)[br]{\tiny $x_2$}}
\put(10,-5.5){\makebox(0,0)[tr]{\tiny $x_3$}}
\put(19.5,6.1){\makebox(0,0)[l]{\tiny $< d_2$}}
\put(19.5,-5.9){\makebox(0,0)[l]{\tiny $< d_3$}}
\put(-22.5,0.1){\makebox(0,0)[r]{\tiny $d_1 >$}}
\end{picture}
$$
\end{example}

\begin{corollary} \label {pc0wbyrjo79e8rnn9}
Suppose  $| S(\Teul) | = 1$, $\tD(\Neul)=2$ and $\gcd\setspec{ d_u }{ u \in \Deul } = 1$.
Then $\delta = 3$. Let us write $E = \{ e_1, e_2, e_3 \}$ with $d(e_1) \le d(e_2) \le d(e_3)$,
and define $a_i = a(e_i)$ and $d_i = d(e_i)$.
\begin{enumerata}

\item $(a_1,a_2,a_3) = (1,1,1)$ and $(d_1,d_2,d_3) \in \{ (1,1,1), (1,1,2), (1,2,3) \}$. 

\item If $| \Neul | = 1$ then $\Teul$ is one of the three trees of  Ex.\ \ref{98vvhb349rfcJfe9F}.

\end{enumerata}
\end{corollary}

\begin{proof}
Cor.\ \ref{pc09vnw3oe9c} implies that $\delta = 3$, so 
Lemma \ref{PpPppc0v9vv32487fCd} gives  $\sum_{i=1}^3 (a_i - 1) d_i  = 0$. 
This implies that $a_1=a_2=a_3=1$. So $N = \sum_{i=1}^3 a_i d_i = d_1 + d_2 + d_3$, and consequently each $d_i$ divides $d_1 + d_2 + d_3$. 
Since $\gcd(d_1,d_2,d_3)=1$ and $d_1 \le d_2 \le d_3$, we have $(d_1,d_2,d_3) \in \{ (1,1,1), (1,1,2), (1,2,3) \}$.
This proves (a).
Assume that $| \Neul | = 1$ and write $k_i = k(e_i)$ and $x_i = x(e_i)$.
Since (for each $i$) $N = k_i d_i$ and $k_i = -\det(e_i) = a_i - x_i = 1 - x_i$,
it follows that $x_1,x_2,x_3$ are as stipulated in  \eqref{O0oZ8nvdfjiejyfanscix6f9f}.
So  $\Teul$ is one of the three trees of  Ex.\ \ref{98vvhb349rfcJfe9F}.
\end{proof}

\begin{example} \label {nv63jfy64nvy3}
The trees depicted below are minimally complete abstract Newton trees at infinity satisfying $\Neul = \{v_0\}$,
$\gcd\setspec{ d_u }{ u \in \Deul } = 1$ and $\tD(\Neul)=0$.
In part (b), $a_1,a_2$ are arbitrary positive integers satisfying $\gcd(a_1,a_2)=1$.
In (a) (resp.\ (b)), $v_0$ is a node of type $[1]$ (resp.\ $[1,1]$).
$$
\text{\rm(a)} \scalebox{.9}{
\setlength{\unitlength}{1mm}%
\begin{picture}(28,14)(-2,-11)
\put(0,0){\circle{1}}
\put(15,0){\circle*{1}}
\put(.5,0){\line(1,0){14.5}}
\put(15,0){\vector(1,0){10}}
\put(15,0){\vector(0,-1){10}}
\put(16,-9){\makebox(0,0)[l]{\footnotesize $(0)$}}
\put(0,1){\makebox(0,0)[b]{\footnotesize $v_0$}}
\put(12,1){\makebox(0,0)[br]{\tiny $0$}}
\end{picture}
}
 \qquad\qquad\qquad 
\text{\rm(b)} \scalebox{.9}{
\setlength{\unitlength}{1mm}
\begin{picture}(50,14)(-25,-11)
\put(0,0){\circle{1}}
\put(-15,0){\circle*{1}}
\put(15,0){\circle*{1}}
\put(.5,0){\line(1,0){14.5}}
\put(-.5,0){\line(-1,0){14.5}}
\put(15,0){\vector(1,0){10}}
\put(15,0){\vector(0,-1){10}}
\put(-15,0){\vector(-1,0){10}}
\put(16,-9){\makebox(0,0)[l]{\footnotesize $(0)$}}
\put(-15,0){\vector(0,-1){10}}
\put(-14,-9){\makebox(0,0)[l]{\footnotesize $(0)$}}
\put(0,1){\makebox(0,0)[b]{\footnotesize $v_0$}}
\put(-12,1){\makebox(0,0)[b]{\tiny $-a_2$}}
\put(12,1){\makebox(0,0)[b]{\tiny $-a_1$}}
\put(-15.5,-2){\makebox(0,0)[tr]{\tiny $a_1$}}
\put(14.5,-2){\makebox(0,0)[tr]{\tiny $a_2$}}
\end{picture}
}
$$
\end{example}

\begin{proposition} \label {c0viwjytsdDJjLlxdifFebg}
Suppose that $\Teul$ satisfies $| S(\Teul) | = 1$, $\tD(\Neul) < 2$
and at least one of the following conditions:
\begin{enumerata}

\item[(i)] $\gcd\setspec{ d_u }{ u \in \Deul } = 1$; \qquad {\rm (ii)}  $\tD(\Neul) \ge 0$.

\end{enumerata}
Then $\Teul$ is one of the trees of  Ex.\ \ref{nv63jfy64nvy3}.
\end{proposition}

\begin{proof}
Suppose that $| S(\Teul) | = 1$, that $\tD(\Neul) < 2$, and that one of (i), (ii) is true.
Then $\Teul$ is not a brush by Cor.\ \ref{9vbrtyukey6idckFfFugus}, so  $| \Neul | = 1$ by  Lemma \ref{p309ef2GFT485ryf}.
Use the notations $E = \{ e_1, \dots, e_\delta\}$,  $a_i = a(e_i)$, $d_i = d(e_i)$, $k_i = k(e_i)$  and $x_i = x(e_i)$.
Since $| \Neul | = 1$, we have $k_i = -\det(e_i) = a_i - x_i$ for all $i$.
Cor.\ \ref{pc09vnw3oe9c} implies that $\delta \in \{1,2\}$.

Assume that $\delta=1$.  If (i) holds then $d_1 = \gcd(d_1, \dots, d_\delta)=1$,
so $N = a_1 d_1 = a_1$; since $N = k_1 d_1 = k_1 =  a_1 - x_1$, we get $x_1=0$ and hence $a_1=1$ because $\gcd(a_1,x_1)=1$;
so $\Teul$ is the tree (a) of Ex.\ \ref{nv63jfy64nvy3}.
If (ii) holds then $0 \le \tD(\Neul) = \tD(v_0) = \sigma(v_0) + (-2)(N-1) = (N-d_1) -2(N-1)$, so  $N+d_1\le 2$, so $d_1=1$ and hence (i) holds,
so again $\Teul$ is the tree (a) of Ex.\ \ref{nv63jfy64nvy3}.

Assume that $\delta=2$.
If (i) holds then $N = a_1 d_1 + a_2 d_2$; as $d_1 \mid N$ and $\gcd(d_1,d_2)=1$, we get $d_1 \mid a_2$ and by symmetry $d_2 \mid a_1$.
We have $(a_1-x_1)d_1 = k_1 d_1 = N = a_1 d_1 + a_2 d_2$, so $x_1 = - (a_2/d_1) d_2$, so $d_2 \mid x_1$;
since $d_2 \mid a_1$ and $\gcd(a_1,x_1)=1$, it follows that $d_2=1$.
By symmetry, $d_1=1$. This gives $x_1 = -a_2$ and $x_2 = -a_1$, so $\Teul$ is the tree (b) of Ex.\ \ref{nv63jfy64nvy3}.
If (ii) holds then 
$0 \le \tD(\Neul) = \tD(v_0) = \sigma(v_0) + (-2)(N-1) = (N-d_1)+(N-d_2) -2(N-1)$, so  $d_1+d_2\le 2$, so $d_1=1=d_2$, so (i) holds and hence
$\Teul$ is the tree (b) of Ex.\ \ref{nv63jfy64nvy3}.
\end{proof}

\section{Global structure: decomposition into combs}
\label{Section:GlobalstructureofNeul}

{\it We continue to assume that $\Teul$ is a minimally complete abstract Newton tree at infinity.}

\begin{definition} \label {Xgo6cTYRDh4i8ru237wu}
We shall say that an edge $e$ of $\Neul$ \textit{is of type $\Gamma$} if there exists $\gamma \in \Gamma(\Teul)$ such that $e$ is in $\gamma$.
Given $v \in \Neul$, let $t(v)$ be the number of edges of $\Neul$  incident to $v$ and of type $\Gamma$,
and let $\delta^*(v)$  be the number of edges of $\Neul$ incident to $v$ and not of type $\Gamma$.
\end{definition}

\begin{remark} \label {p0c92n30werf0dof}
The following trivial observations should be kept in mind.
\begin{enumerata}

\item $\delta^*(v) + t(v) = \epsilon(v)$ for all $v \in \Neul$.

\item If $v \in \Neul \setminus S(\Teul)$ then $\delta^*(v)=0$.
\item If $v \in S(\Teul)$ then $\delta^*(v)$ is equal to the valency of $v$ as a vertex of the tree $S(\Teul)$.

\item If $| S(\Teul) | > 1$ then $S(\Teul) = \setspec{ v \in \Neul }{ \delta^*(v)>0 }$.

\item \label {0ck2j3wnsdo9} Let $v \in S(\Teul)$. If $e$ is an edge of $\Neul$ incident to $v$ then $e$ is of type $\Gamma$ if and only if $(v,e)$ is a tooth,
so $t(v)$ is the number of teeth of the form $(v,e)$.
Note that $t(v)>0$ is equivalent to $v \in W(\Teul)$, and if $t(v)>0$ then $t(v) = | V(v) |$.

\end{enumerata}
\end{remark}

\begin{notation} \label {pc0v23e0geCpgspyr23}
Let $P_S(\Teul) = \setspec{ (u,e) \in P(\Teul) }{ u \in S(\Teul) }$.
Observe that if $(u,e), (u',e') \in P_S(\Teul)$ satisfy $(u',e') \succeq (u,e)$ then the set 
$\setspec{ (v,f) \in P(\Teul) }{ (u',e') \succeq (v,f) \succeq (u,e) }$ is included in  $P_S(\Teul)$ and is totally ordered by $\preceq$
(this is because $S(\Teul)$ is a tree).
\end{notation}

\begin{theorem}  \label {cvnv7nd6ykawsujwryf9v}
Define a binary relation $\sim$ on the set $P_S(\Teul)$ by declaring that
$$
(u,e) \sim (u',e') \iff \text{one of $(u,e), (u',e')$ is a comb over the other,}
$$
where $(u,e), (u',e') \in P_S(\Teul)$.
Then $\sim$ is an equivalence relation on $P_S(\Teul)$ and each equivalence class is totally ordered by $\preceq$.
\end{theorem}

The proof of the Theorem is given after the next three lemmas.  We omit the proof of the first one.

\begin{lemma} \label {9823vrZn984rCkjh9onqC238ev874w}
\newcommand{\ab}{\lhd}
\newcommand{\ba}{\rhd}
Let $(Y,\le)$ be a finite poset that is not totally ordered.
Given $y,y' \in Y$, let the notation $y \ab y'$ mean:\ \  $y<y'$ and no element $x$ of $Y$ satisfies $y < x < y'$.
\begin{enumerata}

\item If $Y$ has a greatest element then there exist distinct $y,y_1,y_2 \in Y$ satisfying $y_1 \lhd y \rhd y_2$.
\item If $Y$ has a least element then there exist distinct $y,y_1,y_2 \in Y$ satisfying $y_1 \rhd y \lhd y_2$.

\end{enumerata}
\end{lemma}

%%%%%%%%%%%%%%%%%%%%%%%%%%%%%%%%%%%%%%%%%%%%%%%%%%%%%%%%%%%%%%%%%%%%%%%%%%%%%%%%%%%%%%%%%%%%%%%%%%%%%%%%%%%%%%%%%%%%%%%%
%%%%%%%%%%%%%%%%%%%%%%%%%%%%%%%%%%%%%%%%%%%%%%%%%%%%%%%%%%%%%%%%%%%%%%%%%%%%%%%%%%%%%%%%%%%%%%%%%%%%%%%%%%%%%%%%%%%%%%%%
%	\begin{proof}
%	(a) Since $Y$ is not totally ordered, we have $U \neq \emptyset$ where 
%	$U$ is the set of all $y \in Y$ satisfying: there exists $y' \in Y$ such that $y,y'$ are not comparable.
%	Moreover, $U$ has at least two maximal elements.
%	Let $y_1,y_2$ be distinct maximal elements of $U$. Since $Y$ has a greatest element, $y_1,y_2$ are not maximal in $Y$;
%	so there exist $y_1',y_2' \in Y$ such that  $y_1 \lhd y_1'$ and $y_2 \lhd y_2'$.
%	Since (for each $k=1,2$) $y_k$ is a maximal element of $U$, $y_k' \notin U$; so for all $x \in Y$, $y_k'$ and $x$ are comparable.
%	So $y_1',y_2'$ are comparable, so $y_i' \le y_j'$ for some $(i,j) \in \{ (1,2), (2,1) \}$.
%	If $y_i' \le y_j$ then $y_i < y_i' \le y_j$ contradicts the fact that $y_i$ is maximal in $U$; 
%	so it is not the case that $y_i' \le y_j$; since $y_i', y_j$ are comparable, we get $y_j < y_i'$.
%	Then $y_j < y_i' \le y_j'$ and $y_j \lhd y_j'$ imply $y_i' = y_j'$. Define $y = y_1' = y_2'$, then $y_1 \lhd y \rhd y_2$.
%	This proves (a), and (b) follows by applying (a) to the poset $(Y,\ge)$.
%	\end{proof}
%%%%%%%%%%%%%%%%%%%%%%%%%%%%%%%%%%%%%%%%%%%%%%%%%%%%%%%%%%%%%%%%%%%%%%%%%%%%%%%%%%%%%%%%%%%%%%%%%%%%%%%%%%%%%%%%%%%%%%%%
%%%%%%%%%%%%%%%%%%%%%%%%%%%%%%%%%%%%%%%%%%%%%%%%%%%%%%%%%%%%%%%%%%%%%%%%%%%%%%%%%%%%%%%%%%%%%%%%%%%%%%%%%%%%%%%%%%%%%%%%

\begin{lemma} \label {0cvc023weij0wC3irf9}
Given $(u,e) \in P(\Teul)$, consider the set 
$$
X^-(u,e) = \setspec{ (u',e') \in P(\Teul) }{ \text{$(u,e)$ is a comb over $(u',e')$} }.
$$
Then $\big( X^-(u,e),\preceq \big)$ is nonempty and totally ordered. 
\end{lemma}

\begin{proof}
Since $(u,e) \in X^-(u,e)$, we have $X^-(u,e) \neq \emptyset$.
By contradiction, suppose that $X^-(u,e)$ is not totally ordered.
Since $X^-(u,e)$ has a greatest element (namely $(u,e)$),
Lemma \ref{9823vrZn984rCkjh9onqC238ev874w} implies that there exist 
distinct $(u_1,e_1) , (v,f) , (u_2,e_2) \in X^-(u,e)$ satisfying
\begin{enumerate}

\item[(i)] $(u_1,e_1) \prec (v,f) \succ (u_2,e_2)$

\item[(ii)] no $(x,h) \in X^-(u,e)$ satisfies $(u_1,e_1) \prec (x,h) \prec (v,f)$ or $(u_2,e_2) \prec (x,h) \prec (v,f)$.

\end{enumerate}
By Rem.\ \ref{c0Bj12Wsdh0982EChi}, conditions (i) and (ii) imply that $(u_1,e_1)$ and $(u_2,e_2)$ are immediate predecessors of $(v,f)$.
Thus $u_1=u_2$ and $e_1 \neq e_2$. Write $u^* = u_1 = u_2$, then $(u^*,e_1) \prec (v,f) \succ (u^*,e_2)$.
Since $e_1,e_2,f$ are distinct edges incident to $u^*$, we have $\epsilon(u^*)\ge3$.
For each $k=1,2$, we have  $(u^*,e_k) \prec (v,f) \preceq (u,e)$ and (by definition of $X^-(u,e)$) $(u,e)$ is a comb over $(u^*,e_k)$.
Since $(u,e)$ is a comb over $(u^*,e_2)$, we have $\epsilon(u^*)=3$ and $(u^*,e_1)$ is a tooth.
Since $(u^*,e_1)$ is a tooth, $R(u^*, \{ e_1 \}) > 0$.
Since $(u,e)$ is a comb over $(u^*,e_1)$ and $\epsilon(u^*)=3$, we have $R(u^*, \{ e_1 \}) = 0$, a contradiction. 
\end{proof}

\begin{lemma} \label {0co8nJkjhJKHFGHiu28ue039}
Given $(u,e) \in P_S(\Teul)$, define the set
$$
X^+(u,e) = \setspec{ (u',e') \in P_S(\Teul) }{ \text{$(u',e')$ is a comb over $(u,e)$} } .
$$
Then $\big( X^+(u,e), \preceq \big)$ is nonempty and totally ordered. 
\end{lemma}

\begin{proof}
Since $(u,e) \in X^+(u,e)$, we have $X^+(u,e) \neq \emptyset$.
By contradiction, suppose that $X^+(u,e)$ is not totally ordered.
Since $X^+(u,e)$ has a least element (namely $(u,e)$),
Lemma \ref{9823vrZn984rCkjh9onqC238ev874w} implies that there exist 
distinct $(u_1,e_1) , (v,f) , (u_2,e_2) \in X^+(u,e)$ satisfying
\begin{enumerate}

\item[(i)] $(u_1,e_1) \succ (v,f) \prec (u_2,e_2)$

\item[(ii)] no $(x,h) \in X^+(u,e)$ satisfies $(u_1,e_1) \succ (x,h) \succ (v,f)$ or $(u_2,e_2) \succ (x,h) \succ (v,f)$.

\end{enumerate}
Since $(u_1,e_1) \succ (v,f)$, we may consider the unique immediate predecessor $(u_1',e_1')$ of $(u_1,e_1)$
that satisfies  $(u_1,e_1) \succ (u_1',e_1') \succeq (v,f) \succeq (u,e)$.
Then $(u_1',e_1') \in P_S(\Teul)$ by the claim made in Notation \ref{pc0v23e0geCpgspyr23},
so  $(u_1',e_1') \in X^+(u,e)$ by  Rem.\ \ref{c0Bj12Wsdh0982EChi}.
Thus (ii) implies that $(u_1',e_1') = (v,f)$, i.e.,
$(v,f)$ is an immediate predecessor of $(u_1,e_1)$.
Similarly, $(v,f)$ is an immediate predecessor of $(u_2,e_2)$.
We have
$$
(u_1,e_1) \succ (v,f) \succeq (u,e) \quad \text{and} \quad  (u_2,e_2) \succ (v,f) \succeq (u,e) 
$$
where $e_1 = \{ u_1, v \}$,  $e_2 = \{ u_2, v \}$ and $f$ are distinct edges incident to $v$; so $\epsilon(v) \ge 3$.
Since $(u_2,e_2)$ is a comb over $(u,e)$, we have $\epsilon(v)=3$ and $(v,e_1)$ is a tooth.
Since $(v,e_1)$ is a tooth, it follows that $u_1 \notin S(\Teul)$, contradicting $(u_1,e_1) \in X^+(u,e)$.
\end{proof}

\begin{proof}[Proof of Thm \ref{cvnv7nd6ykawsujwryf9v}]
\newcommand{\combo}{\mathrel{\text{\rm c.o.}}}
Given $(u',e'), (u,e) \in P_S(\Teul)$, let us write ``$(u',e') \combo (u,e)$'' as an abbreviation for the sentence ``$(u',e')$ is a comb over $(u,e)$''.

It is clear that $\sim$ is reflexive and symmetric and we have to show that it is transitive. Suppose that 
$(u,e), (u',e'), (v,f) \in P_S(\Teul)$ satisfy $(u,e) \sim (v,f)$ and $(v,f) \sim (u',e')$; let us prove that $(u,e) \sim (u',e')$ in each 
of the four possible cases:
\begin{center}
\begin{tabular}{ccc}
(i)\ \  $(u,e) \combo (v,f)$ and $(v,f) \combo (u',e')$ && (ii)\ \  $(u',e') \combo (v,f)$ and $(v,f) \combo (u,e)$ \\
(iii)\ \  $(v,f) \combo (u,e)$ and $(v,f) \combo (u',e')$ && (iv)\ \  $(u,e) \combo (v,f)$ and $(u',e') \combo (v,f)$.
\end{tabular}
\end{center}

In case (i), Rem.\ \ref{c0Bj12Wsdh0982EChi} implies that $(u,e) \combo (u',e')$, so  $(u,e) \sim (u',e')$.

Case (ii) is similar to case (i).

In case (iii), we have $(u,e), (u',e') \in X^-(v,f)$.
By Lemma \ref{0cvc023weij0wC3irf9}, we have $(u,e) \preceq (u',e')$ or $(u',e') \preceq (u,e)$; we may assume that $(u,e) \preceq (u',e')$. 
Then $(u,e) \preceq (u',e') \preceq (v,f)$ and $(v,f) \combo (u,e)$, so 
Rem.\ \ref{c0Bj12Wsdh0982EChi} implies that $(u',e') \combo (u,e)$, so  $(u,e) \sim (u',e')$.

In case (iv), we have $(u,e), (u',e') \in X^+(v,f)$.
Using Lemma \ref{0co8nJkjhJKHFGHiu28ue039} and arguing as in case (iii), we find that  $(u,e) \sim (u',e')$.

This shows that $\sim$ is an equivalence relation.
Let $C$ be an equivalence class.
If $(u,e), (u',e') \in C$ then $(u,e) \combo (u',e')$ or $(u',e') \combo (u,e)$;
in the first case we have  $(u,e) \succeq (u',e')$ and in the second case $(u',e') \succeq (u,e)$;
so $(C,\preceq)$ is totally ordered.
This proves Thm \ref{cvnv7nd6ykawsujwryf9v}.
\end{proof}

\begin{definition} \label {cpan3bAfon98gbvqrkdLlxdir6}
We define the set 
$$
\In(\Teul) = \begin{cases}
 \Omega(\Teul) , &  \text{if $\Omega(\Teul) \neq \emptyset$,} \\
 \setspec{ z \in S(\Teul) }{ \delta^*(z) \le 1 } , & \text{if $\Omega(\Teul) = \emptyset$}.
\end{cases}
$$
The elements of $\In(\Teul)$ are called the \textit{initial vertices} of $\Teul$.
Observe:
$$ 
\emptyset  \neq \In(\Teul) \subseteq \setspec{ z \in S(\Teul) }{ \delta^*(z) \le 1 } .
$$
\end{definition}

\begin{remark}
Recall from Lemma \ref{p309ef2GFT485ryf} that $S(\Teul) \neq \emptyset$ and that
$$
\text{$| S(\Teul) | = 1$ $\iff$ $| \Neul | = 1$ or $\Teul$ is a brush.}
$$
In the special case where  $| S(\Teul) | = 1$,
the sets $\Oeul$ and $\overline\Oeul$ (defined in \ref{c20w39w93e9d0dqMne89wcg9}) are empty and consequently most of 
the theory developed below (from \ref{c20w39w93e9d0dqMne89wcg9} to the end of the section) is trivially true.
\end{remark}

\begin{definitions} \label {c20w39w93e9d0dqMne89wcg9}
Choose $z \in \In(\Teul)$.
This choice determines the following objects and notations.

For each $u \in S(\Teul) \setminus \{z\}$, let $e_u$ be the unique edge of $\gamma_{u,z}$ incident to $u$.
Observe that $e_u$ is an edge of the tree $S(\Teul)$ (since $u,z \in S(\Teul)$, $\gamma_{u,z}$ is entirely contained in $S(\Teul)$),
so in particular $(u,e_u)$ is not a tooth. 

\smallskip

\noindent (a) Let $\Oeul = \Oeul(\Teul,z) = \setspec{ (u,e_u) }{ u \in S(\Teul) \setminus \{z\} }$ and observe that
$$
\Oeul \neq \emptyset \iff S(\Teul) \setminus \{z\} \neq \emptyset \iff | S(\Teul) | > 1 .
$$
Also, the first projection $(u,e_u) \mapsto u$ is a bijection from $\Oeul$ to $S(\Teul) \setminus \{z\}$,
and the second projection $(u,e_u) \mapsto e_u$ is a bijection from $\Oeul$ to the set of edges of the tree $S(\Teul)$.
Since $\Oeul \subseteq P_S(\Teul)$, we may consider the equivalence relation $\sim$ on $\Oeul$ obtained by restricting 
the relation $\sim$ of Thm \ref{cvnv7nd6ykawsujwryf9v}. 

\smallskip

\noindent (b) Let $\overline\Oeul = \overline\Oeul(\Teul,z)$ be the set of equivalence classes of $(\Oeul,\sim)$.
Thus $\overline\Oeul \neq \emptyset \Leftrightarrow | S(\Teul) | > 1$.

\smallskip

\noindent (c) Let $C \in \overline\Oeul$. Then (by Thm \ref{cvnv7nd6ykawsujwryf9v}) $(C,\preceq)$ is totally ordered and hence has
a greatest element $(u,e_{u})$ and a least element $(v,e_{v})$ (which may coincide);
define
$$
Y_C = \bar V(u) \cup \big( \Neul(u,e_u) \setminus \Neul(v,e_v) \big) \quad \text{and} \quad \dot c(C) =  c(v,e_v) - c(u,e_u).
$$
We also use the notation $u_C = u$, which makes sense because $u$ is uniquely determined by $C$.

\smallskip

\noindent (d) Assuming that $| S(\Teul) | > 1$,
let $z'$ be the unique element of $S(\Teul)$ that is adjacent to $z$; then $(z', e_{z'}) \in \Oeul$, so there is a unique
element $C_0 \in \overline\Oeul$ satisfying $(z', e_{z'}) \in C_0$. 
Let $u_0 = u_{C_0}$, then $(u_0,e_{u_0})$ and $(z',e_{z'})$ are respectively the greatest and least elements of $C_0$.
(When $|S(\Teul)|=1$, we do not define $z'$, $C_0$ and $u_0$.)
\end{definitions}

\begin{remark}
Choose $z \in \In(\Teul)$.
If $C \in \overline\Oeul = \overline\Oeul(\Teul,z)$, and if $(u,e_u)$ and $(v,e_v)$ are respectively the greatest and least elements of $C$,
then  $(u,e_u)$ is a comb over $(v,e_v)$.
Thus we may view $\overline\Oeul$ as a decomposition of $S(\Teul)$ into combs, this decomposition being determined by the choice of $z$.
Since each element of $\overline\Oeul$ corresponds to a comb, $|\overline\Oeul|$ is the number of combs in the decomposition.
Note that $|\overline\Oeul| \ge 1 \iff | S(\Teul) | > 1$.
\end{remark}

\begin{corollary}  \label {cjvl45otbsAZofg9bgf}
Choose $z \in \In(\Teul)$ and let $C \in \overline\Oeul(\Teul,z)$. Then $\dot c(C)=0$ if and only if
$$
\text{for any $(x,e_x) \in C$ such that $x \neq u_C$ we have $\epsilon(x)=2$ and $\tD(x)=0$.}
$$
\end{corollary}

\begin{proof}
This is \eqref{98mdmosdberdtkj76lca} $\Leftrightarrow$ \eqref{pd90cfn2pw0s} in Prop.\ \ref{0ci19KJTghL872je309}.
\end{proof}

Recall that $| \Omega(\Teul) | \le 2$ by Thm \ref{Xcoikn23ifcdKJDluFYT937}.
In the case $| \Omega(\Teul) | = 2$, the decomposition into combs has a fairly simple description:

\begin{proposition}  \label {A09g2nr9f2wH03rf9rgeF}
Suppose that $| \Omega(\Teul) | = 2$.
Then there exists a $\tD$-trivial path $(z_1, \dots, z_n)$ such that $\Neul = \{z_1, \dots, z_n\} = S(\Teul)$ and
$\In(\Teul) = \Omega(\Teul) = \{ z_1, z_n \}$.
If we write $e_i = \{ z_i, z_{i+1} \}$ for $i=1, \dots, n-1$ then the following hold.
\begin{enumerata}

\item $\overline\Oeul(\Teul,z_1) = \{ C_0 \}$ where $C_0 = \{ (z_2,e_1), \dots, (z_n,e_{n-1}) \}$, $u_0 = z_n$ and $\dot c(C_0)=0$.
\item $\overline\Oeul(\Teul,z_n) = \{ C_0 \}$ where $C_0 = \{ (z_{n-1},e_{n-1}), \dots, (z_1,e_1) \}$, $u_0 = z_1$ and $\dot c(C_0)=0$.

\end{enumerata}
\end{proposition}

\begin{proof}
The assertions made before (a) and (b) follow from Thm \ref{Xcoikn23ifcdKJDluFYT937}.
In (a), it is clear that 
$\overline\Oeul(\Teul,z_1) = \{ C_0 \}$ where $C_0 = \{ (z_2,e_1), \dots, (z_n,e_{n-1}) \}$,
and we have $u_0 = z_n$ because $(z_n,e_{n-1})$ is the greatest element of $(C_0,\preceq)$.
Since for each $i$ such that $1 < i < n$ we have $\epsilon(z_i)=2$ and $\tD(z_i)=0$,
the fact that $\dot c(C_0) = 0$ follows from Cor.\ \ref{cjvl45otbsAZofg9bgf}.
This proves (a), and the proof of (b) is similar.
\end{proof}

\begin{lemma}  \label {cp0v9i230wedqwo}
Assume that $\Teul$ is not a brush.
Let $(u,e) \in P_S(\Teul)$ be such that  $u \notin \Omega(\Teul)$ and $X^+(u,e) = \{ (u,e) \}$.  Then the following hold.
\begin{enumerata}

\item $\tD( \bar V(u) ) \ge \max(1, \epsilon(u)-2)>0$

\item If $\delta^*(u) \ge 2$ then $\tD( \bar V(u) \cup \Neul(u,e) ) \ge | \delta^*(u) - 3 |$.

\item If $\delta^*(u) \ge 3$ and $\tD( \bar V(u) \cup \Neul(u,e) ) = \delta^*(u) - 3$
then $R(u,\{e\}) = 0$, $t(u)=0$, $(u,e)$ is nonpositive, and $u > u'$ where $e = \{u,u'\}$. 

\end{enumerata}
\end{lemma}

\begin{proof}
(a) If $u \in W(\Teul)$ then (since $\Teul$ is not a brush) we have $\tD\big( \bar V(u) \big) \ge \max(1, \epsilon(u)-2)$
by Lemma \ref{p0293efp0cw23ep0hvj} and we are done;
so we may (and we shall) assume that $u \notin W(\Teul)$.
Then  $\bar V(u)  = \{ u \}$, so $\tD( \bar V(u) ) = \tD( u )$.

If $\epsilon(u)=1$ then $\tD(u) >0$, otherwise we would have $u \in Z(\Teul)$, and since $u \notin \Omega(\Teul)$ this would imply that
$u \in V(w)$ for some $w \in W(\Teul)$, which would contradict $u \in S(\Teul)$.
Thus $\epsilon(u)=1$ implies $\tD(u)>0$ and hence $\tD( \bar V(u) ) \ge \max(1, \epsilon(u)-2)$.

Consider the case $\epsilon(u)>1$.
Then  Lemma \ref{90q932r8dhd89cnr9} implies that $N_u>1$ and hence that $\tD(u) \ge \epsilon(u)-2$.
So there only remains to show that if $\epsilon(u)=2$ then $\tD(u)>0$.
By contradiction, assume that $\epsilon(u)=2$ and $\tD(u)\le 0$.  Then in fact $\tD(u) = 0$. 
Let $u^*$ be the unique vertex of $\Neul$ that is adjacent to $u$ and such that $\{ u^*, u \} \neq e$;
let $e^* = \{ u^*, u \}$.
Then $(u^*,e^*) \succ (u,e)$, $\tD(\Neul(u^*,e^*)) = \tD(\Neul(u,e))$,
and $\Omega(\Teul)$ is disjoint from  $\Neul(u^*,e^*) \setminus \Neul(u,e) = \{u\}$; so  $(u^*,e^*)$ is a comb over $(u,e)$ by Prop.\ \ref{0ci19KJTghL872je309}.
Since $(u^*,e^*) \notin X^+(u,e)$ by assumption, it follows that $u^* \notin S(\Teul)$.
Thus $(u, e^*)$ is a tooth and hence $u \in W(\Teul)$, a contradiction.
This proves (a).

Assume that $(u,e)$ is a tooth. Then $\Neul(u,e) \subseteq \bar V(u)$, so part (a) implies that
$\tD( \bar V(u) \cup \Neul(u,e) )  = \tD( \bar V(u) ) \ge \max(1,\epsilon(u)-2)$.
If $\delta^*(u)=2$ then $\tD( \bar V(u) \cup \Neul(u,e) ) \ge 1  = | \delta^*(u)-3 |$ and 
if $\delta^*(u)\ge3$ then $\tD( \bar V(u) \cup \Neul(u,e) ) \ge \epsilon(u)-2 \ge \delta^*(u)-2 >  \delta^*(u)-3 = |\delta^*(u)-3|$.
This shows that (b) and (c) are true whenever $(u,e)$ is a tooth.

From now-on we assume that  $(u,e)$ is not a tooth.  Let us prove:
\begin{equation} \label {0chbqnvu7CVRFMWKZFVcEI9gcf653}
\textit{ if $\delta^*(u)=2$ then $\tD( \bar V(u) \cup \Neul(u,e) ) \ge 1$.}
\end{equation}
Proceeding by contradiction, we assume that  $\delta^*(u)=2$ and $\tD( \bar V(u) \cup \Neul(u,e) ) \le 0$.
Since  $\delta^*(u) = 2$ and $e$ is an edge in $S(\Teul)$ (because $(u,e)$ is not a tooth), there is a unique vertex $u^* \in S(\Teul)$ such that
$u^*$ is adjacent to $u$ and $\{u^*,u\} \neq e$; let $e^* =\{u^*,u\}$.
Then $\tD( \Neul(u^*,e^*) ) = \tD( \bar V(u) \cup \Neul(u,e) ) \le 0$, so $(u^*,e^*)$ is nonpositive. 
As $(u^*,e^*) \succ (u,e)$, it follows that $(u,e)$ is also nonpositive, so $\eta(u^*,e^*)=0=\eta(u,e)$;
as $\Omega(\Teul)$ is disjoint from $\Neul(u^*,e^*) \setminus \Neul(u,e) = \bar V(u)$,
Prop.\ \ref{0ci19KJTghL872je309} implies that $(u^*,e^*)$ is a comb over $(u,e)$.
Since $(u^*,e^*) \notin X^+(u,e)$ by assumption, it follows that $u^* \notin S(\Teul)$, a contradiction.
This proves \eqref{0chbqnvu7CVRFMWKZFVcEI9gcf653}.
It follows that (b) is true (and (c) is trivially true) when $\delta^*(u)=2$.

There remains to prove (b) and (c) when $\delta^*(u)\ge3$.
Let 
$$
A = \{e\} \cup \setspec{ e' \in \Eeul_u }{ \text{$(u,e')$ is a tooth} }.
$$
Since $\epsilon(u)= |A| + \delta^*(u)-1$,
Thm \ref{P90werd23ewods0ci} gives
$ R(u,A) + (\delta^*(u)-2) { \textstyle (1 - \frac1{N_u}) }  =   1 + \frac{ \bD(u,A) - 1 - \sum_{e' \in A}\eta(u,e')}{N_u} $.
Note that $\sum_{e' \in A}\eta(u,e') = \eta(u,e)$ because $\eta(u,e')=0$ whenever $(u,e')$ is a tooth.
After some straightforward manipulations we obtain
$$
\tD( \bar V(u) \cup \Neul(u,e) ) = \bD(u,A) = R(u,A) N_u  +   (\delta^*(u)-3) (N_u-1)  +  \eta(u,e) .
$$
The assumption $\delta^*(u)\ge3$ implies that  $N_u \ge 2$, so
$\tD( \bar V(u) \cup \Neul(u,e) ) \ge \delta^*(u)-3$.
If equality holds then $\eta(u,e) = 0$ (so $(u,e)$ is nonpositive) and $R(u,A)=0$,
so $t(u)=0$ and $R(u,\{e\})=0$; the last condition implies that $M(u,e)=1$,
so Lemma \ref{0hp9fh023wbpchvgrsjs256} gives $u > u'$ where $e = \{u,u'\}$. 
\end{proof}

\begin{remark}  \label {p0v9bhbe57reti9kr8JF5Ftjfn}
{\it If $z \in \In(\Teul)$ and $(u,e) \in \Oeul(\Teul,z)$ then  $X^+(u,e) \subseteq \Oeul(\Teul,z)$.}
(This observation is used in the proofs of Cor.\ \ref{0c9in34dh5sifzsshnkus} and Lemma \ref{A0c9vub23oW8fgbgp0q2arft}.)
\end{remark}

\begin{corollary} \label {0c9in34dh5sifzsshnkus}
Choose $z \in \In(\Teul)$ and let $C \in \overline\Oeul = \overline\Oeul(\Teul,z)$.
\begin{enumerata}

\item $\tD( \bar V(u_C) ) \ge \max(1, \epsilon(u_C)-2)>0$ if and only if $| \Omega(\Teul) | \le 1$.

\item If $\delta^*(u_C) \ge 2$ then $\tD( \bar V(u_C) \cup \Neul(u_C,e_{u_C}) ) \ge | \delta^*(u_C) - 3 |$.

\item If $\delta^*(u_C) \ge 3$ and $\tD( \bar V(u_C) \cup \Neul(u_C,e_{u_C}) ) = \delta^*(u_C) - 3$
then $R(u_C,\{e_{u_C}\}) = 0$, $t(u_C)=0$, $(u_C,e_{u_C})$ is nonpositive, and $u_C > u'$ where $e_{u_C} = \{u_C,u'\}$. 

\end{enumerata}
\end{corollary}

\begin{proof}
If $| \Omega(\Teul) | = 2$ then assertions (b) and (c) are true, because they are implications whose hypotheses are false
(because $\delta^*(u_C)=1$ by Prop.\ \ref{A09g2nr9f2wH03rf9rgeF});
assertion (a) is true because it has the form $\Pscr \Leftrightarrow \Qscr$ with both $\Pscr,\Qscr$ false
(because  $\tD(\bar V(u_C)) = \tD(u_C) \le 0$ by  Thm \ref{Xcoikn23ifcdKJDluFYT937}).

If $| \Omega(\Teul) | \le 1$ then  $\Omega(\Teul) \subseteq \{z\}$  and $u_C \in S(\Teul) \setminus \{z\}$, so $u_C \notin \Omega(\Teul)$.
We have $X^+(u_C,e_{u_C}) \subseteq \Oeul$ by Rem.\ \ref{p0v9bhbe57reti9kr8JF5Ftjfn}, so $X^+(u_C,e_{u_C}) = \{ (u_C,e_{u_C}) \}$ and hence
$(u_C,e_{u_C})$ satisfies all hypotheses of Lemma \ref{cp0v9i230wedqwo}
(note that $\Teul$ is not a brush since $\overline\Oeul \neq \emptyset$ by assumption).
By  Lemma \ref{cp0v9i230wedqwo}, (a), (b) and (c) are all true.
\end{proof}

\begin{lemma} \label {A0c9vub23oW8fgbgp0q2arft}
Assume that $| S(\Teul) | > 1$. 
Choose $z \in \In(\Teul)$ and consider the pair $(u_0,e_{u_0})$  defined in part {\rm (d)} of Def.\ \ref{c20w39w93e9d0dqMne89wcg9}.\footnote{Note
that $(u_0,e_{u_0})$ is defined if and only if $| S(\Teul) | > 1$ (see Def.\ \ref{c20w39w93e9d0dqMne89wcg9}(d)).}
The following are equivalent:
\begin{enumerata}

\item $\Omega(\Teul) \neq \emptyset$

\item $\tD\big( \bar V( z ) ) \le 0$

\item $(u_0,e_{u_0})$  is nonpositive

\item $(u_0,e_{u_0})$  is a maximal element of the set of nonpositive elements of $P_S(\Teul)$.

\end{enumerata}
Moreover, the following hold.
\begin{enumerate}

\item[(i)] If $\tD\big( \bar V( u_0 ) \cup \Neul( u_0, e_{u_0} ) \big) \le 1$ then conditions {\rm (a--d)} are satisfied.

\item[(ii)] If $\delta^*(u_0)=2$ and $\tD\big( \bar V( u_0 ) \cup \Neul( u_0, e_{u_0} ) \big) = 1$ then $R(u_0,A)=1$,
where we define $A = \{e_{u_0}\} \cup \setspec{ e \in \Eeul_{u_0} }{ \text{$(u_0,e)$ is a tooth} }$.

\end{enumerate}
\end{lemma}

\begin{proof}
Consider the case where $| \Omega(\Teul) | = 2$. 
Then conditions (a--d) are equivalent, since they are all true by Prop.\ \ref{A09g2nr9f2wH03rf9rgeF} and Thm \ref{Xcoikn23ifcdKJDluFYT937}. 
Moreover, (i) (resp.\ (ii)) is true because it is an implication whose conclusion is true (resp.\ whose hypothesis is false).
So the Lemma is true whenever $| \Omega(\Teul) | = 2$. 

From now-on, assume that $|\Omega(\Teul)| \le 1$. Note that this implies that $\Omega(\Teul) \subseteq \{z\}$.

It is clear that (d) implies (c).
Consider $(z',e_{z'})$ defined in Def.\ \ref{c20w39w93e9d0dqMne89wcg9}(d) and note that $\Neul(z',e_{z'}) = \bar V(z)$.
If (c) is true then $(z',e_{z'})$ is nonpositive (because $(z',e_{z'}) \preceq ( u_0, e_{u_0} )$), so
$\tD( \bar V(z) )  = \tD( \Neul(z',e_{z'}) ) \le 0$, showing that (c) implies (b).
If (b) is true then $z \notin W$ by Lemma \ref{p0293efp0cw23ep0hvj} (note that $\Teul$ is not a brush since $| S(\Teul) | > 1$),
so $t(z)=0$ and consequently $\epsilon(z)=1$ (because $\delta^*(z)=1$);
we also have $\tD(z) = \tD( \bar V(z) ) \le 0$, so $z \in Z(\Teul)$.
The fact that $z \in S(\Teul)$ implies that no $w \in W$ satisfies $z \in V(w)$, so $z \in \Omega(\Teul)$, showing that (b) implies (a).

Assume that (a) is true.
Then $\Omega(\Teul)=\{z\}$, so $\tD(z)\le0$ and consequently $z \notin W(\Teul)$;
it follows that $\Neul(z',e_{z'}) = \bar V(z) = \{z\}$ and hence that
$\tD( \Neul(z',e_{z'}) ) = \tD(z) \le 0$, which shows that $(z',e_{z'})$ is nonpositive.
So $\eta(z',e_{z'})=0$ by Prop.\ \ref{DKxcnpw93sdo}.
As $( u_0, e_{u_0} )$ is a comb over $(z',e_{z'})$,
we have $\eta( u_0, e_{u_0} ) = \eta(z',e_{z'})=0$ by Prop.\ \ref{0ci19KJTghL872je309}, so 
$( u_0, e_{u_0} )$ is nonpositive by Prop.\ \ref{DKxcnpw93sdo}.  Thus (a) implies (c).

To prove that (c) implies (d), suppose that $( u_0, e_{u_0} ) \preceq (v,f)$ where $(v,f)$ is a nonpositive element of $P_S(\Teul)$;
then $\eta( u_0, e_{u_0} ) = 0 = \eta(v,f)$ and $\Omega(\Teul)$ is disjoint from $\Neul(v,f) \setminus \Neul( u_0, e_{u_0} )$
(because $\Omega(\Teul)\subseteq\{z\} \subseteq \Neul(u_0,e_{u_0})$)
so $(v,f)$ is a comb over $( u_0, e_{u_0} )$ by Prop.\ \ref{0ci19KJTghL872je309}.
Thus $(v,f) \in X^+( u_0, e_{u_0} )$ and hence  $(v,f) \in C_0$ by Rem.\ \ref{p0v9bhbe57reti9kr8JF5Ftjfn}.
Then $(v,f) = ( u_0, e_{u_0} )$ because $( u_0, e_{u_0} )$ is the greatest element of $C_0$. So (c) implies (d).

This shows that (a--d) are equivalent.

To prove (i), assume that $\tD\big( \bar V( u_0 ) \cup \Neul( u_0, e_{u_0} ) \big) \le 1$.
We noted in Def.\ \ref{c20w39w93e9d0dqMne89wcg9} that no element of $\Oeul$ is a tooth, so $( u_0, e_{u_0} )$ is not a tooth
and consequently $\bar V( u_0 ) \cap \Neul( u_0, e_{u_0} ) = \emptyset$; as 
$\tD\big( \bar V( u_0 ) \big) > 0$ by Cor.\ \ref{0c9in34dh5sifzsshnkus}(a),
we have $1 \ge \tD\big( \bar V( u_0 ) \cup \Neul( u_0, e_{u_0} ) \big) > \tD\big( \Neul( u_0, e_{u_0} ) \big)$,
so  $( u_0, e_{u_0} )$ is nonpositive and (i) is proved.

(ii) Assume that $\delta^*(u_0)=2$ and $\tD\big( \bar V( u_0 ) \cup \Neul( u_0, e_{u_0} ) \big) = 1$.
Note that $\bD(u_0,A) =\tD\big( \bar V( u_0 ) \cup \Neul( u_0, e_{u_0} ) \big)$, so $\bD(u_0,A) = 1$.
It follows from (i) that $( u_0, e_{u_0} )$ is nonpositive, so $\eta( u_0, e_{u_0} ) = 0$ and consequently $\eta(u_0,e)=0$ for all $e \in A$.
Since $\delta^*(u_0)=2$ we have $\epsilon(u_0) - |A| - 1 = 0$, 
so Thm \ref{P90werd23ewods0ci} gives
$ R(u_0,A)  =   1 + \frac{ \bD(u_0,A) - 1 - \sum_{e \in A}\eta(u_0,e)}{N_{u_0}} = 1$.
\end{proof}

\begin{smallremark}
The next result assumes that  $| S(\Teul) | > 1$ and $|\Omega(\Teul)| \le 1$ in order to avoid uninteresting complications of the statement.
Indeed, if $| S(\Teul) | = 1$ then $\overline\Oeul=\emptyset$, so part (a) of Thm \ref{0d2h39fh29834jfp0dfgq} holds trivially,
the first equality in (b) is true (because $\bar V( z ) = \Neul$), and the second equality makes no sense because $(u_0,e_{u_0})$ is not defined when
$| S(\Teul) | = 1$.  The case  $|\Omega(\Teul)| = 2$ is treated in Prop.\ \ref{A09g2nr9f2wH03rf9rgeF} and Thm \ref{Xcoikn23ifcdKJDluFYT937},
and does not need to be reiterated in the Theorem.
\end{smallremark}

\begin{theorem} \label {0d2h39fh29834jfp0dfgq}
Assume that $| S(\Teul) | > 1$ and that $|\Omega(\Teul)| \le 1$.
Choose $z \in \In(\Teul)$ and consider $\overline\Oeul = \overline\Oeul(\Teul,z)$.
With notation as in Def.\ \ref{c20w39w93e9d0dqMne89wcg9}, the following hold.
\begin{enumerata}
\setlength{\itemsep}{1mm}

\item For each $C \in \overline\Oeul$ we have $\tD( Y_C) = \tD\big( \bar V(u_C) \big) +  \dot c(C)$,
$\tD\big( \bar V(u_C) \big) \ge \max(1, \epsilon(u_C)-2)$ and $\dot c(C) \ge 0$.

\item $\tD( \Neul )
=  \tD\big( \bar V( z ) \big) + \sum_{ C  \in \overline\Oeul  } \tD( Y_C) 
=  \tD\big( \bar V( u_0 ) \cup \Neul( u_0, e_{u_0} ) \big) + \sum_{ C  \in \overline\Oeul \setminus\{C_0\} } \tD( Y_C)$

\end{enumerata}
\end{theorem}

\begin{proof}
(a) Let $C \in \overline\Oeul$ and let $(u,e_u)$ and $(v,e_v)$ be respectively the greatest and least elements of $(C,\preceq)$.
We already noted that $e_u$ is an edge of $S(\Teul)$;
in particular $(u,e_u)$ is not a tooth, and consequently $\bar V(u) \cap \big( \Neul(u,e_u) \setminus \Neul(v,e_v) \big) = \emptyset$.
This together with $\Neul(v,e_v) \subseteq \Neul(u,e_u)$ implies that
$\tD( Y_C) = \tD\big( \bar V(u) \big) + \tD\big( \Neul(u,e_u) \big) - \tD\big( \Neul(v,e_v) \big)$.
We have $\tD\big( \Neul(u,e_u) \big) - \tD\big( \Neul(v,e_v) \big) = (\eta(u,e_u) - \eta(v,e_v)) + \dot c(C)$
and $\dot c(C) \ge 0$ by Lemma \ref{p0c9vin12q09wsc}.
Since $(u,e_u)$ is a comb over $(v,e_v)$,  $\eta(u,e)=\eta(v,e_v)$ by Prop.\ \ref{0ci19KJTghL872je309}.
Thus $\tD( Y_C) = \tD\big( \bar V(u) \big) +  \dot c(C)$.  
Cor.\ \ref{0c9in34dh5sifzsshnkus}(a) implies that $\tD\big( \bar V(u) \big) \ge \max(1, \epsilon(u)-2)$. 
As $u_C = u$, (a) is proved.

(b) See Def.\ \ref{SJyFDkjyftGkp0c9vb349vby3f6w} for the notion of f-partition.
Since $\overline\Oeul$ is a partition of $\Oeul$ and the first projection $p_1 : \Oeul \to S(\Teul) \setminus \{z\}$ is bijective,
it follows that $\big( p_1(C) \big)_{C \in \overline\Oeul}$ is an f-partition of $S(\Teul) \setminus \{z\}$.
By Lemma \ref{p309ef2GFT485ryf}, $\big( \bar V(x) \big)_{x \in S(\Teul)}$ is an f-partition of $\Neul$.
It is easy to see that $Y_C = \bigcup_{x \in p_1(C)} \bar V(x)$ for each $C \in \overline\Oeul$; it follows that
$\big( Y_C \big)_{C \in \overline\Oeul}$ is an f-partition of $\Neul \setminus \bar V(z)$,
so $\tD( \Neul \setminus \bar V(z) ) =  \sum_{ C  \in \overline\Oeul  } \tD( Y_C)$ and hence
\begin{equation} \label {i87hfyYX6tgygz292w}
\textstyle
\tD( \Neul ) =  \tD\big( \bar V( z ) \big) + \sum_{ C  \in \overline\Oeul  } \tD( Y_C) ,
\end{equation}
which proves the first part of (b).
As in paragraph \ref{c20w39w93e9d0dqMne89wcg9},
let $u_0 = u_{C_0}$ and note that $(u_0,e_{u_0})$ and $(z',e_{z'})$ are respectively the greatest and least elements of $C_0$.
Then $\Neul(z',e_{z'}) = \bar V(z)$, so $Y_{C_0} = \bar V(u_0) \cup \big( \Neul(u_0,e_{u_0}) \setminus \Neul(z',e_{z'}) \big) 
= \bar V(u_0) \cup \big( \Neul(u_0,e_{u_0}) \setminus \bar V(z) \big)$, from which we get
\begin{equation} \label {d0fo239eh9d9h3e7923e}
\tD\big( \bar V( z ) \big) + \tD( Y_{C_0} ) = \tD\big( \bar V(u_0) \cup \Neul(u_0,e_{u_0}) \big).
\end{equation}
Assertion (b) follows from \eqref{i87hfyYX6tgygz292w} and \eqref{d0fo239eh9d9h3e7923e}.
\end{proof}

\begin{lemma} \label {9bvr58sdv3wof9f}
Let $z \in \In(\Teul)$ and $C \in \overline\Oeul(\Teul,z)$.
Then 
$$
0 \le t(C) \le \dot c(C) \in \Nat
$$
where we define $t(C) = | \setspec{ (x,e_x) \in C }{ x \neq u_C \text{ and $t(x)>0$}} |$.\footnote{We can think
of $t(C)$ as being the number of teeth of the comb $C$, not counting the teeth attached to $u_C$.}
\end{lemma}

\begin{proof}
Let $(u_C,e_{u_C})$ and $(v,e_v)$ be respectively the greatest and least elements of $C$;
write $\gamma_{u_C,v} = (x_1, \dots, x_n)$ and (for each $i$) $e_i = e_{x_i}$.
Then 
\begin{equation} \label {c098bymdjctyrfbv8}
\textstyle \dot c(C) = c(x_n,e_n) - c(x_1,e_1) = \sum_{i=2}^n [ c(x_{i},e_{i})  - c(x_{i-1},e_{i-1})].
\end{equation}
Let $i \in \{2,\dots,n\}$.
Since $(x_{i-1},e_{i-1})$ is a comb over $(x_{i},e_{i})$, Prop.\ \ref{0ci19KJTghL872je309}  implies that
$\eta(x_{i-1},e_{i-1}) = \eta(x_{i},e_{i})$;
thus Lemma \ref{p0c9vin12q09wsc} gives
$$
0 \le c(x_{i},e_{i})  - c(x_{i-1},e_{i-1}) = \tD\big( \Neul(x_{i-1},e_{i-1}) \big) - \tD\big( \Neul(x_{i},e_{i}) \big) \in \Integ,
$$
so in particular $c(x_{i},e_{i})  - c(x_{i-1},e_{i-1}) \in \Nat$.
Suppose that $t(x_i)>0$; then $x_i \in W(\Teul)$, so $\tD( \bar V(x_i) ) \ge 1$ by Lemma \ref{p0293efp0cw23ep0hvj}
(note that  $\Teul$ is not a brush, because the assumption implies that $\overline\Oeul(\Teul,z) \neq \emptyset$).
As $\Neul(x_{i-1},e_{i-1}) \setminus \Neul(x_{i},e_{i}) = \bar V(x_i)$,
we get
$c(x_{i},e_{i})  - c(x_{i-1},e_{i-1})
= \tD\big( \Neul(x_{i-1},e_{i-1}) \big) - \tD\big( \Neul(x_{i-1},e_{i-1}) \big) = \tD( \bar V(x_i) ) \ge 1$.
This together with \eqref{c098bymdjctyrfbv8} implies that $\dot c(C) \ge t(C)$, so we are done.
\end{proof}

\begin{corollary} \label {pTo9v2q3hZYa9r1ge0cX3rg}
Choose $z \in \In(\Teul)$ and consider $\overline\Oeul = \overline\Oeul(\Teul,z)$.
\begin{enumerata}

\item $| \overline\Oeul | + \sum_{C \in \overline\Oeul \setminus \{C_0\}} \dot c(C) \le 1 + \max(0,\tD( \Neul ))$

\item If  $| \Neul | > 1$ and $| \overline\Oeul | \ge \tD( \Neul)$ then $\Omega(\Teul) \neq \emptyset$.

\end{enumerata}
\end{corollary}

\begin{proof}
To prove (a) we may assume that $| \overline\Oeul | > 1$, otherwise the result is trivial.
Then $| S(\Teul) | > 1$ (because $\overline\Oeul \neq \emptyset$) and $| \Omega(\Teul) | \le 1$ by Prop.\ \ref{A09g2nr9f2wH03rf9rgeF}.
We also have  $\delta^*(u_0)\ge2$, so Cor.\ \ref{0c9in34dh5sifzsshnkus}(b) gives $\tD\big( \bar V( u_0 ) \cup \Neul( u_0, e_{u_0} ) \big) \ge 0$.
We have $\tD( Y_C) \ge 1 + \dot c(C)$ for each $C  \in \overline\Oeul$ by Thm \ref{0d2h39fh29834jfp0dfgq}(a),
so Thm \ref{0d2h39fh29834jfp0dfgq}(b) gives
$\tD( \Neul ) =  \tD\big( \bar V( u_0 ) \cup \Neul( u_0, e_{u_0} ) \big) + \sum_{ C  \in \overline\Oeul \setminus\{C_0\} } \tD( Y_C)
\ge \sum_{ C  \in \overline\Oeul \setminus\{C_0\} } \tD( Y_C) \ge |\overline\Oeul|-1 +  \sum_{C \in \overline\Oeul \setminus \{C_0\}} \dot c(C)$,
so $| \overline\Oeul |  +  \sum_{C \in \overline\Oeul \setminus \{C_0\}} \dot c(C) \le 1 + \tD( \Neul ) \le 1 +  \max(0,\tD( \Neul ))$, proving (a).

(b) Suppose that $| S(\Teul) | = 1$.
As $| \Neul | > 1$, Lemma \ref{p309ef2GFT485ryf} implies that $\Teul$ is a brush; 
then  $\tD( \Neul) \ge 2$ by Cor.\ \ref{9vbrtyukey6idckFfFugus}, which contradicts 
the hypothesis $0 = | \overline\Oeul | \ge \tD( \Neul)$.

This shows that $| S(\Teul) | > 1$.
We may assume that $| \Omega(\Teul) | \le 1$, otherwise the result is trivial.
So the hypothesis of  Thm \ref{0d2h39fh29834jfp0dfgq} is satisfied.
Part (a) of that result gives $\tD( Y_C) \ge 1$ for each $C  \in \overline\Oeul$, so part (b) gives
$\tD( \Neul ) \ge  \tD\big( \bar V( u_0 ) \cup \Neul( u_0, e_{u_0} ) \big) + |\overline\Oeul|-1$,
so $\tD\big( \bar V( u_0 ) \cup \Neul( u_0, e_{u_0} ) \big) \le 1$.
Then $\Omega(\Teul) \neq \emptyset$ by Lemma \ref{A0c9vub23oW8fgbgp0q2arft}(i).
\end{proof}

\begin{corollary} \label {p0c9ifn2o3w9dcpw0e}
Choose $z \in \In(\Teul)$ and assume that $| \overline\Oeul | > 1$, where $\overline\Oeul = \overline\Oeul(\Teul,z)$.
Define:
\begin{itemize}

\item $B=2$ if  $\delta^*(u_0) = 2$, $B=0$ if  $\delta^*(u_0) \neq 2$;

\item $L= | \setspec{ v \in S(\Teul) }{ \delta^*(v)=1 } | =$ the number of leaves of the tree $S(\Teul)$;

\item $\overline\Oeul_i = \setspec{ C \in \overline\Oeul \setminus \{ C_0 \} } { \delta^*(u_C) = i }$ for $i=1,2$;

\item $\overline\Oeul_{>2} = \setspec{ C \in \overline\Oeul \setminus \{ C_0 \} } { \delta^*(u_C) > 2 }$;

\item $T = \sum_{ C \in \overline\Oeul_1} \max(0,t(u_C)-2)  + \sum_{ C \in \overline\Oeul_2} \max(0,t(u_C)-1)  + \sum_{ C \in \overline\Oeul_{>2}} t(u_C)$;

\item $x_0 = \tD\big( \bar V( u_0 ) \cup \Neul( u_0, e_{u_0} ) \big) -  | \delta^*(u_0) - 3 |$;

\item $x_C = \tD( \bar V(u_C) ) -  \max(1, \epsilon(u_C)-2)$ for each  $C \in \overline\Oeul \setminus \{C_0\}$.

\end{itemize}

Then the following hold.

\begin{enumerata}
\setlength{\itemsep}{1mm}

\item $  \displaystyle
\tD( \Neul ) = B + 2(L-2) + | \overline\Oeul_2 | +  \  T + x_0 +  \!\! \sum_{ C \in \overline\Oeul \setminus \{C_0\} } \!\!\! (\dot c(C) + x_C)$

\item We have $B, L-2, | \overline\Oeul_2 |, T, x_0 \in \Nat$ and, for each  $C \in \overline\Oeul \setminus \{C_0\}$,  $\dot c(C), x_C \in \Nat$.

\item $\tD(\Neul) \ge  B + 2(L-2) + | \overline\Oeul_2 | \ge 2 $

\end{enumerata}
\end{corollary}

\begin{proof}
(a) The assumption $| \overline\Oeul | > 1$ implies that $| S(\Teul) | > 1$ and  $| \Omega(\Teul) | \le 1$,
so the hypothesis of  Thm \ref{0d2h39fh29834jfp0dfgq} is satisfied;
it also implies that $\delta^*(u_0)\ge2$, so $| \delta^*(u_0)-3 | = B + ( \delta^*(u_0)-3 )$.
By definition of $x_0$, it follows that
$$
\tD\big( \bar V( u_0 ) \cup \Neul( u_0, e_{u_0} ) \big) = x_0 + B + ( \delta^*(u_0) - 3 ) .
$$
By Thm \ref{0d2h39fh29834jfp0dfgq} we have $\tD( Y_C) = \tD\big( \bar V(u_C) \big) +  \dot c(C)$ 
for each $C \in \overline\Oeul$, so
$$
\tD( Y_C) =   \max(1, \epsilon(u_C)-2) + x_C +  \dot c(C) \quad \text{for all $C \in \overline\Oeul \setminus \{ C_0 \}$},
$$
by definition of $x_C$. We have
$\tD( \Neul ) =  \tD\big( \bar V( u_0 ) \cup \Neul( u_0, e_{u_0} ) \big) + \sum_{ C  \in \overline\Oeul \setminus\{C_0\} } \tD( Y_C)$
again by Thm \ref{0d2h39fh29834jfp0dfgq}, so the above remarks give
\begin{equation*} 
\tD( \Neul )
\textstyle  =   x_0 + B + ( \delta^*(u_0) - 3 ) + \sum_{ C  \in \overline\Oeul \setminus\{C_0\} }  \max(1, \epsilon(u_C)-2)
+  \sum_{ C  \in \overline\Oeul \setminus\{C_0\} } (x_C +  \dot c(C)).
\end{equation*}
It therefore suffices to show that 
\begin{equation} \label {Pc0oOv9hnt5dyuX}
\textstyle
( \delta^*(u_0) - 3 ) + \sum_{ C  \in \overline\Oeul \setminus\{C_0\} } \max(1, \epsilon(u_C)-2) = 2(L-2) + | \overline\Oeul_2 | + T .
\end{equation}
For each $C \in \overline\Oeul_{>2}$ we have $\epsilon(u_C)-2 \ge 1$, so
$\max(1, \epsilon(u_C)-2) = \epsilon(u_C)-2 = (\delta^*(u_C)-2) + t(u_C)$.
For each $C \in  \overline\Oeul_{1} \cup \overline\Oeul_{2}$, we may write
$
\max(1, \epsilon(u_C)-2)
= 1 + \max(0, \epsilon(u_C)-3) 
= 1 + \max(0, t(u_C) - (3 - \delta^*(u_C)))
= (\delta^*(u_C)-2) + (3-\delta^*(u_C))   + \max(0, t(u_C) - (3 - \delta^*(u_C)))
$, so
$$
\max(1, \epsilon(u_C)-2) =  (\delta^*(u_C)-2) +
\begin{cases}
2 + \max(0, t(u_C)-2) & \text{if $C \in \overline\Oeul_{1}$}, \\
1 + \max(0, t(u_C)-1) & \text{if $C \in \overline\Oeul_{2}$}, \\
t(u_C) & \text{if $C \in \overline\Oeul_{>2}$\, .}
\end{cases}
$$
Since $\big\{ \overline\Oeul_{1}, \overline\Oeul_{2}, \overline\Oeul_{>2} \big\}$ is a partition of $\overline\Oeul \setminus\{C_0\}$,
we get
\begin{equation}  \label {Pc09Jihvn203wchqob78q10wJ}
\textstyle
\sum_{ C  \in \overline\Oeul \setminus\{C_0\} } \max(1, \epsilon(u_C)-2) =
\sum_{ C  \in \overline\Oeul \setminus\{C_0\} } (\delta^*(u_C)-2) + 2 | \overline\Oeul_1 | + | \overline\Oeul_2 | + T .
\end{equation}
It is well known that the sum of the number ``valency$(x)$ minus 2'', over all vertices $x$ of a tree, is equal to $-2$.
This gives the first equality in
\begin{equation} \label {v9bhu7yoOoOleiriobbgt4875}
\textstyle
-2 = \sum_{x \in S(\Teul)} (\delta^*(x)-2)
= (\delta^*(z)-2) + (\delta^*(u_0)-2) + \sum_{ C  \in \overline\Oeul \setminus\{C_0\} } (\delta^*(u_C)-2),
\end{equation}
while the second equality follows from the observation that
$\delta^*(x)=2$ for all $x \in S(\Teul) \setminus \big( \{z\} \cup \setspec{ u_C }{ C \in \overline\Oeul } \big)$.
Thus $\sum_{ C  \in \overline\Oeul \setminus\{C_0\} } (\delta^*(u_C)-2) = 1 - \delta^*(u_0)$.
This and \eqref{Pc09Jihvn203wchqob78q10wJ} give
\begin{equation} \label {c09vbopWEWEWEWE230wd9c2y}
\textstyle
( \delta^*(u_0) - 3 ) + \sum_{ C  \in \overline\Oeul \setminus\{C_0\} } \max(1, \epsilon(u_C)-2) 
= -2  + 2 | \overline\Oeul_1 | + | \overline\Oeul_2 | + T .
\end{equation}
Finally, we note that the set of leaves of $S(\Teul)$ is precisely
$\{z\} \cup \setspec{ u_C }{ C \in \overline\Oeul_1 }$, so $| \overline\Oeul_1 | = L-1$ and 
consequently the right-hand-side of \eqref{c09vbopWEWEWEWE230wd9c2y} is equal to $2(L-2) + | \overline\Oeul_2 | + T$.
This proves \eqref{Pc0oOv9hnt5dyuX} and completes the proof of assertion (a).

(b) It is clear that $B, | \overline\Oeul_2 |, T \in \Nat$.
Since $| S(\Teul) | > 1$, the tree $S(\Teul)$ has at least two leaves; so $L-2 \in \Nat$.
Since $| \overline\Oeul | > 1$, we have $\delta^*(u_0)\ge2$, so $\tD\big( \bar V( u_0 ) \cup \Neul( u_0, e_{u_0} ) \big) \ge | \delta^*(u_0)-3 |$
by Cor.\ \ref{0c9in34dh5sifzsshnkus}(b), so $x_0 \in \Nat$.
Part (a) of Thm \ref{0d2h39fh29834jfp0dfgq} implies that $\dot c(C), x_C \in \Nat$ for all  $C \in \overline\Oeul \setminus \{C_0\}$.

(c) Parts (a) and (b) immediately imply that  $\tD(\Neul) \ge  B + 2(L-2) + | \overline\Oeul_2 |$.
Since $| \overline\Oeul | > 1$, we have $B = 2$ or $L \ge 3$, so $B + 2(L-2) + | \overline\Oeul_2 | \ge B + 2(L-2) \ge 2$. 
\end{proof}

\begin{notation} \label {0vbb359gb}
Choose $z \in \In(\Teul)$ and let $\overline\Oeul = \overline\Oeul(\Teul,z)$.
If $\overline\Oeul \neq \emptyset$ (or equivalently $| S(\Teul) | > 1$) then we define
the rooted tree $\widetilde\Oeul = \widetilde\Oeul(\Teul,z)$ by declaring that
\begin{itemize}

\item the vertex-set of $\widetilde\Oeul$ is $\overline\Oeul$ and the root of  $\widetilde\Oeul$ is $C_0$;

\item distinct elements $C,C'$ of $\overline\Oeul$ are adjacent in $\widetilde\Oeul$ if and only if
there exist $(u,e) \in C$ and $(u',e') \in C'$ such that $u$ is adjacent to $u'$ in $\Neul$.

\end{itemize}
\end{notation}

The rooted tree $\widetilde\Oeul$ is useful for visualizing the decomposition into combs and (when $| \overline\Oeul | > 1$)
for computing the value of $B + 2(L-2) + | \overline\Oeul_2 |$.  
Figure \ref{d9be3yie8dfdokj6rdosefd} gives the values of $B$, $L$ and $ | \overline\Oeul_2 |$
for every rooted tree $\widetilde\Oeul$ whose number of vertices is at least 2 and at most 5.
Recall that  Cor.\ \ref{p0c9ifn2o3w9dcpw0e}(c) implies that $\tD(\Neul) \ge  B + 2(L-2) + | \overline\Oeul_2 |$ whenever  $| \overline\Oeul | > 1$.

\begin{figure}[htb]
\centering
\scalebox{.7}{
$$
\begin{array}[t]{|c|c|c|c|c|c|c|}
\hline
& | \overline\Oeul |  &  \widetilde{\Oeul} \rule{0mm}{5mm} & B & L & |\overline\Oeul_2| & H  \\ \hline \hline
%%%%%%%%%%%%%%%%%%%%%%%%%%%%%%%%%%%%%%%%%%%%%%%%%%%%%%%%%%%%%%%%%%%%%%%%%%%%%%%%%%%%%%%%%%%%%%%%%%%%%
%%%%%%%%%%%%%%%%%%%%%%%%%%%%%%%%%%%%%%%%%%%%%%%%%%%%%%%%%%%%%%%%%%%%%%%%%%%%%%%%%%%%%%%%%%%%%%%%%%%%%
\mbox{(a)} & 2 & \raisebox{1mm}{\scalebox{.7}{ \xymatrix@R=0pt{ \circledast \ar@{-}[r] & \bullet } }}
& 2 & 2 & 0 & 2 \\ \hline
%%%%%%%%%%%%%%%%%%%%%%%%%%%%%%%%%%%%%%%%%%%%%%%%%%%%%%%%%%%%%%%%%%%%%%%%%%%%%%%%%%%%%%%%%%%%%%%%%%%%%
\mbox{(b)} & 3 & \raisebox{4mm}{\scalebox{.7}{ \xymatrix@R=0pt{ 
& \bullet \\
\circledast \ar@{-}[ru] \ar@{-}[rd] \\
& \bullet	} }}
& 0 & 3 & 0 & 2 \\ \hline
%%%%%%%%%%%%%%%%%%%%%%%%%%%%%%%%%%%%%%%%%%%%%%%%%%%%%%%%%%%%%%%%%%%%%%%%%%%%%%%%%%%%%%%%%%%%%%%%%%%%%
\mbox{(c)} & 3 & \raisebox{1mm}{\scalebox{.7}{ \xymatrix@R=0pt{ 
\circledast \ar@{-}[r] & \bullet \ar@{-}[r] & \bullet 
} }} & 2 & 2 & 1 & 3 \\ \hline
%%%%%%%%%%%%%%%%%%%%%%%%%%%%%%%%%%%%%%%%%%%%%%%%%%%%%%%%%%%%%%%%%%%%%%%%%%%%%%%%%%%%%%%%%%%%%%%%%%%%%
\mbox{(d)} & 4 & \raisebox{4mm}{\scalebox{.7}{ \xymatrix@R=0pt{ 
& \bullet \\
\circledast \ar@{-}[r] \ar@{-}[ru] \ar@{-}[rd] & \bullet \\
& \bullet
} }} & 0 & 4 & 0 & 4 \\ \hline
%%%%%%%%%%%%%%%%%%%%%%%%%%%%%%%%%%%%%%%%%%%%%%%%%%%%%%%%%%%%%%%%%%%%%%%%%%%%%%%%%%%%%%%%%%%%%%%%%%%%%
\mbox{(e)} & 4 & \raisebox{4mm}{\scalebox{.7}{ \xymatrix@R=0pt{
& \bullet \\
\circledast \ar@{-}[ru] \ar@{-}[rd] \\
& \bullet \ar@{-}[r] & \bullet 
} }} & 0 & 3 & 1 & 3 \\ \hline
%%%%%%%%%%%%%%%%%%%%%%%%%%%%%%%%%%%%%%%%%%%%%%%%%%%%%%%%%%%%%%%%%%%%%%%%%%%%%%%%%%%%%%%%%%%%%%%%%%%%%
\mbox{(f)} & 4 & \raisebox{4mm}{\scalebox{.7}{ \xymatrix@R=0pt{ 
&& \bullet \\
\circledast \ar@{-}[r] & \bullet \ar@{-}[ru] \ar@{-}[rd] \\
&& \bullet 
} }} & 2 & 3 & 0 & 4 \\ \hline
%%%%%%%%%%%%%%%%%%%%%%%%%%%%%%%%%%%%%%%%%%%%%%%%%%%%%%%%%%%%%%%%%%%%%%%%%%%%%%%%%%%%%%%%%%%%%%%%%%%%%
\mbox{(g)} & 4 & \raisebox{1mm}{\scalebox{.7}{ \xymatrix@R=0pt{ 
\circledast \ar@{-}[r] & \bullet \ar@{-}[r] &  \bullet \ar@{-}[r] &  \bullet 
} }} & 2 & 2 & 2 & 4 \\ \hline
%%%%%%%%%%%%%%%%%%%%%%%%%%%%%%%%%%%%%%%%%%%%%%%%%%%%%%%%%%%%%%%%%%%%%%%%%%%%%%%%%%%%%%%%%%%%%%%%%%%%%
\end{array}
\ \ 
\begin{array}[t]{|c|c|c|c|c|c|c|}
\hline
& | \overline\Oeul |  &  \widetilde{\Oeul} \rule{0mm}{5mm} & B & L & |\overline\Oeul_2| & H  \\ \hline \hline
%%%%%%%%%%%%%%%%%%%%%%%%%%%%%%%%%%%%%%%%%%%%%%%%%%%%%%%%%%%%%%%%%%%%%%%%%%%%%%%%%%%%%%%%%%%%%%%%%%%%%

\mbox{(h)} & 5 & \raisebox{6mm}{\scalebox{.7}{ \xymatrix@R=0pt{
& \bullet \\
& \bullet \\
\circledast \ar@{-}[r] \ar@{-}[ruu] \ar@{-}[ru] \ar@{-}[rd] & \bullet \\
& \bullet
} }} & 0 & 5 & 0 & 6 \\ \hline
%%%%%%%%%%%%%%%%%%%%%%%%%%%%%%%%%%%%%%%%%%%%%%%%%%%%%%%%%%%%%%%%%%%%%%%%%%%%%%%%%%%%%%%%%%%%%%%%%%%%%
\mbox{(i)} & 5 & \raisebox{4mm}{\scalebox{.7}{ \xymatrix@R=0pt{
& \bullet \\
\circledast \ar@{-}[r] \ar@{-}[ru] \ar@{-}[rd] & \bullet \\
& \bullet \ar@{-}[r] & \bullet 
} }} & 0 & 4 & 1 & 5 \\ \hline
%%%%%%%%%%%%%%%%%%%%%%%%%%%%%%%%%%%%%%%%%%%%%%%%%%%%%%%%%%%%%%%%%%%%%%%%%%%%%%%%%%%%%%%%%%%%%%%%%%%%%
\mbox{(j)} & 5 & \raisebox{5mm}{\scalebox{.7}{ \xymatrix@R=0pt{ 
& \bullet \\
\circledast \ar@{-}[ru] \ar@{-}[rd] && \bullet \\
& \bullet \ar@{-}[ru] \ar@{-}[rd] \\
&& \bullet 
} }} & 0 & 4 & 0 & 4 \\ \hline
%%%%%%%%%%%%%%%%%%%%%%%%%%%%%%%%%%%%%%%%%%%%%%%%%%%%%%%%%%%%%%%%%%%%%%%%%%%%%%%%%%%%%%%%%%%%%%%%%%%%%
\mbox{(k)} & 5 & \raisebox{4mm}{\scalebox{.7}{ \xymatrix@R=0pt{
& \bullet \\
\circledast \ar@{-}[ru] \ar@{-}[rd] \\
& \bullet \ar@{-}[r] & \bullet  \ar@{-}[r] & \bullet 
} }} & 0 & 3 & 2 & 4 \\ \hline
%%%%%%%%%%%%%%%%%%%%%%%%%%%%%%%%%%%%%%%%%%%%%%%%%%%%%%%%%%%%%%%%%%%%%%%%%%%%%%%%%%%%%%%%%%%%%%%%%%%%%
\mbox{(l)} & 5 & \raisebox{4mm}{\scalebox{.7}{ \xymatrix@R=0pt{
& \bullet  \ar@{-}[r] & \bullet \\
\circledast \ar@{-}[ru] \ar@{-}[rd] \\
& \bullet \ar@{-}[r] & \bullet 
} }} & 0 & 3 & 2 & 4 \\ \hline
%%%%%%%%%%%%%%%%%%%%%%%%%%%%%%%%%%%%%%%%%%%%%%%%%%%%%%%%%%%%%%%%%%%%%%%%%%%%%%%%%%%%%%%%%%%%%%%%%%%%%
\mbox{(m)} & 5 & \raisebox{4mm}{\scalebox{.7}{ \xymatrix@R=0pt{
& & \bullet \\
\circledast \ar@{-}[r] & \bullet \ar@{-}[r] \ar@{-}[ru] \ar@{-}[rd] & \bullet \\
& & \bullet
} }} & 2 & 4 & 0 & 6 \\ \hline
%%%%%%%%%%%%%%%%%%%%%%%%%%%%%%%%%%%%%%%%%%%%%%%%%%%%%%%%%%%%%%%%%%%%%%%%%%%%%%%%%%%%%%%%%%%%%%%%%%%%%
\mbox{(n)} & 5 & \raisebox{4mm}{\scalebox{.7}{ \xymatrix@R=0pt{
& & \bullet \\
\circledast \ar@{-}[r] & \bullet \ar@{-}[ru] \ar@{-}[rd] \\
& & \bullet \ar@{-}[r] & \bullet 
} }} & 2 & 3 & 1 & 5 \\ \hline
%%%%%%%%%%%%%%%%%%%%%%%%%%%%%%%%%%%%%%%%%%%%%%%%%%%%%%%%%%%%%%%%%%%%%%%%%%%%%%%%%%%%%%%%%%%%%%%%%%%%%
\mbox{(o)} & 5 & \raisebox{4mm}{\scalebox{.7}{ \xymatrix@R=0pt{
&&& \bullet \\
\circledast \ar@{-}[r] & \bullet \ar@{-}[r] & \bullet \ar@{-}[ru] \ar@{-}[rd] \\
&&& \bullet 
} }} & 2 & 3 & 1 & 5 \\ \hline
%%%%%%%%%%%%%%%%%%%%%%%%%%%%%%%%%%%%%%%%%%%%%%%%%%%%%%%%%%%%%%%%%%%%%%%%%%%%%%%%%%%%%%%%%%%%%%%%%%%%%
\mbox{(p)} & 5 & \raisebox{1mm}{\scalebox{.7}{ \xymatrix@R=0pt{ 
\circledast \ar@{-}[r] & \bullet \ar@{-}[r] &  \bullet \ar@{-}[r] &  \bullet  \ar@{-}[r] &  \bullet 
} }} & 2 & 2 & 3 & 5 \\ \hline
%%%%%%%%%%%%%%%%%%%%%%%%%%%%%%%%%%%%%%%%%%%%%%%%%%%%%%%%%%%%%%%%%%%%%%%%%%%%%%%%%%%%%%%%%%%%%%%%%%%%%
\end{array}
$$
}
\caption{See \ref{0vbb359gb}. The number $B + 2(L-2) + |\overline\Oeul_2|$ is called $H$ in the above tables.
The value of $H$ is given for all rooted trees $\widetilde\Oeul$ whose number $|\overline\Oeul|$ of vertices
satisfies $2 \le |\overline\Oeul| \le 5$.  The vertex $\circledast$ is the root of $\widetilde\Oeul$.}
\label {d9be3yie8dfdokj6rdosefd}
\end{figure}

\begin{remark} \label {cp09vw4gririnhe}
Given a rooted tree $T$ we define $H(T) = 2\lambda + n_2 - 2$, where
$\lambda$ is the number of vertices of valency $1$ in $T$ (including the root if it has valency $1$),
and $n_2$ is the number of vertices distinct from the root and of valency $2$ in $T$.
Then it is not hard to see that the following is true:
$$
\textit{Let $z \in \In(\Teul)$ and $\overline\Oeul = \overline\Oeul(\Teul,z)$.  If $| \overline\Oeul | > 1$ then
$H(\, \widetilde \Oeul \,) = B + 2(L-2) + | \overline\Oeul_2 |$.}
$$
So one can compute the value of $B + 2(L-2) + | \overline\Oeul_2 |$ 
just by looking at the picture of $\widetilde\Oeul$ and using $H( \widetilde \Oeul ) = 2\lambda + n_2 - 2$.
\end{remark}

\section{$\Nd^*(\Teul)$}
\label {Sec:Nd_etoile}

Let $\Teul$ be a minimally complete abstract Newton tree at infinity.

The main results of this section are Theorems \ref{p0fkwJjk8eCahbkBJhZuIK83gldf} and \ref{cvintcjowHYtjk2uu78748};
the first one relates the cardinality of $\Nd^*(\Teul)$ to the number $\tD(\Neul)$ and the second gives information about the 
possible locations of the vertices of $\Nd^*(\Teul)$.
This is interesting for reasons that are briefly explained in the introduction.

\begin{definition} \label {iv097383k467k4V498ufhr7}
\begin{enumerata}

\item $\Nd^*(\Teul) = \setspec{ v \in \Nd(\Teul) }{ d(v)=1 }$

\item Let $\xi : \Neul \to \Nat$ be the set map defined as follows (where $v \in\Neul$):
\begin{itemize}

\item If $d(v)\neq1$ then $\xi(v)=0$;

\item if $N_v=1$ then $\xi(v) = 1$ (note that $d(v)=1$ by Lemma \ref{90q932r8dhd89cnr9}\eqref{0239u9e78923});

\item if $N_v\neq 1$ and $d(v)=1$ then $\xi(v) = \max(1,n)$ where $n$ is the number of occurrences of ``$1$'' in the type of $v$
(note that $N_v\neq d(v)$ implies that $v$ is a node, so it makes sense to speak of the type of $v$). 

\end{itemize}
Given a subset $E$ of $\Neul$ we define $\xi(E) = \sum_{v \in E} \xi(v)$.
Observe that  $\xi(v)>0 \Leftrightarrow v \in \Nd^*(\Teul)$ for all $v \in \Neul$;
so $| \Nd^*(\Teul) \cap E | \le \xi(E)$ for every $E \subseteq \Neul$, and in particular $| \Nd^*(\Teul) | \le \xi(\Neul)$.

\end{enumerata}
\end{definition}

\begin{definition}
If $v \in \Neul$ satisfies $\setspec{ d_x }{ x \in \Deul_v } \subseteq \{1,N_v\}$, we say that $v$ is \textit{pure}.
Note that if $v \in \Neul$ satisfies $\sigma(v)=0$ or $N_v=1$ then it is pure.
\end{definition}

\begin{theorem}  \label {p0fkwJjk8eCahbkBJhZuIK83gldf}
\begin{enumerata}

\item $| \Nd^*(\Teul) | \le \xi(\Neul) \le 2 + \max(0,\tD(\Neul))$

\item If  $\xi(\Neul) = 2 + \max(0,\tD(\Neul))$ then each $x \in \Nd^*(\Teul)$ is pure and satisfies $a_x=1$.

\item If $| \Nd^*(\Teul) |  = 2 + \max(0,\tD(\Neul))$ then each $x \in \Nd^*(\Teul)$ is a node of type $[1,N_x,\dots,N_x]$ and satisfies $a_x=1$.

\end{enumerata}
\end{theorem}

The proof of
Thm~\ref{p0fkwJjk8eCahbkBJhZuIK83gldf}
occupies a large portion of this section, and is recapitulated in paragraph \ref{c9b3498siBxkxcY5ykTwybsCte6eu}.

\begin{remark} \label {cCo98sw5vs8g6vblvbL3ervweygd}
{\it To prove Thm \ref{p0fkwJjk8eCahbkBJhZuIK83gldf}, it suffices to show that 
$\xi(\Neul) \le 2 + \max(0,\tD(\Neul))$ and that assertion {\rm (b)} holds.}
Indeed, we already noted that $| \Nd^*(\Teul) | \le \xi(\Neul)$, so (a) follows.
Let us deduce (c) from (a) and (b). 
Assume that $| \Nd^*(\Teul) |  = 2 + \max(0,\tD(\Neul))$;
then (a) implies that (i)~$| \Nd^*(\Teul) | = \xi(\Neul)$ and (ii)~$\xi(\Neul) = 2 + \max(0,\tD(\Neul))$,
and (i) is equivalent to (i$'$)~$\xi(x)=1$ for all $x \in \Nd^*(\Teul)$.
Now consider $x \in \Nd^*(\Teul)$.
We have $\xi(x)=1$ by (i$'$), and (ii) and (b) imply that $x$ is pure and satisfies $a_x=1$.
Since $x$ is pure and satisfies $\xi(x)=1$, it is a node of type $[1,N_x,\dots,N_x]$. This proves (c).
\end{remark}

We prove several lemmas in preparation for the proof of Thm~\ref{p0fkwJjk8eCahbkBJhZuIK83gldf}.

\begin{lemma}  \label {pvOtyxMdcbgOEjo28hw8y}
Let $v \in \Neul$. Then $\sigma(v) \ge N_v - d(v)$,
and if equality holds then either $v$ is not a node or it is a node of type $[d(v), N_v, \dots, N_v]$.
\end{lemma}

\begin{proof}
This is a consequence of the following elementary arithmetic fact, whose verification is left to the reader:
\it Let $N,d_1, \dots, d_s$ ($s\ge1$) be positive integers such that $d_i \mid N$ for all $i$.
Then $\sum_{i=1}^s (N - d_i) \ge N - \gcd(d_1, \dots, d_s)$, and if equality holds then at most one $i$ is such that $d_i \neq N$.
\end{proof}

\begin{lemma} \label {p9uv8w6t376wsw448gF3j}
Let $v \in \Neul$.  Then $\sigma(v) \ge \xi(v)(N_v-1)$, and equality holds if and only if $v$ is pure.
\end{lemma}

\begin{proof}
If either $N_v=1$ or $v$ is not a node then $v$ is pure and the inequality reads ``$0 \ge 0$,'' so the Lemma is true in these cases.
If $v \in \Nd(\Teul) \setminus \Nd^*(\Teul)$ then $\sigma(v)=0$ if and only if $v$ is pure, and the inequality reads ``$\sigma(v)\ge0$,''
so the Lemma is true in this case as well.
Consider the last remaining case, i.e., assume that $v \in \Nd^*(\Teul)$ and that $N_v \neq 1$.
Let $n$ be the number of occurrences of ``$1$'' in the type $[d_1, \dots, d_s]$ of $v$; so $d_i \neq 1 \Leftrightarrow i > n$.
If $n=0$ then $\xi(v)=1$ and we have $1<d_i<N_v$ for at least two values of $i$,
so Lemma \ref{pvOtyxMdcbgOEjo28hw8y} implies that $\sigma(v) > (N_v-d(v)) = \xi(v)(N_v-1)$; as $v$ is not pure, the Lemma is valid.
If $n>0$ then $\xi(v)=n$ and $\sigma(v) - \xi(v)(N_v-1) = \sum_{i>n}(N_v - d_i) \ge0$, so the inequality is valid,
and equality holds if and only if  $\sum_{i>n}(N_v - d_i) =0$, if and only if $v$ is pure. This proves the Lemma.
\end{proof}

\begin{lemma} \label {clibvaulwe7pq}
If $z \in \Nd^*(\Teul) \cap Z(\Teul)$ then $\xi(z)=1$, $\tD(z)=0$, $a_z=1$ and $z$ is a node of type $[1,N_z,\dots,N_z]$.
\end{lemma}

\begin{proof}
By Lemma \ref{8ey8cdody27}, $\Integ \ni \tD({z}) = 1 - d({z})/a_{z} = 1 - 1/a_{z}$, so $a_{z}=1$ and $\tD({z})=0$.
By Lemma \ref{8ey8cdody27} again, $z$ is a node of type  $[d(z), N_{z}, \dots, N_{z}] = [1, N_{z}, \dots, N_{z}]$.
So $\xi(z)=1$.
\end{proof}

\begin{lemma} \label {db8vcmxdJjvdrticu6tfvild}
Let $(x_1,\dots,x_m)$ be a $\tD$-trivial path.
\begin{enumerata}

\item For each $i$ such that $1<i<m$, we have $x_i \notin \Nd^*(\Teul)$, $x_i$ is pure, $a_{x_i}=1$ and $\tD(x_i)=0$.

\item $\Nd^*(\Teul) \cap \{ x_1, \dots, x_m \} \subseteq \{x_1,x_m\}$

\item If $x_1 \in \Nd^*(\Teul)$ and $\tD(x_1) \le 0$ then
$x_1$ is a node of type $[1, N_{x_1}, \dots, N_{x_1}]$ and satisfies $a_{x_1}=1=\xi(x_1)$,
$\tD(x_1)=0$ and $c(x_m, \{x_m, x_{m-1}\})=1$.

\end{enumerata}
\end{lemma}

\begin{proof}
If $1<i<m$ then $\epsilon(x_i)=2$ and $\tD(x_i)=0$,
so Lemma \ref{90q932r8dhd89cnr9} implies that $N_{x_i} > 1$, $\sigma(x_i)=0$ and $a_{x_i}=1$;
now $\sigma(x_i)=0$ and Rem.\ \ref{F0934nofe8rg3p406egfh}
imply that $d(x_i) = N_{x_i}$, so $x_i \notin \Nd^*(\Teul)$.
As $\sigma(x_i)=0$ also implies that $x_i$ is pure, (a) is proved.
Clearly, (b) follows from (a).

Suppose that $x_1 \in \Nd^*(\Teul)$ and $\tD(x_1)\le 0$.  Then Lemma \ref{clibvaulwe7pq} implies that $\xi(x_1)=1$,
$\tD(x_1)=0$, $a_{x_1}=1$ and $x_1$  is a node of type $[1, N_{x_1}, \dots, N_{x_1}]$.
As $\tD( \{ x_1, \dots, x_{m-1} \} ) = 0$, the pair $(x_m, \{x_m, x_{m-1}\})$ is nonpositive, so 
$\tD( \{ x_1, \dots, x_{m-1} \} ) = 1 - c(x_m, \{x_m, x_{m-1}\})$ by Prop.~\ref{DKxcnpw93sdo},
so $c(x_m, \{x_m, x_{m-1}\}) = 1$ and (c) is proved.
\end{proof}

\begin{corollary} \label {0vk3mWsklxd03mcwlfi}
Let $w \in S(\Teul)$.
\begin{enumerata}

\item $\Nd^*(\Teul) \cap \bar V(w) \subseteq \{w\} \cup V(w)$ 

\item Each $x \in \bar V(w) \setminus (\{w\} \cup V(w))$ is pure and satisfies $a_x=1$ and $\xi(x)=0=\tD(x)$.

\item Let $x \in \Nd^*(\Teul) \cap V(w)$. Then $x$ is a node of type $[1,N_x,\dots,N_x]$ and satisfies $a_x=1=\xi(x)$ and $\tD(x)=0$.
If $\epsilon_x$ denotes the edge of $\gamma_{w,x}$ incident to $w$ then $c(w,\epsilon_x)=1$.

\item $\xi\big( \bar V(w) \big) = \xi(w) + \xi\big( V(w) \big) = \xi(w) + | \Nd^*(\Teul) \cap V(w) |$

\end{enumerata}
\end{corollary}

\begin{proof}
Assertions (a), (b) and (c) follow from Lemma \ref{db8vcmxdJjvdrticu6tfvild}.
The first equality in (d) follows from (a), and the second follows from (c).
\end{proof}

\begin{lemma}  \label {Xc0ejrKlPhgx26xk45}
Let $w \in S(\Teul)$ and $\nu = \xi\big( \bar V(w) \big)$.
\begin{enumerata}

\item  \label {0bgkrCsOkviNhytirfV542okjd} Let $A = \setspec{ \epsilon_v }{ v \in V(w) }$,
where for each $v \in V(w)$ we let $\epsilon_v$ denote the edge of $\gamma_{w,v}$ incident to $w$.
Then $R(w,A) \ge \nu ( 1 - \frac1{N_w} )$.

\item \label {plwekj4tdMoO08346fbmt4yigr} $\tD(\bar V(w) ) \ge (\nu + \delta^*(w) - 2)(N_{w}-1)$

\item \label {pv0ero0DRy8jruewoqwd3co}   If $\Neul \neq \{ w \}$ and $w \notin \Omega(\Teul)$ then
$\tD(\bar V(w) ) \ge \nu + \delta^*(w) - 2$. 

\item \label {3j2mwilsi58ozdOondw7BtJie484}
If equality holds in \eqref{0bgkrCsOkviNhytirfV542okjd} or \eqref{plwekj4tdMoO08346fbmt4yigr} or \eqref{pv0ero0DRy8jruewoqwd3co} then
\begin{enumerata}
\item each $x \in \bar V(w)$ is pure and satisfies $a_{ x } = 1$,
\item $V( w ) \subseteq \Nd^*(\Teul)$.
\end{enumerata}

\end{enumerata}
\end{lemma}

\begin{proof}
Let us begin by observing that in order to prove \eqref{3j2mwilsi58ozdOondw7BtJie484} it suffices to show
that if equality holds in \eqref{0bgkrCsOkviNhytirfV542okjd} or \eqref{plwekj4tdMoO08346fbmt4yigr} or \eqref{pv0ero0DRy8jruewoqwd3co} then:
\begin{itemize}
\item[$(*)$]
 $w$ is pure, $a_{ w } = 1$ and  $V( w ) \subseteq \Nd^*(\Teul)$.
\end{itemize}
Indeed, assume that $(*)$ holds. We have to show that if $x \in \bar V(w)\setminus \{w\}$ then $x$ is pure and $a_x=1$.
If $x \in V(w)$ then  $x\in\Nd^*(\Teul)\cap V(w)$ by $(*)$, so  Cor.~\ref{0vk3mWsklxd03mcwlfi}(c) implies that  $x$ is pure and $a_x=1$.
If  $x \notin V(w)$ then Cor.~\ref{0vk3mWsklxd03mcwlfi}(b)  implies that $x$ is pure and $a_x=1$. 
Hence, if $(*)$ is true then so are (i) and (ii).

\medskip

\eqref{0bgkrCsOkviNhytirfV542okjd}
It follows from Cor.~\ref{0vk3mWsklxd03mcwlfi}(d) that $\nu = t_* + \xi(w)$ 
where $t_* = \xi\big( V(w) \big) = |\Nd^*(\Teul) \cap V(w)|$.  
By Lemma \ref{p9uv8w6t376wsw448gF3j},
\begin{equation} \label {ocivbwi4uhr47865tfhn920}
\text{$\textstyle \sum_{x \in \Deul_{w}} (1 - \frac1{k_x}) \ge \xi(w) (1 - \frac1{N_{w}})$, and if equality holds then $w$ is pure.}
\end{equation}
We claim:
\begin{equation} \label {v984kil3wsCkidu65os8r}
\text{$\textstyle \sum_{e \in A}( 1 - \frac1{M({w},e)}) \ge t_*(1 - \frac1{N_{w}})$, and if equality holds then $V(w) \subseteq \Nd^*(\Teul)$.}
\end{equation}
Indeed, Cor.~\ref{0vk3mWsklxd03mcwlfi}(c) implies that $c({w}, \epsilon_v) = 1$ and hence $M({w}, \epsilon_v) = N_{w}$ for all
$v \in \Nd^*(\Teul) \cap V({w})$; this gives the inequality in \eqref{v984kil3wsCkidu65os8r}.
If equality holds then $1 - \frac1{M({w},\epsilon_v)}=0$  for all $v \in V(w) \setminus \Nd^*(\Teul)$,
and since (by Lemma \ref{GRygergGREg8948r})  $M({w},\epsilon_v)>1$ for all $v \in V(w)$ we get  $V(w) \setminus \Nd^*(\Teul) = \emptyset$,
which proves \eqref{v984kil3wsCkidu65os8r}.
By \eqref{ocivbwi4uhr47865tfhn920} and \eqref{v984kil3wsCkidu65os8r},
\begin{align} \label {vop9JhgCIxljdl938}
\textstyle
R(w,A) & = 
\textstyle   \sum_{x \in \Deul_w} (1 - \frac1{k_x}) + \sum_{e \in A}( 1 - \frac1{M(w,e)}) + (1 - \frac1{a_w}) \\
\notag   & \ge  \textstyle   (\xi(w) + t_*) (1 - \frac1{N_w})+ (1 - \frac1{a_w}) = \nu (1 - \frac1{N_w})+ (1 - \frac1{a_w}) \ge \nu (1 - \frac1{N_w}) \, ,
\end{align}
so \eqref{0bgkrCsOkviNhytirfV542okjd} is proved.
Assume that equality holds in \eqref{0bgkrCsOkviNhytirfV542okjd}.
Then $a_w=1$ by \eqref{vop9JhgCIxljdl938}, and equality holds in both \eqref{ocivbwi4uhr47865tfhn920} and \eqref{v984kil3wsCkidu65os8r},
so $(*)$ is satisfied.  So \eqref{3j2mwilsi58ozdOondw7BtJie484} is true when equality holds in \eqref{0bgkrCsOkviNhytirfV542okjd}.

\medskip

\eqref{plwekj4tdMoO08346fbmt4yigr}
We have $\tD(\bar V(w) ) = \bD(w,A)$.
If $e \in A$ then $(w,e)$ is a tooth, so $\eta(w,e)=0$  by Lemma \ref{GRygergGREg8948r}.
So Thm \ref{P90werd23ewods0ci} gives
$$
\textstyle
 R(w,A) + (\epsilon(w) - |A| - 1) (1 - \frac1{N_w})   =   1 + \frac{ \bD(w,A) - 1 }{N_w} .
$$
Then $\frac{ \bD(w,A) }{N_w} = R(w,A) + (\epsilon(w) - |A| - 2) (1 - \frac1{N_w}) = R(w,A) + (\delta^*(w) - 2) (1 - \frac1{N_w})$.
We have $R(w,A) \ge \nu (1 - \frac1{N_w})$ by \eqref{0bgkrCsOkviNhytirfV542okjd}, so
$\frac{ \bD(w,A) }{N_w} \ge (\nu + \delta^*(w) - 2) (1 - \frac1{N_w})$, so
$$
\tD( \bar V(w) ) = \bD(w,A) \ge (\nu + \delta^*(w) - 2) (N_w - 1) \, .
$$
So \eqref{plwekj4tdMoO08346fbmt4yigr} is proved. Note that if equality holds in \eqref{plwekj4tdMoO08346fbmt4yigr}  then $R(w,A) = \nu (1 - \frac1{N_w})$,
so equality holds in \eqref{0bgkrCsOkviNhytirfV542okjd}.
Consequently,  \eqref{3j2mwilsi58ozdOondw7BtJie484} is true when equality holds in \eqref{plwekj4tdMoO08346fbmt4yigr}.

\medskip

\eqref{pv0ero0DRy8jruewoqwd3co}   Assume that $\Neul \neq \{ w \}$ and $w \notin \Omega(\Teul)$. We claim that this implies
\begin{equation}  \label {aojfbjfdsy63i93b81i3d}
\tD( \bar V(w) ) \ge 0 \, .
\end{equation}
Indeed, assume the contrary: $\tD( \bar V(w) ) < 0$.
If $\Teul$ is a brush then $| S(\Teul) | = 1$, so $S(\Teul) = \{w\}$ and $\Neul = \bar V(w)$,
so $\tD( \bar V(w) ) = \tD(\Neul) \ge 2$ by Cor.~\ref{9vbrtyukey6idckFfFugus},
contradicting the hypothesis $\tD( \bar V(w) ) < 0$.  So $\Teul$ is not a brush.
Then  $\tD( \bar V(w) ) < 0$ and Lemma \ref{p0293efp0cw23ep0hvj} imply that $w \notin W(\Teul)$, so $\bar V(w) = \{w\}$ and hence $\tD(w)<0$.
So $\epsilon(w) \in \{0,1\}$ by Lemma \ref{90q932r8dhd89cnr9}\eqref{WFh2h8heq8whoq9u}.
As $\Neul \neq \{w\}$, we have in fact $\epsilon(w)=1$.
Then $w \in Z(\Teul)$, and since $w \notin\Omega(\Teul)$ it follows that 
$w \in V(w')$ for some $w' \in W(\Teul)$; this contradicts $w \in S(\Teul)$, so \eqref{aojfbjfdsy63i93b81i3d} is proved.
Next, we show that
\begin{equation} \label {1hhnejXio8eergjaiqUETDQWKOW}
\text{if $(\nu + \delta^*(w) - 2) (N_w - 2) <  0$, then $\nu + \delta^*(w) - 2 <0$.}
\end{equation}
It is clear that this implication is true when $N_w \neq 1$, so assume that $N_w=1$.
Then Lemma \ref{90q932r8dhd89cnr9} gives $w \in \Nd^*(\Teul)$ and $V(w) = \emptyset$ (because $\tD(w)=0$, so $w \notin W(\Teul)$);
so $\nu=\xi(w)=1$ because $N_w=1$; since $N_w=1$ also implies that $\delta^*(w)\le1$, the hypothesis of \eqref{1hhnejXio8eergjaiqUETDQWKOW} is false,
so \eqref{1hhnejXio8eergjaiqUETDQWKOW} is true.
So \eqref{1hhnejXio8eergjaiqUETDQWKOW} is true in all cases.

We prove \eqref{pv0ero0DRy8jruewoqwd3co} by contradiction: assume that
\begin{equation} \label {Pc09ec5xe8sdifndsrd}
\tD( \bar V(w) ) < \nu + \delta^*(w) - 2 \, .
\end{equation}
Then $(\nu + \delta^*(w) - 2) (N_w - 1) <  \nu + \delta^*(w) - 2$ by \eqref{plwekj4tdMoO08346fbmt4yigr} and \eqref{Pc09ec5xe8sdifndsrd}, so
$(\nu + \delta^*(w) - 2) (N_w - 2) <  0$.
By \eqref{1hhnejXio8eergjaiqUETDQWKOW}, it follows that  $\nu + \delta^*(w) - 2 <0$.
This together with \eqref{Pc09ec5xe8sdifndsrd} gives $\tD( \bar V(w) ) < 0$, which contradicts \eqref{aojfbjfdsy63i93b81i3d}.
This proves \eqref{pv0ero0DRy8jruewoqwd3co}.

\medskip

Since we already know that   \eqref{3j2mwilsi58ozdOondw7BtJie484} is true when equality holds in \eqref{0bgkrCsOkviNhytirfV542okjd} or
\eqref{plwekj4tdMoO08346fbmt4yigr},
to finish the proof of the Lemma it suffices to show that 
$$
\text{if $\Neul \neq \{ w \}$ and $w \notin \Omega(\Teul)$ and equality holds in \eqref{pv0ero0DRy8jruewoqwd3co},
then equality holds in \eqref{plwekj4tdMoO08346fbmt4yigr}.}
$$
We proceed by contradiction. Assume that $\Neul \neq \{ w \}$ and $w \notin \Omega(\Teul)$ and that equality holds in  \eqref{pv0ero0DRy8jruewoqwd3co}
but not in \eqref{plwekj4tdMoO08346fbmt4yigr}.  Then
\begin{equation} \label {pco98vecrhj45so76eur}
(\nu + \delta^*(w) - 2)(N_{w}-1)  < \tD(\bar V(w) ) = \nu + \delta^*(w) - 2 \, ,
\end{equation}
so $(\nu + \delta^*(w) - 2)(N_{w}-2) < 0$;
then we get $\nu + \delta^*(w) - 2 < 0$ by \eqref{1hhnejXio8eergjaiqUETDQWKOW},
which implies that $\tD(\bar V(w))<0$ by \eqref{pco98vecrhj45so76eur}.
This contradicts \eqref{aojfbjfdsy63i93b81i3d}, so the proof is complete.
\end{proof}

We may now prove three special cases of Thm \ref{p0fkwJjk8eCahbkBJhZuIK83gldf}.

\begin{proposition} \label {Pp0drigro8awe7v8csbhsnehry}
Assume that $| \Neul | = 1$.
Then $\xi(\Neul) \le \max(0,\tD(\Neul)) +2$ where equality holds if and only if $\Teul$ is one of the trees of Ex.\ {\rm \ref{nv63jfy64nvy3}(b).}
In particular, Thm \ref{p0fkwJjk8eCahbkBJhZuIK83gldf} is valid when $| \Neul | = 1$.
\end{proposition}

\begin{proof}
Assume that $| \Neul | = 1$.  Then $\Neul = \{ v_0 \}$.
We may assume that $\xi(\Neul)>0$, otherwise $\xi(\Neul) < \max(0,\tD(\Neul)) +2$ and there is nothing to prove.
Note that $\xi(\Neul)>0$ implies that $\gcd\setspec{ d_x }{ x \in \Deul_{v_0} } = 1$.

Consider the case where $\tD(\Neul)\le0$.
Then, by Prop.~\ref{c0viwjytsdDJjLlxdifFebg}, $\Teul$ is one of the trees of Ex.\ \ref{nv63jfy64nvy3}.
If $\Teul$ is as in  \ref{nv63jfy64nvy3}(a) then $\xi(\Neul) = 1 < 2 = \max(0,\tD(\Neul)) + 2$;
if $\Teul$ is as in  \ref{nv63jfy64nvy3}(b) then  $\xi(\Neul) = 2 = \max(0,\tD(\Neul))+2$; so the claim is true whenever $\tD(\Neul)\le0$.

From now-on, assume that $\tD(\Neul) > 0$. Then $\delta_{v_0}>2$ by Cor.~\ref{pc09vnw3oe9c};
since $N_{v_0} \ge \delta_{v_0}$ (see Rem.~\ref{efy872392eujf}), we have $N_{v_0}>2$.
We have $\tD(\Neul) = \tD(v_0) = \sigma(v_0) - 2(N_{v_0}-1)$ by Lemma \ref{90q932r8dhd89cnr9}\eqref{pc9vp23r09}
and $\sigma(v_0) \ge \xi(v_0) ( N_{v_0} - 1 )$ by Lemma \ref{p9uv8w6t376wsw448gF3j}, so 
$$
\xi(\Neul) = \xi(v_0) \le \frac {\sigma(v_0)}{ N_{v_0}-1} =  \frac {\tD(\Neul) + 2(N_{v_0}-1)}{ N_{v_0}-1} =  \frac {\tD(\Neul)}{ N_{v_0}-1} + 2 < \tD(\Neul) + 2,
$$
so $\xi(\Neul) < \max(0,\tD(\Neul))+2$ and the claim is true.
\end{proof}

\begin{proposition}  \label {PiIoO0v9w6ejamuisuergbf83e9f0}
If $\Teul$ is a brush then the following hold.
\begin{enumerata}

\item $\Neul = \bar V(v_0)$ and $\Nd^*(\Teul) \subseteq \{v_0\} \cup V(v_0)$.

\item $\xi(\Neul)  \le \tD(\Neul) + 2$ and if equality holds
then $V(v_0) \subseteq \Nd^*(\Teul)$ and each $x \in \Neul$ is pure and satisfies $a_x=1$.

\end{enumerata}
In particular, Thm \ref{p0fkwJjk8eCahbkBJhZuIK83gldf} is valid when $\Teul$ is a brush.
\end{proposition}

\begin{proof}
By Def.~\ref{pd09Yv3ned09Xse}, there exists $w \in W(\Teul)$ such that $\bar V(w) = \Neul$;
we have $W(\Teul) = \{ v_0 \}$ by Rem.~\ref{pc09n3409vnZiEWOdhFp30ef}, so $\Neul=\bar V( v_0 )$.
Then $\Nd^*(\Teul) \subseteq \{v_0\} \cup V(v_0)$ by Cor.~\ref{0vk3mWsklxd03mcwlfi}.

We have $\Neul \neq \{ v_0 \}$ (because $| \Neul | > 1$ by Rem.~\ref{pc09n3409vnZiEWOdhFp30ef})
and  $v_0 \notin \Omega(\Teul)$ (because $W(\Teul) \cap \Omega(\Teul) = \emptyset$),
so Lemma  \ref{Xc0ejrKlPhgx26xk45}\eqref{pv0ero0DRy8jruewoqwd3co} implies that
$\tD(\bar V(v_0) )
\ge \xi\big(\bar V(v_0)\big)  + \delta^*(v_0) - 2 
= \xi(\Neul)  - 2$ (because $\delta^*(v_0)=0$),
so $\xi(\Neul)  \le \tD(\Neul) + 2$.
If equality holds then  Lemma  \ref{Xc0ejrKlPhgx26xk45} implies that $V(v_0) \subseteq \Nd^*(\Teul)$ 
and that each $x \in \Neul$ is pure and satisfies $a_x=1$.
\end{proof}

\begin{proposition}  \label {hCfa89HFdfjkUY738ihhdJMJi854727}
If $| \Omega(\Teul) | = 2$ then the following hold:
\begin{enumerata}

\item $\Nd^*(\Teul) \subseteq \Omega(\Teul)$, $| \Nd^*(\Teul) |  = \xi(\Neul)  \le 2$, and each $x \in \Nd^*(\Teul)$ has type $[1,N_x,\dots,N_x]$
and satisfies $a_x=1 = \xi(x)$ and $\tD(x)=0$.

\item If $\xi(\Neul) =2$ then each $x \in \Neul$ is pure and satisfies $a_x=1$ and $\tD(x)=0$.

\end{enumerata}
In particular, Thm \ref{p0fkwJjk8eCahbkBJhZuIK83gldf} is valid when $| \Omega(\Teul) | = 2$.
\end{proposition}

\begin{proof}
Suppose that  $| \Omega(\Teul) | = 2$.
By Thm \ref{Xcoikn23ifcdKJDluFYT937}, there exists a $\tD$-trivial path $(x_1, \dots, x_m)$ such that 
$\Neul = \{ x_1, \dots, x_m \}$, $\Omega(\Teul) = \{x_1,x_m\}$,
$\tD(x_1), \tD(x_m) \le 0$ and 
$\tD(\Neul) = \tD(x_1) + \tD(x_m) \le 0$.
Note that  $(x_m, \dots, x_1)$ is also a $\tD$-trivial path;
applying Lemma \ref{db8vcmxdJjvdrticu6tfvild} to both $(x_1, \dots, x_m)$ and $(x_m, \dots, x_1)$ 
implies that $\Nd^*(\Teul) \subseteq \{x_1, x_m\}$
and that each $x \in \Nd^*(\Teul)$ has type $[1, N_x, \dots, N_x]$ and satisfies $a_x=1 = \xi(x)$ and $\tD(x)=0$.
Thus $| \Nd^*(\Teul) | = \xi(\Neul) \le 2$.
To prove (c), first note that the condition ``$x$ is pure and  satisfies $a_x=1$ and $\tD(x)=0$'' is satisfied for each $x = x_i$ with $1<i<m$,
and also (by (b)) for each $x \in \Nd^*(\Teul)$; if $\xi(\Neul) =2$ then this covers all $x \in \Neul$.
\end{proof}

Recall from Lemma \ref{p309ef2GFT485ryf} that if $| S(\Teul) | = 1$ then either $|\Neul| = 1$ or $\Teul$ is a brush.
In view of Propositions \ref{Pp0drigro8awe7v8csbhsnehry}, \ref{PiIoO0v9w6ejamuisuergbf83e9f0} and \ref{hCfa89HFdfjkUY738ihhdJMJi854727},
this means that Thm \ref{p0fkwJjk8eCahbkBJhZuIK83gldf} is valid whenever $| S(\Teul) | = 1$ or $| \Omega(\Teul) | = 2$.
So the remaining part of the proof can be carried out under the assumptions of paragraph \ref{934btgodog9qkej}:

\begin{parag} \label {934btgodog9qkej}
\textbf{\boldmath Let us assume that  $| S(\Teul) | \neq 1$ and $| \Omega(\Teul) | \neq 2$.} It follows that
$$
| S(\Teul) | > 1 \quad \text{and} \quad  | \Omega(\Teul) | < 2 \, .
$$
Let $z \in \In(\Teul)$ and $\overline\Oeul = \overline\Oeul(\Teul,z)$.
Note that $\overline\Oeul \neq \emptyset$, because  $| S(\Teul) | > 1$ (see Def.~\ref{c20w39w93e9d0dqMne89wcg9}). 
Define 
$$
\nu_z = \xi\big( \bar V(z) \big) \, .
$$
For each $C \in \overline\Oeul$, 
let $(v_C,e_{v_C}) \preceq (u_C,e_{u_C})$ be respectively the least and greatest elements of $C$ and define
$$
\beta_C = \xi\big( \Neul(u_C,e_{u_C}) \setminus \Neul(v_C,e_{v_C}) \big) 
\quad \text{and} \quad
\nu_C = \xi\big( \bar V(u_C) \big) \, .
$$
Then 
\begin{equation} \label {87rjHFTRk77bhw7dneo}
\xi( \Neul ) = \nu_z + \sum_{C \in \overline\Oeul} (\beta_C + \nu_C)
= \nu_z  +  \beta_{C_0} + \nu_{C_0} + \sum_{C \in \overline\Oeul \setminus \{ C_0 \} } (\beta_C + \nu_C) \, .
\end{equation}
\end{parag}

\medskip

The assumptions and notations of \ref{934btgodog9qkej} are in effect throughout \ref{9hJTY8gJ73jr9Ij93fe}--\ref{76GgdT276d729gs27g882h}.

\begin{sublemma} \label {9hJTY8gJ73jr9Ij93fe}
Let $C \in \overline\Oeul$. If  $(v,e) \in C$ satisfies $\Neul(v,e) \cap \Nd^*(\Teul) \neq \emptyset$ then 
$$
1 \ge c(v,e) = c(u_C, e_{u_C} ) \, .
$$
\end{sublemma}

\begin{proof}
Since $\Neul(v,e) \cap \Nd^*(\Teul) \neq \emptyset$, Lemma \ref{kjwoeid9cse9c}(c) gives $c(v,e) \le 1$.
So we have $0 < c(u_C, e_{u_C}) \le c(v,e) \le 1$ and consequently
$0 \le c(v,e) - c(u_C, e_{u_C})<1$.
Moreover, $(u_C, e_{u_C})$ is a comb over $(v,e)$, so $\eta(u_C,e_{u_C}) = \eta(v,e)$ by Prop.~\ref{0ci19KJTghL872je309}(a);
then Lemma \ref{p0c9vin12q09wsc} gives
$c(v,e) - c(u_C, e_{u_C}) = \tD( \Neul(u_C, e_{u_C}) ) - \tD( \Neul(v, e) ) \in \Integ$,
so $c(v,e) - c(u_C, e_{u_C}) = 0$.
\end{proof}

\begin{sublemma} \label {88yHvcxhl93hf7}
Let $C \in \overline\Oeul$ and let $(v_C,e_{v_C}) \preceq (u_C,e_{u_C})$ be respectively the least and greatest elements of $C$.
Consider the path $\gamma_{v_C,u_C} = (v_1, \dots, v_n)$ and the set
$$
H_C = \big( \Neul(u_C, e_{u_C})  \setminus \Neul(v_C, e_{v_C}) \big) \cap \Nd^*(\Teul) \, .
$$

\begin{enumerata}

\item $|H_C| \le 1$, and if $H_C \neq \emptyset$ then  $\dot c(C) > 0$. 

\item Let $x \in H_C$. Then $a_x=1$, $x$ is a node of type $[1,N_x,\dots,N_x]$, and a unique $i \in \{1, \dots, n-1\}$ satisfies $(\alpha)$ or $(\beta)$, where:
\begin{itemize}

\item[$(\alpha)$] $x = v_i$ and $\epsilon(v_i)=2$;

\item[$(\beta)$]  there exists $(x_1, \dots, x_m) \in \Gamma(\Teul)$ such that $x = x_1$ and $x_m = v_i$.

\end{itemize}
Moreover,  the unique $i$ that satisfies $(\alpha)$ or $(\beta)$ also satisfies
\begin{itemize}

\item[$(\gamma)$] $c(v_i, e_{v_i}) = N_{v_i} > 1 \ge c(v_{i+1}, e_{v_{i+1}}) = c(u_{C}, e_{u_{C}})$.

\end{itemize}

\item $ \beta_C = \big| H_C \big| \in \{0,1\} $

\end{enumerata}
\end{sublemma}

\begin{proof}
Assertion (c) is a consequence of (a) and (b).

(b) Let $x \in H_C$.
Since $x \in \Neul(u_C, e_{u_C})  \setminus \Neul(v_C, e_{v_C})$, 
a unique $i \in \{1, \dots, n-1\}$ satisfies $(\alpha')$ or $(\beta')$, where:
\begin{itemize}

\item[$(\alpha')$] $x = v_i$;

\item[$(\beta')$]  there exists $(x_1, \dots, x_m) \in \Gamma(\Teul)$ such that $x \in \{ x_1, \dots, x_{m-1} \}$ and $x_m = v_i$.

\end{itemize}

Observe that $i$ satisfies $\epsilon(v_i)>1$, so Lemma \ref{90q932r8dhd89cnr9}\eqref{0239u9e78923} gives 
$$
N_{v_i} > 1 \, .
$$

If $(\beta')$ holds then Cor.~\ref{0vk3mWsklxd03mcwlfi}(a) implies that $x \in V(x_m)$, so $x = x_1$, so $(\beta)$ holds.

Suppose that $(\alpha')$ holds but that $(\alpha)$ does not.
Then the definition of comb implies that $\epsilon(v_i) =3$, that $R(v_i, \{ e_{v_i} \}) = 0$,
and that there exists $(x_1, \dots, x_m) \in \Gamma(\Teul)$ such that $x_m=v_i$.
Since $x_m = v_i = x \in \Nd^*(\Teul)$, Lemma \ref{pvOtyxMdcbgOEjo28hw8y} implies that $\sigma(v_i) \ge N_{v_i}-d(v_i) = N_{v_i}-1>0$,
so $\sigma(v_i)>0$, which contradicts  $R(v_i, \{ e_{v_i} \}) = 0$.
This contradiction shows that if $(\alpha')$ holds then so does $(\alpha)$.
So we have shown that  a unique $i \in \{1, \dots, n-1\}$ satisfies $(\alpha)$ or $(\beta)$.
Let us now show that the same $i$ satisfies $(\gamma)$. (We already noted that $N_{v_i}>1$.)

If $i$ satisfies $(\alpha)$ then $R(v_i, \{ e_{v_i} \} ) < 1$ and (by Lemma \ref{pvOtyxMdcbgOEjo28hw8y}) $\sigma(v_i) \ge N_{v_i} - d(v_i) = N_{v_i} - 1 > 0$,
so $M(v_i, e_{v_i} ) =1$, so $c(v_i, e_{v_i} ) = N_{v_i}$. 
If $i$ satisfies $(\beta)$ then $R(v_i, \{ e_{v_i} \} ) =0$,
so $M(v_i, e_{v_i} ) =1$ and hence $c(v_i, e_{v_i} ) = N_{v_i}$.  So in both cases we have 
$$
c(v_i, e_{v_i} ) = N_{v_i} > 1 \, .
$$
Since $\Neul(v_{i+1},e_{v_{i+1}}) \cap \Nd^*(\Teul) \neq \emptyset$, Lemma \ref{9hJTY8gJ73jr9Ij93fe} implies that 
$1 \ge c(v_{i+1}, e_{v_{i+1}}) = c(u_{C}, e_{u_{C}})$, so $i$ satisfies $(\gamma)$.

We claim that $a_x=1$ and that $x$ is a node of type $[1,N_x, \dots, N_x]$.
If $i$ satisfies $(\beta)$, this follows from Cor.~\ref{0vk3mWsklxd03mcwlfi}(c).
We noted in the preceding paragraph that if $i$ satisfies $(\alpha)$ then $\sigma(x)>0$ and $R(x,\{e_x\})<1$;
so $a_x=1$; as $\sigma(x)/N_x \le R(x,\{e_x\})<1$, the type of $x$ is $[d(x), N_x, \dots, N_x]$,
and since $x \in \Nd^*(\Teul)$ the type is  $[1,N_x, \dots, N_x]$.
So (b) is proved.

(a) Consider $x,x' \in H_C$. By (b), there exist $i,i' \in \{1, \dots, n-1\}$ such that
$(x,i)$ satisfies  $(\alpha)$ or $(\beta)$, and $(x',i')$ satisfies  $(\alpha)$ or $(\beta)$.
Since $i$ and $i'$ must then satisfy $(\gamma)$, we have $i=i'$; considering the statements  $(\alpha)$ and $(\beta)$, it
follows that $x=x'$.  So $|H_C| \le 1$.
If $H_C \neq \emptyset$ then there is an $i$ satisfying $(\gamma)$, so $\dot c(C) \ge c(v_i, e_{v_i}) - c(v_{i+1}, e_{v_{i+1}}) > 0$.
So (a) is true.

(d) We have $\beta_C = \xi\big( \Neul(u_C,e_{u_C}) \setminus \Neul(v_C,e_{v_C}) \big)$ by definition
and  $\xi\big( \Neul(u_C,e_{u_C}) \setminus \Neul(v_C,e_{v_C}) \big) = \xi( H_C )$ because $\xi$ is zero outside of $\Nd^*(\Teul)$.
Since (b) implies that $\xi(x)=1$ for each $x \in H_C$, it follows that $\xi(H_C) = | H_C |$. 
\end{proof}

\begin{sublemma}  \label {ovF8ym5jDdrfjkerdjkteowos} 
Let $X = \bar V(u_0) \cup \Neul(u_0,e_{u_0})$.  
\begin{enumerata}

\item \label {P0v3b49wsrpqwo0q59RHYJ}  $\xi(X) \le \max \big(0, \tD( X ) \big) - \delta^*(u_0) + 3$

\item If equality holds in \eqref{P0v3b49wsrpqwo0q59RHYJ} then 
$N_{ u_0 }>1$, 
$V( u_0 ) \subseteq \Nd^*(\Teul)$, 
each $x \in \bar V(u_0)$ is pure and satisfies $a_{ x } = 1$,
and each $x \in \Nd^*(\Teul) \cap X$ is pure 
and satisfies $a_x=1$.

\item Assume that $\tD(X) \le 0$. Then $\Nd^*(\Teul) \cap \Neul(u_0,e_{u_0}) = \emptyset$,
and if equality holds in \eqref{P0v3b49wsrpqwo0q59RHYJ} then $M(u_0,e_{u_0})=1$.

\end{enumerata}
\end{sublemma}

\begin{proof}
Let us first prove the case where $\Omega(\Teul) = \emptyset$.
Lemma \ref{A0c9vub23oW8fgbgp0q2arft} implies that $(u_0,e_{u_0})$ is not nonpositive
and Cor.\ \ref{0c9in34dh5sifzsshnkus} gives $\tD( \bar V(u_0) )>0$, so $\tD(X) = \tD( \Neul(u_0,e_{u_0}) ) + \tD( \bar V(u_0) ) \ge 2$.
In particular, assertion (c) holds trivially. The fact that  $\tD( \bar V(u_0) )>0$ also implies that $N_{u_0}>1$
($N_{u_0}=1$ implies $\bar V(u_0)=\{u_0\}$ and $\tD(u_0)=0$), which is part of assertion (b).
For each $w \in S(\Teul) \cap X$ we have  $\Neul \neq \{ w \}$ (because $| S(\Teul) | > 1$)
and $w \notin \Omega(\Teul)$ (because  $\Omega(\Teul) = \emptyset$), 
so Lemma \ref{Xc0ejrKlPhgx26xk45}\eqref{pv0ero0DRy8jruewoqwd3co} implies that 
\begin{equation}  \label {cgd9s6fjLvIlh7623l923f}
\tD(\bar V(w) ) \ge  \xi\big( \bar V(w) \big) + (\delta^*(w) - 2) \quad \text{for all $w \in S(\Teul) \cap X$.}
\end{equation}
We have $\sum_{w \in S(\Teul)\cap X} \tD(\bar V(w) ) = \tD(X)$ and 
$\sum_{w \in S(\Teul)\cap X}  \xi\big( \bar V(w) \big) = \xi(X)$,
because $\big( \bar V(w) \big)_{w \in S(\Teul) \cap X}$ is an f-partition of $X$.
Since $\delta^*(w)=2$ for all $w \in  (S(\Teul) \cap X) \setminus\{z,u_0\}$,
$\sum_{w \in S(\Teul) \cap X} (\delta^*(w) - 2) = (\delta^*(z)-2) + (\delta^*(u_0)-2) = \delta^*(u_0)-3$.
So \eqref{cgd9s6fjLvIlh7623l923f} implies that $\tD( X ) \ge \xi(X) + \delta^*(u_0) - 3$, which proves \eqref{P0v3b49wsrpqwo0q59RHYJ}.
If equality holds in \eqref{P0v3b49wsrpqwo0q59RHYJ}  then it must hold in \eqref{cgd9s6fjLvIlh7623l923f} for all  $w \in S(\Teul) \cap X$;
then Lemma \ref{Xc0ejrKlPhgx26xk45} implies that a condition stronger than (b) is satisfied, namely:
each $x \in X$ is pure and satisfies $a_x=1$,
and each $w \in S(\Teul) \cap X$ satisfies $V(w) \subseteq \Nd^*(\Teul)$ (and we already know that $N_{u_0}>1$).
So the Lemma is valid whenever $\Omega(\Teul) = \emptyset$.

\medskip

From now-on, and until the end of the proof, we assume that $\Omega(\Teul) \neq \emptyset$.
This implies that $\Omega(\Teul) = \{z\}$. Thus $\tD(z)\le0$, so $z \notin W(\Teul)$ and hence $\bar V(z) = \{z\}$.
Consequently, $\nu_z = \xi\big( \bar V(z) \big) = \xi(z)$, and we have $\xi(z) \in \{0,1\}$ by Lemma \ref{clibvaulwe7pq}, because $z \in Z(\Teul)$.
So Lemmas \ref{9hJTY8gJ73jr9Ij93fe} and \ref{88yHvcxhl93hf7} have the following consequence:
\begin{equation} \label {cjh23ehjiwdjmuduhj84}
\text{$\nu_z + \beta_{C_0} \in \{0,1\}$,
and if  $\nu_z + \beta_{C_0} \neq 0$ then $c(u_0,e_{u_0}) \le 1$.}
\end{equation}

We prove the first part of (c), i.e.,
\begin{equation} \label {p09cvj34e90dcw}
\text{if $\tD(X)\le0$ then $\Nd^*(\Teul) \cap \Neul(u_0,e_{u_0}) = \emptyset$.}
\end{equation}
Assume that $\tD(X)\le0$. We have $0 \ge \tD(X) = \tD( \bar V(u_0) ) + \tD( \Neul(u_0,e_{u_0}) )$ and 
$\tD( \bar V(u_0) ) > 0$ by Cor.\ \ref{0c9in34dh5sifzsshnkus}(a), so $\tD( \Neul(u_0,e_{u_0}) ) < 0$.
Then $(u_0,e_{u_0})$ is nonpositive, so  $\tD( \Neul(u_0,e_{u_0}) ) = 1 - c(u_0,e_{u_0})$ by Prop.\ \ref{DKxcnpw93sdo}.
Thus $c(u_0,e_{u_0})>1$, so \eqref{cjh23ehjiwdjmuduhj84} implies that $\nu_z + \beta_{C_0} = 0$, i.e.,
$\Nd^*(\Teul) \cap \Neul(u_0,e_{u_0}) = \emptyset$, proving \eqref{p09cvj34e90dcw}.

We claim that the Lemma is true whenever $\delta^*(u_0) + \nu_{C_0} <3$.
Indeed, if  $\delta^*(u_0) =1$  and  $\nu_{C_0} \le 1$ then  $\xi(X) = \nu_z + \beta_{C_0} + \nu_{C_0} \le 2 \le \max \big(0, \tD( X ) \big) - \delta^*(u_0) + 3$,
so \eqref{P0v3b49wsrpqwo0q59RHYJ} is true in this case;
if $\delta^*(u_0) =2$ and $\nu_{C_0}=0$ then $\xi(X) = \nu_z + \beta_{C_0} + \nu_{C_0} \le 1$ and Cor. \ref{0c9in34dh5sifzsshnkus} gives 
$\tD( X ) \ge  | \delta^*(u_0) - 3 | = 1 \ge \xi(X)$, so $\xi(X) \le \max(0,\tD(X)) < \max(0,\tD(X)) - \delta^*(u_0) + 3$,
and again \eqref{P0v3b49wsrpqwo0q59RHYJ} is true; so:
\begin{equation} \label {pv09w3vrw7899whdoas9x}
\text{if $\delta^*(u_0) + \nu_{C_0} <3$ then \eqref{P0v3b49wsrpqwo0q59RHYJ} is true.}
\end{equation}
Assume that $\delta^*(u_0) + \nu_{C_0} <3$ and that equality holds in \eqref{P0v3b49wsrpqwo0q59RHYJ}.
Then
\begin{equation}  \label {cp0Kg93h38YJ2POxi}
0 \le \max \big(0, \tD( X ) \big) = \xi(X) + \delta^*(u_0) -3
= (\nu_z + \beta_{C_0}) + (\nu_{C_0} + \delta^*(u_0) -3) \le 0 \, ,
\end{equation}
the last inequality because
$\nu_z + \beta_{C_0} \le 1$ and  $\nu_{C_0} + \delta^*(u_0) -3 \le -1$.
Thus $\tD(X)\le0$ by \eqref{cp0Kg93h38YJ2POxi}, so $\nu_z + \beta_{C_0} = 0$ by \eqref{p09cvj34e90dcw};
then \eqref{cp0Kg93h38YJ2POxi} implies that $0 \le \nu_{C_0} + \delta^*(u_0) -3 \le 0$, which contradicts the hypothesis $\delta^*(u_0) + \nu_{C_0} <3$.
This shows that
\begin{equation} \label {4k3krtemM56jklocvjf56}
\text{if $\delta^*(u_0) + \nu_{C_0} <3$ then equality cannot hold in \eqref{P0v3b49wsrpqwo0q59RHYJ}.}
\end{equation}
So, when $\delta^*(u_0) + \nu_{C_0} <3$, both (b) and the second part of (c) hold trivially.
This together with \eqref{p09cvj34e90dcw} and \eqref{pv09w3vrw7899whdoas9x} means that the Lemma is true when $\delta^*(u_0) + \nu_{C_0} <3$.

To prove the Lemma, we may therefore assume that
\begin{equation}  \label {ckjb2o38eewpqFaJ3yrfv}
\delta^*(u_0) + \nu_{C_0} \ge 3 \, .
\end{equation}
If $N_{u_0} = 1$ then Lemma \ref{90q932r8dhd89cnr9} says that $\epsilon(u_0)=1$ and that $u_0$ is a node of type $[1,\dots,1]$,
so $\nu_{C_0} = 1$ and $\delta^*(u_0)=1$, which contradicts \eqref{ckjb2o38eewpqFaJ3yrfv}. This shows that 
\begin{equation}  \label {9823i5tvbg9w48R}
N_{u_0} > 1 \, .
\end{equation}

Let us now prove (a).
Define $A = \setspec{ \epsilon_v }{ v \in V(u_0) }$,
where for each $v \in V(u_0)$ we let $\epsilon_v$ denote the edge of $\gamma_{{u_0},v}$ incident to ${u_0}$,
and let $A' = A \cup \{ e_{u_0} \}$.
Since $\Omega(\Teul) \neq \emptyset$, Lemma \ref{A0c9vub23oW8fgbgp0q2arft} implies that $(u_0,e_{u_0})$ is nonpositive,
so $\eta(u_0,e_{u_0}) = 0$ by Prop.\ \ref{DKxcnpw93sdo}; we also have $\eta(u_0,e)=0$ for each $e \in A$, by Lemma \ref{GRygergGREg8948r};
so $\eta(u_0,e)=0$ for all $e \in A'$.  Also note that $\tD(X) = \bD({u_0},A')$.
By Thm \ref{P90werd23ewods0ci},
\begin{equation*}
\textstyle
R({u_0},A') + (\epsilon({u_0}) - |A'| - 1) (1 - \frac1{N_{u_0}})
= \textstyle  1 + \frac{\tD(X) - 1  - \sum_{e \in A'} \eta(u_0,e) }{ N_{u_0} } 
= \textstyle  (1 - \frac1{N_{u_0}}) + \frac{\tD(X)}{ N_{u_0} }
\end{equation*}
Since $\epsilon({u_0}) - |A'| = \delta^*({u_0}) -1$,
\begin{equation}  \label {cosmgBVhgf93267c388}
\tD(X) = N_{u_0} R({u_0},A') +  (\delta^*(u_0)-3) (N_{u_0}-1) \, .
\end{equation}
We have $R({u_0},A') = R({u_0},A) + (1 - \frac1{M(u_0,e_{u_0})})$,
and Lemma \ref{Xc0ejrKlPhgx26xk45}\eqref{0bgkrCsOkviNhytirfV542okjd} gives  
\begin{equation}  \label {cp0wjhkmsjkusew3nbosd}
\textstyle   R({u_0},A) \ge \nu_{C_0}(1-\frac1{N_{u_0}}) \, .
\end{equation}
It follows that
\begin{multline}  \label {pf3jerkksseynxzmxed}
N_{u_0} R({u_0},A')
= N_{u_0} R({u_0},A) + (N_{u_0} - c(u_0,e_{u_0})) \\
\ge  \nu_{C_0} (N_{u_0} - 1)  + (N_{u_0} - c(u_0,e_{u_0}))
 = \nu_{C_0} (N_{u_0} - 1)   + \nu_z + \beta_{C_0} + E \, ,
\end{multline}
where we define $E =  N_{u_0} - c(u_0,e_{u_0}) - \nu_z - \beta_{C_0}$.
Note that if equality holds in \eqref{pf3jerkksseynxzmxed} then it holds in \eqref{cp0wjhkmsjkusew3nbosd}, so Lemma \ref{Xc0ejrKlPhgx26xk45} implies:
\begin{equation}  \label {cimao3r8q432q7XNAweyi}
\begin{minipage}[t]{.7\textwidth}
If equality holds in \eqref{pf3jerkksseynxzmxed} then $V( u_0 ) \subseteq \Nd^*(\Teul)$ and
each $x \in \bar V(u_0)$ is pure and satisfies $a_{ x } = 1$.
\end{minipage}
\end{equation}

We claim that $E \ge0$. Indeed, if  $\nu_z + \beta_{C_0} = 0$ then $E =  N_{u_0} - c(u_0,e_{u_0})$ is nonnegative
because $c(u_0,e_{u_0}) \mid N_{u_0}$ by Thm \ref{xncoo9qwdx9},
and if $\nu_z + \beta_{C_0} \neq 0$ then we have $\nu_z + \beta_{C_0} = 1$ and $c(u_0,e_{u_0}) \le 1< N_{u_0}$
by \eqref{cjh23ehjiwdjmuduhj84} and \eqref{9823i5tvbg9w48R},
so $\nu_z + \beta_{C_0} = 1 \le N_{u_0} - c(u_0,e_{u_0})$ and hence $E\ge0$.
Now \eqref{cosmgBVhgf93267c388} and \eqref{pf3jerkksseynxzmxed} give
$\tD(X) \ge  \nu_{C_0} (N_{u_0} - 1)   + \nu_z + \beta_{C_0} + E +  (\delta^*(u_0)-3) (N_{u_0}-1)$
and consequently
\begin{multline}  \label {vfcfgwjek5tfftswwer}
\tD(X) - (\delta^*(u_0)-3) - \xi(X)
= \tD(X) - (\delta^*(u_0)-3) - (\nu_z + \beta_{C_0} + \nu_{C_0}) \\
\ge  \nu_{C_0} (N_{u_0} - 2) + (\delta^*(u_0)-3) (N_{u_0}-2) + E 
= (\delta^*(u_0) + \nu_{C_0} -3) (N_{u_0}-2) + E \ge 0 \, ,
\end{multline}
where the last inequality follows from  \eqref{ckjb2o38eewpqFaJ3yrfv}, \eqref{9823i5tvbg9w48R} and $E\ge0$.
So \eqref{P0v3b49wsrpqwo0q59RHYJ} is proved.

\medskip

(b) Assume that equality holds in \eqref{P0v3b49wsrpqwo0q59RHYJ}; by \eqref{4k3krtemM56jklocvjf56}, this implies
that \eqref{ckjb2o38eewpqFaJ3yrfv} is satisfied, so \eqref{9823i5tvbg9w48R} is also satisfied, i.e., $N_{u_0}>1$.
Since \eqref{ckjb2o38eewpqFaJ3yrfv} is satisfied, so is \eqref{vfcfgwjek5tfftswwer};
since equality holds in \eqref{P0v3b49wsrpqwo0q59RHYJ},
the two inequalities in \eqref{vfcfgwjek5tfftswwer} are in fact equalities,
so equality holds in \eqref{pf3jerkksseynxzmxed} and hence the conclusions of \eqref{cimao3r8q432q7XNAweyi} are valid.
So, to finish the proof of (b), it suffices to show that 
\begin{equation} \label {Opcogjis7etv3p}
\text{each $x \in \Nd^*(\Teul) \cap \Neul(u_0,e_{u_0})$ has type  $[1, N_x, \dots, N_x]$ and satisfies $a_x=1$}
\end{equation}
(because  having type  $[1, N_x, \dots, N_x]$ is stronger than being pure).
By Lemma \ref{88yHvcxhl93hf7}, \eqref{Opcogjis7etv3p} is true when $x \neq z$.
If  $z \in \Nd^*(\Teul)$ then (since $z \in \Omega(\Teul) \subseteq Z(\Teul)$) Lemma \ref{clibvaulwe7pq} implies that 
$a_z=1$ and that $z$ is a node of type $[1,N_z,\dots,N_z]$.
This proves \eqref{Opcogjis7etv3p}, and completes the proof of (b).

\medskip

There only remains to prove the second part of (c).
Assume that $\tD(X)\le0$ and that equality holds in (a).
Then $E=0$ in \eqref{vfcfgwjek5tfftswwer};
as $E =  N_{u_0} - c(u_0,e_{u_0}) - \nu_z - \beta_{C_0}$ and (by \eqref{p09cvj34e90dcw})  $\nu_z + \beta_{C_0} = 0$, we have 
$N_{u_0} = c(u_0,e_{u_0})$ and hence $M(u_0,e_{u_0})=1$.
This completes the proof of the Lemma.
\end{proof}

\begin{sublemma}  \label {76GgdT276d729gs27g882h}
\begin{enumerata}

\item \label {9238ctg8se7p2e} $\xi(\Neul) \le \max(0,\tD(\Neul)) + 2$

\item If equality holds in \eqref{9238ctg8se7p2e}  then 
$N_{u_0}>1$, each $x \in \Neul \setminus \Neul(u_0,e_{u_0})$ is pure and satisfies $a_x=1$,
each $x \in \Nd^*(\Teul)$ is pure and satisfies $a_x=1$, 
and for each $w \in S(\Teul) \setminus \Neul(u_0,e_{u_0})$ we have $V(w) \subseteq \Nd^*(\Teul)$.

\item Assume that $\tD(\Neul)\le0$.
Then $| \overline \Oeul | = 1$ and $\Nd^*(\Teul) \subseteq \{u_0\} \cup V(u_0)$,
and if equality holds in \eqref{9238ctg8se7p2e} then $M(u_0,e_{u_0}) = 1$.

\end{enumerata}
\end{sublemma}

\begin{proof}
Let $X = \bar V(u_0) \cup \Neul(u_0,e_{u_0})$.  

Consider the case where $| \overline \Oeul | = 1$.
Then $X = \Neul$ and $\delta^*(u_0)=1$, so Lemma \ref{ovF8ym5jDdrfjkerdjkteowos} gives
$$
\xi(\Neul)  \le \max(0,\tD( \Neul )) - \delta^*(u_0) + 3 = \max(0,\tD( \Neul )) + 2 \, ,
$$
so \eqref{9238ctg8se7p2e} holds.
Assertion (b) (resp.\ (c)) follows from part (b) (resp.\ part (c)) of  Lemma \ref{ovF8ym5jDdrfjkerdjkteowos},
so the Lemma is true when  $| \overline \Oeul | = 1$.

From now-on, assume that  $| \overline \Oeul | > 1$. 
Then $\tD(X)\ge0$ by Cor.\ \ref{0c9in34dh5sifzsshnkus}(b) and $\tD(\Neul)\ge2$ by Cor.\ \ref{p0c9ifn2o3w9dcpw0e}(c);
so $\max(0,\tD(X)) = \tD(X)$ and $\max(0,\tD(\Neul)) = \tD(\Neul)$.
We have $\xi(X) \le \tD( X ) - \delta^*(u_0) + 3$ by Lemma \ref{ovF8ym5jDdrfjkerdjkteowos}, so
\begin{equation}  \label {cjhbjyrake7ru2ofapt}
\tD( X ) \ge \xi(X) + \delta^*(u_0) - 3 \, .
\end{equation}
For each $w \in S(\Teul) \setminus X$ we have  $\Neul \neq \{ w \}$ (because $| S(\Teul) | > 1$)
and $w \notin \Omega(\Teul)$ (because  $\Omega(\Teul) \subseteq \{z\} \subseteq X$),
so Lemma \ref{Xc0ejrKlPhgx26xk45}\eqref{pv0ero0DRy8jruewoqwd3co} implies that 
\begin{equation}  \label {p0jb4wlfh7msjlwirefwe}
\tD(\bar V(w) ) \ge  \xi\big( \bar V(w) \big) + (\delta^*(w) - 2) \quad \text{for all $w \in S(\Teul) \setminus X$.}
\end{equation}
Since $\big( \bar V(w) \big)_{w \in S(\Teul) \setminus X}$ is an f-partition of $\Neul \setminus X$, we have
$$
\textstyle
\sum_{w \in S(\Teul)\setminus X} \tD(\bar V(w) ) = \tD(\Neul \setminus X) \text{\ \ and\ \ }
\sum_{w \in S(\Teul)\setminus X}  \xi\big( \bar V(w) \big) = \xi\big( \Neul \setminus X \big) \, ,
$$
so \eqref{p0jb4wlfh7msjlwirefwe} gives 
\begin{equation}  \label {903euswtghhdfajsueeuion8}
\textstyle
\tD(\Neul \setminus X) \ge  \xi\big( \Neul \setminus X \big) + \sum_{w \in S(\Teul) \setminus X} (\delta^*(w)-2) \, .
\end{equation}
In any tree, the sum of the quantity ``valency$-2$'' over all vertices is equal to $-2$;
since for each $w \in S(\Teul)$ the number $\delta^*(w)$ is the valency of $w$ in the tree $S(\Teul)$, we have
\begin{align*}
-2 &= \textstyle \sum_{w \in S(\Teul)} (\delta^*(w)-2)
= (\delta^*(z)-2) + (\delta^*(u_0)-2) + \sum_{w \in S(\Teul) \setminus X} (\delta^*(w)-2) \\
&= \textstyle (\delta^*(u_0)-3) + \sum_{w \in S(\Teul) \setminus X} (\delta^*(w)-2) \, ,
\end{align*}
so \eqref{903euswtghhdfajsueeuion8} is equivalent to 
\begin{equation}  \label {78gkjvJhd9498dfw437}
\tD(\Neul \setminus X) \ge   \xi\big( \Neul \setminus X \big) + 1 - \delta^*(u_0) \, .
\end{equation}
Adding \eqref{cjhbjyrake7ru2ofapt} and \eqref{78gkjvJhd9498dfw437} gives $\tD(\Neul) \ge  \xi(\Neul) - 2$,
which proves \eqref{9238ctg8se7p2e}.

(b) We already showed that (b) is true when $| \overline\Oeul | = 1$, so assume that $| \overline\Oeul | > 1$.
Assume that equality holds in \eqref{9238ctg8se7p2e}.
Then equality holds in \eqref{cjhbjyrake7ru2ofapt} and \eqref{78gkjvJhd9498dfw437},
so it holds in \eqref{p0jb4wlfh7msjlwirefwe} for all $w \in S(\Teul) \setminus X$.
By  Lemmas \ref{ovF8ym5jDdrfjkerdjkteowos} and \ref{Xc0ejrKlPhgx26xk45}, the following hold:
\begin{itemize}

\item $N_{ u_0 }>1$, 
$V( u_0 ) \subseteq \Nd^*(\Teul)$, 
each $x \in \bar V(u_0)$ is pure and satisfies $a_{ x } = 1$,
and each $x \in \Nd^*(\Teul) \cap X$ is pure 
and satisfies $a_x=1$.

\item For each $w \in S(\Teul) \setminus X$, 
$V( w ) \subseteq \Nd^*(\Teul)$ and each $x \in \bar V(w)$ is pure and satisfies $a_{ x } = 1$.

\end{itemize}
It follows that (b) is true.  As $\tD(\Neul)\ge2$, (c) holds trivially.  So the Lemma is proved.
\end{proof}

\begin{parag} \label  {c9b3498siBxkxcY5ykTwybsCte6eu}
\textbf{Proof of Thm \ref{p0fkwJjk8eCahbkBJhZuIK83gldf}.}
By Propositions \ref{Pp0drigro8awe7v8csbhsnehry}, \ref{PiIoO0v9w6ejamuisuergbf83e9f0} and \ref{hCfa89HFdfjkUY738ihhdJMJi854727},
Thm~\ref{p0fkwJjk8eCahbkBJhZuIK83gldf} is valid whenever $| S(\Teul) | = 1$ or $| \Omega(\Teul) | = 2$.
By Lemma \ref{76GgdT276d729gs27g882h}, it is valid when  $| S(\Teul) | \neq 1$ and $| \Omega(\Teul) | \neq 2$.
So Thm~\ref{p0fkwJjk8eCahbkBJhZuIK83gldf} is proved. \hfill \qedsymbol
\end{parag}

In fact we proved more than is stated in Thm~\ref{p0fkwJjk8eCahbkBJhZuIK83gldf}:
we also obtained information about the possible locations of the vertices of $\Nd^*(\Teul)$.
In the case $| S(\Teul) | = 1$, that information can be found in Propositions 
\ref{Pp0drigro8awe7v8csbhsnehry} and \ref{PiIoO0v9w6ejamuisuergbf83e9f0}.
In the case $| \Omega(\Teul) | = 2$, see Prop.~\ref{hCfa89HFdfjkUY738ihhdJMJi854727}.
For the case $| S(\Teul) | \neq 1$ and $| \Omega(\Teul) | \neq 2$, the information is spread over paragraph \ref{934btgodog9qkej};
we gather some of it in the following statement:

\begin{theorem} \label {cvintcjowHYtjk2uu78748}
Assume that $| S(\Teul) | \neq 1$ and $| \Omega(\Teul) | \neq 2$.
Choose $z \in \In(\Teul)$ and note that $\overline\Oeul  = \overline\Oeul(\Teul,z)$ is not empty.
Then the following hold.
\begin{enumerata}

\item \label {92g098golJXS4o934}  Each $x \in \Nd^*(\Teul) \setminus \{z\}$ satisfies exactly one of the following conditions:
\begin{enumerata}

\item $x = u_C$ for some $C \in \overline\Oeul$;

\item there exists $w \in W(\Teul)$ such that $x \in V(w)$;

\item $x \in S(\Teul) \setminus \setspec{ u_C }{ C \in \overline\Oeul }$ and $\epsilon(x)=2$.

\end{enumerata}
Moreover, in cases {\rm (ii)} and {\rm (iii)}, $x$ is a node of type $[1, N_x, \dots, N_x]$ and such that $a_x=1$.

\item \label {d9f8g3476238} We have $\dot c(C)=0$ for each $C \in \overline\Oeul$ whose minimal element
$(v_C,e_{v_C})$ satisfies $\Neul(v_C,e_{v_C}) \cap \Nd^*(\Teul) \neq \emptyset$.

\item \label {824c7sKLGF729o}  We have $\Nd^*(\Teul) \setminus \Neul(v,e_{v}) \subseteq \bigcup_{C \in \overline\Oeul} \big( \{u_C\} \cup V(u_C) \big)$
for each $v \in S(\Teul) \setminus \{z\}$ satisfying $\Neul(v,e_{v}) \cap \Nd^*(\Teul) \neq \emptyset$.

\item \label {8s745ki7f0w63rFrg} If $\xi(\Neul) = \max(0,\tD(\Neul)) + 2$ then 
$N_{u_0}>1$, each $x \in \Neul \setminus \Neul(u_0,e_{u_0})$ is pure and satisfies $a_x=1$,
each $x \in \Nd^*(\Teul)$ is pure and satisfies $a_x=1$, 
and for each $w \in S(\Teul) \setminus \Neul(u_0,e_{u_0})$ we have $V(w) \subseteq \Nd^*(\Teul)$.

\item \label {pvejcjkssyeuer23i} Assume that $\tD(\Neul)\le0$.
Then $| \overline \Oeul | = 1$ and $\Nd^*(\Teul) \subseteq \{u_0\} \cup V(u_0)$,
and if $\xi(\Neul)=2$ then $M(u_0,e_{u_0}) = 1$.

\end{enumerata}
\end{theorem}

\begin{proof}
\noindent \eqref{92g098golJXS4o934}   
Let $x \in \Nd^*(\Teul) \setminus \{z\}$. Then either $x \in \bar V(z) \setminus \{z\}$ or there exists  $C \in \overline\Oeul$
such that $x \in \bar V(u_C) \cup \big( \Neul(u_C,e_{u_C}) \setminus \Neul(v_C,e_{v_C}) \big)$,
where $(v_C,e_{v_C}) \preceq(u_C,e_{u_C})$ are respectively the least and greatest elements of $C$.
If $x \in \bar V(z) \setminus \{z\}$ then  (by Cor.~\ref{0vk3mWsklxd03mcwlfi}(a)) $x \in V(z)$, so $x$ satisfies (ii).
If $x \in \bar V(u_C)$ then  (by Cor.~\ref{0vk3mWsklxd03mcwlfi}(a)) either $x=u_C$ or $x \in V(u_C)$, so $x$ satisfies (i) or (ii).
If $x \in \Neul(u_C,e_{u_C}) \setminus \Neul(v_C,e_{v_C})$ then $x \in H_C$,
so we may consider the unique $i$ satisfying $(\alpha)$ or $(\beta)$ in Lemma \ref{88yHvcxhl93hf7}(b);
if $i$ satisfies $(\alpha)$ (resp.\ $(\beta)$) then $x$ satisfies (iii) (resp.\ (ii)).
So $x$ satisfies (exactly) one of (i), (ii), (iii).
There remains to show that if  $x$ satisfies (ii) or (iii) then $x$ is a node of type $[1, N_x, \dots, N_x]$ and such that $a_x=1$.
If $x$ satisfies (ii), then this follows from Cor.~\ref{0vk3mWsklxd03mcwlfi}(c).
If $x$ satisfies (iii), then (as we saw in the above argument) we have $x \in H_C$ for some $C \in \overline\Oeul$, so
the desired conclusion follows from  Lemma \ref{88yHvcxhl93hf7}(b).
So \eqref{92g098golJXS4o934} is proved.

\smallskip

\noindent \eqref{d9f8g3476238}
Let $C \in \overline\Oeul$ and suppose that the  minimal element
$(v_C,e_{v_C})$ of $C$ satisfies $\Neul(v_C,e_{v_C}) \cap \Nd^*(\Teul) \neq \emptyset$.
Then Lemma \ref{9hJTY8gJ73jr9Ij93fe} implies that $1 \ge c(v_C,e_C) = c(u_C, e_{u_C} )$,
so $\dot c(C)=  c(v_C,e_C) - c(u_C, e_{u_C} ) = 0$.

\smallskip

\noindent \eqref{824c7sKLGF729o}  
Let $v \in S(\Teul) \setminus \{z\}$ be such that $\Neul(v,e_{v}) \cap \Nd^*(\Teul) \neq \emptyset$
and consider an element $x$ of $\Nd^*(\Teul) \setminus \Neul(v,e_{v})$.
By contradiction, suppose that $x \notin \bigcup_{C \in \overline\Oeul} \big( \{u_C\} \cup V(u_C) \big)$.
Then $x \notin \bigcup_{C \in \overline\Oeul} \big( \bar V(u_C) \big)$ by Cor.~\ref{0vk3mWsklxd03mcwlfi},
and $x \notin \bar V(z)$ because $\bar V(z) \subseteq\Neul(v,e_{v})$,
so there exists $C \in \overline\Oeul$  such that $x \in \Neul(u_C,e_{u_C}) \setminus \Neul(v_C,e_{v_C})$,
where $(v_C,e_{v_C}) \preceq(u_C,e_{u_C})$ are respectively the least and greatest elements of $C$.
Then $x \in H_C$ and we may consider the unique $i$ satisfying $(\alpha)$ or $(\beta)$ in Lemma \ref{88yHvcxhl93hf7}(b).
Since $x \notin \Neul(v,e_v)$, it follows that $v_i\notin \Neul(v,e_v)$,
so $\Neul(v,e_v) \subseteq \Neul(v_i,e_{v_i})$, so $\Neul(v_i,e_{v_i}) \cap \Nd^*(\Teul) \neq \emptyset$
and consequently $c(v_i,e_{v_i}) \le 1$ by Lemma \ref{9hJTY8gJ73jr9Ij93fe}.
This contradicts the fact that $i$ satisfies condition $(\gamma)$ of Lemma \ref{88yHvcxhl93hf7}(b).
We showed that  $x \in \bigcup_{C \in \overline\Oeul} \big( \{u_C\} \cup V(u_C) \big)$,
which proves~\eqref{824c7sKLGF729o}.

\smallskip

Assertions \eqref{8s745ki7f0w63rFrg} and \eqref{pvejcjkssyeuer23i} follow from parts (b) and (c) of Lemma \ref{76GgdT276d729gs27g882h}. 
\end{proof}

\section{Rational trees}
\label {Section:RationalCase}

\begin{definition} \label {ratttreeee}
By a \textit{rational tree}, we mean an abstract Newton tree at infinity which is minimally complete and satisfies
$\tD( \Neul ) = 0$ and $\gcd\setspec{ d_x }{ x \in \Deul } = 1$.
\end{definition}

The class of rational trees is important because (as explained in the Introduction)
if $f \in \Comp[x,y]$ is a rational polynomial then $\Teul(f;x,y)$ is a rational tree.

\medskip

Recall that if $\Teul$ is any minimally complete abstract Newton tree at infinity then $\In(\Teul) \neq \emptyset$,
so it is always possible to choose $z \in \In(\Teul)$ and to consider the set $\overline\Oeul(\Teul,z)$ determined by
that choice. This is true in particular if $\Teul$ is a rational tree.

\begin{corollary} \label {0d9fpv09ffi209we}
Let $\Teul$ be a rational tree, choose $z \in \In(\Teul)$ and consider $\overline\Oeul = \overline\Oeul(\Teul,z)$.
\begin{enumerata}

\item 
$\overline\Oeul = \emptyset$
$\Leftrightarrow$
$| S(\Teul) | = 1$
$\Leftrightarrow$
$\Neul = \{ v_0 \}$
$\Leftrightarrow$
$\Omega(\Teul) = \emptyset$
$\Leftrightarrow$
$\Teul$ is one of the trees of Ex.\ \ref{nv63jfy64nvy3}.

\item If $\overline\Oeul \neq \emptyset$ then $| \overline\Oeul | = 1$, $\delta^*(u_0)=1$ and 
(see \mbox{\rm \ref{c20w39w93e9d0dqMne89wcg9}(d)} for $u_0$ and \ref{p09cv347dYF6Us98} for $a_{u_0}^*$):
$$
| \setspec{ x \in \Deul_{u_0} }{ k_x>1 } | + a_{u_0}^* + t({u_0}) \le 3 .
$$

\end{enumerata}
\end{corollary}

\begin{proof}
If $| \overline\Oeul | > 1$ then  Cor.\ \ref{p0c9ifn2o3w9dcpw0e}(c) gives 
$\tD(\Neul) \ge 2$, which is not the case. 
Thus $| \overline\Oeul | \le 1$.

(a) We have $S(\Teul) \neq \emptyset$ by Lemma \ref{p309ef2GFT485ryf}. 
It is noted in Def.\ \ref{c20w39w93e9d0dqMne89wcg9}(a) that
$\overline\Oeul = \emptyset$ $\Leftrightarrow$ $| S(\Teul) | \le 1$ $\Leftrightarrow$ $| S(\Teul) | = 1$.
Lemma  \ref{p309ef2GFT485ryf} states that 
$| S(\Teul) | = 1$ $\Leftrightarrow$ $| \Neul | = 1$ or $\Teul$ is a brush,
but $\Teul$ cannot be a brush by Cor.\ \ref{9vbrtyukey6idckFfFugus},
so $| S(\Teul) | = 1$ $\Leftrightarrow$ $| \Neul | = 1$.
As $v_0 \in \Neul$ by Rem.\ \ref{p092pof90e9rf},  $| \Neul | = 1$ $\Leftrightarrow$ $\Neul = \{ v_0 \}$.

If  $\Neul \neq \{v_0\}$ then $|\Neul|>1$ (and $| \overline\Oeul | \ge 0 = \tD(\Neul)$),
so Cor.\ \ref{pTo9v2q3hZYa9r1ge0cX3rg} implies that $\Omega(\Teul) \neq \emptyset$.
Conversely, if  $\Omega(\Teul) \neq \emptyset$ then $Z(\Teul) \neq \emptyset$, so $\Neul \neq \{v_0\}$.
Hence, $\Neul = \{ v_0 \}$ $\Leftrightarrow$ $\Omega(\Teul) = \emptyset$.

If $| S(\Teul) | = 1$ then Prop.\ \ref{c0viwjytsdDJjLlxdifFebg} implies that $\Teul$ is one of the trees of Ex.\ \ref{nv63jfy64nvy3};
conversely, each  tree of Ex.\ \ref{nv63jfy64nvy3} is a rational tree with $\Neul = \{v_0\}$, so (a) is proved.
 
(b) Assume that $\overline\Oeul \neq \emptyset$.
Then $| \overline\Oeul | = 1$,  so $\delta^*(u_0)=1$.
Cor.~\ref{Xxkvp0wedifcwZepd} gives 
$$
| \setspec{ x \in \Deul_{u_0} }{ k_x>1 } | + a_{u_0}^* + \epsilon({u_0})-1 \le 3 ;
$$
as $\epsilon(u_0)-1 = (\delta^*(u_0)-1) + t(u_0) = t(u_0)$, the last claim follows.
\end{proof}

\begin{example} \label {09cq13cgrl38n9ZAyey}
Continuation of Fig.\ \ref{723dfvcjp2q98ewdywe} and Ex.\ \ref{Pc09vbq3j7gabXrZiA}, \ref{cIuCbgepwej6DsJ37655enm},
\ref{cov7btiw67d78vujew9}, \ref{oOo8cvn239vn3p98fgIW}  and \ref{FFlAkjcvwiuer2n3Z3cxs7df9}.
Recall from Ex.\ \ref{cIuCbgepwej6DsJ37655enm} that $\tD(\Neul) = 0$. 
Considering the types of the nodes (given in Ex.\ \ref{Pc09vbq3j7gabXrZiA}), we see that $\gcd\setspec{ d_x }{ x \in \Deul } = 1$,
so $\Teul$ is a rational tree.
We have $S(\Teul) = \{ v_0, v_1, v_2, v_3, v_4 \}$, $\In(\Teul) = \Omega(\Teul) = \{v_0\}$,
$\overline\Oeul = \{ C_0 \}$,
$C_0 = \{ (v_1, e_{v_1}),  (v_2, e_{v_2}),  (v_3, e_{v_3}),  (v_4, e_{v_4}) \}$
where $e_{v_i} = \{v_i, v_{i-1}\}$,
and $(v_4, e_{v_4})$ is a comb over $(v_1, e_{v_1})$.
Note that $u_0 = v_4$  and that $| \setspec{ x \in \Deul_{u_0} }{ k_x>1 } | = 0$, $a_{u_0}^* = 1$ and $t(u_0)=1$
(compare with Cor.\ \ref{0d9fpv09ffi209we}(b)).
\end{example}

See Def.\ \ref{iv097383k467k4V498ufhr7} for $\xi$ and $\Nd^*(\Teul)$.
Recall that $| \Omega(\Teul) |  \in \{0,1,2\}$ by Thm \ref{Xcoikn23ifcdKJDluFYT937}.

\begin{corollary} \label {9vbq34ritgf7ZAOdb9rv}
If $\Teul$ is a rational tree then $| \Nd^*(\Teul) | \le \xi(\Neul) \le 2$ and the following hold.
\begin{enumerata}

\item \label {o2i3brfowe9-c}  If $\Omega(\Teul) = \emptyset$ then $\Teul$ is one of the trees of Ex.\ \ref{nv63jfy64nvy3}.

\item \label {o2i3brfowe9-b} Assume that $\Omega(\Teul)$ is a singleton $\{z\}$.
Then $\In(\Teul) =  \{z\}$ and we may consider $\overline\Oeul(\Teul,z) = \{ C_0 \}$ and the greatest element  $(u_0,e_{u_0})$ of $C_0$.
Then  $\Nd^*(\Teul) \subseteq \{ u_0 \} \cup V( u_0 )$ and each $x \in \Nd^*(\Teul) \cap V( u_0 )$ has type  $[1, N_x, \dots, N_x]$ and satisfies $a_x=1$.
If $\xi(\Neul) =  2$ then $N_{u_0}>1$, $M(u_0,e_{u_0}) = 1$, $V(u_0) \subseteq \Nd^*(\Teul)$, and each $x \in \bar V(u_0)$ is pure and satisfies $a_x=1$.

\item \label {o2i3brfowe9-a}  Assume that $| \Omega(\Teul) | = 2$.
Then $\Nd^*(\Teul) \subseteq \Omega(\Teul)$, $| \Nd^*(\Teul) |  = \xi(\Neul)  \le 2$, and each $x \in \Nd^*(\Teul)$ has type $[1,N_x,\dots,N_x]$
and satisfies $a_x=1 = \xi(x)$ and $\tD(x)=0$.
If $\xi(\Neul) =2$ then each $x \in \Neul$ is pure and satisfies $a_x=1$ and $\tD(x)=0$.

\end{enumerata}
\end{corollary}

\begin{proof}
We have $| \Nd^*(\Teul) | \le \xi(\Neul) \le 2$ by Thm \ref{p0fkwJjk8eCahbkBJhZuIK83gldf},
assertion \eqref{o2i3brfowe9-a}  follows from Prop.\ \ref{hCfa89HFdfjkUY738ihhdJMJi854727},
and assertion \eqref{o2i3brfowe9-c}  from Cor.\ \ref{0d9fpv09ffi209we}(a).
We prove \eqref{o2i3brfowe9-b}. Assume that $\Omega(\Teul)$ is a singleton $\{z\}$.
Then $\In(\Teul) =  \{z\}$ and we may consider $\overline\Oeul = \overline\Oeul(\Teul,z)$.
Since $\Omega(\Teul)  \neq \emptyset$, we have $| \overline\Oeul | = 1$ by  Cor.\ \ref{0d9fpv09ffi209we}, so $\overline\Oeul = \{ C_0 \}$.
By Cor.\ \ref{0vk3mWsklxd03mcwlfi}(c), each $x \in \Nd^*(\Teul) \cap V( u_0 )$ has type  $[1, N_x, \dots, N_x]$ and satisfies $a_x=1$.
The other claims in \eqref{o2i3brfowe9-b}  follow from  parts (e) and (d) of Thm \ref{cvintcjowHYtjk2uu78748}.
\end{proof}

\begin{remark}
Cor.\ \ref{9vbq34ritgf7ZAOdb9rv} has the following consequence (refer to the introduction for the definitions):
{\it If $\Teul = \Teul(f;x,y)$ where $f \in \Comp[x,y]$ is a simple (resp.\ quasi-simple) rational polynomial,
then $\Teul$ has at most $2$ (resp.\ at most $3$) nodes.}
One can check that this claim agrees with the explicit descriptions given in \cite{CND15:simples} and \cite{Sasao_QuasiSimple2006}.
\end{remark}

The following fact is useful in conjunction with Thm \ref{c03hbr9dfgsjlxnzwesmchp}\eqref{ve4876ekrlu8ecBbue}:

\begin{lemma}  \label {0c9gh5thowe6fric}
Let $\Teul$ be a rational tree and suppose that $u \in \Neul$ is such that $(u,e)$ is nonpositive for all $e \in \Eeul_u$.
Let $(d_1, \dots, d_m)$ be the tuple obtained by first concatenating the three families
$$
(d_x)_{x \in \Deul_u}, \quad (c(u,e))_{e \in \Eeul_u} \quad \text{and} \quad ({N_u}/{a_u})
$$
(where the last family has only one term)
and then reordering the terms so that $d_1 \le \cdots \le d_m$.
Then all $d_i$ are positive integers and divisors of $N_u$. Moreover, one of the following holds:
\begin{enumerata}

\item $N_u=1$ and $(d_1, \dots, d_m) = (1, \dots, 1)$;

\item $m\ge2$, $N_u\ge2$ and $(d_1, \dots, d_m) = (1, 1, N_u, \dots, N_u)$;
\item $m\ge3$ and there exists an odd integer $b \ge3$ such that  $N_u=2b$ and $(d_1, \dots, d_m) = (2, b, b, N_u, \dots, N_u)$.

\end{enumerata}
\end{lemma}

The proof of Lemma  \ref{0c9gh5thowe6fric} uses the following fact, whose verification is left to the reader:

\begin{lemma}  \label {c9i2bfxcgrmOxdkgcksi6cf3}
Let $t_1, t_2, t_3, N$ be positive integers satisfying 
$$
\text{$t_i \mid N$ for all $i=1,2,3$}, \quad t_1 + t_2 + t_3 = N+2, \quad t_1 \le t_2 \le t_3 \quad \text{and} \quad \gcd(t_1, t_2, t_3) = 1 .
$$
Then either $(t_1, t_2, t_3) = (1,1,N)$ or there exists an odd integer $b \ge 3$ such that 
$(t_1, t_2, t_3) = (2, b, b)$ and $N = 2b$.
\end{lemma}

\begin{proof}[Proof of \ref{0c9gh5thowe6fric}]
By Def.\ \ref{d23ed9wei9d23}, $N_u/a_u$ is a positive integer (and a divisor of $N_u$).
If $x \in \Deul_u$ then $d_x$ is a positive integer and a divisor of $N_u = k_xd_x$ (Lemma \ref{90q932r8dhd89cnr9}).
If $e \in \Eeul_u$ then $(u,e)$ is nonpositive, so  $c(u,e)$ is a positive integer by Rem \ref{uyhmdytjwhwhrkdftraef4gf23h23},
and $c(u,e) \mid N_u$ by Thm \ref{xncoo9qwdx9}.
So all $d_i$ are positive integers and divisors of $N_u$.
The integer $\gcd(d_1, \dots, d_m)$ is a divisor of $d(u)$ and also a divisor of $c(u,e)$ for each $e \in \Eeul_u$;
since (by Lemma \ref{kjwoeid9cse9c}) $c(u,e)$ divides $\gcd\setspec{d(v)}{ v \in \Nd(u,e) }$ whenever $\Nd(u,e) \neq \emptyset$, it follows that 
$\gcd(d_1, \dots, d_m)$ divides  $\gcd\setspec{d(v)}{ v \in \Nd(\Teul) } = \gcd\setspec{d_x}{ x \in \Deul } = 1$, i.e.,
$$
\gcd(d_1, \dots, d_m) = 1 \, .
$$ 

Clearly, if $N_u=1$ then (a) is satisfied. From now-on, assume that $N_u>1$.
We have
$0 = \tD(\Neul) = \big( R(u, \Eeul_u) - 2 \big) N_u  + 2 +  \sum_{e \in \Eeul_u} \eta(u,e) = \big( R(u, \Eeul_u) - 2 \big) N_u  + 2 $
by Cor.\ \ref{dox9f80293ewf0di} and \ref{DKxcnpw93sdo}, 
and the definitions of $R(u, \Eeul_u)$ and $(d_1,\dots,d_m)$ give
$R(u, \Eeul_u) N_u = \sum_{ i = 1 }^m (N_u - d_i) = \sum_{ i = 1 }^{m_0} (N_u - d_i)$ where we define $m_0$ by 
$$
\text{$d_i  < N_u$ for all $i \le m_0$} \quad \text{and} \quad \text{$d_i  = N_u$ for all $i > m_0$}.
$$
It follows that $\sum_{ i = 1 }^{m_0} (N_u - d_i) = 2(N_u-1)$ where $N_u - d_i \ge \frac{N_u}2$ for each $i = 1, \dots, m_0$. 
So we have $m_0 \in \{2,3\}$. If $m_0=2$ then $d_1=1=d_2$ and hence (b) is satisfied.
If $m_0=3$ then $d_1 + d_2 + d_3 = N_u+2$ and $\gcd(d_1,d_2,d_3) = \gcd(d_1, \dots, d_m)=1$,
so Lemma \ref{c9i2bfxcgrmOxdkgcksi6cf3} implies that (c) is satisfied.
\end{proof}

The next result uses the notations of Fig.\ \ref{902br8hrjkderydyukhyg} in the Introduction.
Note that assertion \eqref{1v0Aeri87llrfdOfitrkj} is obvious
and that \eqref{9f2i39KJG96fgIi8k} (and also \eqref{0vdjo7eHidf782eh39} if it is properly stated) is valid without assuming that $\Teul$ is a rational tree.
We include (\ref{0vdjo7eHidf782eh39}--\ref{1v0Aeri87llrfdOfitrkj}) in the statement
in order to make it clear that Cor.~\ref{0v3y46y7iwoujbfv8e} (in the Introduction) follows from
Thm \ref{c03hbr9dfgsjlxnzwesmchp}.

\begin{theorem} \label {c03hbr9dfgsjlxnzwesmchp}
Let $\Teul$ be a rational tree.
Then either $\Teul$ is one of the trees of Ex.\ \ref{nv63jfy64nvy3} or the following hold.
\begin{enumerata}

\item \label {iHhiHhd8i934ur}
$\Neul$ looks like the tree of Fig.\ \ref{902br8hrjkderydyukhyg} with $n>1$.
In that picture,
the elements of $\Gamma(\Teul)$ are the vertical or oblique branches,\footnote{The elements of $\Gamma(\Teul)$ are
incorrectly but conveniently called ``teeth'' in the Introduction.
This abuse of language is harmless, because of the bijection mentioned just after Def.~\ref{xkclqlwpd90cod}.}
$W(\Teul) \setminus \{z_n\} = \{ z_{i_1}, \dots, z_{i_k} \}$ and $S(\Teul) = \{ z_1, \dots, z_n \}$.

\item \label {pd90fb35t93repfEUTH}
We have $z_1 \notin W(\Teul)$, so if $k\neq0$ then $i_1>1$. Moreover, $z_1$ satisfies one of:
\begin{itemize}

\item $z_1 \notin \Nd(\Teul)$, $z_1 = v_0$ and $\delta_{v_0}=1$.

\item $z_1$ is a node of type $[d, N_{z_1}, \dots, N_{z_1}]$ for some divisor $d$ of $N_{z_1}$,
and if $d \neq N_{z_1}$ then $a_{z_1}=1$.

\end{itemize}

\item \label {c8bGTghU783hceo}
We have $v_0 \in \{ z_1, \dots, z_{n-1} \}$, say $v_0 = z_{i_0}$ with $1 \le i_0 < n$,
and if $k\neq 0$ then $i_0 < i_1$.
Moreover, $\delta_{v_0} \le 2$.

\item \label {ve4876ekrlu8ecBbue}
The hypothesis of Lemma \ref{0c9gh5thowe6fric} is satisfied with $u=z_n$ (so the conclusion of the Lemma is also satisfied).
In particular, $t(z_n) \in \{0,1,2,3\}$. 

\item \label {0vdjo7eHidf782eh39} Let $i$ be such that $1<i<n$.
Then $\epsilon(z_i) \in \{2,3\}$,  $R(z_i, \{ e_{z_i} \} ) =0$ if $\epsilon(z_i)=3$ and $R(z_i, \{ e_{z_i} \} ) <1$ if $\epsilon(z_i)=2$.
In particular, we have $\epsilon'(z_i) \le 3$ where $\epsilon'(z_i)$ is defined in the Introduction.

\item \label {9f2i39KJG96fgIi8k} Consider an element $(x_1, \dots, x_m)$ of $\Gamma(\Teul)$.
Then, for each $j<m$, $x_j$ is a node and $\epsilon'(x_j) \le 2$.

\item \label {1v0Aeri87llrfdOfitrkj}  The set $\Nd^*(\Teul)$ and the function $\xi$ satisfy the conclusions of Cor.~\ref{9vbq34ritgf7ZAOdb9rv}.

\end{enumerata}
\end{theorem}

\begin{proof}
Let $\Teul$ be a rational tree which is not one of the trees of Ex.\ \ref{nv63jfy64nvy3}.
We have $\delta_{v_0} \le 2$ by Prop.~\ref{0c9b45okjrdidjfo}.
We noted in Def.~\ref{cpan3bAfon98gbvqrkdLlxdir6} that $\In(\Teul) \neq \emptyset$;
choose $z \in \In(\Teul)$ and consider $\overline\Oeul = \overline\Oeul(\Teul,z)$.
If $\overline\Oeul = \emptyset$ then  Cor.~\ref{0d9fpv09ffi209we} implies that $\Teul$ is one of the trees of Ex.\ \ref{nv63jfy64nvy3},
which is not the case by assumption. So $\overline\Oeul \neq \emptyset$.
Then  $|\overline\Oeul| = 1$, $| S(\Teul) | > 1$ and $\Omega(\Teul) \neq \emptyset$ by Cor.~\ref{0d9fpv09ffi209we},
so $| \Omega(\Teul) | \in \{1,2\}$ by Thm \ref{Xcoikn23ifcdKJDluFYT937}.
Also note that $\In(\Teul) = \Omega(\Teul)$ by definition of $\In(\Teul)$.

Consider the case where $| \Omega(\Teul) | =2$.
By Thm~\ref{Xcoikn23ifcdKJDluFYT937}, there exists a path $(z_1,\dots,z_n)$
such that  $n > 1$, $\Neul = \{ z_1,\dots,z_n \}$, $\Omega(\Teul) = \{z_1,z_n\}$ and $\tD(z_i)=0$ for all $i = 1, \dots, n$. 
Since each $x \in W(\Teul)$ satisfies $\tD(x)>0$, we must have $W(\Teul) = \emptyset$ and hence $\Gamma(\Teul) = \emptyset$.
Replacing if necessary  $(z_1,\dots,z_n)$ by $(z_n,\dots,z_1)$, we may arrange that $v_0 \neq z_n$.
With this notation, it is clear that all claims of the Theorem are true.

From now-on, assume that $| \Omega(\Teul) | =1$.
As $z \in \In(\Teul) = \Omega(\Teul)$, we have  $\Omega(\Teul) = \{z\}$.
Since $\overline\Oeul$ is a partition of $\Oeul = \setspec{ (u,e_u) }{ u \in S(\Teul) \setminus \{z\} }$,
the unique element $C_0$ of $\overline\Oeul$
satisfies $C_0  = \setspec{ (u,e_u) }{ u \in S(\Teul) \setminus \{z\} }$.
Let $(u_{0}, e_{u_0})$ be the greatest element of $C_0$ and let 
$(z_1, \dots, z_n)$ denote the path from  $z_1=z$ to $z_n = u_{0}$.
Then $S(\Teul) = \{ z_1, \dots, z_n\}$ where  $n = | S(\Teul) | > 1$.
Define $i_1 < \cdots < i_k$ by $\{ z_{i_1}, \dots, z_{i_k} \} = W(\Teul) \setminus \{z_n\}$, then $\Neul$ looks exactly as in Fig.~\ref{902br8hrjkderydyukhyg}
and assertion \eqref{iHhiHhd8i934ur} is true.
Since $z_1 = z \in \Omega(\Teul)$, we have $\tD(z_1)\le 0$ and hence $z_1 \notin W(\Teul)$; so the first part of \eqref{pd90fb35t93repfEUTH} is true;
the second part of \eqref{pd90fb35t93repfEUTH} follows from Lemmas \ref{C0qowbdowX932epug7fnvd} and \ref{8ey8cdody27}.
We have $v_0 \in \{z_1, \dots, z_n\}$ by Lemma \ref{p309ef2GFT485ryf},
and Thm \ref{Xcoikn23ifcdKJDluFYT937} implies that there exists a $\tD$-trivial path $(x_1,\dots,x_h)$ satisfying $x_1 = z_1$ and  $x_{h-1}<x_h$;
so \eqref{c8bGTghU783hceo} is true.

By Lemma \ref{A0c9vub23oW8fgbgp0q2arft}, $(z_n, e_{z_n}) = (u_0, e_{u_0})$ is nonpositive.
If $e \in \Eeul_{z_n} \setminus \{e_{z_n}\}$ then $(z_n,e)$ is a tooth and hence is nonpositive.
So $z_n$ satisfies the hypothesis of Lemma \ref{0c9gh5thowe6fric}.
Note that if $e \in \Eeul_{z_n}$ is such that $(z_n,e)$ is a tooth  then $M(z_n,e)>1$ by Lemma \ref{GRygergGREg8948r}, so $c(z_n,e)<N_{z_n}$,
which means that $e$ produces a term $d_i<N_{z_n}$ in the tuple $(d_1, \dots, d_m)$ of Lemma  \ref{0c9gh5thowe6fric}.
Since at most three $i$ satisfy $d_i<N_{z_n}$, we have $t(z_n) \le 3$, so \eqref{ve4876ekrlu8ecBbue} is true.

If $i$ is such that $1<i<n$
then $(z_n, e_{z_n}) \succ (z_i, e_{z_i})$ and $(z_n, e_{z_n})$ is a comb over $(z_i, e_{z_i})$,
so assertion \eqref{0vdjo7eHidf782eh39} follows from \ref {c09Nc23owsfcnp2q0wuXwoh} and \ref{pv09w4v7AqOmksndfwoa}.

In assertion \eqref{9f2i39KJG96fgIi8k}, $(x_1, \dots, x_m)$ is a $\tD$-trivial path such that $\tD(x_1)\le0$
and $v_0 \notin \{x_1, \dots, x_{m-1}\}$.
So the claim follows from Lemmas \ref{C0qowbdowX932epug7fnvd} and \ref{8ey8cdody27}.

Assertion \eqref{1v0Aeri87llrfdOfitrkj} is clear.
\end{proof}

Corollary 2, in the Introduction, is a consequence of the next result.
This Proposition describes a class of trees that is considerably larger than the set of $\Teul(f;x,y)$ with
$f \in \Comp[x,y]$ a simple rational polynomial.

\begin{proposition}  \label {9988zhgGUjfdlkq4q965vsaqWd39r0}
Let $\Teul$ be a rational tree such that $\Nd(\Teul) = \Nd^*(\Teul)$.
Then the following hold.
\begin{enumerata}

\item If $\Omega(\Teul) = \emptyset$ then $\Teul$ is one of the trees of Ex.\ \ref{nv63jfy64nvy3}.

\item If $| \Omega(\Teul) | = 1$ then $\Neul$ is one of the  following:
$$
\scalebox{1}{
\setlength{\unitlength}{1mm}
\text{\rm(i)\ \ }\raisebox{-13\unitlength}{\begin{picture}(36,17)(69,-14)
\put(70,0){\circle{1}}
\put(80,0){\circle{1}}
\put(100,0){\circle{1}}
\put(80.5,0){\line(1,0){6}}
\put(99.5,0){\line(-1,0){6}}
\put(90,0){\makebox(0,0){\dots}}
\put(70.5,0){\line(1,0){9}}
\put(70,1.5){\makebox(0,0)[b]{\tiny $z_{1}$}}
\put(80,1.5){\makebox(0,0)[b]{\tiny $z_{2}$}}
\put(100,1.5){\makebox(0,0)[b]{\tiny $z_n$}}
\end{picture}} \qquad
\text{\rm(ii)\ \ }\raisebox{-13\unitlength}{\begin{picture}(36,17)(69,-14)
\put(70,0){\circle{1}}
\put(80,0){\circle{1}}
\put(100,0){\circle{1}}
\put(80.5,0){\line(1,0){6}}
\put(99.5,0){\line(-1,0){6}}
\put(90,0){\makebox(0,0){\dots}}
\put(70.5,0){\line(1,0){9}}
\put(99.8787,-.4851){\line(-1,-4){2.25}}
\put(97.5,-10){\circle{1}} \put(97.5,-11.5){\makebox(0,0)[t]{\tiny $y_{1}$}}
\put(70,1.5){\makebox(0,0)[b]{\tiny $z_{1}$}}
\put(80,1.5){\makebox(0,0)[b]{\tiny $z_{2}$}}
\put(100,1.5){\makebox(0,0)[b]{\tiny $z_n$}}
\end{picture}} \qquad
\text{\rm(iii)\ \ }\raisebox{-13\unitlength}{\begin{picture}(36,17)(69,-14)
\put(70,0){\circle{1}}
\put(80,0){\circle{1}}
\put(100,0){\circle{1}}
\put(80.5,0){\line(1,0){6}}
\put(99.5,0){\line(-1,0){6}}
\put(90,0){\makebox(0,0){\dots}}
\put(70.5,0){\line(1,0){9}}
\put(99.8787,-.4851){\line(-1,-4){2.25}}
\put(100.1213,-.4851){\line(1,-4){2.25}}
\put(97.5,-10){\circle{1}} \put(97.5,-11.5){\makebox(0,0)[t]{\tiny $y_{1}$}}
\put(102.5,-10){\circle{1}} \put(102.5,-11.5){\makebox(0,0)[t]{\tiny $y_{2}$}}
\put(70,1.5){\makebox(0,0)[b]{\tiny $z_{1}$}}
\put(80,1.5){\makebox(0,0)[b]{\tiny $z_{2}$}}
\put(100,1.5){\makebox(0,0)[b]{\tiny $z_n$}}
\end{picture}}
}
$$
In the three cases we have $n>1$, $\Omega(\Teul) = \{z_1\} = \{v_0\}$, $\delta_{v_0}=1$ and $S(\Teul)=\{z_1,\dots,z_n\}$. Moreover,
\begin{itemize}

\item in case \text{\rm (i)} we have $\Nd(\Teul) = \{z_n\}$ and the type of $z_n$ is either $[1]$  or $[1,1]$;

\item in case \text{\rm (ii)} we have $V(z_n)=\{y_1\}$ and either $\Nd(\Teul) = \{y_1\}$ or $\Nd(\Teul) = \{z_n,y_1\}$;
moreover, the type of $y_1$ is  $[1, N_{y_1}, \dots, N_{y_1} ]$ and if $z_n \in \Nd(\Teul)$ then its type is $[1]$;

\item in case \text{\rm (iii)} we have $V(z_n)=\Nd(\Teul) = \{y_1,y_2\}$ and (for each $i=1,2$) the type of $y_i$ is $[1, N_{y_i}, \dots, N_{y_i} ]$. 

\end{itemize}

\item If  $| \Omega(\Teul) | = 2$ then $\Neul$ satisfies all of the following conditions:
\begin{itemize}

\item There exists a $\tD$-trivial path $(z_1,\dots,z_n)$ such that $\Neul = \{ z_1, \dots, z_n\}$;

\item $n \in \{2,3\}$ and if $n=3$ then $v_0=z_2$;

\item $a_{z_i}=1$ for all $i = 1, \dots, n$;

\item $\Nd(\Teul) = \{ z_1, z_n \}$ and each node $v$ has type $[d_1, \dots, d_s] = [1, N_v, \dots, N_v]$, where $s=1 \Leftrightarrow v = v_0$.

\end{itemize}
\end{enumerata}
\end{proposition}

\begin{proof}
Assertion (a) reiterates  Cor.~\ref{9vbq34ritgf7ZAOdb9rv}\eqref{o2i3brfowe9-c}.
To prove (b), assume that $\Omega(\Teul)$ is a singleton $\{z\}$.
Then (Cor.~\ref{9vbq34ritgf7ZAOdb9rv}\eqref{o2i3brfowe9-b})
$\In(\Teul) = \{z\}$ and the set $\overline\Oeul = \overline\Oeul(\Teul,z)$ has exactly one element $C_0$.
Let $(u_0,e_{u_0})$ be the greatest elements of $C_0$ and $(z_1,\dots,z_n)$ the path from $z$ to $u_0$ (so $z_1=z$ and $z_n=u_0$).
Since $(u_0,e_{u_0}) \in \Oeul(\Teul,z)$, we have $u_0 \neq z$ by \ref{c20w39w93e9d0dqMne89wcg9}(a), so $n>1$.
Since $| \overline\Oeul | = 1$, $S(\Teul) = \{ z_1, \dots, z_n \}$.
Since $\Nd^*(\Teul) \subseteq \{u_0\} \cup V(u_0)$ by  Cor.~\ref{9vbq34ritgf7ZAOdb9rv}\eqref{o2i3brfowe9-b}
and $\Nd(\Teul) = \Nd^*(\Teul)$ by assumption,
\begin{equation} \label {p0c94e8c092bcf3267jie}
\Nd(\Teul) \subseteq \{u_0\} \cup V(u_0) = \{z_n\} \cup V(z_n) \, .
\end{equation}
If $w \in W(\Teul)$ and $(x_1,\dots,x_m)$ is the path from some element of $V(w)$ to $w$ then $x_m<x_{m-1}$
by definition of $V(w)$, so $v_0 \notin \{x_1,\dots,x_{m-1}\}$,
so Lemma \ref{C0qowbdowX932epug7fnvd} implies that $\{x_1,\dots,x_{m-1}\} \subseteq \Nd(\Teul)$;
by \eqref{p0c94e8c092bcf3267jie}, it follows that $w = z_n$ and $m=2$, so $W(\Teul) \subseteq \{z_n\}$ and: 
\begin{equation}
\text{$V(z_n) \subseteq \Nd(\Teul)$ and each element of $V(z_n)$ is adjacent to $z_n$.}
\end{equation}
Since $W(\Teul) \cap \{z_1,\dots,z_{n-1} \} = \emptyset$ and $(u_0,e_{u_0})$ is a comb over $(z_2,e_{z_2})$,
we obtain $\epsilon(z_i) = 2$ for all $i$ such that $1<i<n$.
We have $\epsilon(z_1)=1$ (because $z_1 \in \Omega(\Teul)$) and  $z_1 \notin \Nd(\Teul)$ (by \eqref{p0c94e8c092bcf3267jie}), 
so Lemma \ref{C0qowbdowX932epug7fnvd}(a) gives
$$
z_1=v_0 \text{\ \ and\ \ }  \delta_{v_0}=1 \, .
$$
We have $| \Nd(\Teul) | \le \xi(\Neul) \le 2$ by Cor.~\ref{9vbq34ritgf7ZAOdb9rv}, so $| V(u_0) | \in \{0,1,2\}$.
This gives the three pictures (i--iii) for $\Neul$, and the following are clear:
\begin{itemize}

\item in case \text{\rm (i)}: $\Nd(\Teul) = \{z_n\}$;  

\item in case \text{\rm (ii)}:  $V(z_n)=\{y_1\}$ and either $\Nd(\Teul) = \{y_1\}$ or $\Nd(\Teul) = \{z_n,y_1\}$;

\item in case \text{\rm (iii)}: $V(z_n)=\Nd(\Teul) = \{y_1,y_2\}$ 

\end{itemize}
It is also clear (by Cor.~\ref{9vbq34ritgf7ZAOdb9rv}\eqref{o2i3brfowe9-b}) that $y_i$ is of type $[1, N_{y_i}, \dots, N_{y_i}]$ in cases (ii) and (iii).
To complete the proof of (b), there only remains to prove the following statement:
\begin{equation} \label {pc092g398gdbcA023fdqpi}
\text{If $z_n$ is a node then its type is $[1]$ or $[1,1]$, and in case (ii) it cannot be $[1,1]$.}
\end{equation}

By contradiction, assume that $z_n$ is a node of type $[1,1]$ in case (ii).
We have $z_n<y_1$ and $N_{y_1}>0$, so $N_{z_n}>1$ by Lemma \ref{8123oiuwhqd8d2} and consequently $\xi(z_n) = 2$,
so $\xi(\Neul) = \xi(z_n) + \xi(y_1) = 2 + 1 > 2$, contradicting Cor.~\ref{9vbq34ritgf7ZAOdb9rv}.
This contradiction proves the second part of assertion \eqref{pc092g398gdbcA023fdqpi}.

Assume that $z_n$ is a node of type $\tau=[d_1,\dots,d_s]$. The proof of (b) will be complete if we can
show that $\tau=[1]$ or $[1,1]$.  We have $s\ge1$, $d_i \mid N_{z_n}$ for all $i$, and
\begin{equation} \label {pc0iqfVvz88qgc8piW9HFG}
\gcd(d_1,\dots,d_s)=1 \, ,
\end{equation}
the last claim because $z_n \in \Nd(\Teul)=\Nd^*(\Teul)$.
Let $q = q(\{z_n,z_{n-1}\}, z_n)$ and consider $\qdic(z_n)$ defined in  \ref{Cjkf0293wdipwos}.
See \ref{0db3ukifxoms7dcdOoOoxde} for the notation ``$A_{y_1}$'' and define
$$
S =
\begin{cases}
0 & \text{in case (i),} \\
\sum_{\alpha \in A_{y_1}} \hat x_{z_n,\alpha} & \text{in case (ii),}
\end{cases}
\qquad 
q_1 = \begin{cases}
1 & \text{in case (i),} \\
q(\{z_n,y_1\},z_n) & \text{in case (ii).}
\end{cases}
$$
Since $v_0=z_1$, we have $\qdic(z_n), q_1 \in \Nat\setminus \{0\}$ and (in case (ii)) $S \in \Nat\setminus \{0\}$;
moreover, $\sum_{\alpha \in A_{z_n}} \hat x_{z_n,\alpha} = \sum_{i=1}^s a_i d_i$ for some $a_1, \dots, a_n \in \Nat\setminus \{0\}$.
In view of \eqref{p0c94e8c092bcf3267jie}, we can write
\begin{equation}  \label {o09c2g38dg1gfo752dherC}
\textstyle      N_{z_n}  = q a_{z_n} \left( q_1 \sum_{i=1}^s a_{i} d_i + \qdic(z_n) S \right) \, .
\end{equation}
As $N_{z_n}>0$, we get $q>0$ and it also follows that
\begin{equation}  \label {c02938ve82f77sco20I}
\textstyle  \sum_{i=1}^s d_i \le N_{z_n} \, .
\end{equation}

If $s=1$ then \eqref{pc0iqfVvz88qgc8piW9HFG} implies $\tau=[1]$;
if $N_{z_n}=1$ then \eqref{c02938ve82f77sco20I} implies $\tau=[1]$; so, from now-on, we may assume that $s>1$ and $N_{z_n}>1$.
Since $s>1$, we have $d_i<N_{z_n}$ for all $i$, by \eqref{c02938ve82f77sco20I}.
By Thm \ref{c03hbr9dfgsjlxnzwesmchp}\eqref{ve4876ekrlu8ecBbue}, we may apply Lemma \ref{0c9gh5thowe6fric} to $u=z_n$.
Let $(\delta_1, \dots, \delta_m)$ be the sequence denoted ``$(d_1,\dots,d_m)$'' in Lemma \ref{0c9gh5thowe6fric}.
Since $N_{z_n}>1$, condition (a) of  Lemma \ref{0c9gh5thowe6fric} cannot hold; so one of the following must be true:
\begin{itemize}

\item[(b$'$)] $(\delta_1, \dots, \delta_m) = (1, 1, N_{z_n}, \dots, N_{z_n})$;

\item[(c$'$)] $m\ge3$ and there exists an odd integer $b \ge3$ such that  $N_{z_n}=2b$ and $(\delta_1, \dots, \delta_m) = (2, b, b, N_{z_n}, \dots, N_{z_n})$.

\end{itemize}
Suppose (c$'$) is true.
If we are in case (ii) then $c(z_n,\{z_n,y_1\})$ is a term of $(\delta_1, \dots, \delta_m)$
and $c(z_n,\{z_n,y_1\})=1$ by Lemma \ref{db8vcmxdJjvdrticu6tfvild}(c), contradicting (c$'$); so we must be in case (i).
Since $(d_1, \dots, d_s)$ is a subsequence of $(\delta_1,\dots,\delta_m)$ and $d_i<N_{z_n}$ for all $i$,
$(d_1, \dots, d_s)$ is a subsequence of $(2,b,b)$; we note that $(d_1, \dots, d_s) \neq (2,b,b)$ by \eqref{c02938ve82f77sco20I},
so \eqref{pc0iqfVvz88qgc8piW9HFG} implies that $(d_1, \dots, d_s) = (2,b)$.
It follows from \eqref{o09c2g38dg1gfo752dherC} (together with $S=0$ since we are in case (i)) that 
there exist integers $A, B\ge 1$ such that $N_{z_n} = A d_1 + B d_2$, so $2b = 2 A + b B$. This is impossible, so (c$'$) cannot hold.

So (b$'$) must hold. As $(d_1, \dots, d_s)$ is a subsequence of $(1, 1, N_{z_n}, \dots, N_{z_n})$ and $d_i<N_{z_n}$ for all $i$,
$(d_1, \dots, d_s)$ is either $(1)$ or $(1,1)$, i.e., the type of $z_n$ is $[1]$ or $[1,1]$.  This proves (b).

\medskip
Proof of (c).
By Thm \ref{Xcoikn23ifcdKJDluFYT937}, there exists a $\tD$-trivial path $(z_1, \dots, z_n)$ such that 
$\Neul = \{ z_1, \dots, z_n \}$, $\Omega(\Teul) = \{z_1,z_n\}$  and $\tD(z_i) = 0$ for all $i = 1, \dots,n$.
Prop.~\ref{hCfa89HFdfjkUY738ihhdJMJi854727} (together with $\Nd(\Teul)=\Nd^*(\Teul)$) implies that $\Nd(\Teul) \subseteq \{z_1,z_n\}$.
Let $i \in \{1,n\}$; if $z_i \notin \Nd(\Teul)$ then Lemma \ref{C0qowbdowX932epug7fnvd}(a) implies that $z_i=v_0$,
so Lemma \ref{90q932r8dhd89cnr9} gives $0 = \tD(z_i) = \tD(v_0) = \sigma(v_0) + (-1)(N_{v_0}-1) + N_{v_0}(1 - \frac1{a_{v_0}}) = 1-N_{v_0}$,
so $N_{v_0}=1$, which implies that $v_0$ is a node by Lemma \ref{90q932r8dhd89cnr9}. This contradiction shows that $z_1,z_n$ are nodes,
so $\Nd(\Teul) = \{z_1,z_n\}$.
If $1<i<n$ then $z_i \notin \Nd(\Teul)$, so Lemma \ref{C0qowbdowX932epug7fnvd}(b) implies that $z_i=v_0$;
this implies that $n \in \{2,3\}$ and that if $n=3$ then $z_2=v_0$.

Cor.~\ref{9vbq34ritgf7ZAOdb9rv}\eqref{o2i3brfowe9-a} implies that each $v \in \Nd(\Teul)$ has type $[d_1,\dots,d_s] = [1,N_v,\dots,N_v]$
and satisfies $a_v=1$; thus no dead end is attached to $v$, so $\delta_v = s+\epsilon(v) = s+1$.
If $s=1$ then $\delta_v=2$, so $v = v_0$ by Def.~\ref{e5e6e7w8e9wee9}(iv).
If $s\neq1$ then $\delta_v>2$, so $v \neq v_0$ by Thm~\ref{c03hbr9dfgsjlxnzwesmchp}\eqref{c8bGTghU783hceo}.
This completes the proof.
\end{proof}

%%%%%%%%%%%%%%%%%%%%%%%%%%%%%%%%%%%%%%%%%%%%%%%%%%%%%%%%%%%%%%%%%%%%%%%%%%%%%%%%%%%%%%%%%%%%%%%%%%%%%%%%%%%%
%%%%%%%%%%%%%%%%%%%%%%%%%%%%%%%%%%%%%%%%%%%%%%%%%%%%%%%%%%%%%%%%%%%%%%%%%%%%%%%%%%%%%%%%%%%%%%%%%%%%%%%%%%%%

\section{The cases $\tD( \Neul ) = 2$ and $\tD( \Neul ) = 4$}
\label {Section:Genus1case}

{\it We continue to assume that $\Teul$ is a minimally complete Newton tree at infinity.}

\medskip

We consider the cases where $\tD( \Neul ) \in \{ 2, 4 \}$.
Note that this implies that 
\begin{equation} \label {pvivy5hjdKo0xswwrJFkuxudgcg}
|\Omega(\Teul)| \le 1
\end{equation}
because Thm \ref{Xcoikn23ifcdKJDluFYT937} implies that if $|\Omega(\Teul)| > 1$ then $\tD( \Neul ) \le 0$.
As discussed in the Introduction, for any tree coming from a primitive polynomial the number  $\tD( \Neul )$ is even and nonnegative
since it is equal to twice the genus of the generic fiber.
So the case $\tD( \Neul ) = 3$ is not interesting for applications, which explains why it is not discussed in this section.

\begin{corollary} \label {p0c9vn3lw9nf9q3we}
Assume that $\tD(\Neul) = 2$, let $z \in \In(\Teul)$ and consider $\overline\Oeul = \overline\Oeul(\Teul,z)$.
\begin{enumerata}

\item $| \overline\Oeul | \in \{0,1,2,3\}$ and if $| \overline\Oeul | \ge 2$ then $| \Omega(\Teul) | = 1$.

\item If $| \overline\Oeul | =2$ (resp.\ $| \overline\Oeul | =3$)
 then the rooted tree $\widetilde\Oeul$ appears in row {\rm (a)} (resp.\ {\rm (b)}) of Figure~\ref{d9be3yie8dfdokj6rdosefd}.

\end{enumerata}
\end{corollary} 

\begin{proof}
Part (a) follows from Cor.\ \ref{pTo9v2q3hZYa9r1ge0cX3rg} and \eqref{pvivy5hjdKo0xswwrJFkuxudgcg}.
If $| \overline\Oeul | = 2$ then $\widetilde\Oeul$ must be the tree in row (a) of Fig.\ \ref{d9be3yie8dfdokj6rdosefd}.
Assume that $| \overline\Oeul | = 3$.
Then $\widetilde\Oeul$ must be one of the trees in rows (b) and (c) of Fig.\ \ref{d9be3yie8dfdokj6rdosefd}.
However it cannot be the one in row (c) because that one satisfies $B + 2(L-2) + | \overline\Oeul_2 | =3 > \tD(\Neul)$,
violating part (c) of Cor.\ \ref{p0c9ifn2o3w9dcpw0e}. 
\end{proof}

We shall now consider separately the four cases of Cor.\ \ref{p0c9vn3lw9nf9q3we} (i.e., the four values of $| \overline\Oeul |$).
If $| \overline\Oeul | = 1$ (resp.\ $2$, $3$) we write $\overline\Oeul = \{ C_0 \}$
(resp.\  $\overline\Oeul = \{ C_0, C_1 \}$, $\overline\Oeul = \{ C_0, C_1, C_2 \}$) where $C_0$ always has the meaning
defined in paragraph  \ref{c20w39w93e9d0dqMne89wcg9}. Also, we use the abbreviation $u_{i} = u_{C_i}$ for $i=0,1,2$.

Keep in mind that the equivalent conditions of Lemma \ref{A0c9vub23oW8fgbgp0q2arft} are satisfied whenever $\Omega(\Teul) \neq \emptyset$.

\begin{corollary} \label {cjerher5dszmmjcnnxdugytl} 
If $\tD(\Neul) = 2$ and $| \overline\Oeul | = 3$ then the following hold.
\begin{enumerata}

\item $| \Omega(\Teul) | = 1$,  $\delta^*(u_0) = 3$, $R(u_0, \{ e_{u_0} \})=0$ and $t(u_0)=0$.

\item $t(u_1), t(u_2) \in \{0,1,2\}$ and $\tD( \bar V( u_{1} ) ) = 1 = \tD( \bar V( u_{2} ) )$.

\item For each $i \in \{1,2\}$, if $\gamma_{u_0,u_i} = (u_{i,0}, \dots, u_{i,n_i})$ then for all $j$ such that $0<j<n_i$ we have
$\epsilon(u_{i,j})=2$ and $\tD(u_{i,j})=0$.

\end{enumerata}
\end{corollary} 

\begin{proof}
By Cor.\ \ref{p0c9vn3lw9nf9q3we},
we have $| \Omega(\Teul) | = 1$  and the rooted tree $\widetilde\Oeul$ appears in row {\rm (b)} of Figure~\ref{d9be3yie8dfdokj6rdosefd}.
Thus $\delta^*(u_0) = 3$.  Since  $B + 2(L-2) + | \overline\Oeul_2 | = 2 = \tD(\Neul)$,
Cor.\ \ref{p0c9ifn2o3w9dcpw0e} implies that $T=0$, $x_0=0$, $\dot c(C_1) = 0 = \dot c(C_2)$ and $x_{C_1}=0=x_{C_2}$.
The condition $T=0$ gives $t(u_1), t(u_2) \in \{0,1,2\}$, and  (by Cor.\ \ref{cjvl45otbsAZofg9bgf}) $\dot c(C_1) = 0 = \dot c(C_2)$ gives assertion (c).
Since $x_0 = 0$, we have $\tD\big( \bar V( u_0 ) \cup \Neul( u_0, e_{u_0} ) \big) = \delta^*(u_0) - 3$,
so Cor.\ \ref{0c9in34dh5sifzsshnkus}(c) implies that $R(u_0,\{e_{u_0}\}) = 0$ and $t(u_0)=0$.
For each $i=1,2$ we have $t(u_i)\le2$, so $\epsilon(u_i) \le 3$; since $x_{C_i}=0$,
it follows that $\tD( \bar V( u_{i} ) ) = \max(1,\epsilon(u_i)-2) = 1$.
\end{proof}

\begin{corollary} \label {pc0qcuh6a8reustb8ws7}
If $\tD(\Neul) = 2$ and $| \overline\Oeul | = 2$ then the following hold.
\begin{enumerata}

\item $| \Omega(\Teul) | = 1$, $\delta^*(u_0) = 2$, $t(u_0) \in \{0,1,2\}$ and $R(u_0,A)=1$,
where we define $A = \{e_{u_0}\} \cup \setspec{ e \in \Eeul_{u_0} }{ \text{$(u_0,e)$ is a tooth} }$.

\item $t(u_1) \in \{0,1,2\}$ and  $\tD( \bar V( u_{1} ) ) = 1$

\item If $\gamma_{u_0,u_1} = (u_{1,0}, \dots, u_{1,n_1})$ then for all $j$ such that $0<j<n_1$ we have
$\epsilon(u_{1,j})=2$ and $\tD(u_{1,j})=0$.

\end{enumerata}
\end{corollary} 

\begin{proof}
Cor.\ \ref{p0c9vn3lw9nf9q3we} implies that $| \Omega(\Teul) | = 1$ and that $\widetilde\Oeul$ is in row (a) of Figure~\ref{d9be3yie8dfdokj6rdosefd}.
Thus $\delta^*(u_0) = 2$. Since  $B + 2(L-2) + | \overline\Oeul_2 | = 2 = \tD(\Neul)$,
Cor.\ \ref{p0c9ifn2o3w9dcpw0e} implies that  $T=x_0=x_{C_1}=\dot c(C_1) = 0$.
The condition $T=0$ gives $t(u_1) \in \{0,1,2\}$ and  $\dot c(C_1) = 0$ gives assertion (c).
The condition $x_0=0$ implies that $\tD\big( \bar V( u_0 ) \cup \Neul( u_0, e_{u_0} ) \big) = | \delta^*(u_0)-3 | = 1$,
so $ R(u_0,A)  =   1$ follows from Lemma \ref{A0c9vub23oW8fgbgp0q2arft}(ii),
and  $t(u_0) \in \{0,1,2\}$ follows from $ R(u_0,A)  =   1$.
We have $t(u_1)\le2$, so $\epsilon(u_1) \le 3$; since $x_{C_1}=0$,
we get $\tD( \bar V( u_{1} ) ) = \max(1,\epsilon(u_1)-2) = 1$.
\end{proof}

\begin{corollary} 
If $\tD(\Neul) = 2$ and $| \overline\Oeul | = 1$ then $\delta^*(u_0) = 1$ and 
$$
| \setspec{ x \in \Deul_{u_0} }{ k_x>1 } | + a_{u_0}^* + t({u_0}) \le 4.
$$
\end{corollary} 

\begin{proof}
$\delta^*(u_0) = 1$ is clear and the last assertion is Cor.\ \ref{Xxkvp0wedifcwZepd}.
\end{proof}

\begin{parag} 
If $\tD(\Neul) = 2$ and $| \overline\Oeul | = 0$ then $|S(\Teul)| = 1$, so Cor.\ \ref{pc0wbyrjo79e8rnn9} applies to this case.
\end{parag} 

\medskip
We shall now say a few words about the case $\tD(\Neul) = 4$.

\begin{corollary} \label {09df2bf09je0dFiHq3hr}
Assume that $\tD(\Neul) = 4$, let $z \in \In(\Teul)$ and consider $\overline\Oeul = \overline\Oeul(\Teul,z)$.
\begin{enumerata}

\item $| \overline\Oeul | \in \{0,1,2,3,4,5\}$ and if $| \overline\Oeul | \ge 4$ then $| \Omega(\Teul) | = 1$. 

\item If $| \overline\Oeul | >1$ then 
 $\widetilde\Oeul$ appears in one of the rows {\rm (a--g, j--l)} of Figure~\ref{d9be3yie8dfdokj6rdosefd}.

\end{enumerata}
\end{corollary} 

\begin{proof}
Part (a) follows from Cor.\ \ref{pTo9v2q3hZYa9r1ge0cX3rg} and \eqref{pvivy5hjdKo0xswwrJFkuxudgcg}.
Assume that $| \overline\Oeul | >1$; then $\widetilde\Oeul$ is a rooted tree whose number $| \overline\Oeul |$ of vertices 
satisfies $2 \le | \overline\Oeul | \le 5$;
as all such rooted trees are listed in Fig.\ \ref{d9be3yie8dfdokj6rdosefd}, $\widetilde\Oeul$ appears in that Figure.
By Cor.\ \ref{p0c9ifn2o3w9dcpw0e}(c),  $\widetilde\Oeul$ must satisfy the condition $B + 2(L-2) + | \overline\Oeul_2 | \le \tD(\Neul) = 4$,
so  $\widetilde\Oeul$ is in one of the ten rows (a--g), (j--l).
\end{proof}

In six of the ten cases of Cor.\ \ref{09df2bf09je0dFiHq3hr}(b) we have  $B + 2(L-2) + | \overline\Oeul_2 | = 4$, so
\begin{equation} \label {ciujbv3948utdydrhgo}
\textstyle
T + x_0 +  \sum_{ C \in \overline\Oeul \setminus \{C_0\} }  (\dot c(C) + x_C) \ = \ 0
\end{equation}
by Cor.\ \ref{p0c9ifn2o3w9dcpw0e}(a)
and consequently  $T =0= x_0$ and $\dot c(C) =0 = x_C$ for all $C \in \overline\Oeul \setminus \{C_0\}$.
As we saw in the proofs of Corollaries \ref{cjerher5dszmmjcnnxdugytl} and \ref{pc0qcuh6a8reustb8ws7},
consequences can be deduced from that.
In the remaining four cases the sum \eqref{ciujbv3948utdydrhgo} is equal to $1$ or $2$, which leads to a large number of cases.

We shall not further elaborate the case $\tD(\Neul) = 4$.

%%%%%%%%%%%%%%%%%%%%%%%%%%%%%%%%%%%%%%%%%%%%%%%%%%%%%%%%%%%%%%%%
%%%%%%%%%%%%%%%%%%%%%%%%%%%%%%%%%%%%%%%%%%%%%%%%%%%%%%%%%%%%%%%%
\bibliographystyle{alpha}

\end{document}

%% file: descadaNT1.pdf_t
\begin{picture}(0,0)%
\includegraphics{descadaNT1.pdf}%
\end{picture}%
\setlength{\unitlength}{2763sp}%
\begingroup\makeatletter\ifx\SetFigFont\undefined%
\gdef\SetFigFont#1#2#3#4#5{%
  \reset@font\fontsize{#1}{#2pt}%
  \fontfamily{#3}\fontseries{#4}\fontshape{#5}%
  \selectfont}%
\fi\endgroup%
\begin{picture}(12353,7333)(61,-8459)
\put(8626,-1486){\makebox(0,0)[lb]{\smash{{\SetFigFont{8}{9.6}{\familydefault}{\mddefault}{\updefault}{\color[rgb]{0,0,0}$0$}%
}}}}
\put(5026,-1486){\makebox(0,0)[lb]{\smash{{\SetFigFont{8}{9.6}{\familydefault}{\mddefault}{\updefault}{\color[rgb]{0,0,0}$-2$}%
}}}}
\put(10486,-1486){\makebox(0,0)[lb]{\smash{{\SetFigFont{8}{9.6}{\familydefault}{\mddefault}{\updefault}{\color[rgb]{0,0,0}$-1$}%
}}}}
\put(8551,-3211){\makebox(0,0)[lb]{\smash{{\SetFigFont{8}{9.6}{\familydefault}{\mddefault}{\updefault}{\color[rgb]{0,0,0}$-1$}%
}}}}
\put(6826,-3286){\makebox(0,0)[lb]{\smash{{\SetFigFont{8}{9.6}{\familydefault}{\mddefault}{\updefault}{\color[rgb]{0,0,0}$-8$}%
}}}}
\put(5026,-3286){\makebox(0,0)[lb]{\smash{{\SetFigFont{8}{9.6}{\familydefault}{\mddefault}{\updefault}{\color[rgb]{0,0,0}$-74$}%
}}}}
\put(1426,-3286){\makebox(0,0)[lb]{\smash{{\SetFigFont{8}{9.6}{\familydefault}{\mddefault}{\updefault}{\color[rgb]{0,0,0}$-891$}%
}}}}
\put(6700,-3680){\makebox(0,0)[lb]{\smash{{\SetFigFont{8}{9.6}{\familydefault}{\mddefault}{\updefault}{\color[rgb]{0,0,0}$3$}%
% \put(6751,-3586){\makebox(0,0)[lb]{\smash{{\SetFigFont{8}{9.6}{\familydefault}{\mddefault}{\updefault}{\color[rgb]{0,0,0}$3$}%
}}}}
\put(4900,-3680){\makebox(0,0)[lb]{\smash{{\SetFigFont{8}{9.6}{\familydefault}{\mddefault}{\updefault}{\color[rgb]{0,0,0}$3$}%
}}}}
\put(3125,-3680){\makebox(0,0)[lb]{\smash{{\SetFigFont{8}{9.6}{\familydefault}{\mddefault}{\updefault}{\color[rgb]{0,0,0}$2$}%
}}}}
\put( 16,-3420){\makebox(0,0)[lb]{\smash{{\SetFigFont{8}{9.6}{\familydefault}{\mddefault}{\updefault}{\color[rgb]{0,0,0}$2$}%
}}}}
\put(8476,-4936){\makebox(0,0)[lb]{\smash{{\SetFigFont{8}{9.6}{\familydefault}{\mddefault}{\updefault}{\color[rgb]{0,0,0}$-3$}%
}}}}
\put(6901,-6736){\makebox(0,0)[lb]{\smash{{\SetFigFont{8}{9.6}{\familydefault}{\mddefault}{\updefault}{\color[rgb]{0,0,0}$-4$}%
}}}}
\put(8090,-6736){\makebox(0,0)[lb]{\smash{{\SetFigFont{8}{9.6}{\familydefault}{\mddefault}{\updefault}{\color[rgb]{0,0,0}$-4$}%
}}}}
\put(9601,-6736){\makebox(0,0)[lb]{\smash{{\SetFigFont{8}{9.6}{\familydefault}{\mddefault}{\updefault}{\color[rgb]{0,0,0}$-4$}%
}}}}
\put(9451,-1400){\makebox(0,0)[lb]{\smash{{\SetFigFont{8}{9.6}{\familydefault}{\mddefault}{\updefault}{\color[rgb]{0,0,0}\small $v_0$}%
}}}}
\put(9301,-1336){\makebox(0,0)[lb]{\smash{{\SetFigFont{8}{9.6}{\familydefault}{\mddefault}{\updefault}{\color[rgb]{0,0,0}}%
}}}}
\put(9826,-1336){\makebox(0,0)[lb]{\smash{{\SetFigFont{8}{9.6}{\familydefault}{\mddefault}{\updefault}{\color[rgb]{0,0,0}}%
}}}}
%%%%%%%%%%%%%%%%%%%%%%%%%%%%%%%%%%%%%%%%%%%%%%%%%%%%%%%%%%%%%%%%%%%%%%%%%%%%%%%%%%%%%%%%%%%%%%%%%%%%%%%%%%%%%%%%%%%
\put(8150,-1400){\makebox(0,0)[lb]{\smash{{\SetFigFont{8}{9.6}{\familydefault}{\mddefault}{\updefault}{\color[rgb]{0,0,0}\small $v_1$}}}}}
\put(8100,-3200){\makebox(0,0)[lb]{\smash{{\SetFigFont{8}{9.6}{\familydefault}{\mddefault}{\updefault}{\color[rgb]{0,0,0}\small $v_2$}}}}}
\put(6400,-3200){\makebox(0,0)[lb]{\smash{{\SetFigFont{8}{9.6}{\familydefault}{\mddefault}{\updefault}{\color[rgb]{0,0,0}\small $v_3$}}}}}
\put(4600,-3200){\makebox(0,0)[lb]{\smash{{\SetFigFont{8}{9.6}{\familydefault}{\mddefault}{\updefault}{\color[rgb]{0,0,0}\small $v_4$}}}}}
\put(2800,-3200){\makebox(0,0)[lb]{\smash{{\SetFigFont{8}{9.6}{\familydefault}{\mddefault}{\updefault}{\color[rgb]{0,0,0}\small $v_5$}}}}}
\put(8050,-5200){\makebox(0,0)[lb]{\smash{{\SetFigFont{8}{9.6}{\familydefault}{\mddefault}{\updefault}{\color[rgb]{0,0,0}\small $v_6$}}}}}
\end{picture}%